%% file: main.tex
\documentclass{amsart}

\usepackage{amssymb}
\usepackage{comment}
\usepackage{mathrsfs}
\usepackage[stretch=10,shrink=10]{microtype}
\usepackage[showframe=false, margin = 1.5 in]{geometry}
\usepackage{mathtools}
\usepackage[shortlabels]{enumitem}
\usepackage{ifthen}
\newcounter{dummy}
\usepackage{xcolor}
\usepackage{stmaryrd}
\usepackage{tikz,tikz-cd}
\usepackage{array}
\usepackage{color}
\usepackage{colortbl}
\usepackage{hyperref}
\hypersetup{
     colorlinks   = true,
     citecolor    = black
}
\usepackage{theoremref}

\makeatletter
\def\thmref@flush{%
   \ifx\thmref@last\empty\else
      \ifthmref@comma, \thmref@finaltrue\fi \thmref@commatrue
      \thmref@last \ifx\thmref@stack\empty\else s\fi \thmref@num 0
      \let\do\thmref@one \thmref@stack
      \ifcase\thmref@num\or\space and\else\thmref@finaltrue, and\fi
      ~\ref{\thmref@head}\let\thmref@stack\empty\fi}
\def\thmref@one#1{\ifnum\thmref@num>0,\fi
   \space\ref{#1}\advance\thmref@num 1\relax}
\makeatother

\makeatletter
\newcommand\myitem[1][]{\item[#1]\refstepcounter{dummy}\def\@currentlabel{#1}}
\makeatother
\newcommand*{\mprime}{\ensuremath{'}}

\renewcommand{\emptyset}{\varnothing}
\newcommand{\E}{\mathbf{E}}

\renewcommand{\P}{\mathbf{P}}
\renewcommand{\Pr}{\P}
\newcommand{\1}{\mathbf{1}}

\newcommand{\f}{\frac}
\newcommand{\SOC}{self-organized criticality }
\newcommand{\SOCperiod}{self-organized criticality. }

\newcommand{\sgn}{\mathrm{sgn}}
\newcommand{\Sig}{\mathsf{Sig}}
\newcommand{\too}{\tikz[baseline=-0.5ex,shorten <=2pt,shorten >=2pt] \draw[-cm double to] (0,0) to (1.3em,0); }

\DeclareMathOperator{\Ber}{Bernoulli}
\DeclareMathOperator{\Bin}{Bin}
\DeclareMathOperator{\Geo}{Geo}

\DeclarePairedDelimiter\abs{\lvert}{\rvert}%
\DeclarePairedDelimiter\floor{\lfloor}{\rfloor}%
\DeclarePairedDelimiter\ceil{\lceil}{\rceil}%
\DeclarePairedDelimiter\ii{\llbracket}{\rrbracket}%
\newcommand{\bigmid}{\;\big\vert\;}
\newcommand{\Bigmid}{\:\Big\vert\:}

\newcommand{\eqd}{\overset{d}=}
\newcommand{\topr}{\xrightarrow[]{\text{prob}}}
\newcommand{\layer}{\mathrm{layer}}

\newcommand{\RR}{\mathbb{R}}

\newcommand{\erase}{\mathsf{GreedyPath}}

\newcommand{\badbox}{\mathsf{BadBox}}
\newcommand{\badboxdelta}{\theta\text{-}\mathsf{BadBox}}
\newcommand{\infectbox}{\theta\text{-}\mathsf{InfectedBox}}
\newcommand{\badcell}{\mathsf{BadCell}}

\newcommand{\ZZ}{\mathbb{Z}}

\newcommand{\Fscr}{\mathscr{F}}

\newcommand{\Blb}{\mathcal{R}}

\newcommand{\NN}{\mathbb{N}}

\newcommand{\instr}{\mathsf{instr}}
\newcommand{\Instr}{\mathsf{Instr}}
\newcommand{\eosos}[1]{\mathcal{O}_{#1}}
\newcommand{\ip}[1]{\mathcal{I}_{#1}}
\newcommand{\s}{\mathfrak{s}}

\newcommand{\critical}{\rho}
\newcommand{\lpcritical}{\rho_*}

\newcommand{\contained}{\mathsf{Contained}}

\newcommand{\widebar}[1]{\overline{#1}}

\newcommand{\ymin}{y^{\min}}
\newcommand{\xmin}{x^{\min}}
\newcommand{\smax}{s^{\max}}
\newcommand{\smin}{s^{\min}}
\newcommand{\sminp}{\tilde{s}^{\min}}
\newcommand{\rt}[2]{\mathcal{R}_{#2}(#1)}
\newcommand{\lt}[2]{\mathcal{L}_{#2}(#1)}
\newcommand{\emax}{e_{\max}}
\newcommand{\infections}{\mathsf{Infections}}
\newcommand{\opositions}{\mathsf{Positions}}
\newcommand{\stable}{\mathsf{Stable}}

\newcommand{\bx}[4][k]{\mathcal{B}^{#2}_{#1}(#3,#4)}
\DeclareFontEncoding{LS1}{}{}
\DeclareFontSubstitution{LS1}{stix}{m}{n}
\DeclareSymbolFont{stixletters}{LS1}{stix}{m}{it}
\DeclareMathAccent{\cev}{\mathord}{stixletters}{"91}
\DeclareMathAccent{\vec}{\mathord}{stixletters}{"92}
\DeclareMathAccent{\vecev}{\mathord}{stixletters}{"95}
\newcommand{\greedy}[1][k]{\text{$k$-\erase}}
\makeatletter

\newcommand{\crist}[1][\@nil]{%
\def\tmp{#1}%
   \ifx\tmp\@nnil
       \lpcritical
    \else
       \lpcritical^{(#1)}
    \fi}
\newcommand{\infset}[1][\@nil]{
  \def\tmp{#1}%
   \ifx\tmp\@nnil
       \zeta
    \else
       \zeta^{#1}
    \fi}

\newcommand{\altinfset}[1][\@nil]{
  \def\tmp{#1}%
   \ifx\tmp\@nnil
       \widetilde{\zeta}
    \else
       \widetilde{\zeta}^{#1}
    \fi}
\newcommand{\binfset}[1][\@nil]{
  \def\tmp{#1}%
   \ifx\tmp\@nnil
       \cev{\zeta}
    \else
       \cev{\zeta}^{#1}
    \fi}
\makeatother

\newcommand{\Left}{\ensuremath{\mathtt{left}}}
\newcommand{\Right}{\ensuremath{\mathtt{right}}}
\newcommand{\Sleep}{\ensuremath{\mathtt{sleep}}}

\newtheorem{thm}{Theorem}[section]
\newtheorem{lemma}[thm]{Lemma}
\newtheorem{prop}[thm]{Proposition}
\newtheorem{cor}[thm]{Corollary}

\newtheorem*{conjecture*}{Density Conjecture}

\theoremstyle{remark}
\newtheorem{remark}[thm]{Remark}
\theoremstyle{definition}
\newtheorem{define}[thm]{Definition}
\newtheorem{coupling}[thm]{Coupling}
\newtheorem{example}[thm]{Example}

\definecolor{hancolor}{rgb}{0.0 0.0, 1.0}
\newcommand{\critFE}{\critical_{\mathtt{FE}}}
\newcommand{\critDD}{\critical_{\mathtt{DD}}}
\newcommand{\critCY}{\critical_{\mathtt{CY}}}
\newcommand{\critPS}{\critical_{\mathtt{PS}}}
\newcommand{\ManyParticles}{\mathsf{ManyParticles}}
\newcommand{\FewSleepers}{\mathsf{FewSleepers}}

\newcommand{\mo}{\mathfrak{m}}
\begin{document}

 \title{The density conjecture for activated random walk}  

 \author{Christopher Hoffman}
 \address{Christopher Hoffman, Department of Mathematics, University of Washington}
 \email{\texttt{choffman@uw.edu}}
 \author{Tobias Johnson}
 \address{Tobias Johnson, Departments of Mathematics,  College of Staten Island, City University of New York}
\email{\texttt{tobias.johnson@csi.cuny.edu}}

 \author{Matthew Junge}
 	\address{Matthew Junge, Department of Mathematics, Baruch College, City University of New York}
	\email{\texttt{matthew.junge@baruch.cuny.edu}}

\begin{abstract}
  Bak, Tang, and Wiesenfeld developed their theory of \emph{self-organized criticality}
  in the late 1980s to explain why many real-life processes exhibit signs of critical behavior 
  despite the absence of a tuning parameter. 
  A decade later,
  Dickman, Mu\~noz, Vespignani, and Zapperi explained self-organized criticality as
  an external force pushing
  a hidden parameter toward the critical value of a traditional absorbing-state phase transition.
  As evidence, they observed empirically that for various sandpile models,
  the particle density in a finite box under driven-dissipative dynamics
  converges to the critical density of an infinite-volume version of the model.
  We give the first proof of this well-known \emph{density conjecture} in any setting by establishing it
  for activated random walk in one dimension. We prove that two other natural versions
  of the model have the same critical value, further establishing activated random
  walk as a universal model of self-organized criticality.
\end{abstract}
  \maketitle

 \setcounter{tocdepth}{2}
 \tableofcontents

\section{Introduction}

     Many complex systems, such as tectonic plates, snow slopes, and financial markets, seem to be driven to a critical state where minor disturbances can lead to major events---earthquakes, avalanches, and financial crashes. In 1987, Bak, Tang and Wiesenfeld proposed a general explanation for this phenomenon which they called \emph{self-organized criticality} 
     \cite{BakTangWiesenfeld87}. They theorized that the critical state arises without any external tuning, but rather from the steady accumulation and occasional dissipation of energy. Their work was paradigm-shifting and ranked among the most cited papers in physics over the ensuing decade \cite{bak2013nature}.

    The canonical physical example of \SOC is a pile formed by sprinkling sand over a flat table \cite{BakTangWiesenfeld88}. The sandpile grows until it reaches a critical slope and can grow no steeper. The in- and outflow of sand roughly balances in the long run, but sometimes small local additions cause large avalanches. Shortly after \cite{BakTangWiesenfeld87} was published, several experiments and observational studies in nature were conducted and generally found to support the theory of \SOC \cite{held1990experimental, czirok1993experimental, jaeger1989relaxation, knight1995density}.

    Systems exhibiting \SOC are ubiquitous and share many common characteristics. Thus the key features of these systems should not depend on any specific detail of the model, but should be more universal in nature.    In his book \emph{How Nature Works},  Bak wrote on this principle, ``The critical state must be robust with respect to modifications. This is of crucial importance for the concept of self-organized criticality to have any chance of describing the real world. In fact, this is the whole idea."  Bak goes on to describe the search for the simplest model of \SOC \cite{bak2013nature}.
    Though several processes have been proposed as universal models of self-organized criticality
    based on simulations, rigorous proofs of their properties remain elusive.

The proposed models of self-organized criticality typically feature particles on a graph 
that disperse when their density in a region becomes too large. 
Bak, Tang and Wiesenfeld proposed
the \emph{abelian sandpile model} 
which has a deterministic dispersion mechanism \cite{BakTangWiesenfeld87}.
Mathematicians from analysts to combinatorialists to algebraists have explored
the model's rich mathematical structure; one highlight is the work of Levine,
Pegden, and Smart analyzing the model's scaling limit 
  as the viscosity solution of a PDE \cite{PegdenSmart,levine2016apollonian, levine2017apollonian}.
     Unfortunately this rich structure 
     prevents the model from exhibiting the universality that a model of \SOC should have,
     rendering it highly sensitive to the initial configuration and other details of the model
     \cite{fey2010approach, dhar1999abelian, jo2010comment}. 

     Manna proposed a probabilistic variant of the abelian sandpile model that was later named the \emph{stochastic sandpile model} 
     \cite{manna1991two, dickman2002nonequilibrium}. This model has seen some rigorous analysis,
     but focus has mostly shifted to a close variant called
     \emph{activated random walk} (ARW). This is widely believed to be the most promising tractable mathematical model exhibiting \SOCperiod Physicists believe that all such sandpile models with stochastic dynamics
     belong to the \emph{Manna universality class} and have
     the same core behaviors as ARW \cite{lubeck2004universal}.

     ARW can be formulated as an interacting particle system on a graph with active and sleeping  particles. Sleeping particles remain in place, but become active if an active particle moves to their site. Active particles perform simple random walk at exponential rate 1. When an active particle is alone, it falls asleep at exponential rate $\lambda \in (0,\infty)$. A precise description is given in Section \ref{sec:construction}. The cases $\lambda=0$ and $\lambda = \infty$ are known as the frog model and IDLA, respectively \cite{hoffman2017recurrence, jerison2012logarithmic}.

The surveys \cite{dickman2010activated, rolla2020activated, levine2023universality}
outlined many fundamental questions about the model, most of which remain unsolved.
As described in all three surveys, one of the most central open problems in ARW is
known as the \emph{density conjecture}. It claims that finite-volume versions of ARW that
resemble the canonical sandpile on a table described earlier have a limiting critical density;
furthermore, this critical density is universal and coincides with the critical density
of models of ARW with conservative dynamics that do not exhibit self-organized criticality but
rather have a traditional phase transition.
This conjecture is a crucial step along the path to showing that ARW exhibits the universality that Bak called ``the whole idea" of self-organized criticality. 
In this paper, we prove the density conjecture in full for ARW in one dimension,
the first rigorous confirmation that ARW actually organizes itself into a critical state.
Our proof develops new tools that have the potential to solve other major
problems about ARW and related models.
    
\subsection{The density conjecture}

    Dickman, Mu\~{n}oz, Vespignani, and Zapperi    
    gave an explanation for self-organized criticality, expounded in \cite{dickman2000paths}.
    According to their theory, when a model exhibits self-organized criticality, it will
    have a variant with a conventional phase transition in some explicit parameter. 
    The original model has a hidden version of the parameter, and
    self-organized criticality occurs because it is driven by some external force to the critical
    value for the conventional phase transition. This theory is now widely accepted by physicists.
    
    In the context of ARW,
    the conventional phase transition occurs in the \emph{fixed-energy} version
    of the system, which takes place on a torus or infinite lattice with no boundary.
    Its particle dynamics are conservative, meaning that particles are not created or destroyed. 
    The particle density is an explicit parameter of the initial distribution and remains
    constant in time. The model undergoing self-organized criticality is the \emph{driven-dissipative}
    version of ARW. Here, the process runs on a finite graph with particles killed at the boundary. Eventually, it reaches an absorbing state $\xi_1$, a configuration in which all particles are sleeping. Then a new particle is added and the system runs again until it reaches its next absorbing state $\xi_2$, and so on. The sequence of absorbing states $\xi_1,\xi_2,\hdots$ is called the driven-dissipative Markov chain; its state space consists of stable configurations. This chain is mixing on any given finite graph, and so it has a unique stationary distribution that is the distributional limit of $\xi_n$.    

Dickman et al.\ made several claims in support of their theory, originally
for the stochastic sandpile model but in subsequent publications for ARW as well. 
First, they claimed that the fixed-energy model has a traditional absorbing-state phase transition parametrized by the density of the system: for small density, all particles eventually sleep forever, while for large density activity continues at all sites forever. Both phases are nontrivial, i.e., the critical density $\critFE$ is strictly between 0 and 1 on a $d$-dimensional lattice. Next, Dickman et al.\ claimed that the dynamics of the driven-dissipative
model push its density to $\critFE$, in the following sense:

\begin{conjecture*}[\cite{dickman2000paths}] The mean density of the invariant distribution
  for the driven-dissipative system on a $d$-dimensional box of width $n$ converges as $n \to \infty$ 
  to some value $\critDD$, which coincides with 
  $\critFE$.
\end{conjecture*}

Dickman et al.\ supported these claims with numerical evidence from simulations.
Mathematicians were inspired to study activated random walk, and starting in the 2010s significant progress
was made on the model, primarily on the conjecture
that $0 < \critFE < 1$ for the fixed-energy model on infinite lattices.
This line of research has been a great success, as we will discuss
in Section~\ref{sec:prior.results}, with both upper and lower bounds recently established in all dimensions.
But this work has not translated into progress toward the density conjecture.
One obstacle is that few results have applied to the driven-dissipative model;
another is that most of the fixed-energy results prove fixation or nonfixation far from the critical
density.
The density conjecture has remained a distant goal, 
despite being the main basis for the belief in ARW as a universal model of
self-organized criticality.


\subsection{Main result}
\label{sec:results}
We prove the density conjecture in one dimension, showing that the driven-dissipative model naturally
drives itself to the critical density of the fixed-energy model on the line.
We further show that the fixed-energy model on the cycle has a phase transition at this same critical
density, and that a large number of particles originating at a single site naturally spread out
until they reach this critical density. Our proof that these critical densities coincide
for all four models provides strong evidence for activated random walk as a universal
model for self-organized criticality, and it begins the investigation into its critical state.

The four models of ARW we consider are defined as follows. All the models
and all the critical densities depend on the sleep rate parameter $\lambda>0$. 
\begin{description}
    \item[Driven-dissipative model] The ARW dynamics take place on $\ii{0,n} :=\{0,\dots,n\}$ with sinks at $-1$
      and $n+1$. At each step of the chain one active particle
      is added and then the system evolves until only sleeping particles remain in $\ii{0,n}$.
      This produces a Markov chain whose state space consists of configurations of sleeping
      particles on $\ii{0,n}$. The active particle may be added uniformly at random or at a fixed site,
      with no effect on the stationary distribution of the chain \cite{levine2021exact}.  
    We let $S_{n}$ be the number of sleeping particles left in $\ii{0,n}$ in a sample from the stationary distribution.  Define the limiting expected density of sleeping particles
    $$\critDD:= \lim_{n \to \infty}\f{\E[S_{n}] }{n+1},$$ should it exist.

    \item[Point-source model]
    The process starts with $N$ active particles at the origin of $\mathbb Z$ and no particles elsewhere. It continues until all $N$ particles are sleeping. 
    Let $L_N$ be length of the shortest interval containing all the sleeping particles after stabilization. 
    Define the limiting expected density of sleeping particles
    $$\critPS:=\lim_{N \to \infty}\E\biggl[ \f N {L_N}\biggr],$$
    should it exist.
    
    \item[Fixed-energy model on $\ZZ$] This version takes place on $\mathbb Z$ with the initial number of active particles at each site chosen according to some stationary distribution with mean $\rho$. The system \emph{fixates} if, 
      under the ARW dynamics, each particle eventually sleeps forever.
    Rolla, Sidoravicius, and Zindy proved that there exists a critical value $\critFE$ such that if $\rho > \critFE$ then the system a.s.\ does not fixate, and if $\rho < \critFE$ then the system fixates a.s.\ \cite{rolla2019universality}.
    \item[Fixed-energy model on the cycle] The underlying graph in this model is the cycle of length~$n$. Start with $\lfloor \rho n\rfloor$ particles placed uniformly at random and run the process until all particles are sleeping. Let $\tau_n$ be the total number of jump and sleep instructions followed by all of the particles.    Should it exist, 
      $\critCY$ is the density such that
      that for some $b,c>0$, we have  $\P(\tau_n > e^{c n}) \to 1$ if $\rho>\critCY$
      and $\P(\tau_n<n^b)\to 1$ if $\rho< \critCY$.
\end{description}

Prior to this work only $\critFE$ was even known to exist.
The density conjecture is the statement that $\critDD$ exists
and is equal to $\critFE$. Physicists sometimes view $\ZZ$
as the limiting case of the cycle and instead state the density conjecture
as the existence and equality of $\critDD$ and $\critCY$, as in \cite{DickmanVespignaniZapperi98}.
In any event, we settle the conjecture in dimension one.

\begin{thm}\thlabel{thm:universal}
For each $\lambda>0$, the critical densities
$\critDD$, $\critPS$, $\critFE$, and $\critCY$ exist and are equal.
\end{thm}

We in fact prove stronger versions of this result, including exponential concentration bounds around
the critical value for the finite-volume models and statements about the driven-dissipative model
from arbitrary starting configurations; see Sections~\ref{sec:peanut.butter} and
\ref{sec:final.density.proofs}.
We further note that the celebrated bounds $\critFE>0$ and $\critFE<1$ (for small $\lambda$), first proven
in \cite{rolla2012absorbing} and \cite{BasuGangulyHoffman18}, are nearly trivial consequences
of our approach, as we describe in Section~\ref{sec:critical.bounds}.

Key to our analysis is a new stochastic process we develop called \emph{layer percolation}.
ARW and other abelian models are typically studied via their odometer functions, which count the number
of times a particle moves or falls asleep at each site.
The odometer function describing the behavior of the system is the minimal element of a wider class
of feasible odometer functions. Our innovation is to represent these odometer functions
as paths  of infections in a $(2+1)$-dimensional oriented infection process.
The growth rate of the infection path along one of the dimensions corresponds to the density of
particles left sleeping by the odometer function. The set of infected sites
in layer percolation has a limiting growth rate $\crist$ along this dimension, and
we prove \thref{thm:universal} by showing that each of the four critical densities is equal
to $\crist$. 
We believe that our approach has the potential to solve many of the remaining problems about ARW 
described in \cite{dickman2010activated, rolla2020activated, levine2023universality}.

\subsection{Prior results}
\label{sec:prior.results}
The bulk of mathematical research on ARW has been on the fixed-energy model
on infinite lattices, and more specifically on showing $0<\critFE<1$ in all dimensions. 
The first major result on ARW was Rolla and Sidoravicius's proof of the lower bound for $d=1$ \cite{rolla2012absorbing}. The upper bound for $d=1$ proved more difficult and was established first only for
sufficiently small values of the sleep parameter $\lambda$ 
\cite{BasuGangulyHoffman18} and then later for all choices of $\lambda$ \cite{HoffmanRicheyRolla20}. 
Both sides of the bound are now known for $d\geq 2$
by the collective efforts of
\cite{SidoraviciusTeixeira17,StaufferTaggi18,hu2022active,forien2022active,asselah2022critical}. Additionally, \cite{rolla2019universality} proved that $\critFE$ is the same for fixed-energy ARW on $\ZZ^d$ with any ergodic initial configuration of active particles, bolstering ARW's candidacy as a model of self-organized criticality.

The other models of ARW considered in this paper have been studied as well, though the literature on them
is smaller.
The fixed-energy model on a cycle of length~$n$ or a torus of width~$n$
has been shown to have two phases,
with fixation occurring either in polynomial or exponential time
\cite{BasuGangulyHoffmanRichey19, forien2022active,asselah2022critical}, but with
no proof that the transition is sharp or that it coincides with $\critFE$.
The point-source model is studied only in \cite{levine2021source}, which relates
the density of the model after stabilization to $\critFE$ and $\critDD$, without proving
existence of $\critDD$ or $\critPS$.

For the driven-dissipative model, there are also few results.
The model shows up implicitly in criteria for proving fixation or activity of the fixed-energy model
like \cite[Proposition~3]{RollaTournier18}.
The mixing time of the driven-dissipative model is studied
in \cite{levine2021source, BS}. This work is relevant to universality
of ARW---fast mixing suggests universality because the chain forgets its initial configuration quickly---but
 its main technique is to approximate ARW with internal 
diffusion-limited aggregation (the $\lambda=\infty$ case of ARW), which gives away too much to
be of use in proving the density conjecture.
The strongest result on the driven-dissipative process comes in the recent paper
\cite{nicolas2024macroscopic} by Forien:
starting with initial density $\rho > \critFE$ of active particles
on an interval of length~$n$, 
with probability bounded from 0 the final density on the interval after stabilization is strictly less than 
$\rho$ with positive probability. This result is the first to connect 
$\critFE$ to the driven-dissipative model in a significant way, and it holds
arbitrarily close to the critical density. Our paper proves the stronger result that
for any $\epsilon>0$, the final density after stabilization is less than $\critFE+\epsilon$
with probability converging exponentially to $1$.

\subsection{Proof sketch}

To prove the density conjecture, we show that the critical densities for all four models
are equal to a constant $\crist=\crist(\lambda)$ to be defined in terms of our layer percolation process.
Let us start with the driven-dissipative model. Levine and Liang proved that if we stabilize
ARW on the interval $\ii{0,n}$ starting with one active particle on each site, we obtain
an exact sample from the invariant distribution of the driven-dissipative Markov 
chain \cite[Theorem~1]{levine2021exact}. Thus, to prove that $\critDD=\crist$,
we seek to show that the number of sleeping particles on $\ii{0,n}$ after stabilization
is within $(\crist-\epsilon)n$ and $(\crist+\epsilon)n$ with high probability.

\subsubsection*{Stable odometers}

As is typical, we use the \emph{sitewise construction} of ARW. Each site contains
a stack of instructions telling particles to move or sleep when present on the site.
The final state of the system after stabilization
can be determined solely from the \emph{odometer},
the function counting the number of instructions executed at each site.
According to the \emph{least-action principle},
the true odometer is the minimal element of a larger collection we call the \emph{stable odometers},
which if executed would produce feasible flows of particles leading to a stable configuration.
Hence, the number of sleeping particles left on $\ii{0,n}$ by the true odometer is the most
of any stable odometer. This background material is covered in Section~\ref{sec:sitewise}.

Thus, if we can construct \emph{any} stable odometer leaving
$(\crist-\epsilon)n$ particles on $\ii{0,n}$, then we obtain our lower bound
on particle density for the driven-dissipative model.
Conversely, we can obtain the upper bound by showing the nonexistence of any stable odometer
leaving more than $(\crist+\epsilon)n$ particles.
Broadly speaking, constructing an odometer and applying the least-action
principle is the idea used in fixation results like \cite[Theorem~2]{rolla2012absorbing}.
Similarly, some existing nonfixation results
proceed by showing the nonexistence of stable odometers that leave many sleeping particles
(see \cite[Section~6]{BasuGangulyHoffman18}, for example).

\subsubsection*{Embedding stable odometers in layer percolation}

The main innovation of our paper is a stochastic process called \emph{layer percolation}
that helps us understand the set of stable odometers.
The process can be thought of as a sequence $(\zeta_k)_{k\geq 0}$ of subsets of $\NN^2$.
We think of a point $(r,s)\in\zeta_k$ as a \emph{cell} in column~$r$ and row~$s$ 
at step~$k$ of layer percolation that has been \emph{infected}.
At each step, every cell infects cells in the next step at random; the set $\zeta_{k+1}$
consists of all cells infected by a cell in $\zeta_k$.
The infections are defined in terms of the random instructions from the sitewise representation
of layer percolation. Each stable odometer on $\ii{0,n}$ is embedded in
layer percolation as
an \emph{infection path}, a chain of infections ending at some cell $(r,s)\in\zeta_n$.
Under this correspondence, the ending row $s$ of the infection path
is equal to the number of particles that the odometer leaves sleeping on the interval.
We define layer percolation and give this correspondence in Section~\ref{sec:LP.ARW}.

\subsubsection*{Analysis of layer percolation}

Since the number of particles left sleeping by a stable odometer on $\ii{0,n}$ corresponds
to the ending row of an infection path in layer percolation, and the true odometer maximizes
the number of particles left sleeping out of all stable odometers, 
the question turns to the maximum row present
in the set $\zeta_n$ of infected cells. Using properties of regularity and superadditivity present
in layer percolation but not in ARW, we show that this maximum row 
grows linearly and that its growth rate converges to a deterministic constant we call $\crist$
(see \thref{prop:subadditive}). 
Thus it is likely that layer percolation contains
infection paths ending in row $(\crist-\epsilon)n$ but no infection paths ending in row $(\crist+\epsilon)n$.

\subsubsection*{Back to odometers}
The existence of infection paths ending in row $(\crist-\epsilon)n$ but not 
in row $(\crist+\epsilon)n$ should translate back to the existence of stable
odometers leaving as many as $(\crist-\epsilon)n$ but no more
than $(\crist+\epsilon)n$ particles sleeping on $\ii{0,n}$, which would
complete the proof that $\critDD=\crist$.
But the correspondence between stable odometers in ARW and infection paths in layer percolation
is not as neat as we have made it out to be.
A full accounting of these complications is best given later,
but the summary is that they force us to prove a stronger version of
the statement that layer percolation contains
infection paths ending in row $(\crist-\epsilon)n$ but not in row $(\crist+\epsilon)n$.
In Section~\ref{sec:box}, we prove that there likely exist infection paths ending 
in row $(\crist-\epsilon)n$ in any column in a wide, deterministic range.
This predictable behavior helps us construct stable odometers from infection paths,
and it also allows us to prove that the probability
of infection paths ending at row $(\crist+\epsilon)n$ not only vanishes but does
so at exponential rate, the subject
of Section~\ref{sec:upper.bound}.

\subsubsection*{The other three models of ARW}

Similar techniques can be used in all models of ARW we consider. For example,
in the point-source model we consider an initial configuration of $N$ particles at the
origin rather than one particle everywhere; but the correspondence with layer percolation works
in the same way regardless of the initial configuration. Several of the bounds
are transferred from one model to another without applying layer percolation again.
These final proofs are given in Section~\ref{sec:critical.values}.

\section{Odometers} \label{sec:construction}

In Section~\ref{sec:sitewise}, we present background material on the sitewise construction of activated
random walk. In this representation of the process, particles move about the graph
and fall asleep according to stacks of random instructions at each site, executed at random times.
Because the process then obeys an abelian property---executing the same instructions in different
orders results in the same stable configuration---this representation works well for determining
the final state of the system when run until stabilization, and it is used in nearly
all work on activated random walk. We mostly follow \cite{rolla2020activated}, and we have
tried to leave as many technicalities aside as possible.

In Section~\ref{sec:stability}, we describe what we mean by a \emph{stable odometer}
and prove the \emph{least-action principle}, which states that the true odometer
resulting from running activated random walk until all particles stabilize is the minimal
stable odometer. 

\subsection{The sitewise construction}\label{sec:sitewise}

An activated random walk \emph{configuration} is a placement of particles on a set of sites $V$;
particles can be sleeping or active, but multiple particles on the same site must all be active.
We represent a configuration as an element of $\{\s,0,1,2,\ldots\}^V$, where $\s$ represents a single
sleeping particle and a natural number represents that quantity of active particles.
We define $\abs{\s}=1$ so that we can write
$\abs{\sigma(v)}$ for a configuration $\sigma$ 
to refer to the number of particles, sleeping or active, at $v$.
We say that a configuration $\sigma$ is \emph{stable on $U$} for $U\subseteq V$ if $\sigma(v)\in\{0,\s\}$
for all $v\in U$, and we call it \emph{stable} if it is stable on all sites $V$.

Activated random walk is a continuous-time Markov chain in which active particles
jump to neighbors at rate~$1$ according to some given set of transition probabilities
and when alone on a site fall asleep at rate~$\lambda$ for a given parameter $0<\lambda<\infty$.
See \cite[Sections~2 and 11]{rolla2020activated} for further formalities and existence results for the Markov chain.
We specialize now to the case of 
nearest-neighbor symmetric random walk on $\ZZ$ or a subinterval, which is all we consider in this paper,
but all of Section~\ref{sec:construction} generalizes in an obvious
way to arbitrary graphs and random walk transition probabilities.

As is typical, we analyze activated random walk via its \emph{sitewise representation}, for which
we follow the approach of \cite[Section~2.2]{rolla2020activated}.
Assign each site $v$ a list of \emph{instructions}, which in our one-dimensional setting
consist of the symbols \Left, \Right, and \Sleep. We write $\Instr_v(k)$ to denote the $k$th instruction
at site~$v$, and we take $\bigl(\Instr_v(k),\ k\geq 1\bigr)$ to be i.i.d.\ with
\begin{align}\label{eq:LRS.distribution}
  \Instr_v(k)=\begin{cases} \Left&\text{with probability $\frac{1/2}{1+\lambda}$,}\\
    \Right&\text{with probability $\frac{1/2}{1+\lambda}$,}\\
    \Sleep&\text{with probability $\frac{\lambda}{1+\lambda}$.}
  \end{cases}
\end{align}
At a site with at least one active particle, it is \emph{legal} to execute the next
unexecuted instruction from its list, which is called \emph{toppling} the site.
To execute a \Left\ (resp.\ \Right) instruction at a site~$v$ with an active particle,
we subtract $1$ from the configuration at $v$ and add $1$ at
$v-1$ (resp.\ $v+1$), interpreting $\s+1$ as $2$.
To execute a \Sleep\ instruction at a site~$v$ with an active particle, we alter the configuration
at $v$ to $\s$ if it is currently $1$ and make no change if it is $2$ or more.
If at each site~$v$ we execute instructions from its list
at rate equal to $1+\lambda$
times the number of active particles at $v$, the resulting Markov chain is activated random walk
as defined previously.
The advantage of the sitewise representation is that the long-term
state of the system is determined solely by the initial configuration and list of instructions, with no role
played by the timing or order of executions. To make this statement precise, fix the instruction lists
and the initial configuration and define the
\emph{odometer} for any sequence of legal topplings
as the function $u\colon \ZZ\to\NN$ where $u(v)$ gives the number of topplings of vertex~$v$.
(Note that we will generalize this notion of odometer in Section~\ref{sec:stability}.)
For any finite set $V\subseteq\ZZ$, we say that a finite sequence of topplings \emph{stabilizes}
$V$ if it leaves a stable configuration
on $V$ (i.e., all active particles in $V$ are left
sleeping or are driven out of $V$). While we can stabilize $V$ in more than one order,
in the end we arrive at the same place:
\begin{lemma}[{\cite[Lemma~2.4]{rolla2020activated}}]\thlabel{lem:abelian}
  All sequences of topplings within a finite set $V$ that stabilize $V$ have the same odometer.
\end{lemma}
If this odometer exists, we call it the \emph{true odometer stabilizing $V$}.
It is a simple corollary, to be discussed in the next section (see \eqref{eq:initbalance}),
that all sequences of topplings in $V$ that stabilize $V$ result in the same stable
configuration.
We call this configuration the \emph{stabilization} of the initial configuration on $V$.
It is also easy to see that with random instruction lists as defined previously, the true odometer
stabilizing $V$ exists almost
surely. We can always find it by toppling
sites in $V$ in any order until all particles have fallen asleep or left $V$.
We can think of this procedure as piling up active vertices on the outer boundary of $V$
or as stabilization with sinks on the outer boundary of $V$.

\subsection{Stable odometers and the least-action principle}\label{sec:stability}
Generalizing the odometers derived from sequences of topplings in the previous section,
we call any function $u\colon\ZZ\to\NN$ an \emph{odometer}.
We say that $u$ is an \emph{odometer on $V\subseteq\ZZ$} if it is zero everywhere off $V$.
Whether or not an odometer arises from a sequence of topplings,
we still think of it as indicating a quantity of instructions executed at each site.
For an odometer $u$, we write $\lt{u}{v}$ and $\rt{u}{v}$ to refer to the number of \Left\ 
and \Right\ instructions, respectively, executed at site~$v$; that is,
\begin{align*}
  \lt{u}{v} = \sum_{i=1}^{u(v)}\1\{\Instr_v(i)=\Left\} \quad\quad\text{and}\quad\quad
  \rt{u}{v} = \sum_{i=1}^{u(v)}\1\{\Instr_v(i)=\Right\}.
\end{align*}

If $u$ is an odometer derived from a sequence of topplings starting from an initial configuration $\sigma$, 
the number of particles ending at site~$v$ is
\begin{align}\label{eq:initbalance}
  \abs{\sigma(v)} + \rt{u}{v-1} + \lt{u}{v+1} - \lt{u}{v} - \rt{u}{v}.
\end{align}
The terms $\rt{u}{v-1}$ and $\lt{u}{v+1}$ count the number of times a particle arrives at $v$,
while $\lt{u}{v}$ and $\rt{u}{v}$ count the number of times a particle departs it.
For the true odometer stabilizing $V$, (a) the quantity \eqref{eq:initbalance} is equal to $0$
or $1$ at each site $v\in V$, and (b) it is equal to $1$ if and only if the final instruction
executed at $v$ is \Sleep.
For a general odometer $u$ that need not correspond to any sequence of legal topplings,
we still think of \eqref{eq:initbalance} as the number of particles left on $v$
by the odometer. With this interpretation in mind, we abstract (a) and (b) for a general
odometer as follows:
\begin{define}\thlabel{def:stable}
  Let $u$ be an odometer on $\ZZ$ and 
  let $\sigma$ be an ARW configuration with no sleeping particles.
  We call $u$ \emph{stable on $V$} for the initial configuration $\sigma$ and 
  instructions $\bigl(\Instr_v(i),\,v\in \ZZ,\,i\geq 1\bigr)$ if for all $v\in V$,
  \begin{enumerate}[(a)]
    \item $h(v) := \sigma(v)+\rt{u}{v-1} + \lt{u}{v+1} - \lt{u}{v} - \rt{u}{v} \in\{0,1\}$;\label{i:balance}
    \item $h(v)=1$ if and only if $\Instr_v(u(v))=\Sleep$.
      \label{i:end.in.sleep}
  \end{enumerate}
  We call $u$ \emph{weakly stable on $V$} if for all $v\in V$ it
  satisfies \ref{i:balance} and
  \begin{enumerate}
    \myitem[(b\mprime)] if $h(v)=1$ then $\Instr_v(u(v))=\Sleep$.\label{i:weak.end.in.sleep}
  \end{enumerate}
  For any odometer $u$ stable or weakly stable on $V$, we say that
  $\sum_{v\in V}h(v)$ is the quantity of particles \emph{left on $V$ by $u$}.
\end{define}
The true odometer stabilizing $V$ is always stable on $V$.
Note that this relies on our assumption that $\sigma$ contains no sleeping particles here, since if it did
the true odometer might execute zero instructions at a site starting with a sleeping particle.
We could drop this assumption by modifying \ref{i:end.in.sleep} to allow this, but we have no need
to consider such initial configurations in this paper.

There exist other stable odometers besides the true stabilizing odometer. Some of them can be obtained
if we permit the toppling of sleeping particles: these are the
\emph{acceptable} but non-legal topplings of \cite{rolla2020activated}.
But there are even more stable odometers beyond these, which can be thought of as 
arising from sequences of topplings with
particle counts permitted to be negative, say by carrying out all topplings indicated by the odometer
except any terminating \Sleep\ instructions in any order, and then executing the final \Sleep\ instruction.
But as commented in \cite[Section~2.2]{rolla2020activated}, the abelian property fails if instructions
can be executed with no particles present, and hence it is unwise to consider toppling sequences
along these lines. We instead view the stable odometers as the fundamental objects,
and we do not associate them with any sequences of topplings.

We now give our least-action principle, stating that the true stabilizing odometer is minimal among
all weakly stable odometers. This least-action principle differs slightly from
the most commonly cited \cite[eq.~(2.3)]{rolla2020activated}, which states
that the true odometer is minimal among all odometers obtained from sequences
of acceptable topplings, but it is not really new. It is the same
as the ``strong form of the least action principle'' used in \cite{FeyLevinePeres10} for the abelian
sandpile model, and it could also be obtained from the very general
framework in \cite{BondLevine16}.
We give a direct proof here since it is not complicated to do so.

\begin{lemma}[Least-action principle]\thlabel{lem:lap}
  Let $u$ be the true odometer stabilizing finite $V\subseteq\ZZ$ with given instructions and initial configuration
  with no sleeping particles.
  Let $u'$ be an odometer on $\ZZ$ that is weakly stable on $V$ for the same instructions and 
  initial configuration.
  Then
  \begin{align*}
    u(v)\leq u'(v)
  \end{align*}
  for all $v\in V$.
\end{lemma}
\begin{proof}
  Consider the following toppling procedure: Starting with our initial configuration $\sigma$, arbitrarily
  choose any nonstable site $v\in V$ that has been toppled fewer than $u'(v)$ times and topple it.
  Continue until no such sites exist, i.e., all sites $v\in V$ are stable or have been
  toppled $u'(v)$ times, which is guaranteed to occur eventually since $V$ is finite and $u'(v)<\infty$
  for all $v$. Let $w$ be the odometer on $V$ derived from this sequence of topplings. By
  construction $w\leq u'$.
  
  We claim that $w=u$, which completes the proof.
   We need only show that our toppling procedure stabilizes
  $V$, since then $w=u$ by \thref{lem:abelian}.
  Suppose that some site $v$ is left unstable. Then $w(v)=u'(v)$, and either
  multiple particles are left on $v$ or one active particle is left on $v$.
  In the first case,
  \begin{align*}
    \sigma(v)+\rt{w}{v-1} + \lt{w}{v+1} - \lt{w}{v} - \rt{w}{v} >1.
  \end{align*}
  But since $w(v)=u'(v)$ and $w\leq u'$, this means that
  \begin{align*}
    1&< \sigma(v)+\rt{w}{v-1} + \lt{w}{v+1} - \lt{w}{v} - \rt{w}{v}\\
    &\leq
      \sigma(v) + \rt{u'}{v-1} + \lt{u'}{v+1} - \lt{u'}{v} - \rt{u'}{v},
  \end{align*}
  demonstrating that \thref{def:stable}\ref{i:balance} fails for $u'$, a contradiction
  since $u'$ is weakly stable.
  In the second case,
  \begin{align*}
    \sigma(v)+\rt{w}{v-1} + \lt{w}{v+1} - \lt{w}{v} - \rt{w}{v} =1
  \end{align*}
  but $\Instr_v(w(v))\neq\Sleep$. This time, from $w(v)=u'(v)$ and $w\leq u'$ we obtain
  \begin{align*}
    1 &= \sigma(v)+\rt{w}{v-1} + \lt{w}{v+1} - \lt{w}{v} - \rt{w}{v}\\
      &\leq \sigma(v) + \rt{u'}{v-1} + \lt{u'}{v+1} - \lt{u'}{v} - \rt{u'}{v}
  \end{align*}
  with $\Instr_v(u'(v))\neq\Sleep$. If $\sigma(v) + \rt{u'}{v-1} + \lt{u'}{v+1} - \lt{u'}{v} - \rt{u'}{v}>1$
  then \thref{def:stable}\ref{i:balance} fails for $u'$.
  If $\sigma(v) + \rt{u'}{v-1} + \lt{u'}{v+1} - \lt{u'}{v} - \rt{u'}{v}=1$, then
  \thref{def:stable}\ref{i:weak.end.in.sleep} fails for $u'$. Either way yields a contradiction
  since $u'$ is weakly stable.
\end{proof}
\begin{remark}
  We used \thref{lem:abelian} from \cite{rolla2020activated} in this proof, but in fact
  we could use this proof to establish \thref{lem:abelian}. First, restate
  the least-action principle as applying to any odometer $u$
  obtained from topplings in $V$ that stabilize $V$. Then if $u$ and $\tilde{u}$ are two
  such odometers, both are stable on $V$ and hence $u\leq\tilde{u}$
  and $\tilde{u}\leq u$, thus proving $u=\tilde{u}$.
\end{remark}

Since the true odometer stabilizing an interval is minimal among weakly stable odometers, it in particular
executes the minimal number of instructions on the boundary of the interval.
Thus, we can interpret the least-action principle as saying that the true odometer
maximizes the number of particles left sleeping on the interval.
\begin{lemma}\thlabel{lem:dual.lap}
  Let $u'$ be an odometer on $\ii{a,b}$ that is weakly stable on $\ii{a,b}$, for given initial configuration
  and instruction lists on the interval.
  The true odometer stabilizing $\ii{a,b}$ leaves at least as many
  particles sleeping on $\ii{a,b}$ as does $u'$,
   in the sense given in \thref{def:stable}.
\end{lemma}
\begin{proof}
  The number of particles left by $u'$ is by definition
  \begin{align*}
    \sum_{v=a}^b\Bigl( \abs{\sigma(v)} + \rt{u'}{v-1} + \lt{u'}{v+1} - \lt{u'}{v} - \rt{u'}{v}\Bigr).
  \end{align*}
  Since all terms $\lt{u'}{v}$
  and $\rt{u'}{v}$ cancel except for $\rt{u'}{a-1}+\lt{u'}{b+1}-\lt{u'}{a}-\rt{u'}{b}$,
  and the first two of these terms are zero since we have assumed that $u'$ is zero off of
   $\ii{a,b}$, this sum is equal to
  \begin{align}\label{eq:u'.left}
    \sum_{v=a}^b\abs{\sigma(v)} -\lt{u'}{a}-\rt{u'}{b},
  \end{align}
  which is the number of particles in the initial configuration minus the number
  of particles pushed off the interval by $u'$.
  The true number of particles left on the interval is
  \begin{align*}
    \sum_{v=a}^b\abs{\sigma(v)} -\lt{u}{a}-\rt{u}{b},
  \end{align*}
  where $u$ is the true odometer $u$ stabilizing $\ii{a,b}$, and this quantity is at least
  \eqref{eq:u'.left} since
  $u(a)\leq u'(a)$ and $u(b)\leq u'(b)$ by the least-action principle.
\end{proof}

\section{Layer percolation} \label{sec:LP.ARW}

Our approach to activated random walk is to understand the set of stable odometers, given
the initial configuration and instructions at each site. Layer percolation encodes this set
as an infection process in a $(2+1)$-dimensional
oriented percolation model. 
At each time step, a two-dimensional set of sites randomly infects sites in the next step.
We call the sites \emph{cells} to reserve
the term \emph{site} for locations of particles in activated random walk. We refer to the cells
at step~$k$ as $(r,s)_k$ for $r,s\geq 0$. 
Typically we will consider $(0,0)_0$ to be the lone infected cell at step~$0$, denoting
this singleton set as $\infset_0$. Then we define $\infset_1$ as the set of cells at step~$1$ infected
by $(0,0)_0$, then $\infset_2$ as the set of cells at step~$2$ infected by a cell in $\infset_1$, and so on.

An instance of layer percolation can be defined from the instructions in activated random walk;
the instructions on site~$k+1$ of activated random walk will determine the infections in
layer percolation going from step~$k$ to step~$k+1$. When activated random walk and layer percolation
are coupled in this way, the odometers stable on the interior of $\ii{0,n}$ 
correspond to length~$n$ \emph{infection paths} in layer percolation, sequences
of cells each infecting the next starting at $(0,0)_0$. We will not prove this correspondence
until Section~\ref{sec:ARW.percolation.connection}, but in Section~\ref{sec:tree}
we give an example to give a sense of the correspondence and motivate layer percolation.
We define layer percolation in Section~\ref{sec:define}, and then
in Section~\ref{sec:example2} we revisit the example from the perspective of
layer percolation.

\subsection{Example: finding stable odometers}
\label{sec:tree}

Consider ARW on an interval $\ii{0,n}$ with initial configuration $\sigma\equiv 1$
where the instructions at sites $0$, $1$, and $2$ are as follows, with
\texttt{L}, \texttt{R}, and \texttt{S} short for \Left, \Right, and \Sleep.
{\setlength{\tabcolsep}{0.5pt}
\newcolumntype{i}{>{\tt}l}
\begin{align}\label{eq:ex.instructions}
\begin{tabular}{w{l}{6em}iiiii|iiiii|iiiii|iiiii|iiiii|l}
  Site~0:&S&R&S&S&L&L&R&L&R&L&R&R&L&R&R&L&R&S&L&R&R&L&S&R&R&\ldots\\
  Site~1:&R&L&S&L&R&R&S&S&R&L&R&S&L&L&R&R&L&R&L&R&S&R&S&L&R&\ldots\\
  Site~2:&S&L&R&R&L&S&L&S&R&L&L&R&S&L&L&S&L&S&R&R&S&S&L&L&S&\ldots
\end{tabular}
\end{align}
}

Suppose we want to construct all odometers on $\ii{0,n}$ stable on the interior of the interval.
To get started, we choose some $u_0$ and restrict ourselves to constructing odometers $u$ satisfying
$u(0)=u_0$. We also make a choice $f_0$, representing a net flow from $0$ to $1$, and further
restrict ourselves to odometers satisfying $\rt{u}{0}-\lt{u}{1}=f_0$.
In the following example, we take $u_0=20$ and $f_0=2$.

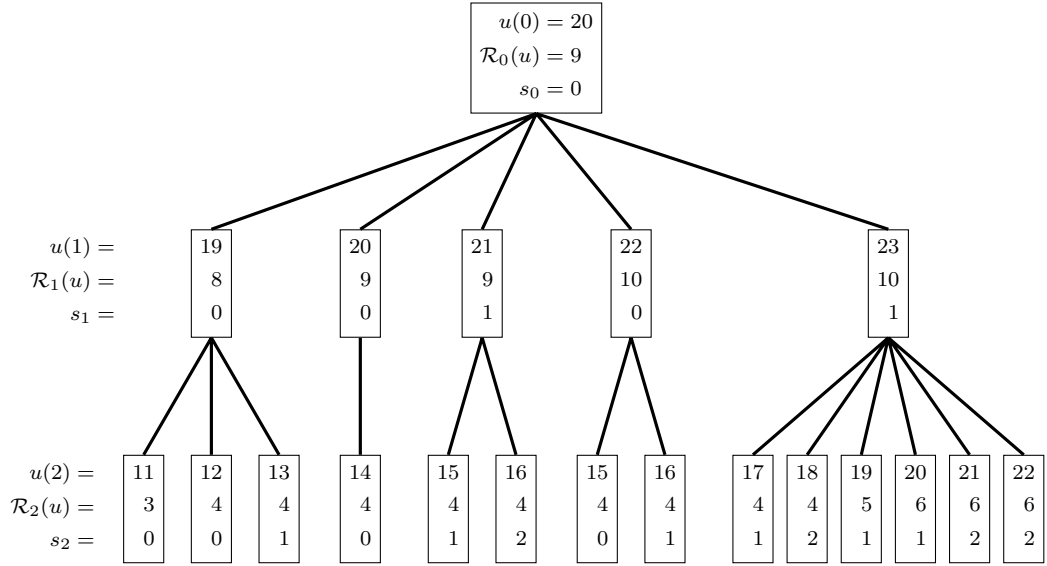
\begin{figure}
  \begin{tikzpicture}[every node/.style={draw, font=\footnotesize},xscale=.9]
    \path
      (-1.2,0) node[] (20) {$\begin{aligned}u(0)&=20\\\rt{u}{0}&=9\\s_0&=0\end{aligned}$};
    \path (0,0)
      +(-8,-3) node[align=left,draw=none] {$\begin{aligned}u(1)&={}\\\rt{u}{1}&={}\\s_1&={}\end{aligned}$}
      ++(-6,-3) node (20-19) {$\begin{aligned}19\\8\\0\end{aligned}$}
      ++(2.2,0) node (20-20) {$\begin{aligned}20\\9\\0\end{aligned}$}
      ++(1.8,0) node (20-21) {$\begin{aligned}21\\9\\1\end{aligned}$}
      ++(2.2,0) node (20-22) {$\begin{aligned}22\\10\\0\end{aligned}$}
      ++(3.8,0) node (20-23) {$\begin{aligned}23\\10\\1\end{aligned}$};
    \path (20-19)
      ++(-2.3,-3) node[align=left,draw=none] {$\begin{aligned}u(2)&={}\\\rt{u}{2}&={}\\s_2&={}\end{aligned}$}
      ++(1.3,0) node (20-19-11) {$\begin{aligned}11\\3\\0\end{aligned}$}
      ++(1,0) node (20-19-12) {$\begin{aligned}12\\4\\0\end{aligned}$}
      ++(1,0) node (20-19-13) {$\begin{aligned}13\\4\\1\end{aligned}$}
      ;
    \path (20-20)
      ++(0,-3) node (20-20-14) {$\begin{aligned}14\\4\\0\end{aligned}$};
    \path (20-21)
      ++(-0.5,-3) node (20-21-15) {$\begin{aligned}15\\4\\1\end{aligned}$}
      ++(1,0) node (20-21-16) {$\begin{aligned}16\\4\\2\end{aligned}$};
    \path (20-22)
      ++(-0.5,-3) node (20-22-15) {$\begin{aligned}15\\4\\0\end{aligned}$}
      ++(1,0) node (20-22-16) {$\begin{aligned}16\\4\\1\end{aligned}$};
    \path (20-23)
      ++(-2,-3) node (20-23-17) {$\begin{aligned}17\\4\\1\end{aligned}$}
      ++(.8,0) node (20-23-18) {$\begin{aligned}18\\4\\2\end{aligned}$}
      ++(.8,0) node (20-23-19) {$\begin{aligned}19\\5\\1\end{aligned}$}
      ++(.8,0) node (20-23-20) {$\begin{aligned}20\\6\\1\end{aligned}$}
      ++(.8,0) node (20-23-21) {$\begin{aligned}21\\6\\2\end{aligned}$}
      ++(.8,0) node (20-23-22) {$\begin{aligned}22\\6\\2\end{aligned}$}
      ;
    \draw[very thick] (20.south)--(20-19.north) (20.south)--(20-20.north) (20.south)--(20-21.north) (20.south)--(20-22.north) (20.south)--(20-23.north)
      (20-19.south)--(20-19-11.north)
      (20-19.south)--(20-19-12.north)  
      (20-19.south)--(20-19-13.north)    
      (20-20.south)--(20-20-14.north)
      (20-21.south)--(20-21-15.north)
      (20-21.south)--(20-21-16.north)
      (20-22.south)--(20-22-15.north)
      (20-22.south)--(20-22-16.north)
      (20-23.south)--(20-23-17.north)
      (20-23.south)--(20-23-18.north)
      (20-23.south)--(20-23-19.north)
      (20-23.south)--(20-23-20.north)
      (20-23.south)--(20-23-21.north)
      (20-23.south)--(20-23-22.north)
          ;
    (20-23)--(20-23-17.north) (20-23)--(20-23-18.north)
    ;
  \end{tikzpicture}

\caption{Each of the fourteen possibilities for $\bigl(u(0),\,u(1),\,u(2)\bigr)$ to make a stable
odometer consistent with \eqref{eq:ex.instructions},
$\sigma\equiv 1$, $u_0=20$ and $f_0=2$ is represented as a leaf in the tree. 
For example, the leftmost leaf represents $\bigl(u(0),\,u(1),\,u(2)\bigr)=(20,19,11)$, while
the rightmost represents $\bigl(u(0),\,u(1),\,u(2)\bigr)=(20,23,22)$.
The quantity $s_k$ denotes the number of sites $1,\ldots,k$ on which the final instruction
executed under the odometer is \Sleep, representing the number of particles left on $\ii{1,k}$
by the odometer.
}
\label{fig:decision.tree}  
\end{figure}

From $u_0$ and the instructions at site~$0$, we have
$\rt{u}{0}=9$ for any odometer in this class. From $f_0$, we must have $\lt{u}{1}=7$, and therefore
$u(1)\in\ii{19,23}$. Each choice of $u(1)$ then yields different possibilities for $u(2)$.
\begin{enumerate}[label=\textbf{Case \arabic*:},leftmargin=*,labelindent=0pt,labelwidth=0pt,itemindent=20pt]
  \item $u(1)=19$.\\
    Then $\rt{u}{1}=8$, and $u$ leaves no particle sleeping at $1$ since the final
    instruction executed there is \Left. The net flow from site~$0$ to site~$1$ was $2$, and therefore
    the net flow from site~$1$ to site~$2$ must be $3$ to remove the particle initially at site~$1$.
    Hence $\lt{u}{2}=8-3=5$, yielding $u(2)\in\ii{11,13}$.
  \item $u(1)=20$.\\
    Then $\rt{u}{1}=9$, and $u$ leaves no particle sleeping at $1$ since the final
    instruction executed there is \Right. The net flow from site~$0$ to site~$1$ was $2$, and therefore
    the net flow from site~$1$ to site~$2$ must be $3$ to remove the particle initially at site~$1$.
    Hence $\lt{u}{2}=9-3=6$, yielding $u(2)=14$.
  \item $u(1)=21$.\\
    Then $\rt{u}{1}=9$, and $u$ leaves a particle sleeping at $1$ since the final
    instruction executed there is \Sleep. The net flow from site~$0$ to site~$1$ was $2$, and therefore
    the net flow from site~$1$ to site~$2$ must also be $2$.
    Hence $\lt{u}{2}=9-2=7$, yielding $u(2)\in\ii{15,16}$.
  \item $u(1)=22$.\\
    Then $\rt{u}{1}=10$, and $u$ leaves no particle sleeping at $1$ since the final
    instruction executed there is \Right. The net flow from site~$0$ to site~$1$ was $2$, and therefore
    the net flow from site~$1$ to site~$2$ must be $3$ to remove the particle initially at site~$1$.
    Hence $\lt{u}{2}=10-3=7$, yielding $u(2)\in\ii{15,16}$.
  \item $u(1)=23$.\\
    Then $\rt{u}{1}=10$, and $u$ leaves a particle sleeping at $1$ since the final
    instruction executed there is \Sleep. The net flow from site~$0$ to site~$1$ was $2$, and therefore
    the net flow from site~$1$ to site~$2$ must be $2$.
    Hence $\lt{u}{2}=10-2=8$, yielding $u(2)\in\ii{17,22}$.
\end{enumerate}
The fourteen possible odometers produced so far are illustrated in decision tree form
in Figure~\ref{fig:decision.tree}. Given the instructions at sites $2,\ldots,n$, we could
continue forming the decision tree to find all stable odometers for our specified $u_0$
and $f_0$. As illustrated by our example, at each step the choice made so far of $u(0),\ldots,u(k)$ determines
$\rt{u}{k}$ and the net flow from site~$k$ to site~$k+1$ required to make $u$ stable at $k$,
thus determining $\lt{u}{k+1}$. Then $\lt{u}{k+1}$ determines the possible choices
of $u(k+1)$, ranging from the $\lt{u}{k+1}$th
\Left\ instruction up to but excluding the $(\lt{u}{k+1}+1)$th \Left\ instruction.

We have depicted the set of stable odometers in tree form, but different branches of the tree
are highly dependent. For example, in Figure~\ref{fig:decision.tree}, the level~$1$ nodes
representing $\bigl(u(0),\,u(1)\bigr)=(20,21)$ and $\bigl(u(0),\,u(1)\bigr)=(20,22)$ have nearly
identical children because both are generated from the same range of the site~$2$ instructions.
On the other hand, their grandchildren will not match up. This complex dependency poses an obstacle
to analyzing the stable odometers.

The key insight in making sense of the odometers is that we do not need to know the entire past history
$u(0),\ldots,u(k)$ to determine the possible choices for $u(k+1)$. Rather, two pieces of information
suffice: The first is $\rt{u}{k}$. The second is the number of sites~$1,\ldots,k$ on which the final
instruction executed under $u$ is \Sleep, a quantity we denote by $s_k$. 
From these two pieces of information, we can find the number
of \Left\ instructions needed at site~$k+1$ to make $u$ stable at $k$, thus determining
the range of instructions at site~$k+1$ that are possible choices for $u(k+1)$.
Each increase of $\rt{u}{k}$ forces $\lt{u}{k+1}$ to increase by one to balance the flow of particles
onto site~$k$.
Likewise, each increase of $s_k$ forces $\lt{u}{k+1}$ to increase by one, since it means that
an additional particle sleeps on $\ii{1,k}$ and hence the net flow from site~$k$ to site~$k+1$
must decrease by one. The explanation for why
$\bigl(u(0),\,u(1)\bigr)=(20,21)$ and $\bigl(u(0),\,u(1)\bigr)=(20,22)$ lead to the same possible
values $u(2)\in\ii{15,16}$ is that $\rt{u}{1}+s_1=10$ for both choices of $\bigl(u(0),\,u(1)\bigr)$.

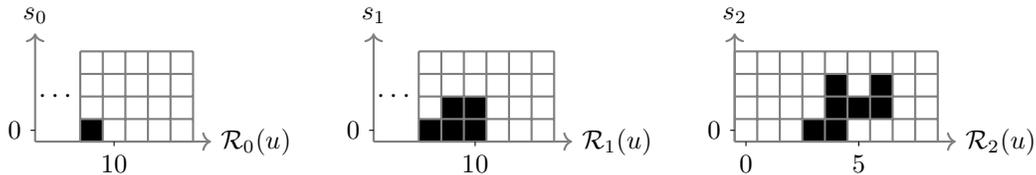
\begin{figure}
  \centering
\begin{tikzpicture}[scale=.3]
\begin{scope}[shift={(0,0)}]
  \begin{scope}[shift={(2,0)}]
    \fill (0,0) rectangle +(1,1);
    \draw[thick,gray] (0,0) grid[step=1] (5,4);
  \end{scope}
  \draw[thick,gray,->] (0,0) -- (7.8,0) node[black,right] {$\rt{u}{0}$};
  \draw[thick,gray,->] (0,0) -- (0,4.8) node[black,above] {$s_0$};
  \draw (3.50000000000000,0)-- ++(0,-.2) node[below] {10};
  \draw (0,0.500000000000000)-- ++(-.2,0) node[left] {0};
  \path (1,2) node[font=\large] {$\cdots$};
\end{scope}
\begin{scope}[shift={(17,0)}]
  \fill (0,0) rectangle +(1,1);
  \fill (1,0) rectangle +(1,1);
  \fill (2,0) rectangle +(1,1);
  \fill (1,1) rectangle +(1,1);
  \fill (2,1) rectangle +(1,1);
  \draw[thick,gray] (0,0) grid[step=1] (6,4);
  \draw[thick,gray,->] (-2,0) -- (6.8,0) node[black,right] {$\rt{u}{1}$};
  \draw[thick,gray,->] (-2,0) -- (-2,4.8) node[black,above] {$s_1$};
  \draw (2.500000000000000,0)-- ++(0,-.2) node[below] {10};
  \draw (-2,0.500000000000000)-- ++(-.2,0) node[left] {0};
  \path (-1,2) node[font=\large] {$\cdots$};
\end{scope}
\begin{scope}[shift={(34,0)}]
  \fill (0,0) rectangle +(1,1);
  \fill (1,0) rectangle +(1,1);
  \fill (1,1) rectangle +(1,1);
  \fill (1,2) rectangle +(1,1);
  \fill (2,1) rectangle +(1,1);
  \fill (3,1) rectangle +(1,1);
  \fill (3,2) rectangle +(1,1);
  \draw[thick,gray] (-3,0) grid[step=1] (6,4);
  \draw[thick,gray,->] (-3,0) -- (6.8,0) node[black,right] {$\rt{u}{2}$};
  \draw[thick,gray,->] (-3,0) -- (-3,4.8) node[black,above] {$s_2$};
  \draw (-2.500000000000000,0)-- ++(0,-.2) node[below] {0};
  \draw (2.5,0)-- ++(0,-.2) node[below] {5};
  \draw (-3,0.500000000000000)-- ++(-.2,0) node[left] {0};
\end{scope}
\end{tikzpicture}

  \caption{Each node of the tree in Figure~\ref{fig:decision.tree} is represented as a \emph{cell}
  in the above diagram. A node at level~$k$ of the tree is plotted at step~$k$ of this diagram,
  with $\rt{u}{k}$ plotted on the horizontal axis and $s_k$ plotted on the vertical axis.
  }
  \label{fig:layer.perc.version}
\end{figure}

We can thus capture all information in each level of the decision tree in Figure~\ref{fig:decision.tree} 
as a collection of pairs $(\rt{u}{k},s_k)$, as depicted in Figure~\ref{fig:layer.perc.version}.
The nodes at level~$k$ of the tree become infected \emph{cells}, depicted as black squares, at step~$k$ of our
layer percolation process. If one node is the parent of another in the tree, we say that the cell
corresponding to the parent \emph{infects} the cell corresponding to the child.
For example, the cell $(\rt{u}{2},s_2)=(4,0)$ at step~$2$ is infected by three different cells
$(\rt{u}{1},s_1)=(8,0),(9,0),(10,0)$ at step~$1$. This process in which cells with two-dimensional
coordinates infect cells in the next step is \emph{layer percolation}, to be
defined formally in the next section.

We can describe the odometer-generating
algorithm purely from the perspective of layer percolation. The initial cell
$(\rt{u}{0},s_0)=(9,0)$ infects cells according to instructions $19,\ldots,23$
at site~$1$ (\texttt{LRSRS}),
starting at the fifth \Left\ instruction and ending immediately before the sixth.
These five choices correspond to the five infected cells in the middle image in 
Figure~\ref{fig:layer.perc.version}. The three non-\Sleep\ instructions (i.e., choosing $u(1)$
to be $19$, $20$, or $22$) 
lead respectively to the infection of three consecutive cells in row~$0$: 
the row is $0$ because these choices
of instructions do not leave a particle sleeping at site~$1$, and the columns span three cells
because each additional \Right\ instruction executed by the odometer moves the infected cell to the right.
Choosing a \Sleep\ instruction ($u(1)=21$ or $u(1)=23$) leads to an infection in row~$1$,
since a particle is left sleeping a site~$1$; the column infected is the same for the first
non-\Sleep\ instruction prior to the \Sleep\ instruction.

We can apply the same analysis to determine the cells infected by any cell $(\rt{u}{k},s_k)$.
The infected cells are determined by some string of instructions at site~$k+1$
starting with a \Left\ instruction
and ending immediately prior to the next \Left\ instruction. A stretch of cells in row~$s_k$ is infected
whose length is given by the number of non-\Sleep\ instructions in the stretch.
And some of the cells above this stretch in row~$s_k+1$ are also infected, depending on the presence
of \Sleep\ instructions in between successive non-\Sleep\ instructions.
From \eqref{eq:LRS.distribution}, the stretch of cells infected in row~$s_k$ has length $1+\Geo(1/2)$,
and each cell in row $s_k+1$ in the stretch is infected independently with probability $\lambda/(1+\lambda)$.
(Here and elsewhere, $\Geo(p)$ denotes the geometric distribution on the nonnegative integers
with parameter~$p$, i.e., the one placing probability $(1-p)^kp$ on $k$ for $k\in\NN=\{0,1,\ldots\}$.
We abuse notation slightly by conflating distributions and random variables to permit an expression
like $1+\Geo(1/2)$.)

Next, we try to unravel the dependence between the infections
arising from different cells. First, we observe that two cells at step~$k$ with the same value of
$\rt{u}{k}+s_k$ infect cells according to the same range of instructions at site~$k+1$.
The cells they infect have the same shape but are in shifted rows from each other
depending on the value of $s_k$ for each infector (see for example the cells
infected by $(\rt{u}{1},s_1)=(10,0)$ and $(\rt{u}{1},s_1)=(9,1)$).
On the other hand, the infections stemming from cells with values of $\rt{u}{k}+s_k$---in other words, 
cells in different antidiagonals---are based on distinct ranges of instructions and hence
have independent shapes.

In the description so far, we have discussed the width of each infected range of cells while
obscuring the exact columns infected, which we consider now.
First, all infections from a given antidiagonal
occur at the same columns---the infected cells are shifted in rows but not in columns.
If we move from one antidiagonal to the next, the range of instructions determining
the infections moves ahead by one \Left\ instruction. That is,
suppose that a given antidiagonal at step~$k$ uses instructions from the $j$th \Left\ up to but excluding
the $(j+1)$th \Left\ instruction at site~$k+1$. Then the next antidiagonal will use instructions from
$(j+1)$th \Left\ instruction up to but excluding the $(j+2)$th.
The final instruction before the $(j+1)$th \Left\ instruction corresponds to the same
value of $\rt{u}{k+1}$ as the $(j+1)$th \Left\ instruction itself.
Hence the rightmost column infected from a given antidiagonal is the same as the leftmost
column infected from the next. This phenomenon is the reason we call our process layer percolation:
each successive antidiagonal at one step infects a layer of cells in the next, and the layers accumulate
one after the other.

In the next section, we give a formal description of the layer percolation process,
motivated by this discussion but without reference to odometers and instructions.
The one notable difference between the process as sketched here
and as formally described in the next section is that the first coordinate of cells will not match
up exactly with $\rt{u}{k}$. Instead all cells in a given step
will be shifted at each step by a constant to put the leftmost
infected cell in column~$0$.

\subsection{Definition of layer percolation}
\label{sec:define}

We define \emph{layer percolation} with parameter $\lambda>0$ as a process in which \emph{cells}
denoted $(r,s)_k$ for integers $r,s,k\geq 0$ infect one or more cells $(r',s')_{k+1}$.
We say that a cell $(r,s)_k$ is at \emph{step~$k$} of layer percolation; since
we plot $r$ on the horizontal and $s$ on the vertical axis, we call $r$ the cell's \emph{column}
and $s$ the cell's \emph{row}. 
We write $(r,s)_k\to(r',s')_{k+1}$ to denote the event that $(r,s)_k$ infects $(r',s')_{k+1}$.
Formally, layer percolation up to step~$n\leq\infty$ is the collection of indicators
\begin{align*}
  \Bigl(\1\bigl\{(r,s)_{k-1}\to(r',s')_k\bigr\},\ r,s\geq 0,\, 1\leq k\leq n\Bigr).
\end{align*}
Note that the
process is oriented, in the sense that cells at step~$k$ only infect cells at step~$k+1$.
The infections $\bigl(\1\bigl\{(r,s)_{k-1}\to(r',s')_k\bigr\},\ r,s\geq 0\bigr)$ for different
values of $k$ are i.i.d.
To define the process, it therefore suffices to state the distribution of 
$\bigl(\1\bigl\{(r,s)_{k-1}\to(r',s')_k\bigr\},\ r,s\geq 0\bigr)$ for any particular $k\geq 1$,
which we give now.

Take $k$ as fixed.
We refer to the $j+1$ cells $\{(r,s)_{k-1}\colon r+s=j\}$ as \emph{diagonal~$j$} at step~$k-1$.
Cells in diagonal~$j$ at step~$k-1$ infect cells only in \emph{layer~$j$} at step~$k$,
which we define now.
Let $R_0,R_1,\ldots$ be independent with distribution $\Geo(1/2)$. 
For $j=0,1,\ldots$, we define layer~$j$ at step~$k$
as the vertical strip of cells starting at column $R_0+\cdots+R_{j-1}$
and ending at column $R_0+\cdots+R_j$, including these endpoints.
Let
\begin{align}\label{eq:layer.def}
  \layer(j)=\layer_k(j)=\ii{R_0+\cdots+R_{j-1},\, R_0+\cdots+R_j},
\end{align}
so that layer~$j$ at step~$k$ consists of all cells $(r,s)_k$ for $r\in\layer(j)$.
We call $1+R_j$ the \emph{width} of layer~$j$, the number of columns it spans.
More informally, the layer structure at step~$k$ is defined as follows.
The widths of the layers are independent with distribution $1+\Geo(1/2)$.
Layer~$0$ is a strip of columns starting at $0$ with the given width. Layer~$1$
begins in the final column of layer~$0$, overlapping it by one column, and then extends to
the right for its given width. Then layer~$2$
begins in the final column of layer~$1$, and so on.
By the memoryless property of the geometric distribution, we can also imagine forming the layers
starting in column~$0$ with layer~$0$. At each step, with equal probability we extend the current
layer by one column, or we remain at the same column and start a new layer, continuing in this way
forever to carve up the cells into vertical strips each overlapping the last by one column.

\begin{figure}
\centering
  \begin{tikzpicture}[scale=.3]
    \fill (0,0) rectangle +(1,1);
    \fill (1,0) rectangle +(1,1);
    \fill (2,0) rectangle +(1,1);
    \fill (3,0) rectangle +(1,1);
    \fill (4,0) rectangle +(1,1);
    \fill (0,1) rectangle +(1,1);
    \fill (3,1) rectangle +(1,1);
  \draw[thick,gray] (0,0) grid[step=1] (5,2);
  \end{tikzpicture}
  
  \caption{A cell $(r,s)_{k-1}$ in diagonal~$j$ (i.e., $r+s=j$) 
  infects an interval of $R_j+1$ cells in row~$s$ at step~$k$, along with
  some of the cells above them in row~$s+1$ as determined by the Bernoulli random variables
  $\bigl(B^j_0,\ldots,B^j_{R_j}\bigr)$. The example above shows the shape of cells infected
  when $R_j=4$ and $\bigl(B^j_0,\ldots,B^j_{r}\bigr)=(1,0,0,1,0)$.
   }
  \label{fig:shape}
\end{figure}
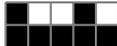

Now, we explain which cells in layer~$j$ at step~$k$ are infected by each cell in diagonal~$j$
at step~$k-1$. For each $j$, let $B^j=(B^j_0,\ldots,B^j_{R_j})$ consist of 
Bernoulli random variables with $\E B^j_i=\lambda/(1+\lambda)$, independent for all $i$ and $j$.
The random variables $(R_j,B^j)$ encodes the \emph{shape} of cells infected at step~$k$
by a cell in diagonal~$j$ at step~$k-1$; see Figure~\ref{fig:shape}.
A cell $(r,s)_{k-1}$ in diagonal~$j$ infects all $1+R_j$ cells at row~$s$ in layer~$j$,
and it also infects the $i$th cell in row~$s+1$ in layer~$j$ if $B^j_i=1$ (here we call
column $R_0+\cdots+R_{j-1}$ the $0$th column of layer~$j$). That is,
we have $(r,s)_{k-1}\to(u,t)_k$ if either
\begin{enumerate}[(i)]
  \item $u\in \layer(r+s)$ and $t=s$, or
  \item $u\in\layer(r+s)$ and $t=s+1$ and $B^j_i=1$, where $i=u-(R_0+\cdots+R_{r+s-1})$.
\end{enumerate}
This completes our description of the infections from step~$k-1$ to step~$k$.
As we said before, the infections from other steps are independent with the same distribution,
and thus our definition of layer percolation is complete.
\begin{remark}\thlabel{rmk:what.is.layer.percolation}
  We defined layer percolation up to step~$n$ as the collection of indicators on events
  $(r,s)_{k-1}\to(r',s')_k$ for all $r,s\geq 0$ and $1\leq k\leq n$.
  We note that for any $k$, this collection of indicators
  can be generated from $(R_j)_{j\geq 0}$ and $(B^j_0,\ldots,B^j_{R_j})_{j\geq 1}$ and vice versa, 
  and thus layer percolation could equally well be defined by $(R_j)_{j\geq 0}$ and 
  $(B^j_0,\ldots,B^j_{R_j})_{j\geq 1}$ for each $1\leq k\leq n$.
\end{remark}

We take a break from formality now to look back to our example from
Section~\ref{sec:tree}. The random variables $R_j$ and $B^j$ specify the shape
of cells infected in step~$k$ by a cell in diagonal~$j$ at step~$k-1$. In terms of the tree
in Figure~\ref{fig:decision.tree}, they describe the children of all nodes
at level~$k-1$ with a certain value of $\rt{u}{k-1}+s_{k-1}$. The quantity $R_j$
corresponds to the number of \Right\ instructions between a certain pair of \Left\ instructions
at site~$k$. And the Bernoulli random variables $B^j_0,\ldots,B^j_{R_j}$ are indicators
on the presence of \Sleep\ instructions occurring between the non-\Sleep\ instructions
in the string of instructions from the two \Left\ instructions.
Our main interest in layer percolation will be
the set of infected cells at each step originating from the single infected cell
$(0,0)_0$, for which we will give notation below. A chain of infections of length~$n$ starting at $(0,0)_0$
will correspond to an odometer on $\ii{0,n}$ stable on $\ii{1,n-1}$, modulo various complications.

Finally, we explain again why we call the process
\emph{layer percolation}. Given the infected cells in diagonal~$j$ at step~$k-1$, 
we can determine the \emph{shape} of the set of cells they infect from $R_j$ and $B^j$ alone.
But to determine the actual set of cells infected, we must first find
the cells infected by diagonals $0,\ldots,j-1$. In effect, we need to find the shapes of the
sets of cells infected by diagonals $0,\ldots,j$ and then paste them together from left to right.
Thus our infection sets are built up layer by layer. This differs from classical (and even modern \cite{hartarsky2021generalised}) percolation
processes and to us seems essential. 

We also mention that layer percolation can be generalized by defining the random variables
$R_0,R_1,\ldots$ to have distribution $\Geo(\theta)$ for a parameter $\theta$. For $\theta$
other than $1/2$, this layer percolation will match up with activated random walk in which particles
move as nearest-neighbor random walk with a bias in one direction. This fundamentally
alters the behavior of layer percolation by changing $\E R_i$ from $1$,
thus making key branching processes subcritical or supercritical rather than critical.
We plan to examine this process in future work.

We close the section with some additional notation. We have now defined the event
of $(r,s)_k$ infecting $(r',s')_{k+1}$, denoting it by $(r,s)_k\to(r',s')_{k+1}$.
If we want to say that $(r,s)_k$ infects $(r',s')_{k+1}$, we will often just write
$(r,s)_k\to(r',s')_{k+1}$ rather than ``the event $(r,s)_k\to(r',s')_{k+1}$ holds''.
We write
\begin{align*}
  (r_0,s_0)_k\to(r_1,s_1)_{k+1}\to\cdots\to(r_n,s_n)_{n+k}
\end{align*}
as a shorthand for the event that $(r_i,s_i)_{k+i}\to(r_{i+1},s_{i+1})_{k+i+1}$ for all $0\leq i< k$.
In this case we call this sequence of cells an \emph{infection path} from $(r_0,s_0)_k$ to $(r_n,s_n)_{n+k}$.
For $k\geq 2$, we write $(r,s)_k\to(r',s')_{n+k}$ for the event that there exists an infection path
from $(r,s)_k$ to $(r',s')_{n+k}$.
For consistency of notation, we say that the event $(r,s)_k\to(r',s')_k$ holds if and only
if $r=r'$ and $s=s'$.

Define the \emph{infection set from $(r,s)_k$ after $n$ steps}, denoted $\infset[(r,s)_k]_n$ (and
often abbreviated to $\infset_n$),
as the set of cells at step~$k+n$ infected starting from the given cell $(r,s)_k$.
Formally,
\begin{align}
  \infset[(r,s)_k]_n = \bigl\{ (u,t)_{k+n}\colon (r,s)_k\to(u,t)_{k+n} \bigr\}. \label{eq:forward}
\end{align}
Similarly, the \emph{backward infection set from $(r,s)_k$ after $n$ steps}, denoted
$\binfset[(r,s)_k]_n$, consists of all cells at step $k-n$ that infect
$(r,s)_k$, or
\begin{align}
  \binfset[(r,s)_k]_n = \bigl\{ (u,t)_{k-n}\colon (u,t)_{k-n}\to(r,s)_k \bigr\}. \label{eq:backward}
\end{align}
As we will see in Section~\ref{sec:couplings}, layer percolation has a sort of duality
that relates the distribution of backward infection sets to forward ones.

\subsection{Example: stable odometers from the layer percolation perspective}
\label{sec:example2}
In Section~\ref{sec:tree}, we described an algorithm to find all odometers stable on the interior
of an interval
given the initial configuration and instructions from activated random walk, the value
$u_0$ of the odometer at site~$0$, and the net flow $f_0$ of particles from site~$0$ to site~$1$
under the odometer. We sketched a representation of these odometers in terms of layer percolation.
Now that we have defined layer percolation and given notation for it, we return to the
example from Section~\ref{sec:tree} and show how we construct an instance of layer percolation
from the instructions of the activated random walk.
We will be informal in this section, trying to get the idea across of the correspondence
between activated random walk and layer percolation while
saving rigor for Section~\ref{sec:ARW.percolation.connection}.

Recall that there are two key quantities for an odometer $u$ at each site~$k$. The
first is $\rt{u}{k}$,
the number of \Right\ instructions executed at site~$k$ under the odometer.
The second is $s_k$, the number out of the indices $i\in\ii{1,k}$ for which the $u(i)$th
instruction at site~$i$ is \Sleep; that is, it represents the number of times the odometer $u$
leaves a particle sleeping at site~$i\in\ii{1,k}$. The importance of $\rt{u}{k}$ and $s_k$
is that they provide enough information to determine all possible extensions of $u$
from $\ii{1,k}$ to $\ii{1,n}$.

Recall the instructions \eqref{eq:ex.instructions} and that in the example we fix $u_0=20$
and $f_0=2$.
Our goal is to generate an instance of layer percolation that corresponds to 
Figure~\ref{fig:layer.perc.version} and encodes the stable odometers.
To generate $R_0,R_1,\ldots$ and $B^0,B^1,\ldots$ for generating the infections
from step~$0$ to step~$1$ of layer percolation, we first recall that from $u_0$
and $f_0$, the minimum (and only) value that $\lt{u}{1}$ could have is $7$.
Thus we start our generation of layer percolation at the seventh
\Left\ instruction at site~$1$. We set $R_0$ equal to the number
of \Right\ instructions between the $7$th and $8$th \Left\ instructions,
then $R_1$ equal to the number of \Right\ instructions between the $8$th and $9$th \Left\ instructions,
and so on. Since the number of \Right\ instructions between successive \Left\ instructions
are independent with distribution $\Geo(1/2)$ by the strong Markov property,
our random variables $R_0,R_1,\ldots$ are i.i.d.-$\Geo(1/2)$ as desired. In this example,
$R_0=2$.

To determine $B^j_0,\ldots,B^j_{R_j}$, we look at the portion of instructions from the $(7+j)$th
to the $(8+j)$th \Left\ instruction, which contains $R_j$ \Right\ instructions.
We let $B^j_i$ be an indicator on the presence of a \Sleep\ instruction 
immediately following the $i$th of these \Right\ instructions (or following the initial
\Left\ instruction for $i=0$). Thus the instructions \texttt{LRSRS} from the 7th \Left\ instruction
up to the 8th give us $(B^0_0,B^0_1,B^0_{2})=(0,1,1)$ because there is not a \Sleep\ instruction
following the initial \Left\ instruction but there are \Sleep\ instructions following the first
and second \Right\ instructions. Since the instruction following any given instruction
is \Sleep\ with probability $\lambda/(1+\lambda)$, these random variables $B^j_i$
are i.i.d.-$\Ber\bigl(\lambda/(1+\lambda)\bigr)$.

These values $R_0=2$ and $(B^0_0,B^0_1,B^0_{2})=(0,1,1)$ represent the shape
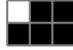
\begin{align*}
    \begin{tikzpicture}[scale=.3]
    \fill (0,0) rectangle +(1,1);
    \fill (1,0) rectangle +(1,1);
    \fill (2,0) rectangle +(1,1);
    \fill (1,1) rectangle +(1,1);
    \fill (2,1) rectangle +(1,1);
  \draw[thick,gray] (0,0) grid[step=1] (3,2);
  \end{tikzpicture}
\end{align*}
To form the infection set at step~$1$ starting from $(0,0)_0$,
the cell $(0,0)_0$ in diagonal~$0$ infects a set of cells of the shape above, starting at column~$0$
(because the leftmost cell infected by the diagonal~$0$ cell is always in column~$0$)
and at row~$0$ (because a cell in row~$s$ always infects cells in row~$s$ and $s+1$).
Thus the infection set at step~$1$ consists of the following cells $(r,s)_1$:
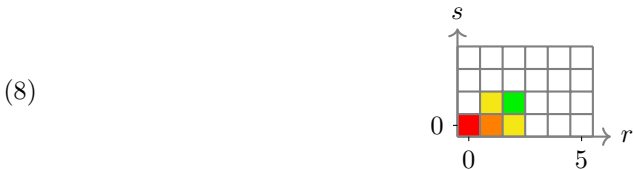
\begin{align}\label{eq:step1}
  \begin{tikzpicture}[scale=.3,baseline=.5cm]
    \fill[red] (0,0) rectangle +(1,1);
    \fill[orange] (1,0) rectangle +(1,1);
    \fill[yellow!95!black] (2,0) rectangle +(1,1);
    \fill[yellow!95!black] (1,1) rectangle +(1,1);
    \fill[green!95!black] (2,1) rectangle +(1,1);
    \draw[thick,gray] (0,0) grid[step=1] (6,4);
    \draw[thick,gray,->] (0,0) -- (6.8,0) node[black,right] {$r$};
    \draw[thick,gray,->] (0,0) -- (0,4.8) node[black,above] {$s$};
    \draw (0.500000000000000,0)-- ++(0,-.2) node[below] {0};
    \draw (5.50000000000000,0)-- ++(0,-.2) node[below] {5};
    \draw (0,0.500000000000000)-- ++(-.2,0) node[left] {0};
  \end{tikzpicture}
\end{align}
The colors are to be used in illustrating the next step; they correspond to the diagonal
of each cell.
Note that these infected cells match step~$1$ in Figure~\ref{fig:layer.perc.version}
except for being shifted along the horizontal axis.

Next, we form the random variables $R_0,R_1,\ldots$ and $B^0,B^1,\ldots$ that
generate the infections from step~$1$ to step~$2$ of layer percolation.
We will form them from the site~$2$ instructions, setting $R_j$ equal to
the number of \Right\ instructions between the $(m+j)$th and $(m+j+1)$th \Left\ instructions
and setting $(B^j_0,\ldots,B^j_{R_j})$ as indicators on the presence of \Sleep\ instructions
in this span. But we must first determine the value of $m$, which we can obtain from the index of
the first possible instruction that $u(2)$ might take. In this example, the minimal
value of $u(2)$ is $11$, corresponding to the 5th \Left\ instruction, yielding $m=5$.
In general, to find these minimal odometer values at each step, we follow the leftmost branch of
the tree in Figure~\ref{fig:decision.tree}, simply choosing the minimal value of $u(k+1)$ at each step
given our previous choices of $u(0),\ldots,u(k)$. The odometer obtained in this way
is called the \emph{minimal odometer} and plays an important role in our theory.\label{pg:min.odometer}
(But note that while it is minimal among odometers stable on $\ii{1,n-1}$
with a fixed choice of $u_0$ and $f_0$,
it is typically not stable at the boundary points $0$ and $n$
 and hence it is not minimal in the sense of \thref{lem:lap}).

Counting up the number of \Right\ instructions between successive \Left\ instructions
at site~$2$ starting at the 5th \Left\ instruction,
\begin{align*}
  (R_0,R_1,R_2,R_3,\ldots)=(1,0,0,2,\ldots).
\end{align*}
Examining the presence of \Sleep\ instructions interspersed between the non-\Sleep\ instructions,
\begin{align*}
  B^0 &= (0,1),\qquad
  B^1 = (0),\qquad
  B^2 = (1),\qquad
  B^3 = (1,0,1),\ldots
\end{align*}
Thus the shapes for diagonals $0,\ldots,4$ are:
\begin{align}\label{eq:shapes}
  \begin{tikzpicture}[scale=.3,baseline=.25cm]
    \begin{scope}[fill=red]
      \fill (0,0) rectangle +(1,1);
      \fill (1,0) rectangle +(1,1);
      \fill (1,1) rectangle +(1,1);
      \draw[thick,gray] (0,0) grid[step=1] (2,2);
    \end{scope}
    \begin{scope}[shift={(7,0)},fill=orange]
      \fill (0,0) rectangle +(1,1);
      \draw[thick,gray] (0,0) grid[step=1] (1,2);
    \end{scope}
    \begin{scope}[shift={(13,0)},fill=yellow!95!black]
      \fill (0,0) rectangle +(1,1);
      \fill (0,1) rectangle +(1,1);
      \draw[thick,gray] (0,0) grid[step=1] (1,2);
    \end{scope}
    \begin{scope}[shift={(19,0)},fill=green!95!black]
      \fill (0,0) rectangle +(1,1);
      \fill (1,0) rectangle +(1,1);
      \fill (2,0) rectangle +(1,1);
      \fill (0,1) rectangle +(1,1);
      \fill (2,1) rectangle +(1,1);
      \draw[thick,gray] (0,0) grid[step=1] (3,2);
    \end{scope}
  \end{tikzpicture}
\end{align}
Here the shapes are colored to match up the diagonals in \eqref{eq:step1}.
Each cell of a certain color in \eqref{eq:step1} will infect a block of cells in the shape
above with the same color. 

Using the formula $\layer(j) = \ii[\big]{\sum_{i=0}^{j-1}R_i,\sum_{i=0}^{j}R_i}$,
the layers in step~$2$ consist of the columns
\begin{align*}
  \layer(0) = \ii{0,1},\qquad
  \layer(1) = \ii{1,1},\qquad
  \layer(2) = \ii{1,1},\qquad
  \layer(3) = \ii{1,3},\ldots
\end{align*}
A cell $(r,s)_1$ in diagonal~$j$ (i.e., $r+s=j$) infects cells in the columns found in $\layer(j)$.
It will infect all cells $(r,s)_2$ and some of the cells $(r,s+1)_2$ for $r\in\layer(j)$.
Here we show the layers one at a time. Layers~0, 1, and 3 each consist of a single copy of one
of the shapes in \eqref{eq:shapes}, while layer~2 is formed from two copies of the yellow shape,
one starting in row~$0$ and one in row~$1$:
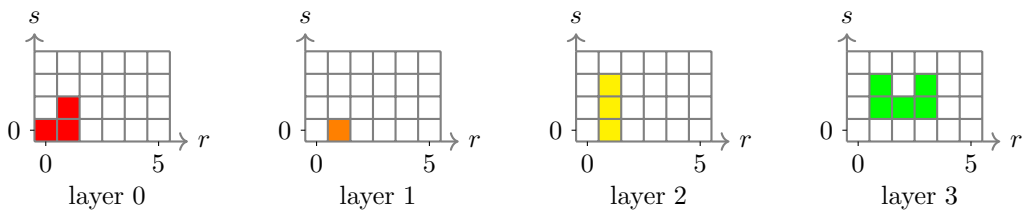
\begin{center}
  \begin{tikzpicture}[scale=.3]
  \begin{scope}[shift={(0,0)}]
  \fill[red] (0,0) rectangle +(1,1);
  \fill[red] (1,0) rectangle +(1,1);
  \fill[red] (1,1) rectangle +(1,1);
  \draw[thick,gray] (0,0) grid[step=1] (6,4);
  \draw[thick,gray,->] (6,0) -- ++(.8,0) node[black,right] {$r$};
  \draw[thick,gray,->] (0,4) -- ++(0,.8) node[black,above] {$s$};
  \draw (0.500000000000000,0)-- ++(0,-.2) node[below] {0};
  \draw (5.50000000000000,0)-- ++(0,-.2) node[below] {5};
  \draw (0,0.500000000000000)-- ++(-.2,0) node[left] {0};
  \node[anchor=west] at (1.1,-2.5) {layer~0};
  \end{scope}
    \begin{scope}[shift={(12,0)}]
  \fill[orange] (1,0) rectangle +(1,1);
  \draw[thick,gray] (0,0) grid[step=1] (6,4);
  \draw[thick,gray,->] (6,0) -- ++(.8,0) node[black,right] {$r$};
  \draw[thick,gray,->] (0,4) -- ++(0,.8) node[black,above] {$s$};
  \draw (0.500000000000000,0)-- ++(0,-.2) node[below] {0};
  \draw (5.50000000000000,0)-- ++(0,-.2) node[below] {5};
  \draw (0,0.500000000000000)-- ++(-.2,0) node[left] {0};
  \node[anchor=west] at (1.1,-2.5) {layer~1};
  \end{scope}
    \begin{scope}[shift={(24,0)}]
  \fill[yellow] (1,0) rectangle +(1,1);
  \fill[yellow] (1,1) rectangle +(1,1);
  \fill[yellow] (1,2) rectangle +(1,1);
  \draw[thick,gray] (0,0) grid[step=1] (6,4);
  \draw[thick,gray,->] (6,0) -- ++(.8,0) node[black,right] {$r$};
  \draw[thick,gray,->] (0,4) -- ++(0,.8) node[black,above] {$s$};
  \draw (0.500000000000000,0)-- ++(0,-.2) node[below] {0};
  \draw (5.50000000000000,0)-- ++(0,-.2) node[below] {5};
  \draw (0,0.500000000000000)-- ++(-.2,0) node[left] {0};
  \node[anchor=west] at (1.1,-2.5) {layer~2};
  \end{scope}
  \begin{scope}[shift={(36,0)}]
  \fill[green] (1,1) rectangle +(1,1);
  \fill[green] (2,1) rectangle +(1,1);
  \fill[green] (3,1) rectangle +(1,1);
  \fill[green] (1,2) rectangle +(1,1);
  \fill[green] (3,2) rectangle +(1,1);
  \draw[thick,gray] (0,0) grid[step=1] (6,4);
  \draw[thick,gray,->] (6,0) -- ++(.8,0) node[black,right] {$r$};
  \draw[thick,gray,->] (0,4) -- ++(0,.8) node[black,above] {$s$};
  \draw (0.500000000000000,0)-- ++(0,-.2) node[below] {0};
  \draw (5.50000000000000,0)-- ++(0,-.2) node[below] {5};
  \draw (0,0.500000000000000)-- ++(-.2,0) node[left] {0};
  \node[anchor=west] at (1.1,-2.5) {layer~3};
  \end{scope}
  \end{tikzpicture}
\end{center}
The infection set at step~$2$ is the union of these cells:
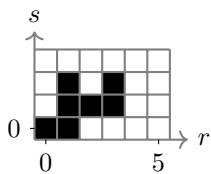
\begin{center}
\begin{tikzpicture}[scale=.3]
  \fill[] (0,0) rectangle +(1,1);
  \fill[] (1,0) rectangle +(1,1);
  \fill[] (1,1) rectangle +(1,1);
  \fill[] (1,2) rectangle +(1,1);
  \fill[] (2,1) rectangle +(1,1);
  \fill[] (3,1) rectangle +(1,1);
  \fill[] (3,2) rectangle +(1,1);
  \draw[thick,gray] (0,0) grid[step=1] (6,4);
  \draw[thick,gray,->] (6,0) -- ++(.8,0) node[black,right] {$r$};
  \draw[thick,gray,->] (0,4) -- ++(0,.8) node[black,above] {$s$};
  \draw (0.500000000000000,0)-- ++(0,-.2) node[below] {0};
  \draw (5.50000000000000,0)-- ++(0,-.2) node[below] {5};
  \draw (0,0.500000000000000)-- ++(-.2,0) node[left] {0};
\end{tikzpicture}
\end{center}
This set is the one pictured in step~$2$ in Figure~\ref{fig:layer.perc.version},
except that it is shifted to begin in column~$0$.

\section{The correspondence between ARW and layer percolation}
\label{sec:ARW.percolation.connection}

We are now ready to connect layer percolation to activated random walk, as
alluded to in Sections~\ref{sec:tree} and \ref{sec:example2}.
The idea is to define layer percolation by constructing the random
variables $R_0,R_1,\ldots$ and $B^0,B^1,\ldots$ that yield the infections
from step~$k-1$ to step~$k$ from the ARW instructions at site~$k$.
The layer widths $R_0,R_1,\ldots$ are given by the counts of \Right\ instructions
between successive \Left\ instructions; the Bernoulli random variables $B^j_i$ are determined
by the presence of \Sleep\ instructions.

The end result of this construction will be a correspondence given in \thref{bethlehem} between
what we call \emph{extended odometers} in ARW
and infection paths in the coupled layer percolation. The correspondence is between
a set of extended odometers $\eosos{n}(\instr, \sigma, u_0, f_0)$ on $\ii{0,n}$
and a set of length~$n$ infection paths
$\ip{n}(\instr, \sigma, u_0, f_0)$ in the coupled layer percolation
and is given by a map $\Phi$ that is a bijection up to identification of odometers that yield
the same final configuration.
We get started now on the task of defining
$\eosos{n}(\instr, \sigma, u_0, f_0)$,
$\ip{n}(\instr, \sigma, u_0, f_0)$, and  $\Phi$.
The proof is technical, but the ideas in it are simple and are apparent in the
example from Sections~\ref{sec:tree} and \ref{sec:example2}.

\subsection{Extended odometers}

As we described in Section~\ref{sec:tree}, we consider odometers on $\ii{0,n}$
stable on $\ii{1,n-1}$
with a given value $u_0$ at $0$ and given net flow $f_0$ from site~$0$ to site~$1$.
Our goal is to relate this set of odometers to
the set of infection paths of length~$n$ starting from $(0,0)_0$ in layer percolation.
But there is an obstacle not apparent in our example
in Sections~\ref{sec:tree} and \ref{sec:example2}, which is that
things can go wrong when $u_0$ is too small or $f_0$ is too large.
For example,
suppose we take $u_0=10$ and $f_0=20$. Then there are no odometers at all with the given properties,
because we cannot have net flow $\rt{u}{0}-\lt{u}{1}=20$ unless $\rt{u}{0}\geq 20$, which cannot
occur if $u(0)=10$. 
Since our odometers are to correspond to a set of infection paths in layer percolation
with no dependence on $u_0$ and $f_0$, we need to make its size the same (in distribution)
regardless of the choice of $u_0$ and $f_0$.

We carry this out by introducing \emph{extended odometers} that can take negative values.
We extend each stack of instructions to be two-sided.
We think of negative odometer values as representing the execution of instructions on the negative-index
portion of the instruction list, but with the reverse of their normal effects:
a $\Left$ instruction at site~$1$ takes a particle from $0$ to $1$ rather than sending
a particle from $1$ to $0$, for example. Thus we can produce an odometer satisfying $u_0=10$ and $f_0=20$
by making the odometer negative at site~$1$. These extended odometers have no useful interpretation
from the perspective of activated random walk, but by allowing them we regularize the set of
stable odometers and make them independent of the choice of $u_0$ and $f_0$.

We make this precise now.
For $v\geq 0$, let $\instr_v$ denote the two-sided list
of instructions at site~$v$, with instruction number~$i$ notated as $\instr_v(i)$
for $i\in\ZZ$, and with $\instr_v(0)=\Left$ for all $v$.
In Section~\ref{sec:sitewise} we defined $\Instr_v$ to be a one-sided list of i.i.d.\ instructions.
Here we think of $\instr_v$ as being any deterministic list of instructions with
$\instr_v(0)=\Left$. The correspondence
between ARW and layer percolation is deterministic, 
and probability plays no role here until Section~\ref{sec:Instr}.

If $u(v)\geq 0$, we define $\lt{u}{v}$ and $\rt{u}{v}$
as before as the number of \Left\ and \Right\ instructions, respectively,
in $\instr_v(1),\ldots,\instr_v(u(v))$. Note that we have $\lt{u}{v}=\rt{u}{v}=0$ if $u(v)=0$.
For $u(v)<0$, we define $\lt{u}{v}$ and $\rt{u}{v}$
as the negative of the number of \Left\ and \Right\ instructions, respectively,
in $\instr_v(u(v)+1),\ldots,\instr_v(0)$. Note that because $\instr_v(0)=\Left$,
we always have $\lt{u}{-1}=-1$ and $\rt{u}{-1}=0$. This is the point of
requiring $\instr_v(0)=\Left$, since it makes it so that extending our instruction lists
to be two-sided does not create new values of $i$ with $\lt{u}{i}=0$.

We define an \emph{extended odometer} on $\ii{a,b}$
as a function $u\colon\ZZ\to\ZZ$ that is zero off of $\ii{a,b}$.
We call an extended odometer \emph{stable}, either at a site or a set of sites, by the same criteria
given in \thref{def:stable} but with our extended definition of
$\lt{\cdot}{v}$ and $\rt{\cdot}{v}$.
And we define
$\eosos{n}(\instr, \sigma, u_0, f_0)$ to be the set of extended odometers $u$ on
$\ii{0,n}$ for the instructions $\instr$ and initial configuration $\sigma$
that satisfy $u(0)=u_0$ and $\rt{u}{0}-\lt{u}{1}=f_0$ and are stable on $\ii{1,n-1}$.
This is the set that we will embed in an instance of layer percolation.

We start by giving a criterion for an extended odometer $u$ on $\ii{0,n}$ to be
stable on $\ii{1,n-1}$. Essentially we are
just rewriting \thref{def:stable} in a convenient form.

\begin{lemma}\thlabel{lem:flow}
  Let $u$ be an extended odometer on $\ii{0,n}$.
  Let $f_v=\rt{u}{v}-\lt{u}{v+1}$, the net flow from $v$ to $v+1$ under the odometer.
  Let $s_v=\sum_{i=1}^v\1\{\instr_i(u(i))=\Sleep\}$.
  Then $u$ is stable on $\ii{1,n-1}$ if and only if
  \begin{align}\label{eq:lemflow}
    f_v = f_0 + \sum_{i=1}^v\abs{\sigma(i)} - s_v\quad\text{for all $v\in\ii{0,n-1}$.}
  \end{align}
\end{lemma}
\begin{proof}
  Suppose $u$ is stable on $\ii{1,n-1}$. Then from
  \thref{def:stable},
  \begin{align}\label{eq:f-f}
    f_i-f_{i-1} &= \abs{\sigma(i)}  -  \1\{\instr_i(u(i))=\Sleep\}
  \end{align}
  for all $i\in\ii{1,n-1}$. Summing this equation from $i=1$ to $i=v$ yields \eqref{eq:lemflow}.
  
  Conversely, suppose \eqref{eq:lemflow} holds. Then \eqref{eq:f-f} holds, which shows that
  the criteria for stability given in \thref{def:stable} hold.
\end{proof}

\subsection{The minimal odometer}
Our embedding of $\eosos{n}(\instr, \sigma, u_0, f_0)$ into layer percolation
will take the smallest odometer in $\eosos{n}(\instr, \sigma, u_0, f_0)$
to the infection path  $(0,0)_0\to(0,0)_1\to\cdots\to(0,0)_n$;
see the discussion of the \emph{minimal odometer} on page~\pageref{pg:min.odometer}.
We now define this minimal odometer before proving in \thref{lem:min.is.min} that it is
indeed the minimal element of $\eosos{n}(\instr, \sigma, u_0, f_0)$.

\begin{define}\thlabel{def:minimal.odometer}
The \emph{minimal odometer} $\mo$ of $\eosos{n}(\instr, \sigma, u_0, f_0)$ is the extended
odometer defined by the following inductive procedure.
First let $\mo(0)=u_0$. Now suppose that $\mo(v-1)$ has already been defined. We define
$\mo(v)$ to be the minimum integer such that
\begin{align*}
  \lt{\mo}{v} = \rt{\mo}{v-1} - f_0 - \sum_{i=1}^{v-1}\abs{\sigma(i)}.
\end{align*}
\end{define}
This procedure always makes $\instr_v(\mo(v))=\Left$ for $v\in\ii{1,n}$, since if $\instr_v(\mo(v))$
were anything other than \Left, we could decrease $\mo(v)$ by one without changing $\lt{\mo}{v}$
(note that this argument applies even when $\mo(v)\leq 0$).

The idea of the definition is that at each step we are choosing $\mo(v)$
to be the minimal value making $\mo$ stable at $v-1$.
If we ever made $\mo(v)$ a larger value, then we would execute
an additional \Right\ instruction or cause a particle to sleep at $v$,
either of which would require an additional \Left\ instruction at site~$v+1$.
\begin{lemma}\thlabel{lem:min.is.min}
  The odometer $\mo$ constructed in \thref{def:minimal.odometer} is an element
  of $\eosos{n}(\instr,\sigma,u_0,f_0)$ satisfying
  \begin{align*}
     \mo(v)\leq u(v)\text{ for all $v\in\ii{0,n}$}
  \end{align*}
  for any $u \in \eosos{n}(\instr,\sigma,u_0,f_0)$.
  The net flow from site~$v$ to $v+1$ under $\mo$ is
  \begin{align}\label{eq:min.o.constr}
    \rt{\mo}{v}-\lt{\mo}{v+1}=f_0+\sum_{i=1}^v\abs{\sigma(i)} \quad\text{for $v\in\ii{0,n-1}$.}
  \end{align}
\end{lemma}
\begin{proof}
  Equation~\eqref{eq:min.o.constr} is immediate from the construction of the minimal odometer.
  With $s_v=\sum_{i=1}^v\1\{\instr_i(\mo(i))=\Sleep\}$, we have $s_v=0$ for all $v\geq 0$
  since $\instr_i(\mo(i))=\Left$ for $i\in\ii{1,n}$.
  Hence \eqref{eq:lemflow}
  is satisfied, and \thref{lem:flow} proves that $\mo$ is stable on $\ii{1,n-1}$.
  By construction, the minimal odometer satisfies $\mo(0)=u_0$
  and by \eqref{eq:min.o.constr} we have $\rt{\mo}{0}-\lt{\mo}{1}=f_0$.
  Thus $\mo$ is an element of $\eosos{n}(\instr,\sigma,u_0,f_0)$.
  
  To prove the minimality of $\mo$, fix $u\in \eosos{n}(\instr,\sigma,u_0,f_0)$
  and proceed by induction on $v$. Assume that $\mo(v)\leq u(v)$ for some $v\in\ii{0,n-1}$,
  and we will prove that $\mo(v+1)\leq u(v+1)$. Since $u$ is stable on $\ii{1,n-1}$, 
  \thref{lem:flow} gives us
  \begin{align*}
    \rt{u}{v}-\lt{u}{v+1} &= \rt{u}{0}-\lt{u}{1} + \sum_{i=1}^v\abs{\sigma(i)} - \sum_{i=1}^v\1\{\instr_i(u(i))=\Sleep\}\\ &= f_0+ \sum_{i=1}^v\abs{\sigma(i)} - \sum_{i=1}^v\1\{\instr_i(u(i))=\Sleep\}\\
    &\leq f_0+ \sum_{i=1}^v\abs{\sigma(i)}=\rt{\mo}{v}-\lt{\mo}{v+1}.
  \end{align*}
  Thus $\lt{u}{v+1}\geq\rt{u}{v}-\rt{\mo}{v}+\lt{\mo}{v+1}$.
  Since $u(v)\geq \mo(v)$, we have $\rt{u}{v}-\rt{\mo}{v}\geq 0$, proving that $\lt{u}{v+1}\geq\lt{\mo}{v+1}$.
  If $\lt{u}{v+1}$ is strictly greater than $\lt{\mo}{v+1}$, then $u(v+1)>\mo(v+1)$.
  And if $\lt{u}{v+1}=\lt{\mo}{v+1}=\ell$, then since $\mo(v+1)$ was chosen as the minimum value making
  $\lt{\mo}{v+1}$ equal to $\ell$, we have $u(v+1)\geq \mo(v+1)$.
\end{proof}
We emphasize that despite being minimal in $\eosos{n}(\instr,\sigma,u_0,f_0)$,
the minimal odometer is not minimal in the sense of the least-action principle and is not typically
the true odometer.
While the true odometer stabilizing $\ii{0,n}$ is minimal out of all odometers stable on $\ii{0,n}$,
the minimal odometer of $\eosos{n}(\instr,\sigma,u_0,f_0)$ is stable only on $\ii{1,n-1}$
and is usually only an extended odometer, not an odometer.

\subsection{Statement of the correspondence}
\label{sec:correspondence.statement}
We are nearly ready to state the main result of Section~\ref{sec:ARW.percolation.connection}.
First we construct  layer percolation from the ARW instructions.
Recall that layer percolation can be represented as random variables
$(R_j)_{j\geq 0}$ and $(B^j_0,\ldots,B^j_{R_j})_{j\geq 1}$ for each $v\in\ii{1,n}$
that determine the infections from step~$v-1$ to $v$ (see \thref{rmk:what.is.layer.percolation}). 
\begin{define}[$\Theta_n$]
  We define a mapping of
  realizations of ARW to realizations of layer percolation, expressed as
  \begin{align*}
     \Theta_n\colon\Bigl((\instr_v)_{v\in\ii{0,n}},\, \sigma,\, u_0,\,f_0\Bigr)\mapsto 
     \Bigl((R_j(v))_{j\geq 0},\,(B^j_0(v),\ldots,B^j_{R_j}(v))_{j\geq 1}\Bigr)_{v\in\ii{1,n}},
  \end{align*}
  where $R_j(v)$ are deterministic nonnegative integers and $B^j_0(v),\ldots,B^j_{R_j}(v)$
  are deterministic values in $\{0,1\}$.
  So long as we are working with one value of $v$ at a time, we will suppress
  $v$ from the notation by fixing it and referring to
  $(R_j)_{j\geq 0}$ and $(B^j_0,\ldots,B^j_{R_j})_{j\geq 0}$ as the component
  of $\Theta_n(\instr,\sigma,u_0,f_0)$ for step~$v$.

Given $(\instr_v)_{v\in\ii{0,n}}$, $\sigma$, $u_0$, and $f_0$, let $\mo$ be the minimal
odometer of $\eosos{n}(\instr,\sigma,u_0,f_0)$. Fix $v\in\ii{1,n}$
and consider the instruction list
\begin{align}\label{eq:reduceddef}
  \instr_v\bigl(\mo(v)\bigr),\, \instr_v\bigl(\mo(v)+1\bigr),\,\instr_v\bigl(\mo(v)+2\bigr),\ldots,
\end{align}
which always starts with \Left. 
We define $\Theta_n(\instr,\sigma,u_0,f_0)$ by setting
$R_0,R_1,\ldots$ for step~$v$ as
the number of \Right\ instructions between successive \Left\ instructions in this list.
That is, taking the initial \Left\ instruction on the list as the $0$th \Left\ instruction,
$R_j$ is the number of \Right\ instructions occurring between the $j$th and $(j+1)$th \Left\ instruction.

Now fix $j$ and consider this portion of the list
from the $j$th to the $(j+1)$th \Left\ instruction with $R_j$ many \Right\ instructions between them.
In the $R_j+1$ regions of the string between each of these $R_j+2$ many instructions, there might or
might not be a string of \Sleep\ instructions.
We define $B^j_0,\ldots,B^j_{R_j}$ as indicators on the presence of a sleep instruction in
the respective slots.
Carrying this out for each $v$, we have defined $\Theta_n(\instr,\sigma,u_0,f_0)$.
\end{define}
Note that this construction matches our example in Section~\ref{sec:example2},
where we constructed layer percolation from the sitewise instructions \eqref{eq:ex.instructions}.
Also note that all objects so far are deterministic. We have constructed a deterministic
map from a realization of ARW to a realization of layer percolation, and soon we will establish
a deterministic bijection between the extended odometers for the realization of ARW
and the infection paths in the realization of layer percolation.
Only in Section~\ref{sec:Instr} will we add randomness, 
when we show that the image of $\Theta_n$ with random instructions is
layer percolation as defined in Section~\ref{sec:LP.ARW},
i.e., it consists of independent $\Geo(1/2)$ and $\Ber(\lambda/(1+\lambda))$
random variables for each $v\in\ii{1,n}$.

Next we give the map from extended odometers to infection paths:

\begin{define}[$\ip{n}(\instr,\sigma,u_0,f_0)$ and $\Phi_n$]
Let $\ip{n}(\instr,\sigma,u_0,f_0)$ be the set of infection paths of length~$n$ starting from
$(0,0)_0$ in the realization of layer percolation given by $\Theta_n(\instr,\sigma,u_0,f_0)$.
And let $\Phi\colon \eosos{n}(\instr,\sigma,u_0,f_0)\to\ip{n}(\instr,\sigma,u_0,f_0)$
be defined as follows.
For any extended odometer $u$ on $\ii{0,n}$, we define $\Phi(u)$ as the sequence of cells
 $\bigl((r_v,s_v)_v,\,0\leq v\leq n\bigr)$ given by
\begin{align*}
  r_v &= \rt{u}{v} - \rt{\mo}{v},\\
  s_v &= \sum_{i=1}^v\1\{\Instr_i(u(i))=\Sleep \}.
\end{align*}
It is not clear that the resulting sequence of cells
is an infection path, i.e., that in the realization of layer
percolation $\Theta_n(\instr,\sigma,u_0,f_0)$, we have $(r_{v-1},s_{v-1})_{v-1}\to(r_v,s_v)_v$ in 
for all $1\leq v\leq n$, but we will prove it as part of \thref{bethlehem}.
\end{define}

To summarize, we start with a realization of ARW on $\ii{0,n}$.
To form the corresponding realization of layer percolation $\Theta_n(\instr,\sigma,u_0,f_0)$, 
we consider the instruction list at site~$v$ starting at index $\mo(v)$,
where $\mo$ is the minimal odometer.
We break up the instruction list into blocks separated by \Left\ instructions.
The layer widths $R_0(v),R_1(v),\ldots$ are determined by the number of \Right\ instructions
in each block; the Bernoulli random variables $(B^j_i(v))$ determining the upward
spread of layer percolation are determined by the presence of \Sleep\ instructions.

Then, we consider $\eosos{n}(\instr,\sigma,u_0,f_0)$, the extended odometers stable on $\ii{1,n-1}$
with given value $u_0$ at site~$0$ and given net flow $f_0$ from site~$0$ to $1$,
and $\ip{n}(\instr,\sigma,u_0,f_0)$, the infection paths in the layer percolation $\Theta_n(\instr,\sigma,u_0,f_0)$.
The map $\Phi$ goes from $\eosos{n}(\instr,\sigma,u_0,f_0)$ to
$\ip{n}(\instr,\sigma,u_0,f_0)$ by taking an extended odometer $u$
to the sequence of cells $\bigl((r_v,s_v)_v,\,0\leq v\leq n\bigr)$ where
$r_v$ gives the number of extra right instructions
executed by odometer $u$ at position~$v$ compared to the minimal odometer, and $s_v$ gives the number
of particles that the odometer leaves sleeping at positions $1,\ldots,v$.

Can we expect $\Phi$ to be injective?
Suppose an extended odometer $u\in\eosos{n}(\instr,\sigma,u_0,f_0)$ leaves a particle sleeping
at site~$v$, i.e., it satisfies $\instr_v(u(v))=\Sleep$.
Suppose that the instruction at the next index $u(v)+1$ at site~$v$ is also \Sleep.
By our definition of $\Phi$, the value of $\Phi(u)$ is unchanged if we increase $u(v)$ by $1$.
Thus $\Phi$ is not injective. On the other hand, it makes sense to identify these two extended odometers,
because an odometer may take any index within a string of consecutive \Sleep\ instructions
without altering any of its properties from the perspective of \thref{def:stable}.
The following result shows that if we identify odometers differing in this way,
then $\Phi$ becomes a bijection:
\begin{prop}\thlabel{bethlehem}
The map $\Phi$ is a surjection from $\eosos{n}(\instr,\sigma,u_0,f_0)$
onto $\ip{n}(\instr,\sigma,u_0,f_0)$. If $u,u'\in\eosos{n}(\instr,\sigma,u_0,f_0)$
are two extended odometers with $\Phi(u)=\Phi(u')$, then for all $v\in\ii{0,n}$,
either $u(v)=u'(v)$
or $u(v)$ and $u'(v)$ are two indices in $\instr_v$ in a string of consecutive \Sleep\ instructions.
\end{prop}

It is a technical but not difficult task to prove this result.
The idea of the bijection is more apparent from an example than a proof, and we urge
our readers to consult the example from Sections~\ref{sec:tree} and \ref{sec:example2}.
The basic idea, visible in the example from Sections~\ref{sec:tree} and \ref{sec:example2},
is to match each choice of odometer value at site~$v$ with a collection of infections from step~$v-1$
to $v$ in layer percolation.

\subsection{Reduced instructions}
When we construct $\Theta_n(\instr,\sigma,u_0,f_0)$,
we first compute the minimal odometer $\mo$ from $\instr_0$, $\sigma$, $u_0$, and $f_0$.
Then from the portion of $\instr_v$ starting at index $\mo(v)$,
we compute the component of $\Theta_n(\instr,\sigma,u_0,f_0)$ for step~$v$.
It will be helpful to represent this portion of the instruction list in a different form:
\begin{define}\thlabel{def:reduced1}
We define the \emph{reduced instructions from $(\instr,\sigma,u_0,f_0)$} at site~$1\leq v\leq n$
as the two sequences
  $a_1,a_2,\ldots\in\{\Left,\Right\}$ and $b_1,b_2,\ldots\in\{0,1\}$ defined as follows.
  Let $\mo$ be the minimal odometer of $\eosos{n}(\instr,\sigma,u_0,f_0)$.
  Let $\mo(v)=i_1<i_2<\cdots$ be the indices of the non-\Sleep\ instructions in $\instr_v$
  starting at index~$\mo(v)$. We define $a_1,a_2,\ldots$
  to be $\instr_v(i_1),\,\instr_v(i_2),\ldots$, the subsequence of \Left\ and \Right\ instructions
  in $\instr_v$ starting at index~$\mo(v)$. Note that $a_1$ is always equal to \Left.
  Next, for $j\geq 1$ we set $b_j=1$ if $\instr_v$ contains a \Sleep\ instruction
  between indices $i_j$ and $i_{j+1}$ and $b_j=0$ otherwise.
\end{define}

It also makes sense to start with layer percolation and 
arrive at reduced instructions from the other direction.
\begin{define}\thlabel{def:reduced2}
  Let $R_0,R_1,\ldots$ and $B^0,B^1,\ldots$ be the values defining the infections from
  step~$v-1$ to step~$v$ in a realization of layer percolation
  (see \thref{rmk:what.is.layer.percolation}).
  We define the \emph{reduced instructions} corresponding to this step of layer percolation
  as the two sequences
  $a_2,a_3,\ldots\in\{\Left,\Right\}$ with $a_1=\Left$ and $b_1,b_2,\ldots\in\{0,1\}$ defined as follows.
  We form $a_1,a_2,\ldots$ as a \Left\ instruction followed by $R_0$ many \Right\ instructions,
  then a \Left\ instruction followed by $R_1$ many \Right\ instructions, and so on.
  To define $b_1,b_2,\ldots$, let
  $1=\ell_0<\ell_1<\cdots$ be the indices in $a_1,a_2,\ldots$ of \Left\ instructions.
  Then for each $j\geq 0$, we define
  \begin{align}\label{eq:Bb.def}
    \bigl( b_{\ell_j},\,b_{\ell_j+1},\ldots,b_{\ell_{j+1}-1}\bigr)
      &= \bigl(B^j_0,\ldots,B^j_{R_j}\bigr).
  \end{align}
  
  Similarly, we define the step of layer percolation corresponding to reduced instructions
  $a_2,a_3,\ldots\in\{\Left,\Right\}$ with $a_1=\Left$ and $b_1,b_2,\ldots\in\{0,1\}$ as follows.
  Let
  $1=\ell_0<\ell_1<\cdots$ be the indices in $a_1,a_2,\ldots$ of \Left\ instructions.
  Let $R_j$ be the number of \Right\ instructions in $a_1,a_2,\ldots$ 
  between indices $a_{\ell_j}$ and $a_{\ell_{j+1}}$, or equivalently $R_j=\ell_{j+1}-\ell_j-1$.
  And define $\bigl(B^j_0,\ldots,B^j_{R_j}\bigr)$ from $b_1,b_2,\ldots$
  according to \eqref{eq:Bb.def}.  
\end{define}
The following lemma is a simple exercise:
\begin{lemma}\thlabel{lem:reduced=RB}
  The maps in \thref{def:reduced2} between $(R_j)_{j\geq 0},\,(B^j)_{j\geq 0}$
  and $(a_j)_{j\geq 1},\,(b_j)_{j\geq 1}$ are inverses.
  Hence reduced instructions are in bijection with realizations of
  $(R_j)_{j\geq 0},\,(B^j)_{j\geq 0}$ and provide another way to specify a step of layer percolation.
\end{lemma}
Finally, we observe that \thref{def:reduced1,def:reduced2} are compatible with each other.
\begin{prop}\thlabel{prop:same.reduced}
  The reduced instructions from $(\instr,\sigma,u_0,f_0)$ at site~$v\in\ii{1,n}$
  are the same as the reduced instructions corresponding to step~$v$ of
  the layer percolation $\Theta_n(\instr,\sigma,u_0,f_0)$.
\end{prop}
\begin{proof}
  Let $a_1,a_2,\ldots$ and $b_1,b_2,\ldots$ be the reduced instructions from $(\instr,\sigma,u_0,f_0)$
  at site~$v$
  and let $a'_1,a'_2,\ldots$ and $b'_1,b_2',\ldots$ be the reduced instructions
  corresponding to step~$v$ of
  $\Theta_n(\instr,\sigma,u_0,f_0)$.
  Let $R_0,R_1,\ldots$ and $B^0,B^1,\ldots$ be the component of
  $\Theta_n(\instr,\sigma,u_0,f_0)$ for step~$v$.
  
  The sequence $a_1,a_2,\ldots$ is defined as the subsequence of non-\Sleep\ instructions
  within $\instr_v$ starting at index~$\mo(v)$.
  Meanwhile, $R_0,R_1,\ldots$ are defined as the number of \Right\ instructions
  between successive \Left\ instructions in $\instr_v$ starting at $\mo(v)$,
  and then $(a'_1,a'_2,\ldots)=(a_1,a_2,\ldots)$ by \thref{def:reduced2}.
  
  With $1=\ell_0<\ell_1<\cdots$ the indices in $a_1,a_2,\ldots$ of \Left\ instructions,
  we have $B^j_i=b_{\ell_j+i}$ because both are defined as indicators on the presence
  of a \Sleep\ instruction in the same portion of $\instr_v$. Then $b'_{\ell_j+i}$ for $0\leq i\leq R_j$
  is defined to be equal to $B^j_i$, demonstrating that $(b_1,b_2,\ldots)=(b'_1,b'_2,\ldots)$.
\end{proof}
\begin{example}\thlabel{ex:running}
  We return to the example presented in Sections~\ref{sec:tree} and \ref{sec:example2}.
  We take the positive-index portions of $\instr_0,\instr_1,\instr_2$ to be given by
  \eqref{eq:ex.instructions}, and we take $\sigma\equiv 1$, $u_0=20$,
  and $f_0=2$. Here are the site~$2$ instructions, with the portion starting
  at the minimal odometer index highlighted:  
  {\setlength{\tabcolsep}{0.5pt}
  \definecolor{Gray}{gray}{0.9}
  \newcolumntype{I}{>{\columncolor{Gray}\tt}l}
\newcolumntype{i}{>{\tt}l}
\begin{align*}
\begin{tabular}{iiiii|iiiii|IIIII|IIIII|IIIII|I}
  S&L&R&R&L&S&L&S&R&L&L&R&S&L&L&S&L&S&R&R&S&S&L&L&S&\ldots
\end{tabular}
\end{align*}
}
The reduced instructions from $(\instr,\sigma,u_0,f_0)$ at site~$2$ are
{\setlength{\tabcolsep}{0.5pt}
\newcolumntype{i}{>{\tt}l}
\begin{align}\label{eq:ex.reduced}
\begin{tabular}{w{l}{2em}iiiiiiiiil}
  $a$:&L&R&L&L&L&R&R&L&L&\ldots\\
  $b$:&0&1&0&1&1&0&1&0&1&\ldots
\end{tabular}
\end{align}
}
The component of $\Theta_n(\instr,\sigma,u_0,f_0)$ for site~$v$ is
\begin{gather*}
  (R_0,R_1,\ldots)=(1,0,0,2,0,\ldots),\\
  B^0 = (0,1),\ B^1 = (0),\ B^2=(1),\ B^3=(1,0,1),\ B^4=(0),\ldots
\end{gather*}
The reduced instructions corresponding to step~$2$ of $\Theta_n(\instr,\sigma,u_0,f_0)$
are also equal to \eqref{eq:ex.reduced}. The sequence $R_0,R_1,\ldots$ 
gives the length of runs of \Right\ instructions in $a_1,a_2,\ldots$, and $B^0,B^1,\ldots$
when concatenated form the sequence $b_1,b_2,\ldots$.
\end{example}

\subsection{Proof of the correspondence}
To prove \thref{bethlehem}, we first construct a bijection between 
infections in a step of layer percolation and locations within the reduced
instructions corresponding to this step.
We start by defining a set $\infections_v$ consisting of all infections from 
step~$v-1$ to step~$v$ in a realization of layer percolation.
\begin{align*}
  \infections_v &=\{(r,s,r',s')\colon (r,s)_{v-1}\to(r',s')_v\}.
\end{align*}
Taking $a_1,a_2,\ldots$ and $b_1,b_2,\ldots$ as the reduced instructions 
corresponding to this step of layer percolation,
we define a set $\opositions_v$ consisting of every possible position
in the reduced instructions, represented as follows:
\begin{align*}
  \opositions_v &=
  \bigl\{(i,z)\colon i\geq 1,\,z\in\{0,1\},\,z\leq b_i\bigr\}.
\end{align*}
In the definition of $\opositions_v$, 
the value of $i$ represents an index in the reduced instructions, while $z$ represents
whether a $\Sleep$ instruction is chosen, if it is available.

To give a sense of the bijection before we rigorously define it, we again
consider our running example. In the realization of layer percolation
described in Section~\ref{sec:example2}, we have $(1,1)_1\to(1,2)_2$.
The infections from $(1,1)_1$ stem from indices $15$ and $16$ at site~$2$ (see Case~3
in Section~\ref{sec:tree}), with the infection of $(1,2)_2$ corresponding
to $u(2)=16$. Looking at \thref{ex:running}, index~$16$ in $\instr_2$
is the \Sleep\ instruction after the 
fourth non-\Sleep\ instruction starting at index $\mo(2)=11$. 
Hence index~$16$ in $\instr_2$ corresponds to $(4,1)\in\opositions_2$.
This will be represented in the following lemma by a map $\Psi\colon\infections_v\to\opositions_v$
taking the infection $(1,1)_1\to(1,2)_2$ (represented as $(1,1,1,2)\in\infections_2$)
to $(4,1)\in\opositions_2$.
This map is surjective, but it is not bijective because all of the ``parallel'' infections
from cells along a common diagonal are mapped to the same element of $\opositions_v$.
\begin{lemma}\thlabel{lem:Psi}
  Let $a_1,a_2,\ldots$ and $b_1,b_2,\ldots$ be the reduced instructions
  corresponding to step~$v$ of layer percolation represented by
  $(R_j)_{j\geq 0}$ and $(B^j_0,\ldots,B^j_{R_j})_{j\geq 0}$, and consider
  the associated sets $\infections_v$ and $\opositions_v$.
  \begin{enumerate}[(a)]
    \item\label{i:Psi.welldef}
      For each $(r,s,r',s')\in\infections_v$, there
      is a unique index $j$ such that $a_1,\ldots,a_j$ contains
      $r+s+1$ $\Left$ and $r'$ $\Right$ instructions and $b_j\geq s'-s$.
  \end{enumerate}
  Thus we can define a map $\Psi_v$ on $\infections_v$ taking $(r,s,r',s')$
  to $(j,s'-s)$, where $j$ is the index given in \ref{i:Psi.welldef}.
  \begin{enumerate}[(a),resume]
    \item \label{i:Psi.surjective}
      The map $\Psi_v$ is a surjection from
      $\infections_v$ to $\opositions_v$.
    \item \label{i:Psi.preimages}
      For each $(j,z)\in\opositions_v$,
    \begin{align*}
      \Psi_v^{-1}\bigl((j,z)\bigr) = \bigl\{ (r,s,r',s')\colon \text{$r,s,s'\in\NN$ such that
        $r+s=q-1$ and $s'-s=z$}\bigr\},
    \end{align*}
    where $q$ and $r'$ are the number of $\Left$ and $\Right$ instructions, respectively,
    in $a_1,\ldots,a_j$.
  \end{enumerate}
\end{lemma}
\begin{proof}
  Let $1=\ell_0<\ell_1<\cdots$ be the indices of \Left\ instructions in $a_1,a_2,\ldots$.
  For $j\in \ii{\ell_{r+s},\,\ell_{r+s+1}-1}$ and no other values of $j$, there are exactly $r+s+1$ \Left\ 
  instructions in $a_1,\ldots,a_j$. According to \thref{def:reduced2}, the number of \Right\ instructions
  in $a_1,\ldots,a_{\ell_{r+s}}$ is $R_0+\cdots+R_{r+s-1}$, and the number of \Right\ instructions in
  $a_1,\ldots,a_{\ell_{r+s+1}-1}$ is $R_0+\cdots+R_{r+s}$.
  
  Suppose that $(r,s)_{v-1}\to(r',s')_v$. Then
  \begin{align*}
    R_0+\cdots+R_{r+s-1}\leq r'\leq R_0+\cdots+R_{r+s}
  \end{align*}
  by definition of layer percolation. Hence there is a unique $j\in\ii{\ell_{r+s},\,\ell_{r+s+1}-1}$
  such that $a_1,\ldots,a_j$ contains $r'$ \Right\ instructions, namely 
  \begin{align}\label{eq:baf.j}
    j=\ell_{r+s} + r' - (R_0+\cdots+R_{r+s-1}).
  \end{align}
  Thus we have found a unique index $j$ so that $a_1,\ldots,a_j$ contains $r'$ \Right\ instructions,
  within a unique stretch of indices $j$ so that $a_1,\ldots,a_j$ contains $r+s+1$ \Left\ instructions.
  To prove \ref{i:Psi.welldef}, it just remains to prove that for this choice of $j$
  we have $b_j=1$ when $s'-s\geq 1$. By \thref{def:reduced2}
  we have $b_j = B^{r+s}_{j-\ell_{r+s}}$.
  If $s'-s\geq 1$, then $B^{r+s}_{r'-(R_0+\cdots+R_{r+s-1})}=1$ by definition
  of layer percolation, and $B^{r+s}_{r'-(R_0+\cdots+R_{r+s-1})} = B^{r+s}_{j-\ell_{r+s}}$
  by \eqref{eq:baf.j}. Thus \ref{i:Psi.welldef} is proven, and $\Psi_v$ is well defined.
  
  To prove \ref{i:Psi.preimages}, let $(j,z)\in\opositions_v$.
  Let $q$ and $r'$ be the number of $\Left$ and $\Right$ instructions, respectively,
  in $a_1,\ldots,a_j$. 
  By \thref{def:reduced2},
  \begin{align*}
    R_0+\cdots+R_{q-2}\leq r'\leq R_0+\cdots+R_{q-1},
  \end{align*}
  and $r'-(R_0+\cdots+R_{q-2})=j-\ell_{q-1}$
  Suppose that $r+s=q-1$. Then $r'\in\layer_v(r+s)$, using the notation
  given in \eqref{eq:layer.def}. If $z=0$, then we have $(r,s)\to(r',s)$.
  Hence $(r,s,r',s)\in\infections_v$, and $\Psi_v(r,s,r',s)=(j,0)$
  since $a_1,\ldots,a_j$ contains $q=r+s+1$ \Left\ instructions and
  $r'$ \Right\ instructions. If $z=1$, then $b_j=1$ by definition of $\opositions_v$.
  By \thref{def:reduced2}, we have $B^{q-1}_{j-\ell_{q-1}}=b_j=1$.
  Hence $B^{r+s}_{r'-(R_0+\cdots+R_{r+s-1})}=1$, and so $(r,s,r',s+1)\in\infections_v$
  and $\Psi_v(r,s,r',s+1)=(j,1)$. This proves that $\Psi_v^{-1}\bigl((j,z)\bigr)$
  contains all $(r,s,r',s')$ such that $r+s=q-1$ and $s'-s=z$.
  For the converse, observe that if $r+s\neq q-1$, then $\Psi_v(r,s,r',s')=
  (j',s'-s)$ with $j'\neq j$, since $j'$ is by definition
  of $\Psi_v$ an index such that $a_1,\ldots,a_{j'}$ contains $r+s+1$ \Left\ instructions.
  And if $r+s=q-1$ but $s'-s\neq z$, of course $\Psi_v(r,s,r',s')\neq(j,z)$,
  since $\Psi_v(r,s,r',s')=(j',z')$ where $z'=s'-s$. This completes the proof
  of \ref{i:Psi.preimages}.
  
  To prove \ref{i:Psi.surjective}, by \ref{i:Psi.preimages}
  we need only assert that for every $(j,z)\in\opositions_v$, there exists
  $r,s,r',s'$ such that $r+s=q-1$ and $s'-s=z$.
  In fact, with $q,r'$ defined as above, there are exactly $q$ choices
  of $(r,s,r',s')$, obtained by taking $r\in\ii{0,q-1}$, then defining $s=q-r$
  and $s'=z+s$. And $q\geq 1$ since $a_1=\Left$.
\end{proof}

Proving \thref{bethlehem} amounts to repeatedly applying \thref{lem:Psi}.
We first need to define a map $\tau_v$ taking indices for the 
instructions at site~$v$ to indices of the reduced instructions.
Let $a_1,a_2,\ldots$ and $b_1,b_2,\ldots$ be the reduced instructions from $(\instr,\sigma,u_0,f_0)$
at site~$v$, and let $\mo(v)=i_1<i_2<\cdots$ be the indices of non-\Sleep\ instructions in $\instr_v$
as in \thref{def:reduced1}.
We now define $\tau_v(k)$ for $k\geq \mo(v)$.
If $\instr_v(k)\in\{\Left,\Right\}$, then $k$ is the index of the $j$th non-\Sleep\ instruction
at or after index $\mo(v)$ for some $j\geq 1$. Then we define $\tau_v(k)=(j,0)$.
If $\instr_v(k)=\Sleep$, then let $k'<k$ be the first index before $k$ of a \Left\ or \Right\ 
instruction. Note that $k'\geq \mo(v)$ since $\mo(v)$ is the index of a \Left\ instruction.
Then $\tau_v(k')$ has already been defined as $(j,0)$ for some $j\geq 1$, and we define
$\tau_v(k)$ to be $(j,1)$. It is easy to see that $\tau_v$ maps $\{\mo(v),\mo(v+1),\ldots\}$
surjectively onto $\opositions_v$, and that $\tau_v(k)=\tau_v(k')$ only when $k=k'$
or $k$ and $k'$ are indices within a string of consecutive \Sleep\ instructions in
$\instr_v$.

\begin{proof}[Proof of \thref{bethlehem}]
  First, we demonstrate that $\Phi$ maps an element  $u\in\eosos{n}(\instr,\sigma,u_0,f_0)$
  into $\ip{n}(\Instr,\sigma,u_0,f_0)$.
  Let $\bigl((r_v,s_v)_v,\,0\leq v\leq n\bigr)=\Phi(u)$.
  We have $r_0=0$ since $u(0)=u_0=\mo(0)$, and $s_0=0$ by definition. To show that 
  $\Phi(u)\in\ip{n}(\Instr,\sigma,u_0,f_0)$, then, we must show that $(r_{v-1},s_{v-1})_{v-1}\to(r_v,s_v)_v$
  for each $v\in\ii{1,n}$. By definition of $\Phi$,
  \begin{align*}
    r_{v-1} &= \rt{u}{v-1}-\rt{\mo}{v-1},\qquad r_v=\rt{u}{v}-\rt{\mo}{v},\\
    s_v-s_{v-1} &= \1\{\Instr_v(u(v))=\Sleep \}.
  \end{align*}
  Let $a_1,a_2,\ldots$ and $b_1,b_2,\ldots$ be the reduced instructions 
  from $(\instr,\sigma,u_0,f_0)$ at $v$. 
  Let $(j,z)=\tau_v(u(v))$, the element of $\opositions_v$ corresponding to $u(v)$.
  By definition of $\tau_v$, there are $\rt{u}{v}-\rt{u}{\mo}=r_v$
  \Right\ instructions and $\ell_v:=\lt{u}{v}-\lt{\mo}{v}+1$ \Left\ instructions
  in $a_1,\ldots,a_j$, and $z=\1\{\instr_v(u(v))=\Sleep\}=s_v-s_{v-1}$.
  Putting these facts together,
  \begin{align*}
    r_{v-1}-\ell_v &= \rt{u}{v-1} - \lt{u}{v}-\bigl(\rt{\mo}{v-1}-\lt{\mo}{v}\bigr)-1\\
      &= \Biggl(f_0 + \sum_{i=1}^{v-1}\abs{\sigma(i)} - s_{v-1}\Biggr)
        - \Biggl(f_0 + \sum_{i=1}^{v-1}\abs{\sigma(i)} \Biggr)-1 = -s_{v-1}-1,
  \end{align*}
  applying \thref{lem:flow} and \eqref{eq:min.o.constr} to get to the second line.
  Hence $r_{v-1}+s_{v-1}=\ell_v-1$. By \thref{lem:Psi}\ref{i:Psi.preimages},
  we have $(r_{v-1},s_{v-1},r_v,s_v)\in\Psi_v^{-1}\bigl((j,z)\bigr)$.
  In particular, $(r_{v-1},s_{v-1})_{v-1}\to (r_v,s_v)_v$.
  This shows that $\Phi(u)\in\ip{n}(\instr,\sigma,u_0,f_0)$.
  
  Next, we define a map $\chi$ from $\ip{n}(\instr,\sigma,u_0,f_0)$ back
  to $\eosos{n}(\instr,\sigma,u_0,f_0)$ that we will use to prove surjectivity of $\Phi$.
  Consider an infection path
  \begin{align}\label{eq:chi.apply}
    (0,0)_0=(r_0,s_0)_0\to\cdots\to(r_n,s_n)_n\in\ip{n}(\instr,\sigma,u_0,f_0).
  \end{align}
  Now we construct an odometer $u$ that we define as the image of this infection path
  under $\chi$. First set $u(0)=u_0$. Now fix $v\in\ii{1,n}$,
  and again let $a_1,a_2,\ldots$ and $b_1,b_2,\ldots$ be the reduced instructions at $v$.
  Let $(j,z)=\Psi_v(r_{v-1},s_{v-1},r_v,s_v)$. Define $u(v)$ as the smallest
  value making $\tau_v(u(v))=(j,z)$, which exists by surjectivity of $\tau_v$ onto
  $\opositions_v$. By definition of $\tau_v$, we have
  $\instr_v(u(v))=\Sleep$ if and only if $z=1$.
  By definition of $\Psi_v$, we have $z=s_v-s_{v-1}$, yielding
  \begin{align}\label{eq:vf1}
    s_v-s_{v-1}=\1\{\instr_v(u(v))=\Sleep\}.
  \end{align}
  Also by definition of $\Psi_v$, there are $r_{v-1}+s_{v-1}+1$ \Left\ instructions and $r_v$
  \Right\ instructions in $a_1,\ldots,a_j$.
  Since $\tau_v(u(v))=(j,z)$,
  \begin{align}\label{eq:uv9}
    \rt{u}{v} &= r_v + \rt{\mo}{v},\\
    \lt{u}{v} &= r_{v-1}+s_{v-1}+\lt{\mo}{v}.\label{eq:lv9}
  \end{align}
  These equalities hold for $v\in\ii{1,n}$,
  and \eqref{eq:uv9} also holds for $v=0$ since $u(0)=\mo(0)$ and $r_v=0$.
  
  Now we show that $u\in\eosos{n}(\instr,\sigma,u_0,f_0)$.
  For any $v\in\ii{1,n}$, by \eqref{eq:uv9} and \eqref{eq:lv9}
  \begin{align}\label{eq:chi.flow}
    \rt{u}{v-1}-\lt{u}{v} &= \rt{v-1}{\mo}-\lt{v}{\mo}-s_{v-1}
       = f_0+\sum_{i=1}^{v-1}\abs{\sigma(i)} - s_{v-1},
  \end{align}
  applying \thref{lem:min.is.min} for the last equality.
  By \eqref{eq:vf1}, we have
  \begin{align}\label{eq:sv9}
    s_v=\sum_{i=1}^v\1\{\instr_i(u(i))=\Sleep\}
  \end{align}
  in agreement with the definition of $s_v$ in \thref{lem:flow},
  which we can then apply to show that $u$ is stable on $\ii{1,n-1}$.
  We have $u(0)=u_0$ by construction. By the $v=1$ case of \eqref{eq:chi.flow},
  we have $\rt{u}{0}-\lt{u}{1}=f_0$. Hence $u\in\eosos{n}(\instr,\sigma,u_0,f_0)$,
  showing that $\chi$ does indeed map $\ip{n}(\instr,\sigma,u_0,f_0)$ back into
  $\eosos{n}(\instr,\sigma,u_0,f_0)$.
  From \eqref{eq:uv9} and \eqref{eq:sv9} and the definition of $\Phi$, 
  we see that $\Phi(u)$ is our original infection path
  \eqref{eq:chi.apply}. That is, $\Phi\circ\chi$ is the identity on $\ip{n}(\instr,\sigma,u_0,f_0)$.
  This proves that $\Phi$ is surjective.
  
  Finally, suppose $\Phi(u)=\Phi(u')=\bigl((r_v,s_v)_v,\,0\leq v\leq n\bigr)$ for
  $u,u'\in\eosos{n}(\instr,\sigma,u_0,f_0)$.
  From the definition of $\Phi$,
  \begin{align*}
    \rt{u}{v}&=\rt{u'}{v}\quad\text{for $v\in\ii{0,n}$,}
  \end{align*}
  and
  \begin{align*}
    \instr_v(u(v))=\Sleep \text{ if and only if } \instr_v(u'(v))=\Sleep\quad \text{for $v\in\ii{1,n}$.}
  \end{align*}
  Using these facts together with the stability of $u$ and $u'$ on $\ii{1,n-1}$,
  we have $\lt{u}{v}=\lt{u'}{v}$ for $v\in\ii{1,n}$. Thus at each site in $\ii{1,n}$,
  the odometers $u$ and $u'$ execute identical
  numbers of \Left\ and \Right\ instructions. Hence $u(v)$ and $u'(v)$
  can only differ by pointing to different \Sleep\ instructions in a consecutive block of them.
  \end{proof}

\subsection{Layer percolation constructed from ARW is correctly distributed}
\label{sec:Instr}  

Extending our definition of $\Instr_v$ from Section~\ref{sec:sitewise}, we let
$\bigl(\Instr_v(k),\,v\in\ZZ\bigr)$ be independent with
\begin{align*}
  \Instr_v(k)=\begin{cases} \Left&\text{with probability $\frac{1/2}{1+\lambda}$,}\\
    \Right&\text{with probability $\frac{1/2}{1+\lambda}$,}\\
    \Sleep&\text{with probability $\frac{\lambda}{1+\lambda}$,}
  \end{cases}
\end{align*}
for $k\neq 0$ and $\Instr_v(0)=\Left$.
To finish the correspondence, we show now that the layer percolation
$\Theta_n(\Instr,\sigma,u_0,f_0)$ constructed from these random instructions
is actually distributed as layer percolation as defined in Section~\ref{sec:LP.ARW}.
Note that we will leave the instructions at 0 as the arbitrary, deterministic list $\instr_0$,
which will prove technically helpful later on.

\begin{prop}\thlabel{prop:reduced.to.layerperc}
  Let
  $(R_j(v))_{j\geq 0}$ and $(B^j_i(v))_{j\geq 0,\,0\leq i\leq R_j}$ be the component
  of
  \begin{align*}
    \Theta_n\bigl( (\instr_0,\Instr_1,\ldots,\Instr_n),\, \sigma,\, u_0,\, f_0 \bigr)
  \end{align*}
  for step~$v$.
  For different choices of $v$, these random variables are independent. 
  For any particular $v\in\ii{1,n}$, the random variables $(R_j(v))_{j\geq 0}$
  are i.i.d.-$\Geo(1/2)$, and
  conditional on $(R_j(v))_{j\geq 0}$, the random variables $(B^j_i(v))_{j\geq 0,\,0\leq i\leq R_j}$
  are i.i.d.-$\Ber\bigl(\lambda/(1+\lambda)\bigr)$.
\end{prop}

\begin{proof}
  For any $v\geq 0$, let $\Fscr_v$ be the $\sigma$-algebra generated by $\Instr_1,\ldots,\Instr_v$,
  and observe that 
  the random variables $(R_j(u))_{j\geq 0}$ and $(B^j_i(u))_{j\geq 0,\,0\leq i\leq R_j}$ 
  for $u\leq v$ are measurable with respect to $\Fscr_v$.
  Thus it suffices to show that for any $v\geq 1$, the distribution
  of $(R_j(v))_{j\geq 0}$ and  $(B^j_i(v))_{j\geq 0,\,0\leq i\leq R_j}$
  conditional on $\Fscr_{v-1}$ is as stated.
  
  Let $\mo$ be the minimal odometer of $\eosos{n}(\Instr,\sigma,u_0,f_0)$.
  We claim that the instructions
  \begin{align}\label{eq:rand.seq}
    \Instr_v\bigl(\mo(v)\bigr),\,\Instr_v\bigl(\mo(v)+1\bigr),\,\Instr_v\bigl(\mo(v)+2\bigr)\ldots
  \end{align}
  used to generate $(R_j(v))_{j\geq 0}$ and $(B^j_i(v))_{j\geq 0,\,0\leq i\leq R_j}$
  are distributed conditional on
  $\Fscr_{v-1}$ as a \Left\ instruction followed by i.i.d.\ instructions with the usual distribution, namely
  \Left\ and \Right\ each with probability $1/2(1+\lambda)$ and \Sleep\ with 
  probability $\lambda/(1+\lambda)$. The lemma follows, by
  simple properties of i.i.d.\ sequences and by definition of
  $(R_j(v))_{j\geq 0}$ and $(B^j_i(v))_{j\geq 0,\,0\leq i\leq R_j}$.
  
  Now we establish the claim.
  For a two-sided sequence of instructions $\ldots,x_{-1},x_0,x_1,\ldots$ with $x_0=\Left$, for $\ell\in\ZZ$
  we refer to the $\ell$th instance of \Left\ in the sequence to mean
  the $\ell$th instance after index~$0$ when $\ell$ is positive, the $\abs{\ell}$th
  instance before index~$0$ when $\ell$ is negative, and $x_0$ itself when $\ell=0$.
  We define $T_\ell(\ldots,x_{-1},x_0,x_1,\ldots)$ as the shift of the sequence
  that puts the $\ell$th occurrence of \Left\ at position~$0$.
  
  Let $\ell_v=\rt{\mo}{v-1} - f_0 - \sum_{i=1}^{v-1}\abs{\sigma(i)}$
  and recall that $\mo(v)$ is the index of the $\ell_v$th
  \Left\ instruction in $\Instr_v$.
  Since $\ell_v$ is measurable with respect to $\Fscr_{v-1}$, and 
  conditioning on $\Fscr_{v-1}$ does not alter the distribution of $\Instr_v$,
  the claim is equivalent to showing that the sequence $\Instr_v$
  is invariant under the action of $T_\ell$ for any deterministic $\ell$.
  By symmetry it suffices to show this claim for $\ell\geq 1$,
  and by iteration it is enough to show it for $T_1$.
  And this is a simple technical fact following from the strong Markov property for an i.i.d.\ sequence.
\end{proof}

\begin{cor}\thlabel{cor:ip}
  For any initial configuration $\sigma$ on $\ii{0,n}$ and any $u_0,f_0\in\ZZ$,
  the set $\ip{n}(\Instr,\sigma,u_0,f_0)$ is distributed as the collection of length~$n$
  infection paths from $(0,0)_0$ in layer percolation.
  Furthermore the conclusion holds for random $u_0$ measurable with respect to $\Instr_0$.  
\end{cor}
\begin{proof}
  Since \thref{prop:reduced.to.layerperc} is stated for deterministic instructions at site~$0$,
  we can apply it to show that $\Theta_n(\Instr,\sigma,u_0,f_0)$ conditional on $\Instr_0$
  is distributed as layer percolation. Hence it is distributed unconditionally as layer percolation,
  and $\ip{n}(\Instr,\sigma,u_0,f_0)$ has the claimed distribution.
\end{proof}
Together, \thref{bethlehem,cor:ip} show that
the set of extended odometers $\eosos{n}(\Instr,\sigma,u_0,f_0)$
is (nearly) in bijection with the length~$n$ infection paths
in layer percolation.

Since the distribution of $\ip{n}(\Instr,\sigma,u_0,f_0)$ does not depend on $\sigma$, $u_0$,
and $f_0$, it follows from the definition of $\Phi$ that
the distribution of the extended odometers $\eosos{n}(\Instr,\sigma,u_0,f_0)$ 
depends on $\sigma$, $u_0$, and $f_0$ only in that they affect the minimal odometer.

We have taken care to allow $u_0$ to depend on $\Instr_0$ for the following reason.
All extended odometers in $\eosos{n}(\Instr,\sigma,u_0,f_0)$ are stable on $\ii{1,n-1}$
by definition. While it will take quite a bit of work to form odometers stable also at $n$ (see
Section~\ref{sec:box}), we can make them stable at $0$ by the right choice of $u_0$:
\begin{prop}\thlabel{prop:stable.at.0}
  Given initial configuration $\sigma$ on $\ii{0,n}$ and $f_0<0$,
  let $u_0$ be the index of the $(\sigma(0)-f_0)$th
  occurrence of \Left\ within $\Instr_0(1),\Instr_0(2),\ldots$. Then all 
  $u\in\eosos{n}(\Instr,\sigma,u_0,f_0)$ are stable at $0$.
\end{prop}
\begin{proof}
  Since $u(-1)=0$ by definition for $u\in\eosos{n}(\Instr,\sigma,u_0,f_0)$, the stability
  condition for $u$ at site~$0$ is that
  \begin{align*}
    \sigma(0)  -\lt{u}{0} - \rt{u}{0} + \lt{u}{1} =\1\bigl\{\Instr_0(u(0))=\Sleep\bigr\}.
  \end{align*}
  Since $u(0)=u_0$ and $\rt{u}{0}-\lt{u}{1}=f_0$, the extended odometer $u$ is stable
  at $0$ if and only if
  \begin{align*}
    \lt{u}{0} = \sigma(0) -f_0 -\1\bigl\{\Instr_0(u_0)=\Sleep\bigr\}.
  \end{align*}
  The proof is complete since our choice of $u_0$ makes this equation hold.
\end{proof}

\section{Basic properties of layer percolation}

We establish fundamental properties of infection in layer percolation that we will use throughout the
rest of the paper. In Section~\ref{sec:infection}, we prove that 
forward and backward infection sets are connected, in slightly different senses.
In Section~\ref{sec:bounds.on.infection.paths}, we translate bounds on the rows of an infection path into bounds on the columns. Additionally we give general lower and upper bounds on infection sets. We include a lower bound on the minimal odometer (which is associated with ARW rather than layer percolation), since it uses
similar techniques. In Section~\ref{sec:couplings} we describe how the law of layer percolation is basically invariant with respect to the starting cell, and we describe a coupling that establishes a duality between forward and reverse infection.
In Section~\ref{sec:greedy}, we stitch together infection paths of locally optimal growth in the second
coordinate (the row) and prove growth estimates in both row and column for the resulting paths.
Then in Section~\ref{sec:critical.density}, we apply these greedy paths to establish the existence
of a limiting speed $\crist$ for the maximum row infected and prove an exponential lower tail bound
for it.

\label{sec:layer.perc.misc}
\subsection{Connectivity of infection sets} \label{sec:infection}

\begin{figure}
  \centering
  \begin{tikzpicture}[scale=.28,forward/.style={blue!40!white}, backward/.style={red!40!white}, both/.style={red!60!blue!70}, fpath/.style={blue},bpath/.style={red},bothpath/.style={purple}]
    \input{Y.tikz}
  \end{tikzpicture}
  \begin{tikzpicture}[scale=.28,forward/.style={blue!40!white}, backward/.style={red!40!white}, both/.style={red!60!blue!70}, fpath/.style={blue},bpath/.style={red},bothpath/.style={purple}]
    \input{Z.tikz}
  \end{tikzpicture}
  \caption{Examples of forward (blue) and backward (red) infection sets at a  step.  \thref{prop:connectivity.forward} proves that each cell (except the bottom-left) in the forward infection set has another infected cell to its south or west. \thref{prop:connectivity.backward} proves that each cell (except the top-left) in the backward infection set has another cell to its west or northwest.}
  \label{fig:Y}
\end{figure}
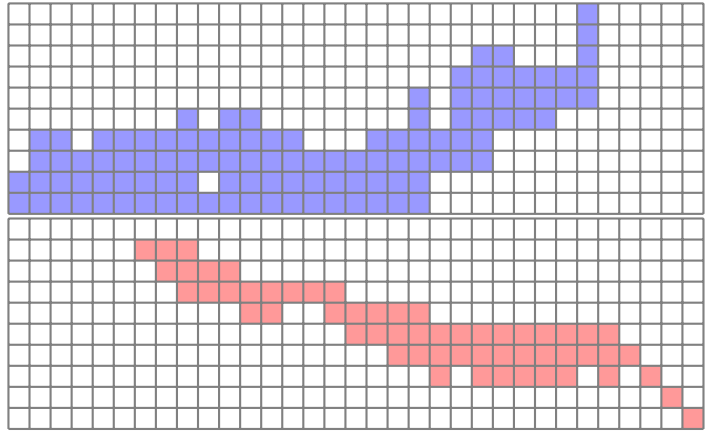

At the end of Section~\ref{sec:define}, we define $\infset[(r,s)_k]_n$ as the set of cells
at step~$k+n$ infected starting from $(r,s)_k$, and we define
$\binfset[(r,s)_k]_n$ as the set of cells at step~$k-n$ that infect $(r,s)_k$.
These sets have geometric properties that are apparent from examples (see Figure~\ref{fig:Y})
but which would be challenging even to state in terms of odometers.
We show now that each cell in an infection set other than the most southwestern has a neighbor immediately
to its south or west. Hence we can move within an infection
set from any cell to any other by taking steps in compass directions.

\begin{prop}\thlabel{prop:connectivity.forward} 
      Fix $(u,t)_0$ and consider the forward infection set $\infset_n=\infset[(u,t)_0]_n$ as defined in \eqref{eq:forward}.
      For any $(r,s)_n\in \infset_n$ other than the leftmost cell in row~$t$,
      either $(r-1,s)_n\in\infset_n$ or $(r,s-1)_n\in\infset_n$.
\end{prop}

\begin{proof}
  We proceed by induction and assume that the connectivity property holds for $\infset_{n-1}$.
  Now, we consider a cell $(r,s)_n\in\infset_n$ and show that $(r-1,s)_n\in\infset_n$
  or $(r,s-1)_n\in\infset_n$ or $(r,s)$ is the leftmost cell in row~$t$ in $\infset_n$.
  We break the proof into cases based on the source of infection
  of $(r,s)_n$ in $\infset_{n-1}$.
  \begin{enumerate}[label=\textbf{Case \arabic*:},leftmargin=*,labelindent=0pt,labelwidth=0pt,itemindent=20pt]
    \item $(r,s)_n$ is infected by some $(r',s-1)_{n-1}\in\infset_{n-1}$.\\
      If $(r',s-1)_{n-1}\to(r,s)_n$, then $(r',s-1)_{n-1}\to(r,s-1)_n$ too.
      Thus $(r,s-1)_n\in\infset_n$.
  \end{enumerate}
  In Cases~2--4, $(r,s)_n$ is infected by a cell in $\infset_{n-1}$
  in row~$s$. We let $(r',s)_{n-1}\in\infset_{n-1}$ be the leftmost such cell.
  \begin{enumerate}[label=\textbf{Case \arabic*:},leftmargin=*,labelindent=0pt,
                    labelwidth=0pt,itemindent=20pt,resume]
    \item $(r,s)_n$ is infected by $(r',s)_{n-1}\in\infset_{n-1}$ 
      and column~$r$ is not the minimum value in $\layer_n(r'+s)$.\\[4pt]
      Then $(r',s)_{n-1}\to (r-1,s)_n$, and hence $(r-1,s)_n\in\infset_n$.
    \item $(r,s)_n$ is infected by $(r',s)_{n-1}\in\infset_{n-1}$,
      column~$r$ is the minimum value in $\layer_n(r'+s)$,
      and $(r',s)_{n-1}$ is not the leftmost cell in row~$t$ of $\infset_{n-1}$.\\[4pt]
      Since $r$ is the minimum of $\layer_n(r'+s)$, it is the maximum
      of $\layer_n(r'+s-1)$. Thus $(r'-1,s)_{n-1}\to(r,s)_n$.
      Since $(r',s)_{n-1}$ is the leftmost cell in row~$s$ of $\infset_{n-1}$
      infecting $(r,s)_n$, we can conclude that $(r'-1,s)_{n-1}\notin\infset_{n-1}$.
      And since $(r',s)_{n-1}$ is not the leftmost cell in row~$t$ of $\infset_{n-1}$,
      the inductive hypothesis yields that $(r',s-1)_{n-1}\in\infset_{n-1}$.
      And $(r',s-1)_{n-1}\to(r,s-1)_n$, because $r\in\layer_n(r'+s-1)$.
      Thus we have shown that $(r,s-1)_n\in\infset_n$.
      
    \item $(r,s)_n$ is infected by $(r',s)_{n-1}\in\infset_{n-1}$, 
      column~$r$ is the minimum value in $\layer_n(r'+s)$, and $(r',s)_{n-1}$ is the leftmost
      cell in row $t$ of $\infset_{n-1}$.\\[4pt]
      We have $s=t$. Infections in row~$t$ can only come from
      cells in row~$t$, since $t$ is the minimum row present in these infections sets.
      Since $r$ is the minimum value in the layer infected by the leftmost
      cell in row~$t$ at step~$n-1$, the cell $(r,t)_n$ is the leftmost cell 
      in $\infset_n$ in row~$t$.\qedhere
  \end{enumerate}
\end{proof}

Similarly, any cell in the backward infection set other than the top-left one has a neighbor either
immediately to its left or upper-left:
\begin{prop}\thlabel{prop:connectivity.backward}
  Fix $(u,t)_{m+n}$ and consider the backward infection set $\binfset_n=\binfset[(u,t)_{m+n}]_n$
  as defined in \eqref{eq:backward}. For any $(r,s)_m\in\binfset_n$ other than the leftmost
  cell in row~$t$, if $r>0$ then either
  $(r-1,s)_m\in\binfset_n$ or $(r-1,s+1)_m\in\binfset_n$.
\end{prop}
\begin{proof}
  The proof is nearly the same as the previous one.
  We assume the connectivity property holds for $\binfset_{n-1}$ and extend it to 
  $\binfset_{n}$. Suppose that $(r,s)_m\in\binfset_n$. We must show
  that $(r-1,s)_m\in\binfset_n$ or $(r-1,s+1)_m\in\binfset_n$ or $(r,s)_m$ is the leftmost
  cell in row~$t$ of $\binfset_n$.
  Again, we consider cases based on the cell in $\binfset_{n-1}$ that $(r,s)_m$ infects.
  We assume throughout that $r>0$.
  
  \begin{enumerate}[label=\textbf{Case \arabic*:},leftmargin=*,labelindent=0pt,labelwidth=0pt,itemindent=20pt]
    \item $(r,s)_m$ infects some $(r',s+1)_{m+1}\in\binfset_{n-1}$.\\
      We have $r'\in\layer_{m+1}(r+s)$. Hence
      $(r-1,s+1)_m\to(r',s+1)_{m+1}$, which makes $(r-1,s+1)_m\in\binfset_{n}$.
  \end{enumerate}
  In Cases~2--4, $(r,s)_m$ infects a cell in $\infset_{n-1}$
  in row~$s$. We let $(r',s)_{m+1}\in\binfset_{n-1}$ be the leftmost such cell.
  \begin{enumerate}[label=\textbf{Case \arabic*:},leftmargin=*,labelindent=0pt,
                    labelwidth=0pt,itemindent=20pt,resume]
    \item $(r,s)_m\to(r',s)_{m+1}\in\binfset_{n-1}$ and column~$r'$ is the minimum value
      in $\layer_{m+1}(r+s)$.\\
      Then $r'$ is the maximum in $\layer_{m+1}(r+s-1)$. Since $r>0$, we have a
      cell $(r-1,s)_m$ also infecting $(r',s)_{m+1}$. Hence $(r-1,s)_m\in\binfset_n$.
    \item $(r,s)_m\to(r',s)_{m+1}\in\binfset_{n-1}$ and column~$r'$ is not the minimum
      value in $\layer_{m+1}(r+s)$ and $(r',s)_{m+1}$ is not the leftmost cell in row~$t$ in $\binfset_{n-1}$.\\[4pt]
      Since $r'$ is not the minimum of $\layer_{m+1}(r+s)$, we have $r'>0$ and
      $(r,s)_m\to(r'-1,s)_{m+1}$.
      Since $(r',s)_{m+1}$ was the leftmost cell in $\binfset_{n-1}$ infected by $(r,s)_m$,
      we conclude that $(r'-1,s)_{m+1}\notin\binfset_{n-1}$. Since $(r',s)_{m+1}$ is not
      the leftmost cell in row~$t$ in $\binfset_{n-1}$,  the inductive hypothesis yields
      $(r'-1,s+1)_{m+1}\in\binfset_{n-1}$.
      And $(r-1,s+1)_m\to(r'-1,s+1)_{m+1}$ since $r'-1\in\layer_{m+1}(r+s)$, yielding
      $(r-1,s+1)_m\in\binfset_m$.
    \item $(r,s)_m\to(r',s)_{m+1}\in\binfset_{n-1}$ and column~$r'$ is not the minimum
      value in $\layer_{m+1}(r+s)$ and $(r',s)_{m+1}$ is the leftmost cell in row~$t$ in 
      $\binfset_{n-1}$.\\[4pt]
      We have $s=t$. Since $r'$ is not the minimum value in $\layer_{m+1}(r+t)$, 
      we have $\layer_{m+1}(p+t)\subseteq\ii{0,r'-1}$ for all $p<r$.
      Since $(r',t)_{m+1}$ is the leftmost cell in row~$t$ in $\binfset_{n-1}$ and
      cells $(p,t)_m$ for $p<r$ only infect cells in columns $r'-1$ and below,
      cell $(r,s)_m=(r,t)_m$ is the leftmost cell in row~$t$ of $\binfset_n$.\qedhere
  \end{enumerate}
\end{proof}

\subsection{Bounds on infection paths}
\label{sec:bounds.on.infection.paths}

Branching processes with migration---where the population size is increased or decreased by a time-dependent
but deterministic quantity at each step---feature
prominently in our analysis of layer percolation. It is helpful to allow the population size to decrease below
zero, so that migration always shifts the population by a given amount and its expectation can be computed.
To accomplish this, we define a signed Galton--Watson process with migration, with the following dynamics:
The child distribution is as usual on the nonnegative integers.
If the population is nonnegative, each member gives birth to a number of children sampled
independently from the child distribution as usual,
and then the amount of migration (possibly negative) is added to get the size of the next generation.
If the population is negative, then each member gives birth to a number of children sampled
independently from the child distribution; the size of this next generation is the negative
of the resulting quantity of children, plus the amount of migration. We give the formal definition now.
\begin{define}\thlabel{def:GW}
We call $(X_j)_{j\geq 0}$ a \emph{signed Galton--Watson process with child distribution $L$
and migration $(e_j)_{j\geq 1}$}
if $X_{j+1}$ is distributed conditional on $X_j$ as
\begin{align*}
  \sgn(X_j)\sum_{i=1}^{\abs{X_j}}L_i + e_{j+1},
\end{align*}
where $L_1,L_2,\ldots$ are independent copies of $L$. The migration $e_j$ may be positive or negative.
If $e_j\geq 0$ for all $j$, we call it an \emph{immigration} process.
When the migration counts are all negative, we often call it a signed Galton--Watson process
with \emph{emigration $(f_j)_{j\geq 1}$} where $f_j=-e_j\geq 0$.
When the child distribution is $\Geo(1/2)$, we call the process 
a \emph{critical geometric branching process}.
\end{define}

The aim of our first two lemmas is to transfer bounds
on the rows of an infection path (i.e., their second coordinates) into bounds on the columns.
\begin{lemma}\thlabel{lem:smin.bound}
  Let $\smin_0,\ldots,\smin_n$ and $r_0$ be nonnegative integers.  Define $r_1,\ldots,r_n$ inductively
  by setting $r_{i}$ to be the minimum value in $\layer_{k+i}(r_{i-1}+\smin_{i-1})$.
  Then
  \begin{enumerate}[(a)]
    \item for any infection path $(u_0,s_0)_k\to\cdots\to(u_n,s_n)_{k+n}$ satisfying
      $u_0\geq r_0$ and
      $s_i\geq\smin_i$ for $0\leq i\leq n$, we have $u_i\geq r_i$ for $1\leq i\leq n$;
      \label{i:smin.bound.a}
    \item the random sequence $r_0+\smin_0,\ldots,r_n+\smin_n$ is a critical geometric branching process
      with immigration $\smin_i$ at step~$i$ for $i\geq 1$. \label{i:smin.bound.b}
  \end{enumerate}
\end{lemma}
\begin{proof}
  To prove \ref{i:smin.bound.a}, we proceed inductively, assuming that $u_i\geq r_i$
  and showing that $u_{i+1}\geq r_{i+1}$.
  The column of the leftmost cell infected by $(u_i,s_i)_{k+i}$ is the minimum value in
  $\layer_{k+i+1}(u_i+s_i)$. Since $u_i\geq r_i$ and $s_i\geq\smin_i$, this column is at least
  the minimum value of $\layer_{k+i+1}(r_i+\smin_i)=r_{i+1}$, proving that $u_{i+1}\geq r_{i+1}$.
  
  Fact~\ref{i:smin.bound.b} follows from the dynamics of layer percolation, since
  the minimum value in $\layer_{k+i+1}(r_i+\smin_i)$ is by definition the sum of $r_i+\smin_i$
  independent random variables with distribution $\Geo(1/2)$.
\end{proof}

The upper bound is nearly the same, and we omit its proof.
The only difference is that $r_i$ must be set to the \emph{maximum} value in
$\layer_{i+k}(r_{i-1}+\smax_{i-1})$, which is equal to the minimum
value in $\layer_{i+k}(r_{i-1}+\smax_{i-1}+1)$
\begin{lemma}\thlabel{lem:smax.bound}
  Let $\smax_0,\ldots,\smax_n$ and $r_0$ be nonnegative integers.  Define $r_1,\ldots,r_n$ inductively
  by setting $r_{i}$ to be the maximum value in $\layer_{k+i}(r_{i-1}+\smax_{i-1})$,
  or equivalently, the minimum value in $\layer_{k+i}(r_{i-1}+\smax_{i-1}+1)$.
  Then
  \begin{enumerate}[(a)]
    \item for any infection path $(u_0,s_0)_k\to\cdots\to(u_n,s_n)_{k+n}$ satisfying
      $u_0\leq r_0$ and
      $s_i\leq\smax_i$ for $0\leq i\leq n$, we have $u_i\leq r_i$ for $1\leq i\leq n$;
      \label{i:smax.bound.a}
    \item the random sequence $r_0+\smax_0+1,\ldots,r_n+\smax_n+1$ is a critical geometric branching process
      with immigration $\smax_i+1$ at step~$i$ for $i\geq 1$. \label{i:smax.bound.b}
  \end{enumerate}
\end{lemma}

Suppose we have a sequence of sets $\xi_0,\ldots,\xi_n$ 
where $\xi_i$ contains cells from step~$k+i$ (in typical applications these will be
infection sets). 
We say that a sequence of cells $(r_0,s_0)_{k},\ldots,(r_n,s_n)_{k+n}$ is a \emph{lower-left
bound} (resp.\ \emph{upper-right bound}) for $\xi_0,\ldots,\xi_n$ if
$r_i\leq r'_i$ and $s_i\leq s'_i$ (resp. $r'_i\leq r_i$ and $s'_i\leq s_i$)
for all $(r'_i,s'_i)_{k+i}\in\xi_i$ and all $0\leq i\leq n$.
\thref{lem:smin.bound,lem:smax.bound} give easy lower-left and upper-right bounds for a
sequence of infection sets. For the lower-left bound,
define the \emph{minimal infection path from $(r,s)_k$} as the infection path 
\begin{align*}
  (r_0,s)_k\to(r_1,s)_{k+1}\to\ldots
\end{align*}
that always chooses $(r_{n},s)_{n+k}$ to be the leftmost cell in row~$s$ infected
by $(r_{n-1},s)_{n+k-1}$. That is,
set $r_0=r$, and then inductively define
$r_{n}$ as the minimum value in $\layer_{n+k}(r_{n-1}+s)$.
This path is very much analogous to the minimal odometer. Indeed, under the bijection
to odometers given in Section~\ref{sec:ARW.percolation.connection}, 
it corresponds to picking the first permissible $\Left$ instruction, exactly as is done
when defining the minimal odometer. Hence it corresponds to a minimal
odometer started at site~$k$ for some initial odometer value at $k$ and net flow from $k$ to $k+1$.

\begin{lemma}\ \thlabel{lem:ll.bound}
  \begin{enumerate}[(a)]
    \item   If $u\geq r$ and $t\geq s$, then the minimal infection path from $(r,s)_k$
      is a lower-left bound
      for the infection sets $\infset[(u,t)_k]_0,\,\infset[(u,t)_{k}]_1,\ldots$.
      \label{i:ll.bound.a}
    \item Let $(r_0,s)_k\to(r_1,s)_{k+1}\to\cdots$ be the minimal infection path
      from $(r,s)_k$. Then the random sequence $r_0+s,r_1+s,\ldots$ is distributed
      as a critical geometric branching process with constant immigration $s$.
      \label{i:ll.bound.b}
  \end{enumerate}
\end{lemma}
\begin{proof}
  Since all infection paths
  from $(u,t)_k$ have all cells at or above row~$s$, the lemma follows
  by applying \thref{lem:smin.bound} with $\smin_i=s$ for all $i\geq 0$.
\end{proof}

For a generic upper-right bound, we define
the \emph{upper-right cell sequence from $(r,s)_k$} as the sequence of cells
\begin{align*}
  (r_0,s)_k,\;(r_1,s+1)_{k+1},\;(r_2,s+2)_{k+2},\ldots
\end{align*}
where $r_0=r$ and $r_{n+1}$ is the maximum of $\layer_{n+k+1}(r_n+s+n)$ for $n\geq 0$.
This sequence is not an infection
path: while $(r_n,s_n+n)_{n+k}$ infects $(r_{n+1},s+n)_{n+k+1}$,
it is not guaranteed to infect $(r_{n+1},s+n+1)_{n+k}$ unless $\lambda=\infty$.
One might think to modify the sequence into an infection path by
incrementing the row only when this infection occurs,
but the resulting infection path is not necessarily an upper-right bound for infection sets.

\begin{lemma}\ \thlabel{lem:ur.bound}
  \begin{enumerate}[(a)]
    \item If $r\geq u$ and $s\geq t$, then the upper-right cell sequence from $(r,s)_k$
      is an upper-right bound
      for the infection sets $\infset[(u,t)_k]_0,\,\infset[(u,t)_{k}]_1,\ldots$.\label{i:ur.bound.a}
    \item Let $(r_0,s)_k,\,(r_1,s+1)_{k+1},\,(r_2,s+2)_{k+2},\ldots$ be the upper-right cell sequence
      from $(r,s)_k$. Then the random sequence $(r_n+s+n+1)_{n\geq 0}$ is distributed
      as a critical geometric branching process with immigration $e_n=s+n+1$ for $n\geq 1$.\label{i:ur.bound.b}
  \end{enumerate}
\end{lemma}
\begin{proof}
  Let $\smax_n = s+n$ for $n\geq 0$.
  Since the row in an infection path can increase by at most one at each step, any infection path
  starting at $(u,t)_k$ has its row dominated by $\smax_0,\smax_1,\ldots$ at each step. Then an application of
  \thref{lem:smax.bound} proves the lemma.  
\end{proof}

Finally, we give a bound on the minimal odometer. We sneak the proof into this section
even though it applies directly to odometers and not to layer percolation because it uses
the same idea as all the other bounds of this section.
\begin{prop}\thlabel{prop:min.odometer.concentration}
  For a given initial configuration $\sigma$ on $\ii{0,n}$, initial odometer value $u_0$ at
  site~$0$, and net flow $f_0$ from site~$0$ to site~$1$, let  
  $\mo$ be the minimal odometer of $\eosos{n}(\Instr,\sigma,u_0,f_0)$.
  Let
  \begin{align*}
    e_i= -f_0-\sum_{v=1}^i\abs{\sigma(v)}
  \end{align*}
  and suppose that $\abs{e_i}\leq\emax$ for some $\emax\geq 1$.
  For some constants $c,C>0$ depending only on $\lambda$, it holds
  for all $t\geq 4\emax$ that
  \begin{align}\label{eq:min.odometer.bound}
    \P\Biggl[\abs[\bigg]{\rt{\mo}{j}-\biggl(\frac{u_0}{2(1+\lambda)}+\sum_{i=1}^je_i\biggr)} \geq t\Biggr]
      \leq C\exp\biggl(-\frac{ct^2}{n\bigl(n\emax + u_0+t\bigr)}\biggr)
  \end{align}
  for all $1\leq j\leq n$.
\end{prop}
\begin{proof}
  Let $Z_j=\rt{\mo}{j}+e_j$.  First, we claim that $Z_0,\ldots,Z_n$ is a critical geometric signed
  branching process with migration $(e_j)_{j\geq 1}$. Suppose that $Z_j$ is positive.
  By definition of the minimal odometer, $\mo(j+1)$ is equal to the index of the $Z_j$th
  \Left\ instruction at site~$j+1$, 
  and $\rt{\mo}{j+1}$ is equal to the number of \Right\ instructions prior to this left instruction.
  Since the counts of \Right\ instructions prior to each left instruction are independent with
  distribution $\Geo(1/2)$,
  this makes the distribution of $Z_{j+1}$ given $Z_j$ equal to the sum of $Z_j$ independent copies
  of $\Geo(1/2)$ plus an additional $e_{j+1}$.
  If $Z_j= 0$, then $\mo(j+1)=0$ by the definition of the minimal odometer (this is a consequence
  of our insistence that $\Instr_{j+1}(0)$ always be equal to \Left). Thus
  $Z_{j+1}=\rt{\mo}{j+1}+e_{j+1}=e_{j+1}$, again consistent with a signed Galton--Watson
  process with migration $e_{j+1}$ at this step.
  And if $Z_j<0$, then $\mo(j+1)$ is equal to the index of the $\abs{Z_j}$th occurence of \Left\ prior
  to index~$0$. Then $\rt{\mo}{j+1}$ is equal to the negative of the number of \Right\ instructions
  at indices $\mo(j+1)+1,\ldots,-1$, which is distributed as the sum of $\abs{Z_j}$ copies of 
  $\Geo(1/2)$, which again gives $Z_{j+1}$ the correct distribution given $Z_j$.
  
  Let $\mu_j=Z_0+\sum_{i=1}^j e_i$.
  We can now invoke \thref{prop:GW.emigration.concentration} conditional on $Z_0$ to obtain
  \begin{align*}
    \P\bigl[\abs{Z_j-\mu_j}\geq t\bigmid Z_0\bigr]
      &\leq C\exp\biggl(-\frac{ct^2}{j(j\emax+\abs{Z_0}+t)}\biggr).
  \end{align*}
  Applying the bounds $\abs{Z_0}\leq u_0+\emax$ and $j\leq n$ and taking expectations,
  \begin{align}
    \P\bigl[\abs{Z_j-\mu_j}\geq t\bigr]
      &\leq C\exp\biggl(-\frac{ct^2}{n\bigl((n+1)\emax + u_0+t\bigr)}\biggr).\label{eq:condZ0}
  \end{align}
  Now we have
  \begin{align*}
    &P\biggl[\abs[\bigg]{\rt{\mo}{j}-\biggl(\frac{u_0}{2(1+\lambda)}+\sum_{i=1}^je_i\biggr)}\geq t\biggr]\\
      &\qquad\qquad\qquad= \P\Biggl[ \abs[\bigg]{Z_j-\mu_j-e_j+\rt{\mo}{0}+e_0-\frac{u_0}{2(1+\lambda)}}\geq t\Biggr]\\
      &\qquad\qquad\qquad\leq \P\biggl[\abs[\big]{Z_j-\mu_j-e_j}\geq \frac{t}{2}\biggr]+
        \P\Biggl[ \abs[\bigg]{\rt{\mo}{0}+e_0-\frac{u_0}{2(1+\lambda)}}\geq \frac{t}{2}\Biggr]\\
        &\qquad\qquad\qquad\leq \P\biggl[\abs[\big]{Z_j-\mu_j}\geq \frac{t}{4}\biggr]+
        \P\Biggl[ \abs[\bigg]{\rt{\mo}{0}-\frac{u_0}{2(1+\lambda)}}\geq \frac{t}{4}\Biggr].
  \end{align*}
  We use $t\geq4\emax$ to arrive at the final inequality above.
  We then bound the first summand with \eqref{eq:condZ0} with $t/4$ substituted for $t$.
  And by Hoeffding's inequality, we bound the second by $2e^{-t^2/8u_0}$.
  Combining the two bounds with modified constants proves
  \eqref{eq:min.odometer.bound}.
\end{proof}
Typically $\emax=n$
and $u_0\leq c'n^2$ for some constant $c'$, in which case the previous proposition yields
\begin{align*}  
  \P\Biggl[\abs[\bigg]{\rt{\mo}{j}-\biggl(\frac{u_0}{2(1+\lambda)}+\sum_{i=1}^je_i\biggr)} \geq tn^{3/2}\Biggr]
  \leq C\exp\biggl(-\frac{ct^2}{1 + c'+t/\sqrt{n}}\biggr),
\end{align*}
which shows that $\rt{\mo}{j}$ has fluctuations on the order of $n^{3/2}$.

\subsection{Shifts and reversals}
\label{sec:couplings}

How does the distribution of the infection set $\infset[(r,s)_0]_n$ depend on
the cell $(r,s)_0$? And how do backward infection sets relate to forward ones?
In both cases, we can connect the different infection sets to each other by a coupling

To address the first question, the lower-left cell of $\infset[(r,s)_0]_n$ lies on the minimal
infection path starting from $(r,s)_0$. But in fact this is the only dependence:
all infection sets $\infset[(r,s)_0]_n$ have the same distribution once we shift
the column by the minimal infection path.
We prove this by constructing a coupling between $\infset[(r,s)_0]_n$ and $\infset[(0,0)_0]_n$.
When we work with two different instances of layer percolation, we use the symbol
$\to$ to denote infection as usual in the first instance and use $\too$
to denote infection in the second.

\begin{prop}\thlabel{prop:shift}
  Let $\infset[(0,0)_0]_n$ as usual denote
  the infection set from $(0,0)_0$ in layer percolation with parameter $\lambda>0$,
  and let $\altinfset[(r,s)_0]_n$ denote the infection set from $(r,s)_0$ in a second instance of layer
  percolation.
  Let
  \begin{align}\label{eq:shift.mip}
    (r,s)_0=(r_0,s)_0\too(r_1,s)_1\too(r_2,s)_2\to\cdots
  \end{align}
  be the minimal infection path starting from $(r,s)_0$ in the second layer
  percolation.
  There exists a coupling of the two instances such that \eqref{eq:shift.mip} is independent
  of the first layer percolation and
  \begin{align}\label{eq:shift.ip}
    \altinfset[(r,s)_0]_n = \bigl\{ (r'+r_n,\,s'+s)_n\colon (r',s')_n\in\infset[(0,0)_0]_n\bigr\}
  \end{align}
  for all $n\geq 0$.
\end{prop}

\begin{proof}
  Observe that $\altinfset[(r,s)_0]_0=\{(r,s)_0\}$ and $\infset[(0,0)_0]_0=\{(0,0)_0\}$,
  consistent with \eqref{eq:shift.ip} for $n=0$.
  Proceeding inductively, we suppose we have a coupling such that \eqref{eq:shift.ip}
  holds for $n$, and we extend the coupling to step~$n+1$.
  We write $(a,b)_n\to(a',b')_{n+1}$ to denote infection in the first layer percolation
  and $(a,b)_n\too(a',b')_{n+1}$ to denote infection in the second.
  
  Let $R_0,R_1,\ldots$ and $B^0,B^1,\ldots$ be the random variables from Section~\ref{sec:define}
  used to define the infections in the first layer percolation from step~$n$ to step~$n+1$
  (see \thref{rmk:what.is.layer.percolation}).
  To form the analogous random variables $\widebar{R}_0,\widebar{R}_1,\ldots$ and $\widebar{B}^0,
  \widebar{B}^1,\ldots$ for the second layer of percolation, we just insert new independent
  random variables to the beginning of $R_0,R_1,\ldots$ and $B^0,B^1,\ldots$.
  More specifically, for $i<r_n+s$ we let $\widebar{R}_i$ be a random variable
  independent of all else with distribution $\Geo(1/2)$
  and we let $\widebar{B}^i=(\widebar{B}^i_0,\ldots,\widebar{B}^i_{\widebar{R}_i})$
  be independent wih common distribution $\Ber\bigl(\lambda/(1+\lambda)\bigr)$.
  For $i\geq r_n+s$, let $\widebar{R}_i=R_{i-r_n-s}$ and $\widebar{B}^i=B^{i-r_n-s}$.
  
  Now, we claim that
  \begin{align}\label{eq:ipext}
    (r',s')_n\to(r'',s'')_{n+1} \Longleftrightarrow (r'+r_n,\,s'+s)_n\too(r''+r_{n+1},\,s''+s)_{n+1}.
  \end{align}
  Indeed, we have
  \begin{align*}
    r_{n+1} = \widebar{R}_0+\cdots+\widebar{R}_{r_n+s-1},
  \end{align*}
  since $r_{n+1}$ is the first column in layer~$r_n+s$ at step~$n+1$ in the second layer percolation.
  Thus layer~$r'+s'+r_n+s$ in the second layer percolation is equal to layer~$r'+s'$
  in the first layer percolation shifted by $r_{n+1}$. Since $\widebar{B}^{r'+s'+r_n+s}=B^{r_n+s}$,
  the cell $(r',s')_n$ in the first layer percolation infects the same cells
  as $(r'+r_n,s'+s)_n$, shifted to the right by $r_{n+1}$ and upward by $s$, confirming
  \eqref{eq:ipext}.
  
  Combining \eqref{eq:shift.ip} with \eqref{eq:ipext} shows that \eqref{eq:shift.ip} holds 
  with $n$ replaced by $n+1$, thus extending the induction.
\end{proof}
Combining \thref{lem:ll.bound}\ref{i:ll.bound.b} with the previous proposition yields
\begin{cor}\thlabel{cor:shift.branch}
  Fix a cell $(r,s)_0$ and
  let $(Z_n)_{n\geq 0}$ be a critical geometric branching process with constant immigration $s$
  starting from $Z_0=r+s$, independent of layer percolation. Then
  \begin{align*}
    \infset[(r,s)]_n\eqd  \bigl\{ (r'-s+Z_n,\,s'+s)_n\colon (r',s')_n\in\infset[(0,0)_0]_n \bigr\}.
  \end{align*}
\end{cor}

Next, we describe the backward infection set $\binfset[(u,t)_m]_n$ as a transformed
version of the forward infection set $\infset[(0,0)_0]_n$.
The coupling connecting these two infection sets can hold only for a finite time,
unlike the coupling in \thref{prop:shift}, since for a fixed $u$ and $t$, the set
$\binfset[(u,t)_m]_n$ will be empty for sufficiently large $n$.

The key idea for the coupling making one layer percolation the reverse of another
is to swap $\Left$ and $\Right$ instructions. Using the material of 
Section~\ref{sec:ARW.percolation.connection} to interpret the two coupled instances
of layer percolation in terms of odometers, the original instance represents 
odometers stable on the interior of an interval with a fixed initial value $u_0$ and left-to-right net 
flow $f_0$ on the left endpoint,
while the second represents odometers with fixed initial value and right-to-left net flow
on the right endpoint.

\begin{coupling}\thlabel{cp:reverse}
  Fix $u,t,m\geq 0$. Consider layer percolation together with 
  an independent signed critical geometric branching process
  $(Z_k)_{k\geq 0}$ with emigration~$t$ at each step and $Z_0=u$.
  We couple these processes with another instance of layer percolation,
  with infections from step~$n$ to step~$n+1$ in the first layer percolation
  coupled with those from step~$m-n-1$ to step~$m-n$ in the second, for all
  $0\leq n< m$.

  Suppose we have constructed the coupling up to step~$n$ in the first layer percolation
  and step~$m-n-1$ in the second.
  Now we construct the coupling for the next step. 
  If $Z_n<0$, then we just allow the two layer percolations and the branching process $Z_{n+1}$ to evolve
  independently  
  When $Z_n\geq 0$, we construct the coupling as follows.
  Start with the reduced instructions
  $a_1,a_2,\ldots$ and $b_1,b_2,\ldots$ corresponding to
  step~$n+1$ of the first layer percolation (see \thref{def:reduced2}). 
  To make reduced instructions corresponding to 
  step~$m-n$ for the second layer percolation, prepend random instructions 
  $\Left$ or $\Right$ with equal probability
  one at a time to the left of the list starting with $a_1,a_2,\ldots$ until
  $Z_n$ $\Left$ instructions have been added.
   (If $Z_n=0$, then no instructions
  are added.) Let $A$ be the number of added instructions. Note that our new
  list always starts with a $\Left$ instruction, since we stop prepending instructions
  after adding $Z_n$ $\Left$ instructions.  Now, take our new list
  consisting of $A$ new instructions followed by $a_1,a_2,\ldots$, and for all but the
  first instruction change $\Left$ instructions to $\Right$ instructions
  and vice versa to produce our new list $a'_1,a'_2,\ldots$.
  From the construction it is evident that $a'_1=\Left$ and that $a'_2,a'_3,\ldots$ consists of independent
  uniformly random instructions $\Left$ and $\Right$.
  
  Define $Z_{n+1}$ as the number of $\Left$ instructions contained in $a'_2,\ldots,a'_A$ minus $t$.
  Since the $A$ instructions added to the beginning of $a_1,a_2,\ldots$
  were independent of $a_1,a_2,\ldots$,
  the value of $Z_{n+1}$ is independent of $a_1,a_2,\ldots$ as desired. And since the number
  of $\Left$ instructions contained in $a'_2,\ldots,a'_A$ was the number of added $\Right$ instructions
  prior to the $\Left$-$\Right$ swap before the $Z_n$th $\Left$ instruction, the value of $Z_{n+1}$ is
  the next step in a branching process distributed as desired.
  
  Finally, we define $b'_1,b'_2,\ldots$ by prepending $A$ random
  $\Ber(\lambda/(1+\lambda))$ random variables
  to the left of $b_1,b_2,\ldots$. 
  Now define step~$m-n$ of the second layer percolation from $a'_1,a'_2,\ldots$
  and $b'_1,b'_2,\ldots$ (see \thref{lem:reduced=RB}). This creates a valid coupling of 
  the two layer percolations and $(Z_j)_{j=0}^{n+1}$.  
\end{coupling}

\begin{prop}\thlabel{cor:reverse}
  Consider \thref{cp:reverse} for given $u,t,m\geq 0$.
  Let $\infset_n=\infset[(0,0)_0]_n$
  denote the infection set for the first layer percolation
  and $\binfset_n=\binfset[(u,t)_m]_n$ the backward infection set for the second layer percolation.
  For all $0\leq n \leq m$,
  \begin{align}\label{eq:reverse.set}
    \text{if $Z_n\geq 0$, then } \binfset_n = \bigl\{(Z_n+r+s,\,t-s)_{m-n} \colon (r,s)_n\in\infset_n,\,s\leq t \bigr\}.
  \end{align}
\end{prop}
\begin{proof}
  Let $\to$ indicate infection in the first layer percolation
  and $\too$ in the second. We claim that the following statement holds for all $0\leq n<m$:
  if $Z_{n+1}\geq 0$, then
  for all $r,r'\geq 0$ and $0\leq s,s'\leq t$,
  \begin{align}\label{eq:infection.forward}
    (r,s)_n&\to(r',s')_{n+1}
  \end{align}
  if and only if
  \begin{align}\label{eq:infection.backward}
    (Z_{n+1}+r'+s',\,t-s')_{m-n-1}&\too (Z_n+r+s,\,t-s)_{m-n}.
  \end{align}
  Indeed, suppose that \eqref{eq:infection.forward} holds. 
  Let
  \begin{align*}
    \Psi\colon\infections_{n+1}\to\opositions_{n+1}\qquad\text{and}\qquad
    \Psi'\colon\infections'_{m-n}\to\opositions'_{m-n}
  \end{align*}
  be the maps
  from infections to reduced instruction locations
  defined in \thref{lem:Psi}, based on the first layer percolation for
  $\Psi$ and the second for $\Psi'$.
  Let $(a_j)_{j\geq 1},\,(b_j)_{j\geq 1}$ and $(a'_j)_{j\geq 1},\,(b'_j)_{j\geq 1}$ 
  be the reduced instructions corresponding
  to the coupled steps of the two layer percolations, as in \thref{cp:reverse}.
  Let $(i,z)=\Psi\bigl((r,s,r',s')\bigr)$. By definition of $\Psi$, 
  there are $r+s+1$ $\Left$ instructions and $r'$ $\Right$ instructions 
  in $a_1,\ldots,a_i$, and $z=s'-s$.
  
  Now consider location $(A+i,z)$, which is a valid location in $\opositions'_{m-n}$
  since $z\leq b_i=b'_{A+i}$. 
  We claim that there are $Z_{n+1}+r'+t+1$ $\Left$ instructions and
  $Z_n+r+s$ $\Right$ instructions in $a'_1,\ldots,a'_{A+i}$.
  First, we dispense with the special case $A=0$, which by our assumption $Z_{n+1}\geq 0$
  can happen only when $t=Z_n=0$. In this case we also have $Z_{n+1}=0$.
  The instructions $a'_1,\ldots,a'_{A+i}$ are obtained by swapping $a_1,\ldots,a_i$ except for
  the first instruction, giving us
  $r'+1$ $\Left$ instructions and $r+s$ $\Right$ instructions.
  
  When $A>0$, we break the computation into four pieces, 
  the number of $\Left$ and $\Right$ instructions in each of
  $a'_1,\ldots,a'_A$ and $a'_{A+1},\ldots,a'_{A+i}$:
  \begin{enumerate}[(i)]
    \item $1+Z_{n+1}+t$ $\Left$ instructions in $a'_1,\ldots,a'_A$\\[2pt]
      By definition of $Z_{n+1}$, there are $Z_{n+1}+t$ $\Left$ instructions in
       $a'_2,\ldots,a'_A$. And also $a'_1=\Left$.\label{i:leftA}
    \item $r'$ $\Left$ instructions in $a'_{A+1},\ldots,a'_{A+i}$\\[2pt]
      This comes from the number of $\Right$ instructions in $a_1,\ldots,a_i$.\label{i:leftAi}
    \item $Z_n-1$ $\Right$ instructions in $a'_1,\ldots,a'_A$\\[2pt]
      We constructed $a'_1,\ldots,a'_A$ by adding random
  instructions until we obtained $Z_n$ $\Left$ instructions and then swapping all but one of them
  to be $\Right$ instructions.\label{i:rightA}
    \item $r+s+1$ $\Right$ instructions in $a'_{A+1},\ldots,a'_{A+i}$\\[2pt]
      This is the number of $\Left$ instructions in $a_1,\ldots,a_i$.\label{i:rightAi}
  \end{enumerate}
  This completes the proof that there are $Z_{n+1}+r'+t+1$ $\Left$ and
  $Z_n+r+s$ $\Right$ instructions in $a'_1,\ldots,a'_{A+i}$.
  
  By \thref{lem:Psi}\ref{i:Psi.preimages},
  \begin{align*}
    \bigl(Z_{n+1}+r'+s',\ t-s',\ Z_n+r+s,\ t-s\bigr)\in\Psi'^{-1}(A+i,z).
  \end{align*}
  Thus we have shown that \eqref{eq:infection.backward} holds.
  
  In the other direction, suppose that \eqref{eq:infection.backward} holds
  and let
  \begin{align*}
    (j,z)=\Psi'\bigl((Z_{n+1}+r'+s',\,t-s',\,Z_n+r+s,\,t-s)\bigr).
  \end{align*}
  Thus $a'_1,\ldots,a'_j$ contains $Z_{n+1}+r'+t+1$ $\Left$ instructions
  and $Z_n+r+s$ $\Right$ instructions.
  Let $i=j-A$. We claim that
  there are $r+s+1$ $\Left$ instructions and $r'$ $\Right$ instructions 
  in $a_1,\ldots,a_i$.
  
  If $A=0$, then as before $t=Z_n=Z_{n+1}=0$.
  The instructions $a'_2,\ldots,a'_j$ are swapped versions of $a_2,\ldots,a_i$
  while $a_1=a'_1=\Left$, giving us $r+s+1$ $\Left$ instructions and
  $r'$ $\Right$ instructions in $a_1,\ldots,a_i$. If $A>0$, then we subtract off
  the number of $\Left$ and $\Right$ instructions in $a_1',\ldots,a'_A$ (see
  \ref{i:leftA} and \ref{i:rightA}) from the number in $a_1',\ldots,a'_j$
  to find that there are $Z_{n+1}+r'+t+1-(1+Z_{n+1}+t)=r'$ $\Left$ and $Z_n+r+s-(Z_n-1)=r+s+1$
  $\Right$ instructions in $a'_{A+1}+\cdots+a'_j$.
  The instructions $a_1,\ldots,a_i$ are these instructions swapped, proving the claim.
  
  Now we have shown that
  there are $r+s+1$ $\Left$ instructions and $r'$ $\Right$ instructions 
  in $a_1,\ldots,a_i$.
  Also, since $b_i=b'_j$ and $z\leq b'_j$, we have $z\leq b_i$ which shows $(i,z)$
  is a valid position in the image of $\Psi$.
  By \thref{lem:Psi}\ref{i:Psi.preimages},
  \begin{align*}
    (r,s,r',s')\in\Psi^{-1}(i,z),
  \end{align*}
  proving that \eqref{eq:infection.forward} holds.

  Now that the equivalence of \eqref{eq:infection.forward} and \eqref{eq:infection.backward}
  is shown, we prove \eqref{eq:reverse.set}. We proceed by induction. The statement holds
  for $n=0$. Now suppose it holds for $n$, and we will extend it to $n+1$.
  If $Z_{n+1}<0$, there is nothing to prove. If $Z_{n+1}\geq 0$, then $Z_n\geq 0$ as well
  by the dynamics of signed branching processes with emigration, and thus
  \begin{align}\label{eq:ihbinfset}
    \binfset_n = \bigl\{(Z_n+r+s,\,t-s)_{m-n} \colon (r,s)_n\in\infset_n,\,s\leq t \bigr\}
  \end{align}
  by the inductive hypothesis.
  Let $(r',s')_{n+1}\in\infset_{n+1}$. Then there exists some $(r,s)_n\in\infset_n$
  such that $(r,s)_n\to(r',s')_{n+1}$.
  We have $(Z_n+r+s,t-s)_{m-n}\in\binfset_{n}$ by \eqref{eq:ihbinfset},
  and $(Z_{n+1}+r'+s',t-s')_{m-n-1}\in\binfset_{n+1}$ by \eqref{eq:infection.backward}.
  Thus
  \begin{align*}
    \bigl\{(Z_{n+1}+r'+s',\,t-s')_{m-n-1} \colon (r',s')_{n+1}\in\infset_{n+1},\,s'\leq t \bigr\}
    \subseteq \binfset_{n+1}.
  \end{align*}
  Conversely, consider a cell $(v,w)_{m-n-1}\in\binfset_{n+1}$. Then $(v,w)_{m-n-1}\too(v',w')_{m-n}$
  for some cell $(v',w')_{m-n}\in\binfset_n$. We claim that
  \begin{align}
     v&\geq Z_{n+1}+t-w,\label{eq:vineq}\\
     v'&\geq Z_n+t-w',\label{eq:v'ineq}\\
     w,w'&\leq t.\label{eq:ww'ineq}
  \end{align}
  Claim~\eqref{eq:ww'ineq} holds because
  all cells in $\binfset_{i}$ are in row~$t$
  or below for all $i$.
  Claim\eqref{eq:v'ineq} holds because by \eqref{eq:ihbinfset}, we have $v'+w'=Z_n+r+t$
  for some $r\geq 0$. For \eqref{eq:vineq},
  let $(i,z)=\Psi'(v,w,v',w')$.
  The smallest $j$ such that
  $a'_1,\ldots,a'_j$ contains $Z_n$ $\Right$ instructions is $j=A+1$
   (see \ref{i:rightA} and observe that $a'_{A+1}=\Right$ except
  when $A=0$, in which case $Z_n=0$). Since the number
  of $\Right$ instructions in $a'_1,\ldots,a'_i$ is $v'$ by definition of $\Psi'$
  and $v'\geq Z_n+t-w'\geq Z_n$, we have $i\geq A+1$.
  There are thus at least $1+Z_{n+1}+t$ $\Left$ instructions in $a'_1,\ldots,a'_i$
  by \ref{i:leftA}, and by definition of $\Psi'$ we have $v+w+1\geq 1+Z_{n+1}+t$, which
  rearranges to show that $v\geq Z_{n+1}+t-w$.
  
  Let $s'=t-w$, $r'=v-Z_{n+1}-s'$, $s=t-w'$, and $r=v'-Z_n-s$. By
  inequalities \eqref{eq:vineq}--\eqref{eq:ww'ineq}, all are nonnegative.
  Now $(v,w)_{m-n-1}=(Z_{n+1}+r'+s',t-s')_{m-n-1}$ and $(v',w')_{m-n}=(Z_n+r+s,t-s)_{m-n}$.
  Since $(v,w)_{m-n-1}\too(v',w')_{m-n}$,
  we have $(r',s')_{n+1}\in\infset_{n+1}$ by \eqref{eq:infection.forward}. Thus
  \begin{align*}
    \binfset_{n+1}\subseteq
    \bigl\{(Z_{n+1}+r'+s',\,t-s')_{m-n-1} \colon (r',s')_{n+1}\in\infset_{n+1},\,s'\leq t \bigr\}.
  \end{align*}
  This extends the inductive hypothesis from $n$ to $n+1$ and completes the proof.
\end{proof}

We mention that we could have used \thref{cor:reverse} to prove 
\thref{prop:connectivity.backward} from \thref{prop:connectivity.forward}, but then the statement
would only hold until the coupling breaks down when $Z_n<0$.

\subsection{Greedy infection paths}
\label{sec:greedy}

In the correspondence between odometers and infection paths, the number of sleeping particles
left by the odometer is equal to the ending row of the infection path.
Since the true odometer leaves the most particles sleeping
of any stable odometer by \thref{lem:dual.lap}, infection paths that reach
high rows will play a special role for us.
A simple way to get infection paths reaching rows close to the highest possible is
to construct them greedily, $k$ steps at a time. We fix a large integer $k$ and choose
an infection path of length~$k$ from a starting cell that reaches the highest row possible,
then from there we choose an infection path of length~$k$ reaching the highest row possible,
and so on:
\begin{define}\thlabel{def:k.greedy}
The \emph{$k$-greedy infection path} is a sequence of cells $(r_0,s_0)_0,\,(r_1,s_1)_1,\ldots$
defined by the following inductive procedure:
From a given starting point $(r_0,s_0)_0$, which we take to be $(0,0)_0$ unless
otherwise mentioned, choose some cell $(r_k,s_k)_k\in\infset[(r_0,s_0)_0]_k$
with $s_k$ maximal, and let $(r_0,s_0)_0\to\cdots\to(r_k,s_k)_k$ be any infection path leading to $(r_k,s_k)_k$.
Then choose $(r_{2k},s_{2k})_{2k}$ from $\infset[(r_k,s_k)_k]_k$ with $s_{2k}$ maximal,
and take $(r_k,s_k)_k\to\cdots\to(r_{2k},s_{2k})_{2k}$ to be any infection path
from $(r_k,s_k)_k$ to $(r_{2k},s_{2k})_{2k}$, and so on.
The choice of $(r_{jk},s_{jk})_{jk}$ in each step of the process and the choice of infection path
$(r_{(j-1)k},s_{(j-1)k})_{(j-1)k}\to\cdots\to(r_{jk},s_{jk})_{jk}$ is not important to us, 
so long as it only depends on information up to step~$jk$ of layer percolation.

Because it will be useful later to have a specific procedure for choosing the $k$-greedy
infection path,
let us specify that given $(r_{(j-1)k},s_{(j-1)k})_{(j-1)k}$,
we take $(r_{jk},s_{jk})_{jk}$ to be the rightmost cell in the maximal row
of $\infset[(r_{(j-1)k},s_{(j-1)k})_{(j-1)k}]_k$.
Then, out of all infection paths from $(r_{(j-1)k},s_{(j-1)k})_{(j-1)k}$ to
$(r_{jk},s_{jk})_{jk}$, select $(r_{jk-1},s_{jk-1})_{jk-1}$
by first maximizing $s_{jk-1}$ and then $r_{jk-1}$.
Then out of all infection paths from $(r_{(j-1)k},s_{(j-1)k})_{(j-1)k}$ to
$(r_{jk-1},s_{jk-1})_{jk-1}$, select $(r_{jk-2},s_{jk-2})_{jk-2}$
by first maximizing $s_{jk-2}$ and then $r_{jk-2}$, and so on.
\end{define}
For the $k$-greedy path, the row reached at step~$jk$ is the sum of $j$ i.i.d.\ random variables
and is hence easy to analyze. Its growth rate is the following constant $\crist[k]$:
\begin{define}\thlabel{def:cristk}
     Let $X_k=\max\bigl\{s\colon (r,s)_k\in\infset[(0,0)_0]_k\bigr\}$ be the highest row contained in the infection set 
  at step~$k$ of layer percolation with sleep rate $\lambda>0$. Then define
    \begin{align}
    \crist[k]=\frac1k\E X_k \label{eq:rholp.k}
    \end{align}
    Note that $0\leq\crist[k]\leq 1$, since $0\leq X_k \leq k$ for all $k$.
\end{define}

In this section,
we seek estimates on the $n$th cell $(r_n,s_n)_n$ of the $k$-greedy infection path.
We take $r_0=s_0=0$, since for other starting points we can simply shift the path
using \thref{prop:shift}. As we mentioned, $s_n\approx\crist[k]n$.
We will also show that
\begin{align}
  r_n&\approx \crist[k]n^2/2.\label{eq:r.heur}
\end{align}
We can heuristically derive the asymptotics for $r_n$ by
estimating $r_{(j+1)k}-r_{jk}$ using \thref{lem:ll.bound,lem:ur.bound}.
The sequence $r_{jk},\ldots,r_{(j+1)k}$ lies
between critical Galton--Watson processes with immigration $s_{jk}$
and with immigration $s_{jk}+k$.
Thus $r_{(j+1)k}-r_{jk}\approx s_{jk}k\approx jk^2\crist[k]$.
We are neglecting the difference of $k$ between the upper and lower bounds because
$k$ will be held fixed while $j$ is taken to infinity, and thus it will be negligible.
Summing these increments, we have
$r_{jk}\approx \frac12j^2k^2\crist[k]$. Since $k$ will be $O(1)$,
this statement will also hold for $r_n$ when
$n$ lies between between $jk$
and $(j+1)k$, explaining \eqref{eq:r.heur}.

The complication in making this heuristic exact is that the sequences
$s_k,s_{2k},\ldots$ and $r_k,r_{2k},\ldots$ are dependent,
and we cannot condition on the first
sequence and then apply the heuristic analysis.
Instead, we do the proof in two steps. First, we choose deterministic sequences
$\smin_1,\smin_2,\ldots$ and $\smax_1,\smax_2,\dots$ with $\smin_j=j\crist[k]-O(\sqrt{n})$
and $\smax_j=j\crist[k]+O(\sqrt{n})$. A classical estimate ensures
that $\smin_j\leq s_j\leq \smax_j$ holds for all $1\leq j\leq n$ with high probability. 
Second, we apply \thref{lem:smin.bound,lem:smax.bound}
to obtain (random) sequences that almost surely bound the columns of \emph{all} infection paths
$(r'_0,s'_0)_0\to\cdots\to(r'_n,s'_n)_n$ for which $s'_1,\ldots,s'_n$ lies between
the bounding sequences $\smin_1,\ldots,\smin_n$ and $\smax_1,\ldots,\smax_n$.
We then complete the proof by bounding the deviations of our stochastic upper and lower bounds
using branching processes.

We start by constructing our bounding sequences.
\begin{lemma}\thlabel{lem:envelopes}
  For fixed $t>0$ and positive integers $k$ and $n$, define $\smin_0,\ldots,\smin_n$ by
  \begin{align*}
    \smin_{jk} &= \max\bigl(0,\,\floor{\crist[k]jk-\tfrac{t}{2}\sqrt{n}}\bigr) 
      \text{ for integers $0\leq j\leq n/k$},\\
    \smin_{jk+i} &= \smin_{jk} \text{ for integers $0\leq j\leq n/k$ and $1\leq i < k$.}
  \end{align*}
  Define $\smax_0,\ldots,\smax_n$ by 
  \begin{align*}
    \smax_{jk} &= \ceil{\crist[k]jk +\tfrac{t}{2}\sqrt{n}}\text{ for integers $0\leq j< n/k+1$}\\
    \smax_{jk-i} &= \smax_{jk}\text{ for integers $1\leq j< n/k+1$ and $1\leq i < k$.}
  \end{align*}
  Then for the $k$-greedy infection path $(0,0)_0\to(r_1,s_1)_1\to(r_2,s_2)_2\to\cdots$,
  \begin{align*}
    \P\Bigl[ \smin_i\leq s_i \leq \smax_i\text{ for all $0\leq i\leq n$} \Bigr] \geq 1 - 2e^{-t^2/3k}.
  \end{align*}
\end{lemma}
\begin{proof}
  Let $\Fscr_i$ be the $\sigma$-algebra generated by layer percolation up to
  step~$i$.
  The random variable $s_{(j+1)k}-s_{jk}$ is independent of $\Fscr_{jk}$
  by \thref{prop:shift}. Thus the random variables $s_{(j+1)k}-s_{jk}$
  for $j=0,1,\ldots$ are i.i.d.\ with mean $\crist[k]k$ and maximum
  value $k$. 
  Let $n'$ be the smallest multiple of $k$ greater than or equal to $n$.
  By the maximal version of Hoeffding's inequality \cite[Theorem~3.2.1]{Roch_2024}, for any $t\geq 0$
  \begin{align*}
    \P\biggl[ \max_{0\leq j\leq n'/k} \abs{s_{jk}-\crist[k]jk}\geq \tfrac{t}{2}\sqrt{n}\biggr]
      &\leq 2\exp\biggl(-\frac{ t^2n}{2n'k}\biggr)\leq 2e^{-t^2/3k}.
  \end{align*}
  If $s_i<\smin_i$, then with $i'$ the largest multiple of $k$ less than or equal
  to $i$, we have $s_{i'}\leq s_i<\smin_i=\smin_{i'}$.
  And if $s_i>\smax_i$, then with $i'$ the smallest multiple of $k$
  greater than or equal to $i$, we have $s_{i'}\geq s_i>\smax_i=\smax_{i'}$.
  Hence the probability that $s_i$ lies outside of $[\smin_i,\smax_i]$ for some $0\leq i\leq n$
  is bounded by the probability that $s_{i'}$ lies outside of $[\smin_{i'},\smax_{i'}]$ for some 
  $i'\in\{0,k,2k,\ldots,n'\}$, completing the proof.
\end{proof}

\begin{prop}\thlabel{prop:greedy.path}
  Let
  $(0,0)_0=(u_0,s_0)_0\to(u_1,s_1)_1\to\cdots$ be the
  $k$-greedy infection path. There exist constants $C,c$
  depending only on $\lambda$ and $k$ such that for all $n$ and all $t\geq 5$,
  \begin{align}
    \P\Biggl[\abs[\bigg]{u_n-\frac{\crist[k]n^2}{2}}\geq tn^{3/2}\Biggr] &\leq C\exp\biggl(-\frac{ct^2}{1+\frac{t}{\sqrt{n}}}\biggr)\label{eq:greedy.path.u},\\\intertext{and}
    \P\biggl[\abs[\Big]{s_n-\crist[k]n}\geq t\sqrt{n}\biggr]&\leq 2e^{-ct^2}.\label{eq:greedy.path.s}
  \end{align}
  
\end{prop}
\begin{proof}
  For the bound on $s_n$, first let $n'$ be the greatest
  multiple of $k$ less than or equal to $n$.
  Since $s_{n'}$ differs from $s_n$
  and $\crist[k]n'$ differs from $\crist[k]n$ by at most $k$,
  \begin{align*}
    \P\Bigl[\abs{s_{n}-\crist[k]n}\geq t\sqrt{n}\Bigr]
      &\leq \P\Bigl[\abs{s_{n'}-\crist[k]n'}\geq t\sqrt{n}-2k\Bigr]
  \end{align*}
  As in the proof of \thref{lem:envelopes}, the random variable $s_{n'}$
  is the sum of $n'/k$ independent random variables with mean $\crist[k]k$
  and upper bound $k$. By
  Hoeffding's inequality,
  \begin{align*}
    \P\Bigl[\abs{s_{n'}-\crist[k]n'}\geq t\sqrt{n}-2k\Bigr] &\leq 2 \exp\biggl(-\frac{2(t\sqrt{n}-2k)^2}
       {n'k}\biggr)\leq 2e^{-c_1t^2}
  \end{align*}
  for a constant $c_1$ depending on $k$ and $\lambda$.
  
  Next we turn to the tail bounds on $u_n$.
  Fix $t>0$ and let $\smin_0,\ldots,\smin_n$ and $\smax_0,\ldots,\smax_n$ be
  as in \thref{lem:envelopes}. We start with the lower tail bound on $u_n$.
  Apply \thref{lem:smin.bound} with $\smin_0,\ldots,\smin_n$ as we have defined it
  and with $r_0=u_0=0$ to generate a lower-bounding infection path
  $(r_0,s_0)_0\to\cdots\to(r_n,s_n)_n$. First, we give lower bounds on $r_n$.
  Let $X_i=r_i+\smin_i$.
  Since $X_0,\ldots,X_n$ is a critical geometric branching process
  with immigration $\smin_i$ at step~$i$, we have
  \begin{align*}
    \E X_n = r_0+\smin_0 + \sum_{i=1}^n\smin_i&\geq \sum_{j=0}^{\floor{n/k}}(\crist[k]jk-\tfrac{t}{2}\sqrt{n}-1)k\\
    &\geq\crist[k]\tfrac12 (n/k)^2k^2-\bigl(\tfrac{t}{2}\sqrt{n}+1\bigr)(n/k+1)k\\
    &= \frac{\crist[k]n^2}{2}-\tfrac{t}{2}\bigl(n^{3/2}+O_k(n)\bigr).
  \end{align*}
  Here we use the notation $O_k(f(n))$ to denote an expression bounded by a constant
  times $f(n)$ for all $n\geq 1$, with the constant permitted to depend on $k$.
  Preparing to apply \thref{prop:GW.emigration.concentration}, we rewrite the deviations
  of $r_n$ in terms of the the deviations of $X_n$:
  \begin{align*}
    \P\Biggl[r_n-\frac{\crist[k]n^2}{2}\leq -tn^{3/2}\Biggr]
    &=
    \P\Bigl[X_n-\E X_n\leq -tn^{3/2}+\smin_n+\tfrac{t}{2}\bigl(n^{3/2}+O_k(n)\bigr)\Bigr]\\
    &\leq \P\Bigl[X_n-\E X_n\leq -\tfrac{t}{2}\bigl(n^{3/2}-O_k(n)\bigr)\Bigr].
  \end{align*}
  Here we have incorporated $\smin_n$ into the $\frac{t}{2}O_k(n)$,
  since $\smin_n\leq n$ and $t$ is bounded away from $0$. Now we apply
  \thref{prop:GW.emigration.concentration} to obtain
  \begin{align}\label{eq:rn.smin.bound}
    \P\Biggl[r_n-\frac{\crist[k]n^2}{2}\leq -tn^{3/2}\Biggr]
    &\leq C_2\exp\biggl(-\frac{c_2t^2n^3}{n(n^2+tn^{3/2})}\biggr)
      = C_2\exp\biggl(-\frac{c_2t^2}{1+\frac{t}{\sqrt{n}}}\biggr)
  \end{align}
  for some constants $c_2=c_2(k)$ and $C_2=C_2(k)$.
  
  By \thref{lem:smin.bound}, we have $u_n\geq r_n$ so long as $u_i\geq\smin_i$
  for $0\leq i\leq n$. Thus,
  \begin{align*}
    \P\Biggl[u_n-\frac{\crist[k]n^2}{2}\leq -tn^{3/2}\Biggr]
      &\leq \P\Biggl[r_n-\frac{\crist[k]n^2}{2}\leq -tn^{3/2}\Biggr]
        + \P\bigl[u_i<\smin_i\text{ for some $0\leq i\leq n$}\bigr]\\
      &\leq C_3\exp\biggl(-\frac{c_3t^2}{1+\frac{t}{\sqrt{n}}}\biggr)
  \end{align*}
  for altered $c_3=c_3(k)$ and $C_3=C_3(k)$, applying \eqref{eq:rn.smin.bound}
  and \thref{lem:envelopes}. The upper tail bound on $u_n$ is obtained
  by a nearly identical argument using \thref{lem:smax.bound} and
  $\smax_0,\ldots,\smax_n$
  in place of \thref{lem:smin.bound} and $\smin_0,\ldots,\smin_n$.
\end{proof}

In Section~\ref{subsec:box}, we will consider the \emph{$k$-greedy infection path capped at row~$s$},
defined as following the $k$-greedy infection path until it reaches row~$s$ and then following
the minimal infection path started at that point. 
\begin{prop}\thlabel{prop:greedy.capped}
  Let $(0,0)_0=(u'_0,s'_0)_0\to(u'_1,s'_1)_1\to\cdots$ be a $k$-greedy infection path capped
  at row~$s$ starting from $(0,0)_0$. Assume that $s/m\leq \crist[k]-\epsilon$ for some $\epsilon>0$.
  Then
  \begin{align}
    \P\Biggl[\abs[\bigg]{u'_m -\biggl(m-\frac{s}{2\crist[k]}\biggr)s}\geq tm^{3/2}\Biggr] &\leq C\exp\biggl(-\frac{ct^2}{1+\frac{t}{\sqrt{m}}}\biggr)\label{eq:greedy.capped.u},\\\intertext{and}
    \P[s'_m\neq s]&\leq e^{-cm},\label{eq:greedy.capped.s}
  \end{align}
  where $c$ and $C$ are constants depending only on $k$, $\lambda$, and $\epsilon$.
\end{prop}
\begin{proof}
  If $s'_m\neq s$, then the greedy infection path failed to reach $s$
  by step~$m$. This event occurs with probability
  bounded by $\P\bigl[s_m-\crist[k]m<-\epsilon m\bigr]$, which by
  \thref{prop:greedy.path} occurs with exponentially vanishing probability,
  with the rates depending on $k$, $\lambda$, and $\epsilon$,
  thus proving \eqref{eq:greedy.capped.s}
  
  Echoing \thref{lem:envelopes}, let
  \begin{align*}
    \sminp_{jk} &= \max\bigl(0,\,\floor{\min(\crist[k]jk,s)-\tfrac{t}{2}\sqrt{m}}\bigr) 
      \text{ for integers $0\leq j\leq m/k$},\\
    \sminp_{jk+i} &= \smin_{jk} \text{ for integers $0\leq j\leq m/k$ and $1\leq i < k$.}
  \end{align*}
  Now,
  following the proof of \thref{prop:greedy.path} using $\sminp_1,\ldots,\sminp_n$
  in place of $\smin_1,\ldots,\smin_n$,
  we apply \thref{lem:smin.bound} with this capped lower bound
  and with $r_0=u_0=0$ to generate a lower-bounding infection path
  $(r_0,s_0)_0\to\cdots\to(r_m,s_m)_m$.
  Let $X_i=r_i+\sminp_i$.
  Just as in \thref{prop:greedy.path} but with a slightly different
  calculation, setting $\rho=s/m$
  \begin{align*}
    \E X_m = r_0+\sminp_0 + \sum_{i=1}^m\sminp_i&\geq \sum_{j\colon jk\leq\rho m/\crist[k]} (\crist[k]jk-\tfrac{t}{2}\sqrt{m}-1)k\\
               &\qquad+ \sum_{j\colon \rho m/\crist[k]\leq jk\leq m}(s-\tfrac{t}{2}\sqrt{m}-1)k\\
             &\geq \biggl(1 - \frac{\rho}{2\crist[k]}\biggr)\rho m^2 -\tfrac{t}{2}\bigl(m^{3/2}+O_k(m)\bigr),
  \end{align*}
  where $O_k(f(m))$ denotes an expression bounded by a constant
  times $f(m)$ for all $m\geq 1$ with the constant permitted to depend on $k$.
  Now we can prove \eqref{eq:greedy.capped.u}
  by the same argument as in \thref{prop:greedy.path} of applying
  \thref{prop:GW.emigration.concentration} and \thref{lem:envelopes}.
  And the upper bound follows similarly.
\end{proof}

\subsection{The critical density}
\label{sec:critical.density}
We now are ready to define the critical density $\crist=\crist(\lambda)$ for layer percolation.
Recall from \thref{def:cristk} that $X_n$ is the highest row
infected at step~$n$ starting at $0$ and $\crist[n]$ is defined as $\E X_n/n$.
We define 
\begin{align*}
  \crist=\limsup_{n\to\infty} \crist[n]=\limsup_{n\to\infty}\frac1n\E X_n.
\end{align*}
It follows from $0\leq\crist[n]\leq 1$ that $0\leq\crist\leq 1$.

We will show that $X_n/n$ converges to this constant and satisfies an exponential lower tail bound:
\begin{prop}\thlabel{prop:subadditive}
  \begin{align}
    \lim_{n\to\infty}\crist[n]={}&{}\crist\label{eq:proper.limit}\\\intertext{and}
    \frac{X_n}{n}\topr{}&{}\crist       \label{matteo}
  \end{align}
and there exist constants $c,C>0$ depending only on $\lambda$ and $\epsilon$
for which the following one-sided tail bound holds:
\begin{equation} \label{baby}
  \Pr\biggl[\crist-\frac{1}{n}X_n>\epsilon\biggr]\leq Ce^{-cn}.
\end{equation}
\end{prop}

\thref{peanut butter} will provide
the corresponding upper tail bound 
$$\Pr\biggl[\frac1n{X_n}-\crist>\epsilon\biggr]\leq Ce^{-cn},$$ 
implying that the convergence is almost sure as well as in probability.

The random variables $X_n$ are superadditive in the following sense:
Choose some cell $(r,s)_n$ from the highest row in the infection set $\infset[(0,0)_0]_n$.
Now choose a cell $(r',s')_{n+m}$ in the highest row of the infection set $\infset[(r,s)_n]_m$.
Then $(r',s')_{n+m}$ belongs to $\infset[(0,0)_0]_{n+m}$, and by
\thref{prop:shift}, we have $s'\eqd X_n+X_m$, where $X_n$ and $X_m$ are independent.
Thus $X_{n+m}$ is stochastically larger than the sum of independent copies of $X_n$ and $X_m$.
\thref{prop:subadditive} will be a straightforward consequence of this superadditivity.

\begin{proof}[Proof of \thref{prop:subadditive}]
  We start by proving \eqref{baby}, for which we have already done most of the work.
  Fix $\epsilon>0$ and choose $k$ large enough that that $\crist[k]>\crist-\epsilon/2$.
  Let $(r_n,s_n)_n$ be the $n$th step in the $k$-greedy infection path from $(0,0)_0$.
  We have $X_n\geq s_n$, and so by \thref{prop:greedy.path},
  \begin{align*}
    \P\biggl[\frac{1}{n}X_n\leq\crist-\epsilon\biggr]\leq \P\biggl[\frac{1}{n}s_n\leq\crist[k]-\epsilon/2\biggr]\leq 2e^{-cn}
  \end{align*}
  for some $c$ depending on $\epsilon$, $\lambda$, and $k$, proving \eqref{baby}.
  As a consequence, $\liminf_{n\to\infty}\E X_n/n\geq \crist$, proving \eqref{eq:proper.limit}.
  To prove \eqref{matteo},
  if $\limsup_{n\to\infty}\P\bigl[X_n/n>\crist+\epsilon\bigr]$ is greater than $0$, then from
  \eqref{baby} we would have $\frac1n\E X_n>\crist$ along a subsequence, contradicting
  the definition of $\crist$. Hence $\P\bigl[X_n/n>\crist+\epsilon\bigr]\to 0$,
  which combined with \eqref{baby} proves
  \eqref{matteo}.
\end{proof}

\section{The box lemma} \label{sec:box}
\label{subsec:box}
In this section we prove the \emph{box lemma} (\thref{lem:box}), which gives a deterministic 
region of cells almost certainly infected by layer percolation.
By \thref{bethlehem,cor:ip}, any infection path in layer percolation ending at row $\rho n$ corresponds to an extended odometer on $\ii{0,n}$ stable on $\ii{1,n-1}$ leaving density $\rho$ of sleeping particles.
By \thref{prop:subadditive}, it is likely that such infection paths exist for any $\rho<\crist$,
i.e., there is some cell infected starting from $(0,0)_n$ in row~$\rho n$.
The box lemma goes one step further, telling us that there is not just one infected cell
in row~$\rho n$ but an entire interval of length on the order of $n^2$.
One application is to construct odometers stable not only on $\ii{1,n-1}$
but also at $n$.
Another is to prove an exponential upper bound on the probability
that layer percolation infects a cell in row $(\crist+\epsilon)n$, which we
carry out in Section~\ref{sec:upper.bound}. Heuristically, the idea
is that the box lemma asserts that layer percolation behaves predictably in the bulk.
Thus no stretch of steps has a large effect on what row is reached by the infection set,
giving rise to concentration for the highest row reached.

To give the idea of the box lemma before we state it,
let $\rho<\crist$ and imagine we set out to infect
a cell in row $\rho n$ at step~$n$.
Suppose we choose an infection path $(0,0)_0=(r_0,s_0)_0\to\cdots\to(r_n,s_n)_n$
with $s_j\approx \rho j$. Then according to \thref{lem:smin.bound,lem:smax.bound},
the sequence $r_0,\ldots,r_n$ is approximately a critical branching process with immigration
$s_j\approx\rho j$ at step~$j$, and $r_n$ should be close to $\sum_{j=1}^n\rho j\approx\rho n^2/2$.
Thus some cell in the vicinity of $(\rho n^2/2,\,\rho n)_n$
is likely infected starting from $(0,0)_0$, for any $\rho<\crist$.
The box lemma states that something much stronger is true: \emph{all} cells in a box of order
$n^2\times n$ around
$(\rho n^2/2,\,\rho n)_n$ are infected with high probability.

\begin{lemma}[Box lemma]\thlabel{lem:box}
  Suppose that $0<\epsilon\leq \rho\leq \crist-\epsilon$.
  There exist constants
  $\delta=\delta(\epsilon)$, $c=c(\lambda,\epsilon)$, and $C=C(\lambda,\epsilon)$
  such that the following statement holds:
  Define the box
  $$\Blb_{n}(\rho,\delta) := \ii[\bigg]{\frac{\rho}{2}(1-\delta) n^2,\,\frac{\rho}{2}(1+\delta) n^2}
  \times  \ii[\Big]{\rho(1-\delta)n,\,\rho(1+\delta) n}.$$
  Then
  \begin{align}
    \P\Bigl[ \Blb_n(\rho,\delta)\subseteq\infset[(0,0)_0]_n \Bigr] \geq 1-Ce^{-cn}.\label{eq:box.stmt}
  \end{align}
  We can take $\delta$ as $\eta\epsilon^2$ for a sufficiently small
  absolute constant $\eta>0$.
\end{lemma}

\subsection{Tools for establishing infection}
To prove that some cell $(r,s)_t$ is infected starting from $(0,0)_0$, it suffices to show
that the forward infection set from $(0,0)_0$ and the backward infection
set from $(r,s)_t$ at some step have a cell in common.
Using the connectivity of infection sets (see Section~\ref{sec:infection}),
the following lemma gives us an easy way to show
that a forward and backward infection set cross each other:
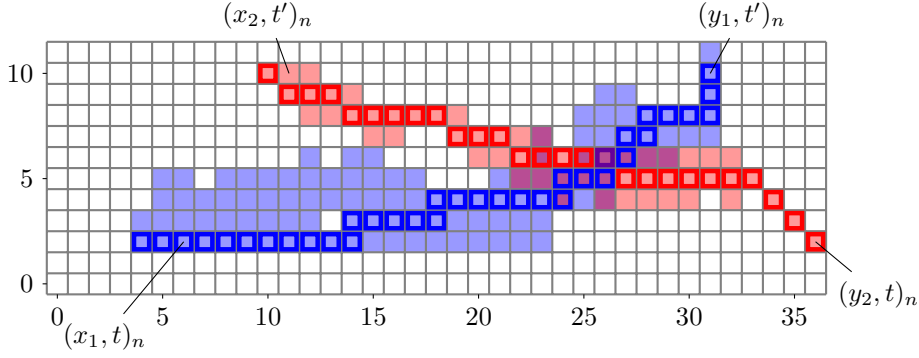
\begin{figure}
  \centering
  \begin{tikzpicture}[scale=.28,forward/.style={blue!40!white}, backward/.style={red!40!white}, both/.style={red!60!blue!70}, fpath/.style={blue},bpath/.style={red},bothpath/.style={purple}]
    \input{X.tikz}
  \end{tikzpicture}
  \caption{Let $\infset$ be an infection set from some cell in row~$t$,
  and let $\binfset$ be a backward infection set from some cell in row~$t'$ for $t'\geq t$,
  both at step~$n$. If $(x_1,t)_n,\,(y_1,t')\in\infset$ and
  $(x_2,t')_n,\,(y_2,t)_n\in\binfset$ and $x_1\leq y_2$ and $x_2\leq y_1$, then the two infection
  sets must cross, as proven in \thref{lem:X}.
  The blue and red shaded regions indicate the forward and backward infections sets, respectively,
  with their overlap shaded purple. The cells with highlighted outlines are the paths (which are not
  infection paths) constructed
  starting at $(y_1,t')_n$ and $(y_2,t)_n$
  in the proof of \thref{lem:X}. It is shown that these two paths must intersect, and in this case
  they do so at cell $(26,6)_n$.}
  \label{fig:X}
\end{figure}
\begin{lemma}\thlabel{lem:X}
  Let $\infset=\infset[(r,t)_m]_{n-m}$ and $\binfset=\binfset[(r',t')_{m'+n}]_{m'}$
  with $t\leq t'$. Suppose that $(x_1,t)_n,\,(y_1,t')_n$ are
  cells in $\infset$ and $(x_2,t')_n,\,(y_2,t)_n$ are cells in $\binfset$
  with $x_1\leq y_2$ and $x_2\leq y_1$ (see Figure~\ref{fig:X}).
  Then $\infset\cap\binfset\neq\emptyset$.
\end{lemma}
\begin{proof}
  We make paths as depicted in Figure~\ref{fig:X} and show that they cross.
  To match the figure, call a cell \emph{blue} if it is an element of $\infset$
  and \emph{red} if it is an element of $\binfset$.  
  By \thref{prop:connectivity.forward}, each blue cell except for the leftmost one
  in row~$t$ has a blue neighbor one cell to its left or below it. Thus, 
  we can make a path of blue cells from $(y_1,t')_n$ to the leftmost blue cell in row~$t$
  by moving left or down at each step, choosing arbitrarily when both are possible.
  The cells in the blue path in a given row form an interval, and the start of
  an interval in some row~$s+1$ is the final column of the interval in row~$s$.
  Likewise, by \thref{prop:connectivity.backward} we can form a path of red cells starting
  from $(y_2,t)_n$ that moves to the left or upper-left in each step, terminating 
  either at the leftmost red cell in row~$t'$ or at some red cell $(0,t_0)_n$
  where $t\leq t_0< t'$. Call the ending cell $(x_0,t_0)_n$ in either case,
  with $x_0\leq x_2$ and $t_0=t'$ in the first case
  and $x_0=0$ and $t_0<t'$ in the second. In either case, the blue path has a cell in row~$t_0$
  at or to the right of column $x_0$, in the first case because $x_0\leq x_2\leq y_1$
  and in the second because $x_0=0$. The cells in the red path in a given row form an interval,
  and the interval in some row~$s-1$ starts one column after the end of the interval in row~$s$.
  
  We claim that the red and blue paths contain a common cell.
  Suppose not. Then in each row, either the blue path is entirely to the left
  of the red path or the red path is entirely to the left of the blue.
  In row~$t$, the blue path contains the leftmost blue cell in the row. Since
  $(x_1,t)_n$ is blue and $(y_2,t)_n$ is red and $x_1\leq y_2$, the blue path
  must be to the left of the red path. As we move to higher rows,
  the blue path must remain to the left. Indeed, suppose the blue
  path in some row~$s$ covers columns $\ii{a,b}$ and the red path covers
  columns $\ii{c,d}$ with $b<c$. Then in row~$s+1$, the blue path has $b$
  as its left endpoint while the red path has $c-1$ as its right endpoint;
  the blue path cannot lie entirely to the right of the red path since $b\leq c-1$.
  Thus the blue path is entirely to the left of the red path in all rows.
  But this is a contradiction, since in row~$t_0$ the red path
  contains $(x_0,t_0)$ and the blue path contains some cell at or to the right of $x_0$. 
  Therefore the red and blue paths have a common element, proving
  that $\infset\cap\binfset\neq\emptyset$.
\end{proof}

In Section~\ref{sec:layer.perc.misc}, we defined minimal and
greedy infection paths. When we prove the box lemma, we will need backward
versions of these paths:
\begin{define}
  The \emph{reverse minimal infection path from $(r,s)_m$} is the infection path
  \begin{align*}
    (r_m,s)_0)\to(r_{m-1},s_{m-1})_1\to\cdots\to(r_1,s)_{m-1}\to(r_0,s)_m=(r,s)_m
  \end{align*}
  in which $(r_{k+1},s)_{m-k-1}$ is the leftmost cell in row~$s$ infecting
  $(r_k,s)_{m-k}$ for each $k\in\ii{0,m-1}$.
  
  The \emph{reverse $k$-greedy path from $(r,s)_m$} is defined analogously
  to the $k$-greedy path (\thref{def:k.greedy}), with $(r_{jk},s_{jk})_{m-jk}$
  chosen from the cells in the minimum row of $\binfset[(r_{(j-1)k},s_{(j-1)k})_{m-(j-1)k}]_k$
  for each $j$ and some infection path chosen from $(r_{jk},s_{jk})_{m-jk}$
  to $(r_{(j-1)k},s_{(j-1)k})_{m-(j-1)k}$. Again, it only really matters that our choice of
  cell $(r_{jk},s_{jk})_{m-jk}$ and path
  from $(r_{jk},s_{jk})_{m-jk}$ to $(r_{(j-1)k},s_{(j-1)k})_{m-(j-1)k}$
  is made using only the information contained in the infections
  from steps~$m-jk$ to $m$, but we give a concrete procedure that will
  be dual to the one from \thref{def:k.greedy} (see \thref{lem:reversed.paths}).
  Given $(r_{(j-1)k},s_{(j-1)k})_{m-(j-1)k}$, choose $(r_{jk},s_{jk})_{m-jk}$ to be the rightmost
  cell in the minimal row of $\infset[(r_{(j-1)k},s_{(j-1)k})_{m-(j-1)k}]_k$.
  Then out of all infection paths from $(r_{jk},s_{jk})_{m-jk}$
  to $(r_{(j-1)k},s_{(j-1)k})_{m-(j-1)k}$, choose $(r_{jk-1},s_{jk-1})_{m-jk+1}$
  by first minimizing $s_{jk-1}$ and then maximizing $r_{jk-1}$.
  Then out of all infection paths from $(r_{jk-1},s_{jk-1})_{m-jk+1}$
  to $(r_{(j-1)k},s_{(j-1)k})_{m-(j-1)k}$
  choose $(r_{jk-2},s_{jk-2})_{m-jk+2}$ by first minimizing $s_{jk-2}$
  and then maximizing $r_{jk-2}$.
  
  The \emph{reverse $k$-greedy path from $(r,s)_m$ capped at row~$t$} for $t\leq s$ is
  defined analogously to forward capped $k$-greedy paths, as the infection
  path that follows the reverse $k$-greedy infection path from $(r,s)_m$ until it
  descends to row~$t$, and then follows the reverse minimal infection path from that
  point on.
\end{define}

Recall that \thref{cor:reverse} yields a coupling of two layer percolations such that
when the coupling is valid, the backward infection set $\binfset[(u,t)_m]_n$ in one layer percolation
is the image of the infection set $\infset[(0,0)_0]_n$ in the other under the transformation
$(r,s)_n\mapsto (Z_n+r+s,t-s)_{m-n}$, where $Z_n$ is a branching process defined in
\thref{cp:reverse}.
It is evident from the definitions that this transformation takes minimal and $k$-greedy
infection paths to reverse minimal and reverse $k$-greedy infection paths:
\begin{lemma}\thlabel{lem:reversed.paths}
  Suppose two layer percolations are coupled by \thref{cp:reverse}
  and that $Z_n\geq 0$ so that \thref{cor:reverse} is in effect
  and infection sets from $(0,0)_0$ in the first layer percolation
  are transformed versions of backward infection sets from $(u,t)_n$
  in the second layer percolation.
  Then an infection path
  \begin{align*}
    (0,0)_0=(r_0,s_0)_0\to\cdots (r_n,s_n)_n
  \end{align*}
  in the first layer percolation is the minimal infection path
  (resp.\ the $k$-greedy infection path, the $k$-greedy infection
  path capped at $s'\geq s$)  from $(0,0)_0$ if and only if
  the infection path
  \begin{align*}
    (Z_n+r_n+s_n,t-s_n)_0\too\cdots\too(Z_0+r_0+s_0,t-s_0)_n=(u,t)_n
  \end{align*}
  is the reverse minimal infection path (resp.\ the reverse $k$-greedy infection path,
  the reverse $k$-greedy infection path capped at $t-s'$) from $(u,t)_n$.
\end{lemma}

  \subsection{Sketch of the proof of Lemma~\ref{lem:box}}
  Suppose that $t=\rho n$ is an integer and
  \begin{align*}
    u\in\ii[\bigg]{\frac{\rho}{2}(1-\delta') n^2,\,\frac{\rho}{2}(1+\delta') n^2}
  \end{align*}
  for $\delta'>0$ to be chosen.
  It will be enough by a union bound to show that
  $\P\bigl[(u,t)_n\in\infset[(0,0)_0]_n\bigr]\geq 1-Ce^{-cn}$.
  Our strategy is to
  show that the forward infection set from $(0,0)_0$ and the backward
  infection set from $(u,t)_n$ are likely to intersect at step $n/2$.
  Let $n_0\approx(1/2-\alpha)n$ for a small constant $\alpha>0$, and let $s\approx\beta n/2$ for $\beta$
  slightly smaller than $\rho$. 
  \subsubsection*{Step~$0$ to step~$n_0$: $(0,0)_0\to (R_1,s)_{n_0}$}
   Run the greedy infection path from $(0,0)_0$ capped at row
   $s$ until step $n_0$, when it will be at a cell $(R_1,s)_{n_0}$.
  \subsubsection*{Step~$n-n_0$ to step~$n$: $(R_2,t-s)_{n-n_0}\to(u,t)_n$}
    Symmetrically to the previous step, 
    run the reverse greedy infection path from $(u,t)_n$ capped at row~$t-s$
    until step $n-n_0$, when it will be at some some cell $(R_2,t-s)_{n-n_0}$.
  \subsubsection*{Step~$n_0$ to step~$n/2$: $(R_1,s)_{n_0}\to (X_1,s)_{n/2}$ and 
  $(R_1,s)_{n_0}\to (Y_1,t-s)_{n/2}$}
      From $(R_1,s)_{n_0}$, run the minimal infection path to $(X_1,s)_{n/2}$
      and the greedy infection path capped at row~$t-s$ to $(Y_1,t-s)_{{n/2}}$.
  \subsubsection*{Step~$n/2$ to step~$n-n_0$: $(X_2,t-s)_{n/2}\to (R_2,t-s)_{n-n_0}$
      and $(Y_2,s)_{n/2}\to (R_2,t-s)_{n-n_0}$}
      Symmetrically to the previous step, from $(R_2,t-s)_{n-n_0}$
      run the reverse minimal infection path to $(X_2,t-s)_{n/2}$ and
      the reverse greedy infection path capped at row~$s$ to $(Y_2,s)_{n/2}$.

  These steps are illustrated in Figure~\ref{fig:s}. In the end, we will have produced cells $(X_1,s)_{n/2}$ and $(Y_1,t-s)_{n/2}$
  in $\infset[(R_1,s)_{n_0}]_{{n/2}-n_0}$ and cells $(X_2,t-s)_{{n/2}}$ and
  $(Y_2,s)_{{n/2}}$ in $\binfset[(R_2,t-s)_{n-n_0}]_{n/2-n_0}$.
  With $\alpha$, $\beta$, and $\delta'$ chosen correctly, we will find that these four cells
  are arranged as in Figure~\ref{fig:XX}, with $X_1\leq Y_2$ and $X_2\leq Y_1$,
  thus producing a cell in $\infset[(R_1,s)_{n_0}]_{{n/2}-n_0}\cap
  \binfset[(R_2,t-s)_{n-n_0}]_{{n/2}-n_0}$ by \thref{lem:X} and proving that
  \begin{align*}
    (0,0)_0\to(R_1,s)_{n_0}\to
       (R_2,t-s)_{n-n_0}\to(u,t)_n.
  \end{align*}

  \begin{figure}

\definecolor{col2}{HTML}{CA0020}
\definecolor{col3}{HTML}{F4A582}
\definecolor{col4}{HTML}{92C5DE}
\definecolor{col1}{HTML}{0571B0}
\begin{tikzpicture}[xscale = .9,yscale=.6]
    \input{s.tikz}
\end{tikzpicture}

\caption{The step (horizontal axis) in layer percolation versus the row (vertical axis) for the paths
constructed in the proof of \thref{lem:box}. 
The black line from step~$0$ to step~$n_0$ is a greedy infection path capped at row~$s$,
and the black line from step~$n-n_0$ to step~$n$ is a reverse greedy infection path
capped at row~$t-s$. The dark and light blue arrows represent a minimal infection path and
a greedy infection path capped at row~$t-s$, respectively, started from the same cell.
The dark and light red arrows represent a reverse minimal infection path
and a reverse greedy infection path capped at row~$s$, respectively, started from the same cell.
The infection sets at step~$n/2$ are depicted in Figure~\ref{fig:XX}.
} \label{fig:s}
\end{figure}
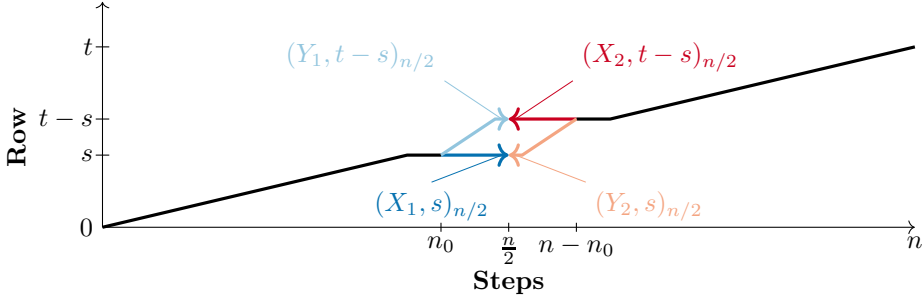

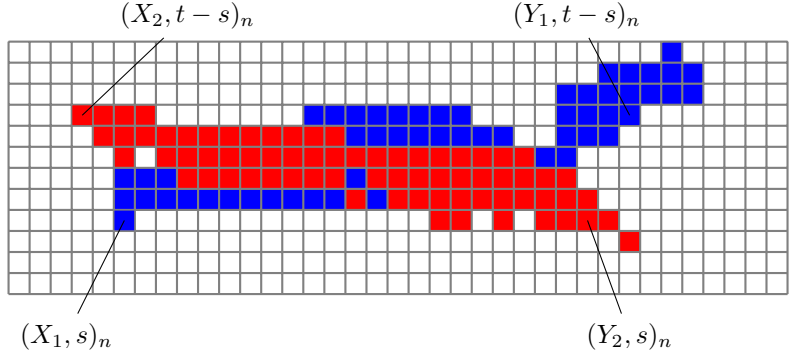
\begin{figure}
    \centering
    \begin{tikzpicture}[scale=.9]
      \input{XX.tikz}
    \end{tikzpicture}
\caption{Two infection sets at step~$n/2$. In blue is $\infset[(R_1,s)_{n_0}]_{n/2-n_0}$
and in red is $\binfset[(R_2,t-s)_{n-n_0}]_{n/2-n_0}$, with their intersection in purple.
The four cells represented by the arrow tips in Figure~\ref{fig:s} are outlined.
Since the paths are constructed to make $X_2\leq Y_1$ and $X_1\leq Y_2$,
these four cells satisfy the conditions of \thref{lem:X}, demonstrating that
that the two infection sets overlap.}
\label{fig:XX}
\end{figure}

  To put this idea into practice, we
  choose a suitable $\alpha$ and $\beta$ and use
  \thref{lem:ll.bound,prop:greedy.capped} to estimate $R_1$, $X_1$, and $Y_1$,
  pinpointing each one to a window of order $n^2$.
  With an additional application of \thref{cor:reverse}, we can do the same
  with $R_2$, $X_2$, and $Y_2$. At this point we will have proven
  that $X_1\leq Y_2$ and $X_2\leq Y_1$ with high probability, allowing us to
  conclude that $(u,t)_n$ is in $\infset[(0,0)_0]_n$ with high probability.
  The technical complexity comes in picking the appropriate values
  of $\alpha$ and $\beta$ and then making the statements about $R_1$, $R_2$, $X_1$,
  $Y_1$, $X_2$, and $Y_2$ precise.

  \subsection{Estimates on the box lemma paths}
  \label{sec:box.paths}
  In this section, we consider a cell $(u,t)_n$ and construct and bound
  the infection paths described in the previous section used to show that
  $(0,0)_0\to(u,t)_n$ with exponentially vanishing probability.
  Throughout this section, we make the following assumptions:
  As in the statement of \thref{lem:box}, we have
  $0<\epsilon\leq\rho\leq\crist-\epsilon$.
  We assume that $t:=\rho n$ is an integer.
  We take $n$ to be even and choose $k$ large enough that
  $\crist[k]>\crist-\epsilon/4$, so that
  \begin{align}\label{eq:rho.hierarchy}
    \epsilon\leq \rho \leq \crist-\epsilon<\crist-\epsilon/4<\crist[k]\leq\crist.
  \end{align}
  Define
  \begin{align}
    \begin{aligned}
    \alpha&=\epsilon/16,&n_0&=n/2-\floor{\alpha n},\\
    \beta &=\rho-\bigl(\crist-\tfrac{\epsilon}{2}\big)\alpha,\quad&\quad
    s&=\floor{\beta n/2},
    \end{aligned}\label{eq:boxdefs}
  \end{align}
  whose roles are laid out in the sketch.
  We let $\delta' = \alpha^2/4$.
  Assume that
  \begin{align}\label{eq:u.range}
    u\in\ii[\bigg]{\frac{\rho}{2}(1-\delta') n^2,\,\frac{\rho}{2}(1+\delta') n^2}.
  \end{align}
  We also define $\delta_0=\rho\delta'/100$, to be used in some error bounds.
  We say that a statement holds with overwhelming probability (w.o.p.)\ if it fails with
  probability at most $Ce^{-cn}$ where $c$ and $C$ are constants that depend only on $\lambda$
  and $\epsilon$.
  
  The following lemma proves several important facts about the quantities defined in \eqref{eq:boxdefs}.
  Equation~\eqref{eq:beta.big} shows in particular that $0<\beta<1$.
  Bounds \eqref{eq:s/n_0} and \eqref{eq:other.rate} confirm that $\alpha$
  and $\beta$ were chosen appropriately so that the greedy
  paths described in the sketch can reach the rows that they need to: 
  \eqref{eq:s/n_0} shows that the capped $k$-greedy paths
  used from steps~$0$ to $n_0$ and from steps~$n-n_0$ to $n$ progress fast enough
  to reach their caps; and \eqref{eq:other.rate} shows that the capped $k$-greedy paths
  used from steps~$n_0$ to $n/2$ and from steps~$n/2$ to $n-n_0$ reach their caps.
  \begin{lemma}\thlabel{lem:alpha.beta}
  For $n$ larger than some constant depending only on $\epsilon$,
  \begin{align}
    \epsilon/2&\leq\beta \leq\rho\label{eq:beta.big},\\
    \frac{s}{n_0}&\leq\crist[k]-\epsilon/2,\label{eq:s/n_0}\\\intertext{and}
    \frac{t-2s}{n/2-n_0}&\leq \crist[k]-\epsilon/5.\label{eq:other.rate}
  \end{align}
  \end{lemma}
\begin{proof}
  The upper bound in \eqref{eq:beta.big} is immediate from the definition of $\beta$.
  For the lower bound, we have $\rho\geq\epsilon$ and $\crist-\epsilon/2\leq 1$, yielding
  $\beta\geq \epsilon-\epsilon/16\geq\epsilon/2$.
  
  For \eqref{eq:s/n_0}, we let $q=(\crist-\epsilon/2)/\rho$ and have
  \begin{align}\label{eq:s/n_0.start}
    \frac{s}{n_0}\leq\frac{\beta}{1-2\alpha}=\frac{\rho(1-q\alpha)}{1-2\alpha}.
  \end{align}
  Now we prove a bound on $(1-q\alpha)/(1-2\alpha)$.
  Observe that $q\geq 1$ by \eqref{eq:rho.hierarchy}.
  If $q\leq 2$, then
  a bit of calculus shows that
  $(1-qx)/(1-2x)\leq 1 + 2(2-q)x$ for $x\leq 1/4$, yielding
  $(1-qx)/(1-2x)\leq1+2x$ for $x\leq 1/4$. 
  If $q>2$, then $(1-qx)/(1-2x)\leq1$ for all $x\leq 1/4$.
  Hence $(1-qx)/(1-2x)\leq1+2x$ for $x\leq 1/4$ holds in all cases.
  Since $\alpha=\epsilon/16$ and $\epsilon\leq 1$,
  \begin{align*}
    \frac{1-q\alpha}{1-2\alpha} &\leq1+\epsilon/8\leq\frac{\rho+\epsilon/8}{\rho}
      \leq \frac{\crist-7\epsilon/8}{\rho}.
  \end{align*}
  Together with \eqref{eq:s/n_0.start} and \eqref{eq:rho.hierarchy}, this yields
  \begin{align*}
    \frac{s}{n_0}\leq\crist-7\epsilon/8<\crist[k]-5\epsilon/8,
  \end{align*}
  proving \eqref{eq:s/n_0}.
  
  For \eqref{eq:other.rate}, just using \eqref{eq:rho.hierarchy} and \eqref{eq:boxdefs} we have
  \begin{align}\label{eq:t-2s}
    \frac{t-2s}{n/2-n_0}=\frac{t-2s}{\floor{\alpha n}} \leq \frac{\rho   -\beta  +1/n }{\alpha  -1/n }
      = \frac{\bigl(\crist-\frac{\epsilon}{2}\bigr)\alpha+1/n}{\alpha-1/n}
      < \frac{\crist[k]-\frac{\epsilon}{4}+1/\alpha n}{1-1/\alpha n}.
  \end{align}
  For large enough $n$, the right-hand side is bounded by $\crist[k]-\epsilon/5$,
  proving \eqref{eq:other.rate}.
\end{proof}
  
  Next, we start defining and giving estimates on
  the cells from our sketch.
  We write $x+\ii{a,b}$ to denote $\ii{a+x,b+x}$.
  \begin{prop}\thlabel{prop:RS}
    Define the following cells:
    \begin{enumerate}[(i)]
     \item
       $(R_1,S_1)_{n_0}$ lies on the $k$-greedy infection path
   from $(0,0)_0$ capped at row~$s$;
       \item 
      $(R_2,S_2)_{n-n_0}$ lies on the reverse $k$-greedy infection path from
   $(u,t)_n$ capped at row~$t-s$.
    \end{enumerate}
    It holds
    with overwhelming probability that
    \begin{align}
      S_1&=s, &
      R_1&\in \biggl(\frac{2 - \beta/\crist[k]}{4}-\alpha\biggr)\frac{\beta n^2}{2} + 
      \ii{-\delta_0n^2,\delta_0 n^2},\label{eq:1loc}\\
      S_2&=t-s,&
           R_2&\in u - \bigl(\tfrac12 -\alpha)\rho n^2 + \biggl(\frac{2 - \beta/\crist[k]}{4}-\alpha\biggr)\frac{\beta n^2}{2} + 
      \ii{-2\delta_0n^2,2\delta_0 n^2}.
      \label{eq:2loc}
  \end{align}
\end{prop}
\begin{proof}
  By \eqref{eq:s/n_0} from \thref{lem:alpha.beta}, we can apply
  \thref{prop:greedy.capped} to estimate $R_1$ and $S_1$.
  Equation~\eqref{eq:greedy.capped.s} shows that the first equality
  in \eqref{eq:1loc} holds w.o.p.
  Using \eqref{eq:greedy.capped.u} and applying \eqref{eq:boxdefs}, we have concentration of $R_1$ around
  \begin{align*}
    \biggl(n_0-\frac{s}{2\crist[k]}\biggr)s = \biggl(\frac{2-\beta /\crist[k]}{4}-\alpha\biggr)
                                               \frac{\beta n^2}{2} +O(n/\epsilon).
  \end{align*}
  (The factor of $\epsilon^{-1}$ in the error term comes from from $\crist[k]$, which is bounded
  below by $\epsilon$.) Then 
  \eqref{eq:greedy.capped.u} proves that the second part of \eqref{eq:1loc} holds w.o.p.
  
  To prove \eqref{eq:2loc}, we must first
  invoke \thref{cor:reverse} to couple $\binfset[(u,t)_n]_{n_0}$
  with $\infset[(0,0)_0]_{n_0}$ in a different layer percolation.
  Let $(Z_k)_{k\geq0}$ be the branching process from \thref{cp:reverse}, which is a critical geometric
  branching process starting from $u$ with constant emigration $t$.
  According to \thref{cor:reverse,lem:reversed.paths}, we have
  \begin{align}\label{eq:R2S2}
    (R_2,S_2)_{n-n_0} = (Z_{n_0}+R_1' + S_1',\,t-S_1')_{n-n_0} \qquad\text{if $Z_{n_0}\geq 0$,}
  \end{align}
  where $(R_1',S_1')_{n_0}$ lies on the $k$-greedy infection path
  from $(0,0)_0$ capped at row~$s$ in the coupled layer percolation. In particular,
  $(R_1',S_1')$ has the same distribution as $(R_1,S_1)$, which we have already analyzed.
  By \thref{prop:GW.emigration.concentration}, we have $Z_{n_0}$
  concentrated around $u - tn_0$, yielding
  \begin{align}\label{eq:Zn0}
    Z_{n_0}\in u - \bigl(\tfrac12 - \alpha\bigr)\rho n^2 + \ii{-\delta_0 n^2,\delta_0 n^2}
    \text{ w.o.p.}
  \end{align}
  Since $u-\frac12 \rho n^2\geq -\frac12 \rho\delta' n^2$ by \eqref{eq:u.range},
  \begin{align*}
    Z_{n_0} &\geq -\frac{\rho\delta' n^2}{2} + \alpha\rho n^2 - \delta_0 n^2 
    = \biggl(1-\frac{51\alpha}{400}\biggr)\rho \alpha n^2>0\text{ w.o.p.},
  \end{align*}
  since $\alpha\leq 1/16$.
  Thus the coupling is valid w.o.p., and from \eqref{eq:R2S2} and
  $S_1'=s$ w.o.p.\ we obtain the first part of \eqref{eq:2loc}.
  And \eqref{eq:Zn0} together with \eqref{eq:1loc} and \eqref{eq:R2S2}
  prove the second part of \eqref{eq:2loc}.
  \end{proof}
  
  Next we give estimates on the cells at step~$n/2$ that we use for applying
  \thref{lem:X}.
  \begin{prop}\thlabel{prop:XYT}
    Define the following cells:
    \begin{enumerate}[(i)]
      \item $(X_1, S_1)_{n/2}$ lies on the minimal infection path
        from $(R_1,S_1)_{n_0}$;
      \item $(Y_1,T_1)_{n/2}$ lies on the $k$-greedy infection
        path from $(R_1,S_1)_{n_0}$ capped at row~$t-s$;
      \item $(X_2,S_2)_{n/2}$ lies on the reverse minimal infection path from
        $(R_2,S_2)_{n-n_0}$;
      \item $(Y_2,T_2)_{n/2}$ lies on the reverse $k$-greedy infection
        path from $(R_2,S_2)_{n-n_0}$ capped at row~$s$.
    \end{enumerate}
    It holds with overwhelming probability that
    \begin{align}
          X_1 &\in \biggl(\frac{2 - \beta/\crist[k]}{4}\biggr)\frac{\beta n^2}{2} + 
      \ii{-3\delta_0n^2,3\delta_0 n^2}\label{eq:X1loc},\\
    Y_1 &\in \biggl(\frac{2 - \beta/\crist[k]}{4}\biggr)\frac{\beta n^2}{2}
       +  \biggl(\alpha  - \frac{\rho-\beta}{2\crist[k]}\biggr)(\rho-\beta)n^2
       +\ii{-4\delta_0n^2,4\delta_0 n^2},\label{eq:Y1loc}\\
    T_1&=t-s,\label{eq:T1loc}\\
    X_2 &\in u - \frac{\rho n^2}{2}  + \biggl(\frac{2 - \beta/\crist[k]}{4}\biggr)\frac{\beta n^2}{2}
     + \ii{-4\delta_0n^2,4\delta_0n^2},\label{eq:X2loc}\\
      Y_2&\in  u - \frac{\rho n^2}{2} + \biggl(\frac{2 - \beta/\crist[k]}{4}\biggr)\frac{\beta n^2}{2}
        +\biggl(\alpha  - \frac{\rho  -\beta  }{2\crist[k]}\biggr)(\rho  -\beta )n^2
    +\ii{-6\delta_0 n^2,6\delta_0 n^2},  \label{eq:Y2loc}\\
    T_2&=s.  \label{eq:T2loc}
    \end{align}
  \end{prop}
\begin{proof}
  For \eqref{eq:X1loc}, we observe that $X_1+s$ is the $\floor{\alpha n}$th step
  of a critical geometric branching process from $R_1+s$ with constant
  immigration $s$ by \thref{lem:ll.bound}.
  An application of \thref{prop:GW.emigration.concentration} conditional on $R_1$ shows that
  $X_1+s$ is contained in $R_1+s\floor{\alpha n}+\ii{-\delta_0n^2,\delta_0n^2}$ w.o.p.
  Applying \eqref{eq:1loc} from \thref{prop:RS} together with $s\floor{\alpha n}=\alpha\beta n^2/2+O(n)$
  yields \eqref{eq:X1loc}.
  
  To prove \eqref{eq:Y1loc} and \eqref{eq:T1loc}, we first consider a $k$-greedy infection path from $(0,0)_0$
  capped at row $t-2s$. Let $(Y',S')_{\floor{\alpha n}}$ lie on this infection path.
  We claim that
  \begin{align}\label{eq:Y'}
    Y'&=\biggl(\alpha  - \frac{\rho  -\beta  }{2\crist[k]}\biggr)(\rho  -\beta )n^2 + 
    \ii{-\delta_0 n^2,\delta_0 n^2} \text{ w.o.p.},\\
    S' &= t-2s\text{ w.o.p.}\label{eq:S'}
  \end{align}
  By \eqref{eq:other.rate} from \thref{lem:alpha.beta}, the conditions of
  \thref{prop:greedy.capped} are satisfied. We can apply it to prove
  \eqref{eq:S'} and to conclude that
  $Y'$ is concentrated around
  \begin{align*}
    \biggl(\floor{\alpha n} - \frac{t-2s}{2\crist[k]}\biggr)(t-2s)
      &= \biggl(\alpha  - \frac{\rho  -\beta  }{2\crist[k]}\biggr)(\rho  -\beta )n^2 + O_\epsilon(n),
  \end{align*}
  which proves \eqref{eq:Y'}.
  
  By \thref{prop:shift}, there exists a coupling of $Y'$ and $S'$ with our layer percolation
  so that $(Y_1,T_1)_{n/2}=(X_1+Y',\,S'+S_1)_{\floor{\alpha n}}$.
  Since $S'=t-2s$ w.o.p.\ and $S_1=s$ w.o.p.\ by \eqref{eq:1loc}, we obtain \eqref{eq:T1loc}.
  And from \eqref{eq:X1loc} and \eqref{eq:Y'}, we prove \eqref{eq:Y1loc}.
  
  The proof of \eqref{eq:X2loc}--\eqref{eq:T2loc} goes
  just like the proof of \eqref{eq:X1loc}--\eqref{eq:T1loc} but with an application
  of \thref{cor:reverse}. This time,
  we couple $\binfset[(R_2,S_2)_{n-n_0}]_{\floor{\alpha n}}$
  with a forward infection set $\infset[(0,0)_0]_{\floor{\alpha n}}$
  and branching process $(\widebar{Z}_k)_{k\geq 0}$ starting from $R_2$
  with constant emigration $S_2$.
  Applying \thref{prop:GW.emigration.concentration} conditional
  on $R_2$ and $S_2$, 
  \begin{align*}
    \widebar{Z}_{\floor{\alpha n}} \in R_2 - \floor{\alpha n}S_2 + \ii{-\delta_0 n^2,\delta_0 n^2} \text{ w.o.p.}
  \end{align*}
  By \eqref{eq:2loc} from \thref{prop:RS} together
  with $\floor{\alpha n}(t-s)=\alpha \rho n^2  - \alpha\beta n^2 /2 + O(n)$, it holds with overwhelming probability that
  \begin{align*}
    R_2 - \floor{\alpha n}S_2 &\in u - \bigl(\tfrac12 -\alpha)\rho n^2 + \biggl(\frac{2 - \beta/\crist[k]}{4}-\alpha\biggr)\frac{\beta n^2}{2} -\floor{\alpha n}(t-s)+ 
      \ii{-2\delta_0n^2,2\delta_0 n^2}\\
        &\subseteq  u - \frac{\rho n^2}{2} + \biggl(\frac{2 - \beta/\crist[k]}{4}\biggr)\frac{\beta n^2}{2}+ 
      \ii{-3\delta_0n^2,3\delta_0 n^2}.
  \end{align*}
  Hence
  \begin{align}\label{eq:Z2}
    \widebar{Z}_{\floor{\alpha n}} \in u - \frac{\rho n^2}{2} + \biggl(\frac{2 - \beta/\crist[k]}{4}\biggr)\frac{\beta n^2}{2}+ 
      \ii{-4\delta_0n^2,4\delta_0 n^2}\text{ w.o.p.}
  \end{align}
  Since $u-\frac12 \rho n^2\geq -\frac12 \rho\delta' n^2$ by \eqref{eq:u.range}
   and $\epsilon/2\leq\beta\leq\rho\leq\crist[k]$ by \eqref{eq:rho.hierarchy}
   and \eqref{eq:beta.big} from \thref{lem:alpha.beta},
  \begin{align*}
    \widebar{Z}_{\floor{\alpha n}}\geq -\frac{\rho\delta' n^2}{2} + \frac14\frac{\epsilon n^2}{4}
      - 4\delta_0 n^2
      = \biggl(1-\frac{27\rho\epsilon}{3200} \biggr)\frac{\epsilon n^2}{16}\geq 0\text{ w.o.p.},
  \end{align*}
  since $\rho$ and $\epsilon$ are both bounded by $1$.
  Hence by \thref{cor:reverse} the coupling of $\binfset[(R_2,S_2)_{n-n_0}]_{\floor{\alpha n}}$
  and $\infset[(0,0)_0]_{\floor{\alpha n}}$ is in effect w.o.p.
  The minimal infection path from $(0,0)_0$ is simply $(0,0)_0\to(0,0)_1\to\cdots$,
  and the image of this path under the transformation $(r,s)_m\mapsto (\widebar{Z}_m + r+s,S_2-s)_m$
  is the reverse minimal infection path from $(R_2,S_2)_{n-n_0}$ under the coupling
  by \thref{lem:reversed.paths}. Hence
  \eqref{eq:Z2} together with $S_2=t-s$ w.o.p.\ from \thref{prop:RS} proves \eqref{eq:X2loc}.
  
  For \eqref{eq:Y2loc} and \eqref{eq:T2loc}, we consider a $k$-greedy infection path from $(0,0)_0$
  capped at $s-2t$, whose location after $\floor{\alpha n}$ steps
  we have already determined in \eqref{eq:Y'} and \eqref{eq:S'}. 
  The image of this path under the transformation $(r,s)_m\mapsto (\widebar{Z}_m + r+s,S_2-s)_m$
  is a reverse $k$-greedy infection path from $(R_2,S_2)_{n-n_0}$ capped at 
  $S_2-(s-2t)=t$ w.o.p.\ by \eqref{eq:2loc}.
  Then $T_2= S_2-S'$ w.o.p., which by \eqref{eq:2loc} and \eqref{eq:S'} proves \eqref{eq:T2loc}.
  In the same way, $Y_2=\widebar{Z}_{\floor{\alpha n}}+Y'+S'$ w.o.p., and combining
  \eqref{eq:Z2}, \eqref{eq:Y'}, and \eqref{eq:S'} proves \eqref{eq:Y2loc}.  
  \end{proof}
  
  \subsection{Proof of Lemma~\ref{lem:box}}
  \label{sec:finalboxproof}
  
  First, we apply the results from the previous section
  to prove the likely infection of a single cell $(u,t)_n$.
  \begin{prop}\thlabel{prop:one.cell.box}
    Let $n$ be even. Let $t$ be an integer and $\rho=t/n$.
    Suppose that $0<\epsilon\leq\rho\leq\crist-\epsilon$.
    Let $\delta'=\epsilon^2/1024$.
    Suppose that
    \begin{align*}
      u\in\ii[\bigg]{\frac{\rho}{2}(1-\delta') n^2,\,\frac{\rho}{2}(1+\delta') n^2}.
    \end{align*}
    For some constants $c=c(\lambda,\epsilon)$ and $C=C(\lambda,\epsilon)$,
    \begin{align*}
      \P\bigl[ (0,0)_0\to(u,t)_n\bigr] \geq 1 - Ce^{-cn}.
    \end{align*}
  \end{prop}
  \begin{proof}
   
  Define $k$, $\alpha$, $\beta$, $n_0$, $s$, and $\delta_0$
  as in Section~\ref{sec:box.paths}, and note that our definition of $\delta'$
  agrees with the definition $\delta'=\alpha^2/4$ in Section~\ref{sec:box.paths}.
  We consider the cells $(X_1,S_1)_{n/2}$, $(Y_1,T_1)_{n/2}$,
  $(X_2,S_2)_{n/2}$, and $(Y_2,T_2)_{n-n_0}$ from \thref{prop:XYT}.
  By construction, cells $(X_1,S_1)_{n/2}$ and
  $(Y_1,T_1)_{n/2}$ are in $\infset[(R_1,S_1)_{n_0}]_{n/2-n_0}$
  and cells $(X_2,S_2)_{n/2}$ and
  $(Y_2,T_2)_{n/2}$ are in $\binfset[(R_2,S_2)_{n-n_0}]_{n/2-n_0}$. Our goal is to
  show these cells are likely to satisfy the criteria of \thref{lem:X}.
  Then $(R_1,S_1)_{n_0}\to(R_2,S_2)_{n-n_0}$, which then
  proves $(0,0)_0\to(u,t)_n$.

  In \thref{prop:RS,prop:XYT}, we showed that $S_1=T_2=s$ and
  $S_2=T_1=t-s$ w.o.p.
  We  now show that $X_1\leq Y_2$ and
  $X_2\leq Y_1$ w.o.p.
  From \thref{prop:XYT},
  \begin{align*}
    Y_2-X_1 &\geq   \biggl(u - \frac{\rho n^2}{2} \biggr)
        +\biggl(\alpha  - \frac{\rho  -\beta  }{2\crist[k]}\biggr)(\rho  -\beta )n^2 
     - 9\delta_0n^2 \text{ w.o.p.,}\\
    Y_1-X_2 &\geq \biggl(\frac{\rho n^2}{2}-u\biggr)+
         \biggl(\alpha  - \frac{\rho-\beta}{2\crist[k]}\biggr)(\rho-\beta)n^2
       -8\delta_0n^2 \text{ w.o.p.}
  \end{align*}
  By our choice of $u$, we have $\abs{u-\rho n^2/2}\leq \rho\delta' n^2/2$.
  Considering the next term in the bounds,
  \begin{align*}
    \biggl(\alpha  - \frac{\rho  -\beta  }{2\crist[k]}\biggr)(\rho  -\beta )n^2
      &= \biggl(1 - \frac{\crist-\epsilon/2}{2\crist[k]}\biggr)\bigl(\crist-\epsilon/2\bigr)\alpha^2n^2.
  \end{align*}
  Since $1 - (\crist-\epsilon/2)/2\crist[k]\geq1/2$
  and $\crist-\epsilon/2\geq\rho$,
  \begin{align*}
    \biggl(\alpha  - \frac{\rho  -\beta  }{2\crist[k]}\biggr)(\rho  -\beta )n^2
      \geq \frac{\rho\alpha^2n^2}{2}=2\rho\delta' n^2.
  \end{align*}
  Hence $Y_2-X_1\geq (2 -1/2 - 9/100)\rho\delta' n^2>0$ w.o.p., and 
  $Y_1-X_2\geq (2-1/2-8/100)\rho\delta' n^2>0$ w.o.p.
  Now \thref{lem:X} applies to prove
  that $\infset[(R_1,S_1)_{n_0}]_{n/2-n_0}$
  and $\binfset[(R_2,S_2)_{n-n_0}]_{n/2-n_0}$ contain a common cell, 
  proving that $(R_1,S_1)_{n_0}\to (R_2,S_2)_{n-n_0}$ w.o.p.
  Since $(0,0)_0\to(R_1,S_1)_{n_0}$ and $(R_2,S_2)_{n-n_0}\to(u,t)_n$ by construction,
  this proves that $(0,0)\to(u,t)_n$ w.o.p.
  \end{proof}
  
  All that remains is to apply the previous proposition
  for all cells $(u,t)_n$ in a box.
  
  \begin{proof}[Proof of \thref{lem:box}]
    Without loss of generality we can take $n$ to be even, since $(0,0)_0\to(0,0)_1$
    always holds and then we can apply the result starting at step~$1$.
    We may also assume that $0<2\epsilon\leq\rho\leq\crist-2\epsilon$
    rather than $0<\epsilon\leq\rho\leq\crist-\epsilon$.
    Take $\delta=\epsilon^2/2048$.
    Let $\rho_t:=t/n$, and observe that
    $\epsilon\leq \rho_t\leq\crist-\epsilon$ for $t\in\ii[\big]{\rho(1-\delta)n,\,\rho(1+\delta) n}$.
    As before, we say that an event happens with overwhelming probability
    if it occurs with probability $1-Ce^{-cn}$ for constants $c,C>0$
    that may depend on $\epsilon$ and $\lambda$.
    Applying \thref{prop:one.cell.box} over a range of $O(n^2)$ choices of $u$
    and $O(n)$ choices of $t$,
    it holds with overwhelming probability that
    \begin{align*}
      \biggl\{ (u,t)_n\colon\ 
      t\in\ii[\Big]{\rho(1-\delta)n,\,\rho(1+\delta) n},\ 
      u\in\ii[\Big]{\frac{\rho_{t}}{2}(1-\delta') n^2,\,\frac{\rho_{t}}{2}(1+\delta') n^2}
       \biggr\}\subseteq\infset[(0,0)_0]_n,
    \end{align*}
    where $\delta'=\epsilon^2/1024=2\delta$.
    To prove \eqref{eq:box.stmt}, it suffices to show that
    for $t\in\ii[\big]{\rho(1-\delta)n,\,\rho(1+\delta) n}$,
    \begin{align*}
      \ii[\Big]{\frac{\rho}{2}(1-\delta) n^2,\,\frac{\rho}{2}(1+\delta) n^2}
      \subseteq \ii[\Big]{\frac{\rho_{t}}{2}(1-\delta') n^2,\,\frac{\rho_{t}}{2}(1+\delta') n^2}.
    \end{align*}
    And this holds because $\delta'=2\delta'$ and hence
    for $t\in\ii[\big]{\rho(1-\delta)n,\,\rho(1+\delta) n}$,
    \begin{align*}
      \rho_t(1-\delta')&\leq \rho(1+\delta)(1-\delta')\leq\rho(1-\delta)\\\intertext{and}
      \rho_t(1+\delta')&\geq \rho(1-\delta)(1+\delta')\geq\rho(1+\delta).\qedhere
    \end{align*}    
  \end{proof}

\section{Upper tail bounds around the critical density}
\label{sec:upper.bound}

Let $X_n$ be the highest row infected by layer percolation at step~$n$,
as in \thref{def:cristk}.
In Section~\ref{sec:critical.density}, we proved that $X_n/n\to\crist$
in probability and gave an exponential bound on the lower tail of $X_n$.
 In this section we complete the picture by proving
an upper tail bound on $X_n$.

\begin{prop}\thlabel{thm:large.deviations}
  Consider layer percolation with parameter $\lambda>0$, and fix $\rho>\crist(\lambda)$. There exist
  $C,c>0$ depending only on $\lambda$ and $\rho$ such that
  \begin{align*}
    \P[ X_n \geq\rho n  ] \leq Ce^{-cn}.
  \end{align*}
\end{prop}

In Section~\ref{sec:peanut.butter}, we will translate this bound to odometers,
thus proving that activated random walk on an interval
rarely stabilizes with a density above $\crist$.

Let $\badcell^{\rho}_n(r,s)_k$ denote the event that $(r,s)_k$ infects
a cell $\rho n$ rows above it in $n$ steps, i.e.,
the infection set $\infset[(r,s)_k]_n$ contains a cell $(r',s')_{n+k}$ with $s'-s\geq \rho n$.
Let $\badbox^\rho_n(r,s)_k$ denote the event that $\badcell^\rho_n(r',s')_k$
holds for some $(r',s')_k$ with $r\leq r'<r+n^2$
and $s\leq s'<s+n$. 
The probabilities of both events depend on $\rho$, $\lambda$, and $n$ but not on the cell $(r,s)_k$:
the step~$k$ is irrelevant since the infections at different steps of layer percolation are i.i.d.,
and $r$ and $s$ are irrelevant by \thref{prop:shift}. We write $\badcell^\rho_n$ and
$\badbox^\rho_n$ as abbreviations for $\badcell^\rho_n(0,0)_0$
and $\badbox^\rho_n(0,0)_0$.

We start with a sketch of the proof of \thref{thm:large.deviations}.
Our goal is to produce an exponential
bound on the probability of $\badcell^\rho_n$.
From the convergence in probability of $X_n/n$ to $\crist$,
we have the weaker statement
\begin{align}\label{eq:badcell}
  \lim_{n\to\infty}\P\bigl[\badcell^\rho_n\bigr]=0.
\end{align}
Our first step is to prove \thref{lem:badbox}, which improves this to
\begin{align}\label{eq:badbox}
  \lim_{n\to\infty}\P\bigl[\badbox^\rho_n\bigr]=0.
\end{align}
We cannot use a union bound here since we have no rate of
convergence for \eqref{eq:badcell}, but the result is
an easy consequence of the regularity established by \thref{lem:box}.

Next, we take $k$ to be a large but fixed constant, and we partition the cells at each level
into boxes of size $k^2\times k$. We fix a sequence of boxes at steps $k,2k,\ldots, jk$ for $j=\floor{n/k}$,
say with lower-left corners $(r_i,s_i)_{ik}$ for $1\leq i\leq j$.
We consider the event that there exists an infection path from $(0,0)_0$ through these
boxes ending beyond row $\rho n$. 
This event can only occur if $\badbox^{\rho'}_k(r_i,s_i)_{ik}$ occurs for a positive
proportion of $i\in\ii{1,j}$, where $\rho'$ is some fixed value between $\crist$ and $\rho$.
Since these events $\badbox^{\rho'}_k(r_i,s_i)_{ik}$ are independent, the probability
of this occurring is exponentially small in $j$, and
by \eqref{eq:badbox} we can control the rate by increasing $k$.

This exponential bound applies only to a single, fixed sequence of boxes.
We would like to apply it in a union bound over all possible sequences of boxes,
but there are too many. Using the tools provided in Section~\ref{sec:bounds.on.infection.paths},
however, we can show that most sequences of boxes are unlikely to have infection paths through them,
since the infection set from any particular box after $k$ steps is concentrated on a $k^2\times k$
set of cells, i.e., on $O(1)$ many boxes. Thus there are effectively only
$C^j$ sequences of boxes to consider, with $C$ independent of $k$.
Increasing $k$ until our exponential probability bound beats this
rate of growth $C^j$, we prove \thref{thm:large.deviations}.

We start now on the proof now, first extending \eqref{eq:badcell} to boxes of size $n^2\times n$.

\begin{lemma}\thlabel{lem:badbox}
  For any fixed $\rho>\crist$ and $\lambda>0$,
  \begin{align*}
    \lim_{n\to\infty}\P\bigl[\badbox^{\rho}_n\bigr]=0.
  \end{align*}
\end{lemma}
\begin{proof}
  For $\theta>0$ to be specified, let $\badboxdelta^{\rho}_n(r,s)_k$ denote the same event
  as $\badbox^{\rho}_n(r,s)_k$, except that the box has dimensions $\theta n^2\times \theta n$ rather
  than $n^2\times n$. That is, $\badboxdelta^\critical_n(r,s)_k$ denotes the event that
  that there exist $r'\in\ii{r,r+\theta n^2-1}$
  and $s'\in\ii{s,s+\theta n-1}$ such that $\badcell^\critical_n(r',s')_k$ holds.
  By a union bound, for any $0<\theta<1$,
  \begin{align}\label{eq:delta.mod}
    \P\bigl[\badbox_n^{\rho}\bigr]\leq 
      \ceil[\bigg]{\frac{1}{\theta^2}}\P\bigl[\badboxdelta^{\rho}_n(0,0)_0\bigr].
  \end{align}
  Thus it suffices to show $\P\bigl[\badboxdelta_n^\rho(r,s)_k\bigr]\to 0$ for 
  some $\theta>0$ not depending on $n$.
  
  For some $r$, $s$, $k$, and $\theta$ to be given, consider the events
  $\badboxdelta^\rho_n(r,s)_k$ and
  \begin{align*}
    \infectbox(r,s)_k := \Bigl\{ (r',s')_k\in\infset[(0,0)_0]_k\text{ for all $r\leq r'<r+\theta n^2$
      and $s\leq s'<s+\theta n$}\Bigr\}.
  \end{align*}
  If $\infectbox(r,s)_k$ and $\badboxdelta^\rho_n(r,s)_k$ both occur, then
  all cells in $\ii{r,r+\theta n^2-1}\times\ii{s,s+\theta n-1}$
  are infected from $(0,0)_0$ at step~$k$, and some cell in this box
  infects a cell at step~$n+k$ at or above row~$s+\rho n$.
  Hence $\badcell_{n+k}^{\rho'}(0,0)_0$ occurs
  for any $\rho'\leq(s+\rho n)/(n+k)$. 
  Restating this conclusion symbolically, for any $\rho'\leq(s+\rho n)/(n+k)$,
  \begin{align}\label{eq:event.rel}
    \badboxdelta^\rho_n(r,s)_k\subseteq(\infectbox(r,s)_k)^c\cup\badcell_{n+k}^{\rho'}(0,0)_0.
  \end{align}
  Now, to bound the probability of $\badboxdelta^\rho_n(r,s)_k$,
  we will bound the failure probability of $\infectbox(r,s)_k$ with \thref{lem:box}
  and bound the probability of $\badcell_{n+k}^{\rho'}(0,0)_0$ with  \eqref{eq:badcell}.

  We choose $r$, $s$, $k$, and $\theta$ now and carry out the proof.
  First choose $\eta>0$ small enough that $\rho':=\rho/(1+\eta)>\crist$,
  and let $k=\floor{\eta n}$.  
  Let $\pi=\crist/2$. Applying \thref{lem:box} with $\pi$ playing the role of $\rho$
  and with $\epsilon=\crist/2$, all cells at step~$k$ in the box
  \begin{align*}
    \Blb(\pi,\delta):=\ii[\bigg]{\frac{\pi}{2}(1-\delta) k^2,\,\frac{\pi}{2}(1+\delta) k^2}
    \times  \ii[\Big]{\pi(1-\delta)k,\,\pi(1+\delta) k}
  \end{align*}
  are infected with probability approaching $1$ as $k\to\infty$, where $\delta$
  is a constant depending only on $\lambda$ (which determines $\epsilon$ here).
  Thus, we set
  \begin{align*}
    r &= \ceil[\bigg]{\frac{\pi}{2}(1-\delta) k^2},\quad\quad
    s = \ceil[\big]{\pi(1-\delta)k},\quad\quad\text{and}\quad\quad
    \theta = \frac{\delta \pi\eta^2}{2}.
  \end{align*}
  Note that $r$, $s$, and $k$ depend on $n$ but $\theta$ is fixed.
  We have
  \begin{align*}
    \ii{r,r+\theta n^2-1}\times\ii{s,s+\theta n-1}\subseteq
      \Blb(\pi,\delta),
  \end{align*}
  and hence $\P\bigl[\infectbox(r,s)_k\bigr] \to 1$  as $n\to\infty$.
  And since $\rho'>\crist$, we have $\P\bigl[\badcell^{\rho'}_{n+k}\bigr]\to 0$
  as $n\to\infty$.
  Since $\rho'\leq (s+\rho n)/(n+k)$, we can apply \eqref{eq:event.rel} to conclude that
  \begin{align*}
    \lim_{n\to\infty}\P\bigl[\badboxdelta^\rho_n(r,s)_k\bigr]=0.
  \end{align*}
  Noting that the probability of $\badboxdelta^\rho_n(r,s)_k$ does not depend on 
  $(r,s)_k$ and applying \eqref{eq:delta.mod} completes the proof.
\end{proof}

Fix a positive integer $k$ and $\crist<\rho'<\rho$ and consider a sequence of cells
\begin{align}\label{eq:signatured}
  (r_0,s_0)_0,\,(r_1,s_1)_1,\,\ldots,\,(r_{jk},s_{jk})_{jk}
\end{align}
that might or might not be an infection path. Typically we let $j=\floor{n/k}$,
so that our potential infection path is a sequence of cells nearly of length~$n$.
We will associate with this sequence
an object $(a,b,c)$ with $a=(a_1,\ldots,a_j)$, $b=(b_1,\ldots,b_j)$, and $c=(c_1,\ldots,c_j)$
that we call the \emph{$k$-signature} of the possible infection path; note that it
will also depend on $\rho'$.

To summarize the $k$-signature before we define it, imagine partitioning the cells at step~$ik$
of layer percolation into boxes
\begin{align*}
  \bx{i}{x}{y}=\ii{xk^2,(x+1)k^2-1}\times\ii{yk,(y+1)k-1},\qquad x\geq 0,\ y\geq 0.
\end{align*}
For $i=0,\ldots,j$, let $x_i=\floor{r_{ik}/k^2}$ and $y_i=\floor{s_{ik}/k^2}$, so that
$\bx{i}{x_i}{y_i}$ specifies the box that \eqref{eq:signatured} goes through at step~$ik$.
\begin{itemize}
  \item $a_1,\ldots,a_j$ encodes the same information as $x_1,\ldots,x_j$, normalized so that
    the columns of each box are specified relative to the minimal infection path from the previous
    box;
  \item $b_1,\ldots,b_j$ encodes the same information as $y_1,\ldots,y_j$, normalized as above;
  \item each $c_i$ encodes whether the infection path increases in row by an unusual amount
    from step~$(i-1)k$ to step~$ik$, i.e., whether $s_{ik}-s_{(i-1)k}$ is unusually large.
\end{itemize}
Eventually, we will bound
the probability of the existence of an infection path with a given $k$-signature.
Then, considering all possible choices of $a_i$ and $b_i$ and all choices of $c_i$ consistent
with an infection path confirming $\badcell^{\rho}_n$, we will take a union bound
over $k$-signatures. In principle we could work directly with $x_0,\ldots,x_j$ and $y_0,\ldots,y_j$ rather
than transforming them to define the $k$-signature,
but it will be simpler to work with the transformed version.

Let us define the $k$-signature now.
Fix $i\in\{1,\ldots,j\}$.
Let $u_0=x_{i-1}k^2$ and $t_0=y_{i-1}k$, so that $(u_0,t_0)_{(i-1)k}$ is the lower-left
corner of the box $\bx{i-1}{x_{i-1}}{y_{i-1}}$ containing cell $(r_{(i-1)k},s_{(i-1)k})_{(i-1)k}$.
Let
\begin{align}\label{eq:minpathi}
  (u_0,t_0)_{(i-1)k}\to(u_1,t_0)_{(i-1)k+1}\to\cdots\to (u_k,t_0)_{ik}
\end{align}
be the minimal infection path starting from $(u_0,t_0)_{(i-1)k}$, as defined in
Section~\ref{sec:bounds.on.infection.paths}.
Let $\xmin_i = \floor{u_k/k^2}$ and $\ymin_i=\floor{t_0/k}=y_{i-1}$.
Thus $\bx{i}{\xmin_i}{\ymin_i}$ is the lower-left box at step~$ik$ containing
a cell infected starting from within $\bx{i-1}{x_{i-1}}{y_{i-1}}$
by \thref{lem:ll.bound}\ref{i:ll.bound.a}.
We now define $a_i=x_i-\xmin_i$ and $b_i=y_i-\ymin_i=y_i-y_{i-1}$.
Finally, define $(c_1,\ldots,c_j)$ by
\begin{align*}
  c_i = \1\{s_{ik} - s_{(i-1)k}> \rho'k\}\in\{0,1\}.
\end{align*}

An arbitrary sequence of cells \eqref{eq:signatured} could have any $a_1,\ldots,a_n$
and $b_1,\ldots,b_n$ in its $k$-signature. But if \eqref{eq:signatured} is truly an infection path, 
we must have $a_i\geq 0$ and $0\leq b_i\leq 1$ for all $i$:
\begin{lemma}\thlabel{lem:allowable.signatures}
  If $(a,b,c)$ is the $k$-signature of an infection path
  \begin{align*}
    (0,0)_0=(r_0,s_0)_0\to(r_1,s_1)_1\to\cdots\to(r_{jk},s_{jk})_{jk},
  \end{align*}  
  then $a_i\geq 0$, $b_i\in\{0,1\}$, and $c_i\in\{0,1\}$ for all $i\in\ii{1,j}$.
\end{lemma}
\begin{proof}
  Fix $i$ and consider the minimal infection path \eqref{eq:minpathi}
  starting from the lower-left corner of the box containing $(r_{(i-1)k},s_{(i-1)k})_{(i-1)k}$.
  We have $(r_{ik},s_{ik})_{ik}\in\infset[(r_{(i-1)k},s_{(i-1)k})_{(i-1)k}]_k$,
  and hence $r_{ik}\geq u_k$ and $s_{ik}\geq t_0$ by \thref{lem:ll.bound}\ref{i:ll.bound.a}.
  By definition of $a_i$ and $b_i$, we have $a_i\geq 0$ and $b_i\geq 0$.
  And $b_i\leq 1$ because an infection path can move up at most one row in each step,
  and hence $s_{ik}-s_{(i-1)k}\leq k$. And $c_i\in\{0,1\}$ since it is an indicator
  by definition.  
\end{proof}

Next, we apply \thref{lem:ll.bound,lem:ur.bound} to show that an infection
path starting from a box is unlikely to move too far from the minimal infection
path from the box. The point of this
is to show that it is unlikely for there to exist an infection path whose
$k$-signature has an especially large $a_i$.
\begin{lemma}\thlabel{lem:box.spread}
  Given $x,y\geq 0$, let $\bx{1}{x_{\min}}{y}$ be the box at step~$k$ containing
  the minimal infection path starting from $(xk^2,yk)_0$, the lower-left corner
  of  $\bx{0}{x}{y}$.
  For any $t\geq 0$, the probability that there is
  an infection path starting within $\bx{0}{x}{y}$ and ending in some
  box $\bx{1}{x'}{y'}$ with $x'-x_{\min}\geq t$ is at most $Ce^{-ct}$
  for absolute constants $c,C>0$.
\end{lemma}
\begin{proof}
  Let $U$ be the maximum column infected at step~$k$ starting within
  $\bx{0}{x}{y}$, and let $(u_n,yk)_k$ lie on the minimal infection
  path from $(xk^2,yk)_0$.
  Suppose there is an infection path starting within $\bx{0}{x}{y}$ and ending in
  box $\bx{1}{x'}{y'}$ with $x'-x_{\min}\geq t$.
  Then there are at least $t-1$ boxes strictly between the ones containing columns $U$ and $u_n$,
  and hence $U-u_n\geq(t-1)k^2$.
  Thus it suffices to show that
  $\P[U-u_n\geq (t-1)k^2]\leq Ce^{-ct}$ for absolute constants $c,C$.
  By \thref{prop:shift}, this can be shown by demonstrating that with $U'$ the maximum column
  infected starting within $\bx{0}{0}{0}$,
  \begin{align}\label{eq:box.spread.goal}
    \P[U'\geq (t-1)k^2]\leq Ce^{-ct}.
  \end{align}
  
  To show this, let 
  \begin{align*}
    (k^2,k)_0=(r_0,k)_0,\ (r_1,k+1)_{1},\ (r_2,k+2)_2,\ldots,\,(r_k,2k)_{k}
  \end{align*}
  be the upper-right cell sequence from $(k^2,k)_0$, defined prior to \thref{lem:ur.bound}.
  By \thref{lem:ur.bound}\ref{i:ur.bound.a}, all infection paths starting within
  $\bx{0}{0}{0}$ arrive in step~$k$ at a column bounded by $r_k$.
  And by \thref{lem:ur.bound}\ref{i:ur.bound.b}, the sequence $(r_j+j+1)_{0\leq j\leq k}$
  is a critical geometric branching process with immigration $j+1\leq k+1$ 
  after each step $1\leq j\leq k$.
  Hence $\E[ r_k+k+1]=k^2 + k(k+1)/2+k=3(k^2+k)/2$, and
  \begin{align*}
    \P[r_k\geq tk^2] &= \P[r_k+k+1 - \E[u_k+k+1]\geq tk^2+k+1-3(k^2+k)/2]\\
      & \leq \P[r_k+k+1 - \E[u_k+k+1]\geq(t-2)k^2]\\
      &\leq C\exp\biggl(-\frac{c(t-2)^2k^4}{k(k(k+1)+k^2+k+1+(t-2)k^2)}\biggr)\leq C'e^{-c'tk}\leq C'e^{-c't}
  \end{align*}
  by \thref{prop:GW.emigration.concentration}, for absolute constants
  $c'$ and $C'$, thus establishing \eqref{eq:box.spread.goal}.
\end{proof}

For given sequences $a=(a_1,\ldots,a_j)$, $b=(b_1,\ldots,b_j)$, and $c=(c_1,\ldots,c_j)$,
we write $\Sig_{j}(a,b,c)$ (with implicit dependence on $k$ and $\rho'$)
to denote the event that there exists an infection path
starting from $(0,0)_0$ with $k$-signature $(a,b,c)$ in an
underlying layer percolation with parameter $\lambda>0$.
We will sometimes
write $\Sig_{j}(a,b,c)$ with sequences $a$, $b$, and $c$ of length greater than $j$, 
interpreting them as being truncated at their $j$th terms. For convenience we interpret
$\Sig_{0}(a,b,c)$ as the entire probability space.

Our next lemma bounds the probability of $\Sig_{j}(a,b,c)$.
Eventually, we will apply this bound via a union bound over all possible $k$-signatures
to complete the proof of \thref{thm:large.deviations}.

\begin{prop}\thlabel{lem:sig.bound}
  Let $p_k=\P\bigl[\badbox^{\rho'}_k\bigr]$. For absolute constants $C,\kappa>0$,
  it holds for all sequences $a$, $b$, and $c$ satisfying $a_i\geq 0$ and $b_i,c_i\in\{0,1\}$ for all $i$ that
  \begin{align}\label{eq:sig.bound}
    \P\bigl[ \Sig_{j}(a,b,c)\bigr]\leq\prod_{i=1}^j\min\bigl(Ce^{-\kappa a_{i}},\,p_k^{c_{i}}\bigr).
  \end{align}
\end{prop}
\begin{proof}
  Let $\Fscr_i$ denote the $\sigma$-algebra generated by all
  information in the underlying layer percolation up to 
  level~$i$, i.e., the $\sigma$-algebra generated by the random variables $R_0,R_1,\ldots$
  and $B^0,B^1,\ldots$ defining the infections from level $\ell-1$ to $\ell$, for all
  $\ell$ from $1$ up to $i$.
  We will show that 
  \begin{align}\label{eq:sig.bound2}
    \P\bigl[ \Sig_{j+1}(a,b,c)\bigmid \Fscr_{jk} \bigr] \leq \min\bigl(Ce^{-\kappa a_{j+1}},\,p_k^{c_{j+1}}\bigr),
  \end{align}
  from which the proposition will follow by induction.
  First, we observe that the information in $\Fscr_{jk}$ determines the boxes on levels
  $1,k,2k,\ldots,jk$ that an infection path must go through if it is to have $k$-signature
  $(a,b,c)$. That is, 
  there is a unique sequence 
  $(x_1,y_1),\,\ldots,(x_{j},y_{j})$ measurable with respect to $\Fscr_{jk}$
  such that if $(r_{ik},s_{ik})_{ik}$ is not in $\bx{i}{x_i}{y_i}$ for any $i\in{1,j}$, then
  \begin{align*}
    (r_0,s_0)_0,\,\ldots,\,(r_{(j+1)k},s_{(j+1)k})_{(j+1)k}
  \end{align*}
  does not have $k$-signature $(a,b,c)$.
  
  To bound the left-hand side of \eqref{eq:sig.bound2} by $p_k^{c_{j+1}}$, we assume
  $c_{j+1}=1$ and must show that
  \begin{align}\label{eq:sig.bound.pk}
    \P\bigl[ \Sig_{j+1}(a,b,c)\bigmid \Fscr_{jk} \bigr] \leq p_k.
  \end{align}
  Since $\Sig_{j+1}(a,b,c)$ can only occur if $\badbox^{\rho'}_k(x_jk^2,y_jk)_{jk}$ occurs,
  and this event still has probability $p_k$ after conditioning on  
  $\Fscr_{jk}$, we obtain \eqref{eq:sig.bound.pk}.
  
  For the other bound on \eqref{eq:sig.bound}, 
  we observe that if $\Sig_{j+1,k}(a,b,c)$ occurs, then some cell in $\bx{j}{x_j}{y_j}$
  infects a cell at step~$(j+1)k$ in a box that is at least $a_{j+1}$ boxes beyond
  the box infected by the minimal infection path from $(x_jk^2,y_jk)_{jk}$.
  By \thref{lem:box.spread}, the probability of this decays exponentially
  in $a_{j+1}$, completing the proof of \eqref{eq:sig.bound2}.
  
  Finally, we observe that if $\Sig_{j+1,k}(a,b,c)$ holds, then $\Sig_{j,k}(a,b,c)$ holds.
  Since $\Sig_{j,k}(a,b,c)$ is measurable with respect to $\Fscr_{jk}$,
  it follows from \eqref{eq:sig.bound2} that
  \begin{align*}
    \P\bigl[\Sig_{j+1,k}(a,b,c)\bigr] &\leq \P\bigl[\Sig_{j,k}(a,b,c)\bigr]
       \min\bigl(Ce^{-\kappa a_{j+1}},\,p_k^{c_{j+1}}\bigr),
  \end{align*}
  and now the proposition follows by induction.
\end{proof}

We are nearly ready to prove \thref{thm:large.deviations}, which we will accomplish by
summing the bound \eqref{eq:sig.bound} over all $k$-signatures that would allow
$\badcell^{\rho}_n(0,0)_0$ to occur. First we give a technical lemma obtained
by summing the right-hand
the sum of \eqref{eq:sig.bound} over all choices of $a=(a_1,\ldots,a_j)$.
\begin{lemma}\thlabel{lem:suma}
  Let $p_k=\P\bigl[\badbox^{\rho'}_k\bigr]$.
  For fixed $b,c\in\{0,1\}^j$,
  \begin{align*}
      \sum_{a\in\NN^j}\P\bigl[\Sig_{j,k}(a,b,c)\bigr]\leq 
        A^{j}\bigl(Bp_k\log\bigl(\tfrac{1}{p_k}\bigr)\bigr)^{\abs{c}},
  \end{align*}
  where $A$ and $B$ are absolute constants and $\abs{c}$ denotes $\sum_{j=1}^jc_j$.
\end{lemma}
\begin{proof}
  Applying \thref{lem:sig.bound}, we bound
  \begin{align*}
    \sum_{a\in\NN^j}\P\bigl[\Sig_{j,k}(a,b,c)\bigr]\leq 
      \sum_{a\in\NN^j}\prod_{i=1}^j\min\bigl(Ce^{-\kappa a_{i}},\,p_k^{c_{i}}\bigr)
      = \prod_{i=1}^j\sum_{\ell=0}^{\infty}\min\bigl(Ce^{-\kappa \ell},\,p_k^{c_{i}}\bigr)
  \end{align*}
  Regardless of $c_i$,
  \begin{align*}
    \sum_{\ell=0}^{\infty}\min\bigl(Ce^{-\kappa \ell},\,p_k^{c_{i}}\bigr)
      &\leq \sum_{a_i=0}^{\infty}Ce^{-\kappa \ell} \leq A,
  \end{align*}
  where $1\leq A<\infty$ is an absolute constant.
  If $c_i=1$, then we let $L=\ceil{\log(C/p_k)/\kappa}$ and break
  the sum in two parts at $L$ to obtain  
  \begin{align}\label{eq:a.bound2}
    \sum_{\ell=0}^{\infty}\min\bigl(Ce^{-\kappa \ell},\,p_k^{c_{i}}\bigr) \leq
    L p_k + \sum_{\ell=L}^{\infty}C e^{-\kappa\ell}\leq
    Lp_k + \frac{p_k}{1-e^{-\kappa}}\leq B p_k\log(1/p_k)
  \end{align}
  for an absolute constant $B$.
  Thus we obtain
  \begin{align*}
    \sum_{a\in\NN^j}\P\bigl[\Sig_{j,k}(a,b,c)\bigr]&\leq 
      A^{j-\abs{c}}(Bp_k\log(1/p_k))^{\abs{c}}.\qedhere
  \end{align*}
\end{proof}

\begin{proof}[Proof of \thref{thm:large.deviations}]
  If $\rho>1$, then the proposition holds trivially since the infection set can move upward
  by at most one row per step, and hence $X_n\leq n$ a.s.
  Thus we can assume that $\crist<\rho\leq 1$.
  
  First, we choose $\rho'\in(\crist,\rho)$ and $\beta\in(0,1)$ sufficiently small that
  $(1-\beta)\rho'+\beta<\rho$. To be concrete, take $\rho'=(\crist+\rho)/2$
  and $\beta=(\rho-\crist)/(4-2\rho-2\crist)$, which yields
  \begin{align*}
    (1-\beta)\rho' + \beta = \frac{\crist + 3\rho}{4}<\rho.
  \end{align*}
  Fix an integer $k$ to be specified later, and let $j=\floor{n/k}$.
  
  Suppose an infection path of length~$jk$ starting from $(0,0)_0$ has $k$-signature
  $(a,b,c)$ with $\abs{c}\leq\beta j$. From step~$(i-1)k$ to $ik$, the infection path
  increases in row by at most $\rho'k$ if $c_i=0$ and by at most $k$ if $c_i=1$.
  Thus at step~$jk$, the infection path is at row at most
  \begin{align*}
    (j-\abs{c})\rho'k + \abs{c}k\leq (1-\beta )\rho'n + \beta n \leq \frac{\crist + 3\rho}{4} n.
  \end{align*}
  If the infection path continues to step~$n$, it can reach at most $k$ rows higher, which
  is still less than $\rho n$ assuming $n\geq n_0(\rho,k)$. Thus we conclude
  that $\badcell^{\rho}_n(0,0)_0$ can occur only if
  $\Sig_{j,k}(a,b,c)$ occurs for $\abs{c}>\beta j$. Now by a union bound,
  \begin{align*}
    \P\bigl[ \badcell^{\rho}_n(0,0)_0 \bigr] &\leq 
      \sum_{\substack{c\in\{0,1\}^j \\\abs{c}>\beta j}} \sum_{b\in\{0,1\}^j}\sum_{a\in\NN^j}
      \P\bigl[\Sig_{j,k}(a,b,c)\bigr]
      \leq \sum_{\substack{c\in\{0,1\}^j \\\abs{c}>\beta j}}
         2^jA^{j}\bigl(Bp_k\log\bigl(\tfrac{1}{p_k}\bigr)\bigr)^{\abs{c}},
  \end{align*}
  applying \thref{lem:suma} and summing over the $2^j$ values of $b\in\{0,1\}^j$.
  By \thref{lem:badbox}, we can choose $k$ large enough that
  \begin{align*}
    4A\Bigl(Bp_k\log\bigl(\tfrac{1}{p_k}\bigr)\Bigr)^\beta \leq 1/2.
  \end{align*}
  Applying this bound and summing over the at most $2^j$ values of $c\in\{0,1\}^j$
  with $\abs{c}>\beta j$,
  \begin{align*}
    \P\bigl[ \badcell^{\rho}_n(0,0)_0 \bigr] &\leq 
      (4A)^j\bigl(Bp_k\log\bigl(\tfrac{1}{p_k}\bigr)\bigr)^{\beta j}\leq 2^{-j}\leq 2^{-n/k}
  \end{align*}
  for all $n\geq n_0(\rho, k)$. Since 
  our choice of $k$ depended only on $\rho$ and $\lambda$, we have proven the theorem.
\end{proof}

\section{Establishing the critical values}
\label{sec:critical.values}

We now translate our results on layer percolation back to odometers to
prove the paper's main results. We prove lower bounds on critical values
by constructing stable odometers leaving density $\crist-\epsilon$, and we prove upper bounds
by showing nonexistence of stable odometers
leaving a density of $\crist+\epsilon$. 
For the constructions, the proof is slightly different for each model.
In Section~\ref{sec:construction.tools}, we prove a few simple lemmas that we will use
together with \thref{lem:box} to construct odometers.
Next in Section~\ref{sec:peanut.butter} we prove a very general nonexistence proof, 
from which the upper bounds for each
individual model follow easily. Then in Section~\ref{sec:final.density.proofs}
we apply these results to the different models of ARW.
Once we are done proving equality
of all the critical values,  we prove upper and lower bounds on
$\crist$ in Section~\ref{sec:critical.bounds} that hence extend to $\critDD,$ $\critPS$, $\critFE$, and $\critCY$.

\subsection{Odometer construction tools}
\label{sec:construction.tools}

Any extended odometer on $\ii{0,n}$ obtained from a length~$n$ infection path in
layer percolation is automatically stable on $\ii{1,n-1}$.
In the next lemma, we give a criterion for stability at $n$.

\begin{lemma}\thlabel{lem:stab.criterion}
  Consider an extended  odometer $u\in\eosos{n}(\Instr,\sigma,u_0,f_0)$ on $\ii{0,n}$
  corresponding to an infection path $(0,0)_0=(r_0,s_0)_0\to\cdots (r_n,s_n)_n$
  in $\ip{n}(\Instr,\sigma,u_0,f_0)$ under the map $\Phi$ from
  Section~\ref{sec:ARW.percolation.connection}. Let $\mo$ be the minimal odometer
  of $\eosos{n}(\Instr,\sigma,u_0,f_0)$.
  Then $u$ is stable at $n$ if and only if
  \begin{align}\label{eq:stab.criterion}
    r_n = f_0 +\sum_{v=1}^n\abs{\sigma(v)} - \rt{\mo}{n} - s_n.
  \end{align}
\end{lemma}
\begin{proof}
  Let $f_v=\rt{u}{v}-\lt{u}{v+1}$. Since $u$ is stable on $\ii{1,n-1}$, we have
  \begin{align}\label{eq:fsc}
    f_v=f_0+\sum_{i=1}^v\abs{\sigma(i)}-s_v
  \end{align}
  for all $v\in\ii{0,n-1}$
  by \thref{lem:flow} together with the definition of $\Phi$.
  By \thref{lem:flow} again, it is stable at $n$ as well if and only if \eqref{eq:fsc} also holds for
  $v=n$. Here $f_n=\rt{u}{n}$ since $u(n+1)=0$, and so \eqref{eq:fsc} for $v=n$ states that
  \begin{align*}
    \rt{u}{n}=f_0+\sum_{i=1}^n\abs{\sigma(i)}-s_n.
  \end{align*}
  And since $\rt{u}{n}=\rt{\mo}{n}+r_n$ under the correspondence given by $\Phi$,
  this statement holds if and only if \eqref{eq:stab.criterion} does.  
\end{proof}

Reordering and labeling \eqref{eq:stab.criterion} makes its similarity to
\thref{def:stable} more apparent:
\begin{align*}
    \underbrace{\sum_{v=1}^n\abs{\sigma(v)}}_{A} + 
         \underbrace{f_0}_{B} 
         - \underbrace{\bigl(r_n +\rt{\mo}{n}\bigr)}_{C}
      = \underbrace{s_n}_{D}
\end{align*}
Term~$A$ is the number of particles starting on $\ii{1,n}$. Terms~$B$ and $C$ record
the flow into $\ii{1,n}$ from $0$ and the flow out $\ii{1,n}$ from $n$, respectively.
And term~$D$ is the number of particles ending on $\ii{1,n}$.

Layer percolation produces extended odometers that may include meaningless negative values.
We will use the following lemmas to show that the ones we generate do not take negative values
and hence can be used to invoke the least-action principle.
  
  \begin{lemma}\thlabel{lem:nonnegative}
    Suppose that $u$ is an extended odometer on $\ii{0,n}$ stable
    on $\ii{1,n-1}$.
    Let $f_v=\rt{u}{v}-\lt{u}{v+1}$, the net flow from $v$ to $v+1$
    under the odometer, as in \thref{lem:flow}.
    \begin{enumerate}[(a)]
      \item For any $v\in\ii{0,n-1}$, if $u(v)\leq 0$ and $f_v\geq 1$ then $u(v+1)<0$.
        \label{i:nonnegative.a}
      \item For any $v\in\ii{1,n}$, if $u(v)<0$ and $f_{v-1}\leq 0$, then $u(v-1)<0$.
        \label{i:nonnegative.b}
    \end{enumerate}
    
  \end{lemma}
  \begin{proof}
    Suppose $u(v)\leq 0$ and $f_v\geq 1$ for some $v\in\ii{0,n-1}$.
    Then $\rt{u}{v}\leq 0$ and $\rt{u}{v}-\lt{u}{v+1}\geq 1$.
    Thus $\lt{u}{v+1}\leq\rt{u}{v}-1\leq -1$, implying $u(v+1)<0$ and proving \ref{i:nonnegative.a}.
    
    Similarly, suppose $u(v)<0$ and $f_{v-1}\leq 0$ for some $v\in\ii{1,n}$.
    Then $\lt{u}{v}\leq -1$ (recall that $\lt{u}{v}$ is strictly negative
    whenever $u(v)$ is), and $\rt{u}{v-1}-\lt{u}{v}\leq 0$, yielding
    $\rt{u}{v-1}\leq -1$. Thus $u(v-1)<0$, proving \ref{i:nonnegative.b}.
  \end{proof}

Now we show that if an extended odometer is nonnegative at its endpoints
and can be broken up into an interval of negative flow followed by an interval of positive
flow, then it is nonnegative everywhere.
\begin{lemma}\thlabel{lem:uniflow}
  Let $u$ be an extended odometer on $\ii{0,n}$ stable on $\ii{1,n-1}$ with $u(0)\geq 0$
  and $u(n)\geq 0$.
  Let $f_v=\rt{u}{v}-\lt{u}{v+1}$.
  Suppose there exists $k$ such that $f_v\leq 0$ for $0\leq v<k$
  and $f_v\geq 1$ for $k\leq v\leq n$. Then $u(v)\geq 0$ for all $v\in\ii{0,n}$.  
\end{lemma}
\begin{proof}
  Suppose $u(v)<0$ for some $v\in\ii{1,n-1}$. If $v\geq k$, then 
  repeated application of \thref{lem:nonnegative}\ref{i:nonnegative.a} shows that
  $u(i)<0$ for all $v\leq i\leq n$, a contradiction since $u(n)\geq 0$.
  If $v<k$, then repeated application of \thref{lem:nonnegative}\ref{i:nonnegative.b}
  shows that $u(i)<0$ for all $0\leq i\leq v$, a contradiction since $u(0)\geq 0$.
\end{proof}

\subsection{Nonexistence of stable odometers above the critical density}
\label{sec:peanut.butter}
In this section we obtain an exponential bound on the probability of activated random
walk leaving a high density of particles on an interval, starting from \emph{any} initial configuration:
\begin{thm}\thlabel{peanut butter}
  Consider activated random walk with sleep rate $\lambda>0$.
  Let $\sigma$ be an initial configuration with no sleeping particles
  on $\ii{0,n}$.
  Let $Y_n$ be the number of particles left sleeping on $\ii{0,n}$ in the stabilization of 
  $\sigma$ on
  $\ii{0,n}$. For any $\rho>\crist(\lambda)$,
  \begin{align*}
    \P[Y_n\geq \rho n] \leq Ce^{-cn}
  \end{align*}
  where $C,c$ are positive constants depending on $\lambda$ and $\rho$ but not on $n$ or $\sigma$.
\end{thm}

\begin{proof}
  \thref{thm:large.deviations} establishes that it is unlikely there is an infection set
  starting at $(0,0)_0$ going beyond row
  $\crist n$ in $n$ steps. Such infection paths correspond to odometers in
  $\eosos{n}(\Instr,\sigma,u_0,f_0)$ leaving more than $\crist n$ sleeping particles
  on $\ii{1,n}$, according to \thref{bethlehem}. The idea of this proof is to
  take a union bound on the existence of such an odometer over all choices of $u_0$ and $f_0$,
  applying \thref{thm:large.deviations} when $u_0$ is small and using an alternate bound
  when $u_0$ is large.

  To start, we may assume
  without loss of generality that $\sigma$ places $0$ or $1$ particle at each site,
  since we can topple each site with two or more particles until no such sites exist.
  
    For any integer $f_0$, nonnegative integer $u_0$, and configuration $\sigma$ in which
  all particles are active,
  let $\stable_{n,\rho}(\sigma,u_0,f_0)$ be the event that there exists an
  odometer $u$ on $\ii{0,n}$ stable on $\ii{1,n-1}$ for initial configuration $\sigma$ 
  with $u(0)=u_0$ and 
  $\rt{u}{0}-\lt{u}{1}=f_0$ and
  \begin{align*}
    \sum_{v=1}^{n}\1\bigl\{\Instr_v(u(v))=\Sleep\bigr\} \geq \rho n,
  \end{align*}
  with $(\Instr_v)_{v\in\ii{0,n}}$ 
  the instructions for activated random walk with sleep rate $\lambda>0$.
  
  We claim that
  \begin{align}\label{eq:stabclaim}
    \P\bigl[\stable_{n,\rho}(\sigma,u_0,f_0)\bigr]\leq Ce^{-cn}
  \end{align}
  for constants $c$ and $C$ depending only on $\lambda$ and $\rho$, not on $\sigma$, $u_0$,
  and $f_0$. Indeed, if $\stable_{n,\rho}(\sigma,u_0,f_0)$ holds, then by \thref{bethlehem,cor:ip} the infection
  set $\infset_n=\infset_n^{(0,0)_0}$ in a coupled instance of layer percolation
  contains a cell in row $\ceil{\rho n}$, 
  thus proving \eqref{eq:stabclaim} by \thref{thm:large.deviations}.

  Suppose that $Y_n\geq\rho n$, and let $u_*$ be the true odometer stabilizing $\sigma$
  on $\ii{0,n}$. Then $u_*$ is stable on $\ii{0,n}$ and leaves at least
  $\rho n-1$ particles on $\ii{1,n}$.
  Hence $\stable_{n,\rho'}(\sigma,u_0,f_0)$ occurs for $\rho'=\rho-\frac{1}{n}$ and some $u_0$ and $f_0$.
  Note that for $n\geq n_0$, where $n_0$ is a constant depending only on $\lambda$ and $\rho$, we have
  $\rho'>\crist(\lambda)$.
  Since there is at most
  one particle on each site initially, we have $-(n-1)\leq f_0\leq 1$. All together, we can say
  that if $Y_n\geq \rho n$, then either the event
  \begin{align}\label{eq:stable.unb}
    \bigcup_{\substack{0\leq u_0\leq 4(1+\lambda)n\\-(n-1)\leq f_0\leq 1}}\stable_{n,\rho'}(\sigma,u_0,f_0)
  \end{align}
  occurs, or $u_*(0)>4(1+\lambda)n$. By \eqref{eq:stabclaim} and
  a union bound, the probability of \eqref{eq:stable.unb}
  is at most $C'e^{-c'n}$ for constants $c,C'$ depending only on $\lambda$ and $\rho$.
  
  Now consider the event that $u_*(0)>4(1+\lambda)n$. 
  Since there are at most $n+1$ particles initially on the interval,
  we always have $\lt{u_*}{0}\leq n+1$. 
  Thus $u_*(0)>4(1+\lambda)n$ can occur only if there are at most
  $n+1$ \Left\ instructions within $\Instr_0(1),\ldots,\Instr_0(\ceil{4(\lambda+1)n})$.
  Thus we have bounded $\P\bigl[u_*(0)>4(1+\lambda)n\bigr]$
  by 
  \begin{align*}
    \P\biggl[\Bin\biggl(\ceil[\big]{4(\lambda+1)n},\frac{1}{2(\lambda+1)}\biggr)\leq n+1\biggr],
  \end{align*}
  which by Hoeffding's inequality is at most $e^{-c''n}$
  with $c''$ depending only on $\lambda$.
  Together with our bound on the probability of \eqref{eq:stable.unb},
  this completes the proof.
\end{proof}

\subsection{All critical densities are equal to \texorpdfstring{$\crist$}{ρ\_*}}
\label{sec:final.density.proofs}

We now apply our results to give upper and lower bounds on the critical
densities of each of the four models discussed in Section~\ref{sec:results}.
For the lower bounds for the driven-dissipative and point-source models, we use layer percolation to produce
stable odometers leaving a density of $\crist-\epsilon$.
We then derive the lower bounds for the fixed-energy and cyclic models from the point-source model bound.
For the upper bounds, we apply \thref{peanut butter} to each model.

\subsubsection{Driven-dissipative model}
\label{sec:driven.dissipative}
We start with the lower bound on the density under the invariant distribution of the driven-dissipative
model. We note that we can work directly with this invariant distribution: as proven
by Levine and Liang in \cite{levine2021exact}, 
the stabilization of a region starting with an initial configuration
of a single active particle at every site exactly has this invariant distribution.
\begin{prop}\thlabel{prop:driven.dissipative.lower}
  Let $S_n$ be distributed as the number of sleeping particles under the invariant
  distribution of the driven-dissipative Markov chain on $\ii{0,n}$.
  For any $\rho<\crist(\lambda)$,
  \begin{align*}
    \P\bigl[ S_n\geq\rho (n+1)\bigr] \geq 1-Ce^{-cn}
  \end{align*}
  for constants $c,C>0$ depending only on $\lambda$ and $\rho$.
\end{prop}
\begin{proof}
  Let $\sigma$ be the initial configuration placing one active particle on each
  site in $\ii{0,n}$. By \cite[Theorem~1]{levine2021exact}, after stabilization this model
  has the invariant distribution of the driven-dissipative model on $\ii{0,n}$.
  
  Let $f_0=-\floor{(1-\rho)(n+1)/2}+1$. Choose $u_0$ to be the smallest value so
  that the first $u_0$ instructions at site~$0$ include $-f_0+1$ \Left\ instructions,
  which according to \thref{prop:stable.at.0} will make all odometers in $\eosos{n}(\Instr,\sigma,u_0,f_0)$
  stable at $0$. We will now construct an odometer $u\in\eosos{n}(\Instr,\sigma,u_0,f_0)$
  stable at $n$ with $\rt{u}{n}=-f_0+1$. By \thref{lem:lap}, the true odometer stabilizing
  $\ii{0,n}$ thus ejects at most density $1-\rho$ of particles and hence retains at least density $\rho$.
      
  We say that an event holds with overwhelming probability (w.o.p.)\ if its probability is bounded for all $n$ by $Ce^{-cn}$,  where $c$ and $C$ may depend on $\lambda$ and $\rho$.
  Let $\mo$ be the minimal odometer of $\eosos{n}(\Instr,\sigma,u_0,f_0)$.
  We claim that $\rt{\mo}{n}$ is concentrated around $-\rho n^2/2$ at a quadratic scale, i.e.,
  for any constant $D=D(\lambda,\rho)>0$,
  \begin{align}\label{eq:rtmnc}
    \abs[\big]{\rt{\mo}{n}+\rho n^2/2}\leq Dn^2 \text{ w.o.p.}
  \end{align}
  To prove this claim, we apply
  \thref{prop:min.odometer.concentration} conditionally on $u_0$ to show that
  $\rt{\mo}{n}$ is concentrated around
  \begin{align}
    \frac{u_0}{2(1+\lambda)} + \sum_{i=1}^n\biggl(-f_0-i\biggr)
      &=\frac{u_0}{2(1+\lambda)} + \Biggl(\floor[\bigg]{{\frac{(1-\rho)(n+1)}{2}}}+1\Biggr)n - \frac{n(n+1)}{2}.
      \label{eq:odom.a.n}
  \end{align}
  Since $u_0$ is a sum of $\abs{f_0}+1$ independent geometric random variables,
  we have $u_0/2(1+\lambda)<Dn^2/3$ w.o.p.\ by \cite[Theorem~2]{Janson}.
  The rest of the right-hand side of \eqref{eq:odom.a.n} is within $O(n)$
  of $-\rho n^2/2$.
  By \thref{prop:min.odometer.concentration}, with overwhelming probability $\rt{\mo}{n}$
  is within $Dn^2/3$ of \eqref{eq:odom.a.n}. Hence $\rt{\mo}{n}$ is within
  $2Dn^2/3+O(n)$ of $-\rho n^2/2$ w.o.p., proving \eqref{eq:rtmnc}.
  
  Now we construct our odometer $u$ stable at $n$ with $\rt{u}{n}=-f_0+1$.
  Let
  \begin{align*}
    r=-f_0+1-\rt{\mo}{n}\qquad\text{and}\qquad s=2f_0+n-1.
  \end{align*}
  According to \thref{bethlehem,cor:ip}, it suffices to construct an infection
  path in $\ip{n}(\Instr,\sigma,u_0,f_0)$ ending at cell $(r,s)_n$.
  For an odometer $u\in\eosos{n}(\Instr,\sigma,u_0,f_0)$ in correspondence
  with this infection path, the condition on $r$ makes $\rt{u}{n}=r+\rt{\mo}{n}=-f_0+1$,
  and \thref{lem:stab.criterion} makes $u$ stable at $n$.

  Let $\delta=\delta(\epsilon)$ be the constant from \thref{lem:box},
  with $\epsilon=\min\bigl(\crist-\rho,\rho)/2$.
  Setting $D=\delta/2$ and applying \eqref{eq:rtmnc}, it holds with overwhelming probability
  that $\rt{\mo}{n}$ is within $(\delta/2)n^2$ of $-\rho n^2/2$.
  Since $-f_0+1 = O(n)$, we have
  $r$ within $\delta n^2$ of $\rho n^2/2$ w.o.p.
  Since $s\leq \rho'n$ for $\epsilon\leq \rho<\rho'\leq\crist-\epsilon$,
  the infection set $\infset[(0,0)_0]_n$ contains
  $(r,s)_n$ w.o.p.\ by \thref{lem:box}, yielding the desired infection path
  ending at $(r,s)_n$.
  
  To summarize, we have produced an infection path such that any corresponding
  extended odometer $u\in\eosos{n}(\Instr,\sigma,u_0,f_0)$ is stable on $\ii{0,n}$
  and leaves at least $s\geq \rho(n+1)$ particles on $\ii{1,n}$.
  The proof follows from \thref{lem:dual.lap} once we show that
  $u(v)\geq 0$ for all $v$.
  Since $\lt{u}{0}$ and $\rt{n}{0}$ are positive, we have $u(0)>0$ and $u(n)>0$.
  Let $f_v=\rt{u}{v}-\lt{u}{v+1}$, consistent with the definition of $f_0$.
  By \thref{lem:flow} and our choice of $\sigma$, we have $f_0\leq\cdots\leq f_{n-1}$.
  Hence $u(v)\geq 0$ for all $v$ by \thref{lem:uniflow}.
  Thus $u$ is an odometer and not just an extended odometer, and the proof is complete.
\end{proof}

The upper bound on the density is just a special case of \thref{peanut butter}:
\begin{prop}\thlabel{prop:driven.dissipative.upper}
  Let $S_n$ be distributed as the number of particles under the invariant
  distribution of the driven-dissipative Markov chain on an interval of length~$n$.
  For any $\rho>\crist(\lambda)$,
  \begin{align*}
    \P\bigl[S_n\leq \rho n\bigr] \geq 1 - Ce^{-cn}
  \end{align*}
  for constants $C,c$ depending only on $\lambda$ and $\rho$.
\end{prop}
\begin{proof}
  Let $\sigma$ be the initial configuration placing one active particle on each
  site of the interval. By \cite[Theorem~1]{levine2021exact}, after stabilization this model
  has the invariant distribution of the driven-dissipative model on the interval.
  The result then follows immediately from \thref{peanut butter}.
\end{proof}

\subsubsection{Point-source model}
The lower bound on density for the point-source model is similar to that of the driven-dissipative
model, though it is a bit more work to construct the odometer.
Note that the lower bound on density for the point-source model is expressed as an upper bound
on the spread of a fixed number of particles starting at the origin.
\begin{prop}\thlabel{prop:point.source.lower}
  Let $\ii{A_N,B_N}$ be the smallest interval containing all sites ever visited by a particle
  in the point-source model on $\ZZ$ with $N$ particles.
  For any $\epsilon>0$,
  \begin{align*}
    \P\Biggl[ \ii{A_N,B_N}\subseteq \ii[\bigg]{- \frac{N}{2(\crist-\epsilon)},\ \frac{N}{2(\crist-\epsilon)}}\Biggr] \geq 1 - Ce^{-cN}
  \end{align*}
  for constants $c,C>0$ depending only on $\lambda$ and $\epsilon$.
\end{prop}
\begin{proof}
  Let our initial configuration
  $\sigma$ consist of $N$ active particles at site~$0$ and no
  particles elsewhere. We will choose $n\approx N/2(\crist-\epsilon)$
  and construct a stable odometer on $\ii{-(n+1),n+1}$.
  
  We say that an event holds with overwhelming probability (w.o.p.)\ if its probability is bounded for all $N$ by $Ce^{-cN}$,  where $c$ and $C$ may depend on $\epsilon$ and $\lambda$.
  In our construction of an odometer, we will want the first instruction
  on the stack at sites~$n+1$ and $-(n+1)$ to be $\Sleep$. Thus we define
  $n$ as the largest integer such that
  \begin{align*}
    \Instr_{n+1}(1)=\Instr_{-(n+1)}(1)=\Sleep\quad\text{and}\quad n\leq \frac{N}{2(\crist-\epsilon/2)}.
  \end{align*}
  Since $\Instr_v(1)=\Sleep$ with probability $\lambda/(1+\lambda)$, independently for each $v$,
  \begin{align}\label{eq:not.too.far}
    n\geq \frac{(1-\epsilon/4)N}{2(\crist-\epsilon/2)}\geq \frac{N}{2(\crist-\epsilon/4)} \text{ w.o.p.}
  \end{align}
  Note that we have revealed information only about the instructions at site~$v$ for
  $\abs{v}>n$;
  we have not imparted any conditioning on $\Instr_{-n},\ldots,\Instr_n$.
  
  Now we use layer percolation to construct our odometer $u$. 
  We will make it so that $\lt{u}{-n}=\rt{u}{n}=1$, and then we set $u(-n-1)=u(n+1)=1$
  so that single particles are sent to $-n-1$ and to $n+1$ and then sleep there. We construct
  $u$ so that the remaining $N-2$ particles sleep on $\ii{-n,n}$.
  Since our notation is set up to construct odometers on an interval from $0$ to a positive
  number, we let $\Instr'_v=\Instr_{v-n}$ and $\sigma'(v)=\sigma(v-n)$ and construct an odometer
  in $\eosos{2n}(\Instr',\sigma',u_0,f_0)$. From now on, $\lt{\cdot}{\cdot}$
  and $\rt{\cdot}{\cdot}$ will use $\Instr'$ rather than $\Instr$.
  
  Let $f_0=-1$ and choose $u_0$ to be the index of the first \Left\ instruction after index~$0$
  in $\Instr'_0$. Let $\rho=(N-2)/2n$, observing that
  $\rho\leq \crist-\epsilon/4$ w.o.p.\ by \eqref{eq:not.too.far}  
  and let $\delta=\delta(\epsilon/4)$ be the constant
  from \thref{lem:box}.
  Let $\mo$ be the minimal odometer of $\eosos{2n}(\Instr',\sigma',u_0,f_0)$.
  According to \thref{prop:min.odometer.concentration}, the value of $\rt{\mo}{2n}$ is concentrated around
  \begin{align*}
    \frac{u_0}{2(1+\lambda)}+\sum_{i=1}^{2n}\bigl(-f_0 - N\1\{i\geq n\}\bigr) = 
    \frac{u_0}{2(1+\lambda)}-2\rho n^2
  \end{align*}
  Since $u_0\sim\Geo\bigl(1/2(1+\lambda)\bigr)$, we have $u_0\leq n$ w.o.p.
  By \thref{prop:min.odometer.concentration}, the value of $\rt{\mo}{2n}$
  is within $(\delta/2)(2n)^2$ of $-2\rho n^2$ w.o.p.
  
  By \thref{cor:ip}, the set $\ip{2n}(\Instr',\sigma',u_0,f_0)$ is distributed as the set
  of length~$0$ infection paths in layer percolation.  
  Let $s=N-2$ and let
  \begin{align}\label{eq:stablrrr}
     r = f_0+\sum_{v=1}^{2n}\abs{\sigma'(v)} - \rt{\mo}{2n} - s = 1-\rt{\mo}{2n}.
  \end{align}
  Now $r$ is within $\delta (2n)^2$ of $2\rho n^2$, and by \thref{lem:box} the infection set
  $\infset[(0,0)_0]_{2n}$ contains $(r,s)_{2n}$ w.o.p.
  Thus there is an infection path in $\ip{2n}(\Instr',\sigma',u_0,f_0)$ ending at $(r,s)_{2n}$ w.o.p.
  By \thref{bethlehem}, this infection path has a corresponding extended 
  odometer $u'\in\eosos{2n}(\Instr',\sigma',u_0,f_0)$
  stable on $\ii{1,2n-1}$ with $\rt{u}{2n}=r+\rt{\mo}{2n}=1$.
  By \thref{lem:stab.criterion} and \eqref{eq:stablrrr}, the odometer $u'$ is stable at $2n$.
  And it is stable at $0$ by \thref{prop:stable.at.0}.
  
  To see that $u'$ takes nonnegative values, observe that since $\lt{u'}{0}>0$
  and $\rt{u'}{2n}>0$, we have $u'(0)>0$ and $u'(2n)>0$.
  Let $f_v=\rt{u'}{v}-\lt{u'}{v+1}$.
  By \thref{lem:flow}, we have $f_v\leq -1$ for $v\leq -1$
  and $f_v\geq 1$ for $v\geq 0$. Hence $u'(v)\geq 0$ for all $v\in\ii{0,2n}$
  by \thref{lem:uniflow}.
  
  By \thref{lem:lap}, with overwhelming probability the stabilization of
  $\ii{0,2n}$ using instructions $\Instr'$ sends at most one particle to site~$-1$ and 
  at most one particle to site~$2n+1$. Shifting back to $\ii{-n,n}$,
  we have shown that with overwhelming probability the stabilization of $\ii{-n,n}$
  sends at most one particle to site~$-n-1$ and at most one particle to site~$n+1$.
  Since $\Instr_{-n-1}(1)=\Instr_{n+1}(1)=\Sleep$, no particle moves to the left
  of $-n-1$ or to the right of $n+1$, showing that $A_n\geq -n-1$ 
  and $B_N\leq n+1$ w.o.p., thus completing the proof since
  $n\leq \frac{N}{2(\crist-\epsilon/2)}$.
\end{proof}

And now we give the upper bound on density:
\begin{prop}\thlabel{prop:point.source.upper}
  Let $\ii{\tilde{A}_N,\tilde{B}_N}$ be the smallest interval containing all sleeping particles after stabilization
  in the point-source model on $\ZZ$ with $N$ particles.
  For any $\epsilon>0$,
  \begin{align*}
    \P\Biggl[ \ii{\tilde{A}_N,\tilde{B}_N}\supseteq \ii[\bigg]{- \frac{N}{2(\crist+\epsilon)},\ \frac{N}{2(\crist+\epsilon)}}\Biggr] \geq 1 - Ce^{-cN}
  \end{align*}
  for constants $c,C>0$ depending only on $\lambda$ and $\epsilon$.
\end{prop}
\begin{proof}
  Let $\sigma$ denote the configuration placing $N$ particles at $0$ and none elsewhere.
  Let $u_*$ be the true odometer on $\ZZ$ obtained by stabilizing $\sigma$.
  If $\abs{\tilde{A}_N}\leq N/2(\crist+\epsilon)$, then one of the following events occurs:
  \begin{enumerate}[(i)]
    \item $\tilde{B}_N> N/2(\crist-\epsilon/2)$;\label{i:psu1}
    \item $\abs{\tilde{A}_N}\leq N/2(\crist+\epsilon)$ and
        $\tilde{B}_N\leq N/2(\crist-\epsilon/2)$.\label{i:psu3}
  \end{enumerate}
  The probability of \ref{i:psu1} vanishes exponentially by \thref{prop:point.source.lower}.
  If \ref{i:psu3} occurs, then the density after
  stabilization is $N/(\tilde{A}_N+\tilde{B}_N+1)\geq\crist+\delta$ for some constant $\delta>0$ depending on
  $\lambda$ and $\epsilon$. 
  By \thref{peanut butter}, the probability of this event decays exponentially in $N$ with the rate
  depending only on $\lambda$ and $\epsilon$.
  This confirms that the probability that $\abs{\tilde{A}_N}\leq N/2(\crist+\epsilon)$ decays exponentially.
  By symmetry, the same holds for the probability that $\tilde{B}_N\leq N/2(\crist+\epsilon)$, completing the proof.
\end{proof}

\begin{cor}[Conjecture 1 from \cite{levine2023universality} for $d= 1$] \thlabel{cor:ps}
Let $\ii{ A_N,  B_N}$ be the smallest interval containing all sites visited until stabilization
  in the point-source model on $\ZZ$ with $N$ particles.
  For any $\epsilon>0$,
  \begin{align*}
    \P\Biggl[ \ii[\bigg]{- \frac{N}{2(\crist+\epsilon)},\ \frac{N}{2(\crist+\epsilon)}}
    \subseteq\ii{ A_N, B_N}\subseteq \ii[\bigg]{- \frac{N}{2(\crist-\epsilon)},\ \frac{N}{2(\crist-\epsilon)}}\Biggr] \geq 1 - Ce^{-cN}
  \end{align*}
  for constants $c,C>0$ depending only on $\lambda$ and $\epsilon$.
\end{cor}
\begin{proof}
The smallest interval containing all sleeping particles after fixation is contained in the smallest
interval containing all sites visited.
    Thus the conjecture follows directly from \thref{prop:point.source.lower} and \thref{prop:point.source.upper}.
\end{proof}
\subsubsection{Fixed-energy model on \texorpdfstring{$\mathbb Z$}{Z}}
It suffices to consider activated random walk on $\mathbb Z$ with an i.i.d.\ Bernoulli$(\rho)$-distributed number of particles initially at each site. By the main result of \cite{rolla2019universality}, if this Bernoulli configuration fixates or remains active almost surely for a given $\rho$, then the same behavior occurs for any stationary configuration with mean $\rho$. 


\begin{prop} \thlabel{prop:fixed.energy.active}
If $\rho > \crist$, then the fixed-energy model almost surely remains active.
\end{prop}

\begin{proof}
  Consider the stabilization of the i.i.d.\ $\Ber(\rho)$ initial configuration on $\ii{-n,n}$,
  and let $M_n$ be the number of particles starting in $\ii{-n,n}$ that terminate
  at the sinks $-n-1$ and $n+1$.
\cite[Theorem 2.11]{rolla2020activated} states that if
\begin{align}
    \limsup_{n\to\infty} \frac{\E M_n }{ 2n+1} >0, \label{eq:suff.active}
\end{align}
then the fixed-energy system on $\mathbb Z$ stays active a.s. 

Let $|\sigma_n|$ be the total number of particles starting in $\ii{-n,n}$ and $S_n$ be the number of particles sleeping in $\ii{-n,n}$ once the system stabilizes. Since mass is conserved, $M_n = |\sigma_n| - S_n$. Fix $0 < \delta < \rho - \crist$. We will show that $|\sigma_n| - S_n$ is larger than $\delta n$ with exponentially high probability, establishing \eqref{eq:suff.active} and, thus, the desired result. 

Observe that $|\sigma_n|\sim\Bin(2n+1,\rho)$.
Let $\ManyParticles_n$ be the event that $|\sigma_n| \geq  (\crist + \delta)(2n+1)$. 
A standard concentration estimate ensures that $$\P(\ManyParticles_n) \geq 1- A e^{-a n}$$ for positive constants $A$ and $a$ that depend only on $\delta$. 
Define $\FewSleepers_n$ to be the event that $S_n \leq \crist + \delta/2$.
It follows from \thref{peanut butter} that 
$$\P(\FewSleepers_n \mid \ManyParticles_n) \geq 1 - B e^{-b n}$$
for positive constants $B$ and $b$ that depend only on $\delta$ and $\lambda$.
Since $M_n = |\sigma_n| - S_n$, on the event $\ManyParticles_n \cap \FewSleepers_n$ we have $M_n \geq \delta n$. Thus, $\P(M_n \geq \delta n) \geq 1-Ce^{-cn}$ for positive constants $C$ and $c$ that depend only on $\delta$.
This implies that
\begin{align*}
    \limsup_{n\to\infty} \frac{\E M_n }{ 2n+1} &\geq\delta/2.\qedhere
\end{align*}
\end{proof}


\begin{prop} \thlabel{prop:fixed.energy.fixation}
    If $\rho < \crist$, then the fixed-energy model almost surely fixates. 
\end{prop}

\begin{proof}
Fix an odd positive integer $k$. Partition $\ZZ$ into successive intervals $\ldots,I_{-1},I_0,I_1,\ldots$
where $\abs{I_n}=k(2\abs{n}+1)$, and let $x_n$ be the integer at the center of $I_n$. 
Declare the points $x_n$ to be sources.
The idea will be to first let all particles move as random walks with no interaction until they reach
a source; this only increases the odometer since we are awakening particles
whenever they sleep. We then allow ARW to run as usual.
By applying \thref{prop:point.source.lower}, we will show that for sufficiently large $k$,
it holds with positive probability that the particles at each source fall asleep
without interacting with particles from any other source. Thus the fixed-energy model
fixates with positive probability, and by the zero-one law \cite[Theorem~2.7]{rolla2020activated}
it fixates a.s.

To carry out this plan, fix $n\in\ZZ$
and arbitrarily topple sites in $\ii{x_n+1,x_{n+1}-1}$
containing particles,
regardless of whether the particles are sleeping, until all particles
settle on $x_n$ or $x_{n+1}$.
Do the same for $\ii{x_{n-1}+1,x_n-1}$ so that all particles starting
in $\ii{x_{n-1},x_{n+1}}$ are placed onto sites $x_{n-1}$, $x_n$,
and $x_{n+1}$, and let $Z_n$ be the number of particles settling on $x_n$.
Then stabilize interval $I_n$, and define $\contained_n$ as the event
that all the particles on $x_n$ stabilize without leaving $I_n$.

We claim that that for large enough $k$, it holds with positive probability
that $\contained_n$ occurs for all $n\in\ZZ$.
To prove this, we give an exponentially decaying failure bound on $\contained_n$.
First consider $Z_n$, the number of particles initially settling on $x_n$.
By symmetry, the sites in $\ii{x_{n-1}+1,x_n-1}$ contribute equally
to $Z_{n-1}$ and $Z_n$ in expectation; likewise, the sites in $\ii{x_{n}+1,x_{n+1}-1}$
contribute equally to $Z_n$ and $Z_{n+1}$  in expectation. 
Adding in the contribution to $Z_n$ from site~$x_n$ itself, we have for any $n\geq 1$
\begin{align*}
  \E Z_n = \tfrac12\rho\abs[\big]{\ii{x_{n-1}+1,x_n-1}} + 
          \tfrac12\rho\abs[\big]{\ii{x_{n}+1,x_{n+1}-1}} + \rho
          =\tfrac12\rho(x_{n+1}-x_{n-1})=\rho\abs{I_n}.
\end{align*}
Fix any $\rho'$ such that $\rho<\rho'<\crist$.
Since $Z_n$ can be realized as a sum of independent Bernoulli random variables,
a Chernoff bound gives
\begin{align}
\P\bigl[ Z_n\leq \rho' \abs{I_n}\bigr] \geq 1- e^{-ank} \label{eq:small}
\end{align}
for constant $a$ not depending on $n$ or $k$.
By \thref{prop:point.source.lower}, conditional on $Z_n\leq \rho' \abs{I_n}$, the 
particles on
$x_n$ will stabilize within 
\begin{align*}
  \ii[\bigg]{x_n-\frac{Z_n}{2\rho'},\ x_n+\frac{Z_n}{2\rho'}}
    \subseteq \ii[\bigg]{x_n-\frac{\abs{I_n}}{2},\ x_n+\frac{\abs{I_n}}{2}}=I_n
\end{align*}
with probability at least $1- Be^{-b nk}$ for positive constants $B$ and $b$ that do not depend 
on $n$ or $k$. Together with \eqref{eq:small}, this shows that
for some constants $c$ and $C$ independent of $n$ and $k$,
we have $\P[\contained_n] \geq 1 - Ce^{-cnk}$ for all $n\geq 1$.
By symmetry $\P[\contained_{-n}]$ obeys the same bound.
Thus by taking $k$ sufficiently large, we can make the probability
of $\cap_{n\neq 0}\contained_n$ arbitrarily close to $1$. Since
$\P[\contained_0]>0$, this proves that for some $k$,
we have $\P[\cap_{n\in\ZZ}\contained_n]>0$.

To prove that the system fixates, we argue
along similar lines as \cite[Section~4]{rolla2020activated} that on the event
$\cap_{n\in\ZZ}\contained_n$,
 for each $n$ we have a finite sequence
of acceptable topplings (i.e., sites with sleeping or active particles may be toppled) stabilizing
$\ii{x_{-n},x_n}$. The odometer for this sequence of topplings provides an upper bound
on the true odometer stabilizing $\ii{x_{-n},x_n}$ by \cite[Lemma~2.1]{rolla2020activated}
(or by our own least-action principle, once it is noted that the odometer
of a sequence of acceptable topplings stabilizing a set is always stable on that set).
And since our upper bound on the odometer at $0$ remains bounded as $n\to\infty$, we have fixation a.s.\ 
by \cite[Theorem~2.7]{rolla2020activated}.
\end{proof}

\subsubsection{The cycle}
\label{sec:CY}
\begin{prop} \thlabel{prop:cycle}
    Fix $\rho \in (0,1)$ and place $\lfloor \rho n \rfloor$ active particles uniformly throughout the sites of the cycle with $n$ vertices. Let $\tau_n$ be the total number of instructions used by all  particles once the system has stabilized. 
    \begin{enumerate}[label = (\roman*)]
        \item If $\rho < \crist$, then $\P(\tau_n > C n \log^2n ) \leq n^{-b}$ and $\P(\tau_n > C'n^4) \leq Be^{-b'n}$ for positive constants $C,C',B,b,b'$ that do not depend on $n$.
        \item If $\rho > \crist$, then $\P(\tau_n < e^{cn}) \leq e^{-cn}$ for a positive constant $c$ that does not depend on $n$.
    \end{enumerate}
\end{prop}

\begin{proof}
   The proofs of (i) and (ii) are  similar to the arguments in \cite{BasuGangulyHoffmanRichey19}. We give a brief sketch and refer the reader to \cite{BasuGangulyHoffmanRichey19} for more details. 
   
   To prove (i), a source scheme like in the proof of \thref{prop:fixed.energy.active} is used. Sources are spaced uniformly at distance $c_0 \log n$ apart. Particles perform random walk until reaching a source, upon which they switch to activated random walk dynamics. It takes $O(n\log^2 n)$ steps to freeze all of the particles. With \thref{prop:point.source.lower}, we can deduce that the particles coming from each source are likely to stabilize in $O(n\log^2n)$ steps without interacting with particles at other sources. Following the proof of \cite[Theorem 1]{BasuGangulyHoffmanRichey19}, this happens for all of the sources with probability at least $1- n^{-b}$ for some $b>0$. To obtain the exponential bound the argument can be repeated with a single source at 0. It is exponentially likely that all particles will reach the source within $n^4$ steps and, by \thref{prop:point.source.lower}, that the particles will then stabilize in no more than $n^4$ additional steps without reaching $\lfloor n/2\rfloor$. 

   To prove (ii), the idea is to stabilize the process on the cycle while freezing points at 0 then to recycle the approximately $(\rho - \crist - \epsilon)n$ particles that are overwhelmingly likely to freeze there by \thref{peanut butter}. These particles are used to reactivate many particles of which many are then frozen at site $n/2$ again by \thref{peanut butter}. The proof of \cite[Theorem 2]{BasuGangulyHoffmanRichey19} shows that this process is overwhelmingly likely to continue for exponentially many steps. Adapting their approach gives (ii). 
\end{proof}
\subsubsection{Proof of \texorpdfstring{\thref{thm:universal}}{Theorem 1.1}}
Let $S_n$ be the number of sleeping particles in a sample from the invariant
distribution of the driven-dissipative chain on $\ii{0,n}$.
\thref{prop:driven.dissipative.lower} proves that 
\begin{align*}
  \liminf_{n \to \infty}\f{\E[S_{n}] }{n+1} &\geq\crist,\\\intertext{and 
         \thref{prop:driven.dissipative.upper} proves that}
  \limsup_{n\to\infty}\f{\E[S_{n}] }{n+1} &\leq\crist,
\end{align*}
demonstrating that $\critDD=\crist$.
Similarly, \thref{prop:point.source.lower,prop:point.source.upper} prove
$\critPS=\crist$, and \thref{prop:fixed.energy.active,prop:fixed.energy.fixation}
prove $\critFE=\crist$, and \thref{prop:cycle} proves 
$\critCY=\crist$.

\subsection{Bounds on critical densities away from 0 and 1}
\label{sec:critical.bounds}
As we mentioned in the introduction, with it now established that
\begin{align*}
  \crist=\critDD=\critPS=\critFE=\critCY,
\end{align*}
we can approach the classical problem of bounding these critical densities away from $0$ and $1$
using layer percolation.
We give simple arguments that the critical density for one-dimensional ARW
is bounded away from $0$
for all $\lambda$ and away from $1$ for small enough $\lambda$. 
We have not tried to optimize our proofs.
Rather, we wish to demonstrate that these two bounds, which were major accomplishments
when first carried out in \cite{rolla2012absorbing} and \cite{BasuGangulyHoffman18},
are easy consequences of our theory.
We start with a slight improvement of Rolla and Sidoravicius's celebrated lower bound
$\critFE\geq \frac{\lambda}{1+\lambda}$ \cite{rolla2012absorbing}:
\begin{prop}
  For any $\lambda>0$, we have $\crist\geq\frac{1}{1/2+\lambda}$.
\end{prop}
\begin{proof}
  Let $X_n$ denote the greatest row present in the infection set $\infset[(0,0)_0]_n$.
  Any given cell in layer percolation infects $1+\Geo(1/2)$ cells in its row in the next step;
  in the row above it, it infects this quantity of cells thinned by $\lambda/(1+\lambda)$,
  which is at least $1$ with probability $\frac{\lambda}{1/2+\lambda}$.
  Thus, conditional on layer percolation up to step~$n$, we have $X_{n+1}=X_n+1$ 
  with at least probability $\frac{\lambda}{1/2+\lambda}$, yielding $\crist\geq\frac{\lambda}{1/2+\lambda}$.
\end{proof}

Next, we reproduce the result of \cite{BasuGangulyHoffman18} that $\critFE$ is
strictly below $1$ for small enough $\lambda$.

\begin{prop}
  If $\lambda<1$, then $\crist<1$.
\end{prop}
\begin{proof}
  A given cell infects on average $2$ cells in its row in the next step
  and on average $2\lambda_0$ cells in the row above, where $\lambda_0=\lambda/(1+\lambda)$.
  Hence, if $Z_n(s)$ is the number of cells in the infection set
  $\infset[(0,0)_0]_n$ in row~$s$, we have
  \begin{align*}
    \E\bigl[ Z_{n+1}(s) \mid Z_n(\cdot)\bigr] \leq 2Z_n(s) + 2\lambda_0 Z_n(s-1).
  \end{align*}
  Setting $\mu_n(s)=\E Z_n(s)$, we thus have $\mu_{n+1}(s)\leq 2\mu_n(s)+2\lambda_0\mu_n(s-1)$.
  We can view this inequality as stating that each component of the vector $\mu_{n+1}(\cdot)$ is bounded
  by the corresponding component of $A\mu_n(\cdot)$, where $A$ is the infinite matrix
  \begin{align*}
    A= \begin{bmatrix}
      2 & 0 & 0 & 0 &\cdots\\
      2\lambda_0 & 2 & 0 &0&\cdots\\
      0 & 2\lambda_0 &2  & 0&\cdots\\
      0 &0 & 2\lambda_0 &2  & \cdots\\
      \vdots&\vdots&\vdots&\vdots&\ddots
    \end{bmatrix}.
  \end{align*}
  We can iterate the inequality to bound each component of $\mu_n(\cdot)$ 
  by that of $A^n\mu_0(\cdot)$, with $\mu_0(s)=\1\{s=0\}$. Hence $\mu_n(\cdot)$
  is bounded by the first column of $A^n$, which after a calculation yields
  \begin{align*}
    \mu_n(s) &\leq 2^n\lambda_0^s\binom{n}{s}.
  \end{align*}
  Now let $s=\rho n$ for $0<\rho<n$ and apply the standard bound $\binom{n}{s}=\binom{n}{n-s}\leq 
  (en/(n-s))^{n-s}$ to get
  \begin{align}
    \P\bigl[Z_n(\rho n)\geq 1\bigr]\leq \E Z_n(\rho n)
    \leq \biggl(2\lambda_0^{\rho } \biggl(\frac{e}{1-\rho}\biggr)^{1-\rho}\biggr)^n.\label{eq:rowreachedbound}
  \end{align}
  The expression $2\lambda_0^\rho((e/(1-\rho))^{1-\rho}$ converges to $2\lambda_0$ as $\rho\to 1$, 
  which is strictly less than $1$ by our assumption $\lambda<1$. Thus there exists
  $\rho<1$ so that the right-hand side of \eqref{eq:rowreachedbound}
  decays exponentially. For such $\rho$,
  the infection set at step~$n$ is exponentially unlikely to contain cells in rows at
  or above $\rho n$,
  which by \thref{prop:subadditive} proves that $\crist\leq\rho$.
\end{proof}

As we mentioned, we have not tried to optimize these results.
We suspect that by similar technique, we could achieve
the optimal lower bound $\crist\geq C\sqrt{\lambda}$ as $\lambda\to 0$
proven in \cite{AsselahSchapiraRolla19}, and that we could show 
$\crist<1$ for all $\lambda>0$ as proven in \cite{HoffmanRicheyRolla20}.

\section{Acknowledgements}
Thanks to Vittoria Silvestri and Joshua Meisel for helpful conversations. Junge was partially supported by NSF grants DMS-2115936 and DMS-2238272.

\appendix
\section{Galton--Watson processes with migration}
\label{sec:appendix}

In this appendix, we prove the following concentration result for
critical geometric branching processes with migration (see \thref{def:GW}).
\begin{prop}\thlabel{prop:GW.emigration.concentration}
  Let $(X_j)_{j\geq 0}$ be a critical geometric branching process with migration $(e_j)_{j\geq 1}$.
  Let $X_0=x_0$ and let $\abs{e_j}\leq \emax$ for all $j$, for some $\emax\geq 1$. Then
  for any $t\geq 0$,
  \begin{align}
        \max\Bigl(\P\bigl[ X_j-\mu_j\geq t\bigr],\ 
        \P\bigl[ X_j-\mu_j\leq -t\bigr]\Bigr) &\leq C \exp\Biggl( -\frac{ct^2}{j(j\emax+\abs{x_0}+t)} \Biggr),
      \label{eq:GW.ec.general}
  \end{align}
  for some absolute constants $c,C>0$ and
  \begin{align*}
    \mu_j = \E X_j = x_0 + \sum_{i=1}^j e_j.
  \end{align*}
\end{prop}
In a typical application, we have $j,\emax=O(n)$, $\abs{x_0}=O(n^2)$, and $t=sn^2$ for $s$ bounded
away from $0$.
The right-hand side of \eqref{eq:GW.ec.general} is then bounded by $Ce^{-csn}$ for some constants $c$ and $C$,
and in particular for any fixed $s$ we obtain an exponential bound in $n$.

A benefit of considering branching processes with geometric offspring distribution
is that the distribution of the $j$th generation can be explicitly calculated.
The following facts are well known and can be proven by simple
inductive arguments.
\begin{prop}\thlabel{prop:GW.distributions}
  Consider a critical geometric branching process $(X_j)_{j\geq 0}$ with $X_0=1$
  and constant migration $m$ after each step.
  \begin{enumerate}[(a)]
    \item If $m=0$, then $\P[X_j>0] = 1/(j+1)$ and the conditional distribution of $X_j$ given $X_j>0$
      is $1+\Geo\bigl(1/(j+1)\bigr)$.
    \item If $m=1$, then $X_j$ has distribution $1+\Geo\bigl(1/(j+1)\bigr)$.
  \end{enumerate}
  
\end{prop}

In the next result, we consider a signed Galton--Watson process $(V_j)_{j\geq 0}$
with migration $(e_j)_{j\geq 1}$ satisfying $\abs{e_j}\leq\emax$.
We then consider two (nonnegative) Galton--Watson processes $(Y_j)_{j\geq 0}$ and $(Z_j)_{j\geq 0}$,
both of which have constant immigration $\emax$ at each step.
The idea is that $(Y_j-Z_j)_{j\geq 0}$ has similar dynamics as $(V_j)_{j\geq 0}$ but should be less
concentrated because of the extra immigration. We make this precise by showing
that after centering both processes  around their means,
the first is dominated by the second in the convex stochastic order, signifying
that it is stochastically less variable than the second (see \cite[Section~3.A]{SS}).
Thus we can bound $V_j$ in terms of $Y_j-Z_j$, whose distribution can be explicitly
calculated when the child distributions are geometric.
\begin{lemma}\thlabel{lem:GW.emigration.comparison}
  Let $(V_j)_{j\geq 0}$ be a signed Galton--Watson process with migration $(e_j)_{j\geq 0}$
  whose child distribution has mean~$1$.
  Suppose that $V_0=v_0$ for
  some integer $v_0$ and that  $\abs{e_j}\leq \emax$ for all $j$.
  Let $(Y_j)_{j\geq 0}$ and $(Z_j)_{j\geq 0}$ be independent Galton--Watson processes
  with constant immigration $\emax$ and
  the same child distribution as $(V_j)$. Let $Y_0=v_0$ and $Z_0=0$
  if $v_0\geq 0$, and let $Y_0=0$ and $Z_0=\abs{v_0}$ if $v_0<0$.  
  For any 
  convex function $\varphi\colon\RR\to\RR$,
  \begin{align}\label{eq:cx.order}
    \E\varphi(V_j-\E V_j) \leq \E \varphi\bigl(Y_j-Z_j - \E[Y_j-Z_j]\bigr).
  \end{align}
\end{lemma}
  \newcommand{\Bplus}{\mathrm{B}+}
  \newcommand{\Bminus}{\mathrm{B}-}
  \newcommand{\Rplus}{\mathrm{R}+}
  \newcommand{\Rminus}{\mathrm{R}-}
\begin{proof}
  \newcommand{\visplus}{\mathrm{vis}+}
  \newcommand{\visminus}{\mathrm{vis}-}
  \newcommand{\invplus}{\mathrm{inv}+}
  \newcommand{\invminus}{\mathrm{inv}-}
  We define a more complicated branching process and embed $(V_j)$, $(Y_j)$,
  and $(Z_j)$ in it. Each population member has a sign (positive or negative)
  and a visibility (visible or invisible). Generation~$0$ consists
  of $\abs{v_0}$ visible members, which are all positive if $v_0\geq 0$
  and negative if $v_0<0$. From generation~$j-1$, we produce generation~$j$ by the following steps:
  \begin{enumerate}[1.]
    \item Each member of generation~$j-1$ gives birth to an independent quantity
      of children, sampled from the offspring distribution. Each child has the same sign
      as its parent and provisionally has the same visibility.
    \item Add $\emax$ new positive invisible population members and $\emax$ new
      negative invisible population members.
    \item 
      \begin{enumerate} \item If $e_j\geq 0$, switch $e_j$ negative visible members to invisible; if there are
      fewer than $e_j$ negative visible members, switch all of them to invisible and then
      switch positive invisible members to visible so that a total of $e_j$ members
      switch visibility.
    \item If $e_j<0$, switch $\abs{e_j}$ positive visible members to invisible;
      if there are
      fewer than $\abs{e_j}$ positive visible members, switch all of them to invisible and then
      switch negative invisible members to visible so that a total of $\abs{e_j}$ members
      switch visibility.
      \end{enumerate}

  \end{enumerate}
  
  Observe that it is always possible to switch the visibility of $\abs{e_j}$ population members
  because $\abs{e_j}\leq\emax$ and we add $\emax$ positive and negative invisible elements
  at each step. Also observe that in each generation, all visible members have the same sign.
  This is true by definition at generation~$0$, and then it holds at successive generations
  by induction:  each successive generation has all its provisionally visible
  elements with the same sign, and then new elements of a given sign are turned visible only when
  there are no visible elements of the opposite sign.
  
  We define $V_j^+$ and $V_j^-$ as the number of visible positive and negative members, respectively,
  at generation~$j$. Let $V_j=V_j^+-V_j^-$.
  Let $I_j^+$ and $I_j^-$ be the number of invisible positive and negative members, respectively,
  at generation~$j$, and let $I_j=I_j^+-I_j^-$.
  Finally let $Y_j = V_j^+ + I_j^+$ and $Z_j = V_j^-+I_j^-$.
  We claim that
  \begin{enumerate}[(i)]
    \item $(V_j)_{j\geq 0}$ is a signed Galton--Watson process
      with migration $e_j$, consistent with its definition in the statement of the 
      proposition;\label{i:Vj}
    \item $(Y_j)_{j\geq 0}$ and $(Z_j)_{j\geq 0}$ are independent Galton--Watson processes
      with immigration~$\emax$, and $V_j+I_j=Y_j-Z_j$;\label{i:Vj+Ij}
    \item $\E[V_j+I_j\mid V_j] = V_j - \sum_{i=1}^je_i$ a.s.\ for all $j\geq 0$.\label{i:condexp}
  \end{enumerate}
  Claim~\ref{i:Vj} is evident by considering the dynamics of the visible population members only.
  The part of claim~\ref{i:Vj+Ij} about $(Y_j)$ and $(Z_j)$ is proven similarly: 
  The processes $(Y_j)$ and $(Z_j)$ give the counts of positive and negative members, respectively,
  with visibility ignored. For each of these processes, the changes in visibility are irrelevant,
  and we simply see two independent Galton--Watson processes with $\emax$ new members added
  at each step. And the statement $V_j+I_j=Y_j-Z_j$ holds by definition.
  
  To prove claim~\ref{i:condexp}, we will show that
  \begin{align}\label{eq:stronger}
    \E\bigl[ I_j\bigmid V_j^+,\,V_j^-\bigr] &= -\sum_{i=1}^j e_i \text{ a.s.}
  \end{align}
  Claim~\ref{i:condexp} follows by adding $V_j$ to both sides of the equation
  and taking conditional expectations with respect to $V_j$.
  When $j=0$, equation~\eqref{eq:stronger} is trivial. Now we proceed inductively, 
  assuming \eqref{eq:stronger}.
  Let $(L_i^{\visplus})_{i\geq 1}$, $(L_i^{\visminus})_{i\geq 1}$,
  $(L_i^{\invplus})_{i\geq 1}$, and $(L_i^{\invminus})_{i\geq 1}$ be
  the offspring counts of each visible positive, visible negative, invisible positive, and invisible negative
  population member, respectively, in generation~$j$. These random variables have mean~$1$
  and
  are independent of each other and of the process
  up to generation~$j$. 
  From step~$j$ to step~$j+1$, the invisible members give birth to new provisionally invisible
  members of the same signs as their parents, a quantity of $\emax$ positive and negative invisible
  elements are added, and $\abs{e_{j+1}}$ elements have their visibility switched.
  Hence,
  \begin{align*}
    I_{j+1} &= \sum_{i=1}^{I_{j}^+}L_i^{\invplus} - \sum_{i=1}^{I_{j}^-}L_i^{\invminus}-e_{j+1}.
  \end{align*}
  Now, we consider $I_{j+1}$ conditional on the full information at generation~$j$ together
  with the offspring counts of the visible elements at generation~$j$. This information
  is independent of $(L_i^{\invplus})_{i\geq 1}$ and $(L_i^{\invminus})_{i\geq 1}$, yielding
  \begin{align*}
    &\E\bigl[ I_{j+1} \bigmid V_{j}^+,\,V_{j}^-,\,I_{j}^+,\,I_{j}^-,\,
                            (L_i^{\visplus})_{i\geq 1},\, (L_i^{\visminus})_{i\geq 1}\bigr]\\
         &\qquad\qquad\qquad\qquad\qquad\qquad= I_{j}^+ - I_{j}^- - e_{j+1} = I_{j}-e_{j+1}
         \text{ a.s.}
  \end{align*}
  Taking a conditional expectation of both sides of this equation given the above information except for 
  $I_{j}^+$ and $I_{j}^-$ gives
  \begin{align*}
    &\E\bigl[ I_{j+1} \bigmid V_{j}^+,\,V_{j}^-,\,
                            (L_i^{\visplus})_{i\geq 1},\, (L_i^{\visminus})_{i\geq 1}\bigr]\\
        &\qquad\qquad\qquad\qquad= \E\bigl[I_{j}\bigmid V_{j}^+,\,V_{j}^-,\,
                            (L_i^{\visplus})_{i\geq 1},\, (L_i^{\visminus})_{i\geq 1}\bigr]
                            -e_{j+1}\text{ a.s.}\\
       &\qquad\qquad\qquad\qquad=  \E\bigl[I_{j}\bigmid V_{j}^+,\,V_{j}^-\bigr]
          -e_{j+1} =-\sum_{i=1}^{j+1} e_i  \text{ a.s.}
  \end{align*}
  The second to last equality is by Doob's conditional independence property 
  \cite[Proposition~6.6]{Kallenberg02}, and the last is by the inductive hypothesis \eqref{eq:stronger}.
  Since $V_{j+1}^+$ and $V_{j+1}^-$ are measurable with respect to $\sigma\bigl(V_{j}^+,\,V_{j}^-,\,
                            (L_i^{\visplus})_{i\geq 1},\, (L_i^{\visminus})_{i\geq 1}\bigr)$,
  we can now take expectations with respect to $V_{j+1}^+$ and $V_{j+1}^-$ to obtain \eqref{eq:stronger}
  with $j$ replaced by $j+1$,
  thus completing the proof of claim~\ref{i:condexp}.
  
  It follows from claim~\ref{i:condexp} that
  \begin{align*}
    V_j - \E V_j = V_j - v_0 - \sum_{i=1}^j e_j &= \E\bigl[V_j+I_j \bigmid V_j\bigr] - v_0.
  \end{align*}
  Finally, by \ref{i:Vj+Ij} we have $V_j+I_j=Y_j-Z_j$ and $\E[Y_j-Z_j]=v_0$, yielding
  \begin{align*}
    V_j - \E V_j
     &= \E\Bigl[Y_j-Z_j - \E[Y_j-Z_j]\Bigmid V_j\Bigr].
  \end{align*}
  Now \eqref{eq:cx.order} follows by Jensen's inequality.
  \end{proof}

\begin{lemma}\thlabel{lem:Janson}
  Let
  \begin{align*}
    X=\sum_{i=1}^N \bigl(1 + \Geo(p_i)\bigr)
  \end{align*}
  and let $p_*=\min_i p_i$.
  If $\E X\leq\nu$, then for all $t\geq 0$
  \begin{align}\label{eq:Janson1}
    \P[ X - \nu\geq t] &\leq \exp\Biggl(-\frac{p_*t^2}{2(\nu+t)}    \Biggr).
  \end{align}
  If $\E X\geq\nu$, then for all $t\geq 0$
  \begin{align}\label{eq:Janson2}
    \P[ X - \nu\leq -t] &\leq \exp\Biggl(-\frac{p_*t^2}{2(\nu+t)}    \Biggr).
  \end{align}
\end{lemma}
\begin{proof}
  Let $\mu=\E X$ and fix $\nu,s\in\RR$. Define
  \begin{align*}
    h(x) &= x\Biggl(\frac{s-x+\nu}{x} - \log\biggl(\frac{s+\nu}{x}\biggr)\Biggr).
  \end{align*}
  Observe that $h'(x)=-\log((s+\nu)/x)$, and hence $h$ is decreasing on $(0,s+\nu]$
  and increasing on $[s+\nu,\infty)$.

  Suppose $\mu\leq\nu$ and $s\geq 0$.
  By \cite[Theorem~2.1]{Janson},
  \begin{align}\label{eq:Jut}
    \P[X-\nu\geq s] \leq e^{-p_*h(\mu)}\leq e^{-p_*h(\nu)},
  \end{align}
  with the final inequality holding because $h$ is decreasing on the interval from
  $\mu$ to $\nu$.
  Similarly, $\mu\geq\nu$ and $-\mu\leq s\leq 0$.
  By \cite[Theorem~3.1]{Janson},
  \begin{align}\label{eq:Jlt}
    \P[X-\nu\leq s] \leq e^{-p_*h(\mu)}\leq e^{-p_*h(\nu)},
  \end{align}
  with the final inequality holding because $h$ is increasing on the interval
  from $\nu$ to $\mu$.
  
  Now we bound the right-hand sides of \eqref{eq:Jut} and \eqref{eq:Jlt}. Assuming that
  $s\geq-\nu$,
  \begin{align*}
    e^{-p_*h(\nu)} &= \exp\Biggl(-p_* \nu\biggl(\frac{s}{\nu} - \log\biggl(1 + \frac{s}{\nu}\biggr)\Biggr)
    \leq \exp\Biggl(-\frac{p_*s^2}{2(\nu+\abs{s})}    \Biggr),
  \end{align*}
  using the inequality $x-\log(1+x)\geq x^2/2(1+\abs{x})$ for all $x\geq -1$.
  (To prove this inequality, observe that the derivative of $x-\log(1+x)-x^2/2(1+\abs{x})$
  is negative for $-1<x<0$ and positive for $x>0$.)
  Applying \eqref{eq:Jut} combined with this bound when $\mu\leq\nu$ and $s\geq 0$
  proves \eqref{eq:Janson1}, setting $t=s$.
  And applying \eqref{eq:Jlt} combined with this bound when $\mu\geq\nu$ and $-\nu\leq s\leq 0$
  proves \eqref{eq:Janson2}, setting $t=-s$, at least for the case $t\leq \nu$.
  And in fact \eqref{eq:Janson2} holds even when $t>\nu$ since the left-hand side
  of \eqref{eq:Janson2} is equal to zero in that case.\end{proof}

\begin{proof}[Proof of \thref{prop:GW.emigration.concentration}]
  Let $(Y_j)$ and $(Z_j)$ be independent Galton--Watson processes
  with child distribution $\Geo(1/2)$ and immigration $\emax$
  and let $Y_0=\abs{x_0}$ and $Z_0=0$.
  First, we work out the distributions of $Y_j$ and $Z_j$.
  By \thref{prop:GW.distributions},
  \begin{align*}
    Z_j \eqd \sum_{1}^{\emax}\Bigl(1+\Geo\bigl(\tfrac{1}{j}\bigr)\Bigr),
  \end{align*}
  with the summands taken as independent random variables with the given distributions.
  Thinking of $Y_j$ as being distributed as $Z_j$ plus $\abs{x_0}$ many independent
  Galton--Watson processes starting from $1$ with no immigration,
  \begin{align}\label{eq:Yjdist}
    Y_j \eqd \sum_1^{\emax}\Bigl(1 + \Geo\bigl(\tfrac{1}{j}\bigr)\Bigr) + \sum_1^{\abs{x_0}}\Ber\bigl(\tfrac{1}{j+1}\bigr)\bigl(1+ \Geo\Bigl(\tfrac{1}{j+1}\bigr)\Bigr).
  \end{align}
  
  Now, we want to prove concentration bounds for $Y_j$ and $Z_j$ that we can then transfer to 
  $X_j$ via \thref{lem:GW.emigration.comparison}.
  We start with $Z_j$. Since $\E Z_j =j\emax$,
  by \thref{lem:Janson} for all $t\geq 0$,
  \begin{align}
    \max\Bigl(\P[Z_j\geq \E Z_j + t],\,\P[Z_j\leq \E Z_j -t]\Bigr)\leq
    \exp\biggl(-\frac{t^2}{2j(j\emax + t)}\biggr).\label{eq:Zj}
  \end{align}
  For $Y_j$, we prove bounds on on the second sum in \eqref{eq:Yjdist}
  by separately controlling the Bernoulli and geometric random variables.
  Let $S\sim\Bin(\abs{x_0},\frac{1}{j+1})$ so that we can express the second sum as
  \begin{align*}
    Y_j':=\sum_{1}^S\bigl(1+ \Geo\Bigl(\tfrac{1}{j+1}\bigr)\Bigr).
  \end{align*}
  We will then bound the four terms on the right-hand sides of
  \begin{align}
    \P[Y'_j \geq \E Y'_j + 2t] \leq \P\biggl[Y'_j \geq \E Y'_j + 2t\text{ and } S\leq\frac{\abs{x_0}+t}{j+1}\biggr]
      + \P\biggl[S>\frac{\abs{x_0}+t}{j+1}\biggr]\label{eq:Yj.bound.ut}\\\intertext{and}
      \P[Y'_j \leq \E Y'_j - 2t] \leq \P\biggl[Y'_j \leq \E Y'_j  -2t\text{ and } S\geq\frac{\abs{x_0}-t}{j+1}\biggr]
      + \P\biggl[S<\frac{\abs{x_0}-t}{j+1}\biggr]\label{eq:Yj.bound.lt}
  \end{align}
  for $t\geq 0$.
  By Bernstein's inequality,
  \begin{align}\label{eq:Bernstein}
    \max\Biggl(\P\biggl[S > \frac{\abs{x_0} + t}{j+1}\biggr],\ 
                \P\biggl[S < \frac{\abs{x_0} - t}{j+1}\biggr]\Biggr)&\leq
    \exp\biggl( -\frac{t^2}{2(j+1)\bigl(\abs{x_0}+t/3\bigr)}\biggr). 
  \end{align}
  
  To bound the other two terms, we apply \thref{lem:Janson} conditionally given $S$.
  Let $\nu = \E Y'_j+t=\abs{x_0}+t$,
  so that the event $\{Y'_j\geq \E Y'_j+2t\}$ can be viewed
  as $\{Y'_j\geq \nu+t\}$.
  On the event $S\leq(\abs{x_0}+t)/(j+1)$, we have
    $\E[Y'_j\mid S] = S(j+1)\leq\nu$ a.s.
  Applying \thref{lem:Janson},
  \begin{align*}
    \P[Y'_j\geq \E Y'_j+ 2t\mid S]\1\biggl\{S\leq\frac{\abs{x_0}+t}{j+1}\biggr\}
      &= \P\Bigl[Y'_j \geq \nu+t\Bigmid S\Bigr]\1\biggl\{S\leq\frac{\abs{x_0}+t}{j+1}\biggr\}\\
      &\leq \exp\Biggl(-\frac{t^2}{2(j+1)(\abs{x_0}+2t)}\Biggr)\text{ a.s.}
  \end{align*}
  Hence
  \begin{align}\label{eq:66.3}
    P\biggl[Y'_j \geq \E Y'_j + 2t\text{ and } S\leq\frac{\abs{x_0}+t}{j+1}\biggr]
      &\leq 
      \exp\Biggl(-\frac{t^2}{2(j+1)(\abs{x_0}+2t)}\Biggr).
  \end{align}
  Similarly, for $\nu'=\E Y'_j-t$, we have $\E[Y'_j\mid S]\geq \nu'$ a.s.\ 
  on the event $S\geq(\abs{x_0}-t)/(j+1)$, and by \thref{lem:Janson} we obtain
  \begin{align*}
    \P[Y'_j-\E Y'_j\leq -2t \mid S]\1\biggl\{S\geq\frac{\abs{x_0}-t}{j+1}\biggr\}
    &= \P\Bigl[Y'_j-\nu' \leq -t\Bigmid S\Bigr]\1\biggl\{S\leq\frac{\abs{x_0}-t}{j+1}\biggr\}\\
      &\leq \exp\Biggl(-\frac{t^2}{2(j+1)(\nu'+t)}\Biggr)\text{ a.s.},
  \end{align*}
  yielding
  \begin{align}\label{eq:66.4}
    P\biggl[Y'_j \leq \E Y'_j - 2t\text{ and } S\geq\frac{\abs{x_0}-t}{j+1}\biggr]
      &\leq 
      \exp\Biggl(-\frac{t^2}{2(j+1)\abs{x_0}}\Biggr).
  \end{align}
  Combining \eqref{eq:Bernstein}, \eqref{eq:66.3}, and \eqref{eq:66.4}, the right-hand sides
  of \eqref{eq:Yj.bound.ut} and \eqref{eq:Yj.bound.lt} are both bounded by
  \begin{align*}
    2\exp\Biggl(-\frac{t^2}{2(j+1)(\abs{x_0}+2t)}\Biggr).
  \end{align*}
  Expressing $Y_j-Z_j$ as $Y'_j+Z_j'-Z_j$ with $Z_j'$ an independent copy of $Z_j$
  and applying our bound on \eqref{eq:Yj.bound.ut} and \eqref{eq:Yj.bound.lt}
  together with \eqref{eq:Zj}, we obtain
  \begin{align}
    \begin{split}
    &\max\biggl(\P\Bigl[Y_j - Z_j - \E[Y_j-Z_j] \geq t\Bigr],\ \P\Bigl[Y_j - Z_j - \E[Y_j-Z_j] \leq -t\Bigr]\bigr)\\
    &\qquad\qquad\qquad \leq C_0\exp\Biggl(-\frac{c_0t^2}{j(j\emax +\abs{x_0}+t)}\Biggr)
    \end{split}\label{eq:Yj-Zj}
  \end{align}
  for some constants $c_0,C_0>0$.
  
  Now, we transfer these estimates to $X_k$ using \thref{lem:GW.emigration.comparison}.
  Fix some $s\geq 2$ and
  define
  \begin{align*}
    \varphi(x) = \biggl(\frac{x}{\sqrt{j(j\emax+\abs{x_0})}} - s+1\biggr)\1\biggl\{\frac{x}{\sqrt{j(j\emax+\abs{x_0})}} - s+1\geq 0\biggr\},
  \end{align*}
  which is convex. Since $\varphi(x)\geq 1$ for $x\geq s\sqrt{j(j\emax+\abs{x_0})}$,
  \begin{align}
    \begin{split}
    \P\Bigl[ X_j-\E X_j \geq s\sqrt{j(j\emax+\abs{x_0})}\Bigr] &\leq \E \varphi\bigl(X_j-\E X_j\bigr)\\
      &\leq \E\varphi\bigl(Y_j-Z_j - \E[Y_j-Z_j]\bigr),
    \end{split}
    \label{eq:GW.Xj.bound1}
  \end{align}
  applying \thref{lem:GW.emigration.comparison} in the second line. (If $x_0<0$, then $Y_j-Z_j$
  should be replaced by $Z_j-Y_j$. For notational simplicity we continue to write $Y_j-Z_j$,
  since our tail bound \eqref{eq:Yj-Zj} on $Y_j-Z_j$ is symmetric anyhow.)
  Continuing the calculation,
  \begin{align}
    \begin{split}
    \E\varphi\bigl(Y_j-Z_j - \E[Y_j-Z_j]\bigr)&=
      \int_0^{\infty}\P\Bigl[ \varphi\bigl(Y_j-Z_j - \E[Y_j-Z_j]\bigr) \geq u\Bigr]\,du\\
      &= \int_{s-1}^{\infty}\P\Bigl[ Y_j-Z_j-\E[Y_j-Z_j] \geq v\sqrt{j(j\emax+\abs{x_0})} \Bigr]\,dv\\
      &\leq\int_{s-1}^{\infty}C_0\exp\Biggl(-\frac{c_0v^2}{1+\frac{v}{\sqrt{\emax  + \abs{x_0}/j}}}\Biggr)\,dv
    \end{split}
    \label{eq:GW.Xj.bound2}
  \end{align}
  using \eqref{eq:Yj-Zj}.
  Let
  \begin{align*}
    h(v)=\frac{c_0v^2}{1+\frac{v}{\sqrt{\emax+\abs{x_0}/j}}}.
  \end{align*}
  The function $h(v)$ is convex and hence for $v\geq v_0\geq 1$
  \begin{align}\label{eq:h(v).convex}
    h(v) - h(v_0)\geq h'(v_0)(v-v_0)\geq h'(1)(v-v_0).
  \end{align}
  We compute
  \begin{align*}
    h'(1) &= \frac{2c_0}{1+\frac{1}{\sqrt{\emax + \abs{x_0}/j}}} - \frac{c_0}{\sqrt{\emax + \abs{x_0}/j} +2 + \frac{1}{\sqrt{\emax+\abs{x_0}/j}} }\\
      &\geq \frac{2c_0}{1+\frac{1}{\sqrt{\emax + \abs{x_0}/j}}} - \frac{c_0}{\sqrt{\emax + \abs{x_0}/j} +2  }.
  \end{align*}
  This expression is increasing in $\emax$ and $\abs{x_0}$, and hence
  it is bounded below by its value when $\emax=1$ and $x_0=0$, yielding
  $h'(1)\geq \frac23 c_0$.
  
  Combining this bound with \eqref{eq:GW.Xj.bound1}--\eqref{eq:h(v).convex} and 
  recalling that we have assumed $s\geq 2$
  which makes $v\geq 1$ in the integral,
  \begin{align*}
    \P\Bigl[ X_j-\E X_j \geq s\sqrt{j(j\emax+\abs{x_0})}\Bigr] &\leq 
      \int_{s-1}^{\infty} C_0e^{-h(s-1)}\exp\Bigl(-\bigl(h(v)-h(s-1)\bigr)\Bigr)\,dv\\
      &\leq C_0 e^{-h(s-1)}\int_{s-1}^{\infty}e^{-\frac23 c_0(v-s+1)}\,dv\\
      &\leq C_1 \exp\Biggl( -\frac{c_1s^2}{1+\frac{s}{\sqrt{\emax+\abs{x_0}/j}}} \Biggr)
  \end{align*}
  for absolute constants $c_1,C_1>0$. Finally, substituting $t=s\sqrt{j(j\emax+\abs{x_0})}$ gives
  \begin{align*}
    \P\bigl[ X_j-\E X_j\geq t\bigr] &\leq C_1 \exp\Biggl( -\frac{c_1t^2}{j(j\emax+\abs{x_0}+t)} \Biggr),
  \end{align*}
  thus proving the upper bound in \eqref{eq:GW.ec.general}
  so long as $t\geq 2\sqrt{j(j\emax+\abs{x_0})}$. The lower bound
  under the same assumption on $t$
  is derived from \eqref{eq:Yj-Zj} in the exact same way,
  using the convex function
  \begin{align*}
        x\mapsto \abs[\bigg]{\frac{x}{\sqrt{j(j\emax+x_0)}} - s-1}\1\biggl\{\frac{x}{\sqrt{j(j\emax+x_0)}} - s-1\leq 0\biggr\},
  \end{align*}
  for fixed $s\leq -2$.
  To remove the condition $t\geq 2\sqrt{j(j\emax+\abs{x_0})}$,
  observe the the right-hand side of \eqref{eq:GW.ec.general} is decreasing in $t$
  and when $t=2\sqrt{j(j\emax+\abs{x_0})}$ is equal to
  \begin{align*}
    C \exp\Biggl( -\frac{4c(j\emax+\abs{x_0})}{j\emax+\abs{x_0}+2\sqrt{j(j+\emax+\abs{x_0})}} \Biggr)
      &\geq Ce^{-4c},
  \end{align*}
  and by increasing $C$ to ensure that $Ce^{-4c}\geq 1$, we obtain \eqref{eq:GW.ec.general}
  trivially for the case that $0\leq t\leq 2\sqrt{j(j\emax+\abs{x_0})}$.  
\end{proof}

\bibliographystyle{amsalpha}
\bibliography{main}

\end{document}

%% file: Y.tikz
\fill[forward] (12,4) rectangle +(1,1);
\fill[forward] (23,4) rectangle +(1,1);
\fill[forward] (17,3) rectangle +(1,1);
\fill[forward] (11,5) rectangle +(1,1);
\fill[forward] (15,5) rectangle +(1,1);
\fill[forward] (6,2) rectangle +(1,1);
\fill[forward] (26,5) rectangle +(1,1);
\fill[forward] (4,2) rectangle +(1,1);
\fill[forward] (27,6) rectangle +(1,1);
\fill[forward] (5,3) rectangle +(1,1);
\fill[forward] (8,2) rectangle +(1,1);
\fill[forward] (19,2) rectangle +(1,1);
\fill[forward] (17,5) rectangle +(1,1);
\fill[forward] (26,7) rectangle +(1,1);
\fill[forward] (6,4) rectangle +(1,1);
\fill[forward] (29,6) rectangle +(1,1);
\fill[forward] (7,3) rectangle +(1,1);
\fill[forward] (18,3) rectangle +(1,1);
\fill[forward] (21,2) rectangle +(1,1);
\fill[forward] (22,3) rectangle +(1,1);
\fill[forward] (27,8) rectangle +(1,1);
\fill[forward] (5,5) rectangle +(1,1);
\fill[forward] (8,4) rectangle +(1,1);
\fill[forward] (19,4) rectangle +(1,1);
\fill[forward] (9,3) rectangle +(1,1);
\fill[forward] (28,7) rectangle +(1,1);
\fill[forward] (26,9) rectangle +(1,1);
\fill[forward] (21,4) rectangle +(1,1);
\fill[forward] (31,8) rectangle +(1,1);
\fill[forward] (22,5) rectangle +(1,1);
\fill[forward] (9,5) rectangle +(1,1);
\fill[forward] (11,2) rectangle +(1,1);
\fill[forward] (13,5) rectangle +(1,1);
\fill[forward] (15,2) rectangle +(1,1);
\fill[forward] (24,5) rectangle +(1,1);
\fill[forward] (25,4) rectangle +(1,1);
\fill[forward] (31,10) rectangle +(1,1);
\fill[forward] (12,6) rectangle +(1,1);
\fill[forward] (14,3) rectangle +(1,1);
\fill[forward] (17,2) rectangle +(1,1);
\fill[forward] (23,6) rectangle +(1,1);
\fill[forward] (11,4) rectangle +(1,1);
\fill[forward] (15,4) rectangle +(1,1);
\fill[forward] (26,4) rectangle +(1,1);
\fill[forward] (16,3) rectangle +(1,1);
\fill[forward] (25,6) rectangle +(1,1);
\fill[forward] (20,3) rectangle +(1,1);
\fill[forward] (14,5) rectangle +(1,1);
\fill[forward] (5,2) rectangle +(1,1);
\fill[forward] (17,4) rectangle +(1,1);
\fill[forward] (15,6) rectangle +(1,1);
\fill[forward] (26,6) rectangle +(1,1);
\fill[forward] (16,5) rectangle +(1,1);
\fill[forward] (7,2) rectangle +(1,1);
\fill[forward] (18,2) rectangle +(1,1);
\fill[forward] (25,8) rectangle +(1,1);
\fill[forward] (29,8) rectangle +(1,1);
\fill[forward] (22,2) rectangle +(1,1);
\fill[forward] (5,4) rectangle +(1,1);
\fill[forward] (9,2) rectangle +(1,1);
\fill[forward] (10,3) rectangle +(1,1);
\fill[forward] (13,2) rectangle +(1,1);
\fill[forward] (7,4) rectangle +(1,1);
\fill[forward] (18,4) rectangle +(1,1);
\fill[forward] (31,7) rectangle +(1,1);
\fill[forward] (22,4) rectangle +(1,1);
\fill[forward] (12,3) rectangle +(1,1);
\fill[forward] (23,3) rectangle +(1,1);
\fill[forward] (30,8) rectangle +(1,1);
\fill[forward] (10,5) rectangle +(1,1);
\fill[forward] (13,4) rectangle +(1,1);
\fill[forward] (24,4) rectangle +(1,1);
\fill[forward] (31,9) rectangle +(1,1);
\fill[forward] (12,5) rectangle +(1,1);
\fill[forward] (14,2) rectangle +(1,1);
\fill[forward] (23,5) rectangle +(1,1);
\fill[forward] (25,5) rectangle +(1,1);
\fill[forward] (16,2) rectangle +(1,1);
\fill[forward] (6,3) rectangle +(1,1);
\fill[forward] (20,2) rectangle +(1,1);
\fill[forward] (31,11) rectangle +(1,1);
\fill[forward] (23,7) rectangle +(1,1);
\fill[forward] (14,4) rectangle +(1,1);
\fill[forward] (4,3) rectangle +(1,1) ;
\fill[forward] (27,7) rectangle +(1,1);
\fill[forward] (28,6) rectangle +(1,1);
\fill[forward] (8,3) rectangle +(1,1);
\fill[forward] (19,3) rectangle +(1,1);
\fill[forward] (25,7) rectangle +(1,1);
\fill[forward] (16,4) rectangle +(1,1);
\fill[forward] (26,8) rectangle +(1,1);
\fill[forward] (6,5) rectangle +(1,1);
\fill[forward] (20,4) rectangle +(1,1);
\fill[forward] (29,7) rectangle +(1,1);
\fill[forward] (21,3) rectangle +(1,1);
\fill[forward] (14,6) rectangle +(1,1);
\fill[forward] (27,9) rectangle +(1,1);
\fill[forward] (28,8) rectangle +(1,1);
\fill[forward] (8,5) rectangle +(1,1);
\fill[forward] (10,2) rectangle +(1,1);
\fill[forward] (9,4) rectangle +(1,1);
\fill[forward] (21,5) rectangle +(1,1);
\fill[forward] (12,2) rectangle +(1,1);
\fill[forward] (23,2) rectangle +(1,1);
\fill[forward] (30,7) rectangle +(1,1);
\fill[forward] (10,4) rectangle +(1,1);
\fill[forward] (11,3) rectangle +(1,1);
\fill[forward] (15,3) rectangle +(1,1);
\draw[thick,gray] (4,2) grid[step=1] (37,12);



%% file: Z.tikz
\fill[backward] (26,5) rectangle +(1,1);
\fill[backward] (27,6) rectangle +(1,1);
\fill[backward] (29,6) rectangle +(1,1);
\fill[backward] (22,5) rectangle +(1,1);
\fill[backward] (24,5) rectangle +(1,1);
\fill[backward] (23,6) rectangle +(1,1);
\fill[backward] (26,4) rectangle +(1,1);
\fill[backward] (25,6) rectangle +(1,1);
\fill[backward] (26,6) rectangle +(1,1);
\fill[backward] (24,4) rectangle +(1,1);
\fill[backward] (23,5) rectangle +(1,1);
\fill[backward] (25,5) rectangle +(1,1);
\fill[backward] (23,7) rectangle +(1,1);
\fill[backward] (28,6) rectangle +(1,1);
\fill[backward] (31,5) rectangle +(1,1);
\fill[backward] (34,4) rectangle +(1,1);
\fill[backward] (12,10) rectangle +(1,1);
\fill[backward] (27,4) rectangle +(1,1);
\fill[backward] (30,6) rectangle +(1,1);
\fill[backward] (32,6) rectangle +(1,1);
\fill[backward] (33,5) rectangle +(1,1);
\fill[backward] (13,8) rectangle +(1,1);
\fill[backward] (29,4) rectangle +(1,1);
\fill[backward] (16,7) rectangle +(1,1);
\fill[backward] (15,8) rectangle +(1,1);
\fill[backward] (20,7) rectangle +(1,1);
\fill[backward] (21,6) rectangle +(1,1);
\fill[backward] (12,9) rectangle +(1,1);
\fill[backward] (22,7) rectangle +(1,1);
\fill[backward] (14,9) rectangle +(1,1);
\fill[backward] (28,5) rectangle +(1,1);
\fill[backward] (17,8) rectangle +(1,1);
\fill[backward] (30,5) rectangle +(1,1);
\fill[backward] (19,8) rectangle +(1,1);
\fill[backward] (32,5) rectangle +(1,1);
\fill[backward] (11,10) rectangle +(1,1);
\fill[backward] (15,7) rectangle +(1,1);
\fill[backward] (20,6) rectangle +(1,1);
\fill[backward] (31,6) rectangle +(1,1);
\fill[backward] (22,6) rectangle +(1,1);
\fill[backward] (36,2) rectangle +(1,1);
\fill[backward] (12,8) rectangle +(1,1);
\fill[backward] (28,4) rectangle +(1,1);
\fill[backward] (14,8) rectangle +(1,1);
\fill[backward] (27,5) rectangle +(1,1);
\fill[backward] (30,4) rectangle +(1,1);
\fill[backward] (19,7) rectangle +(1,1);
\fill[backward] (32,4) rectangle +(1,1);
\fill[backward] (11,9) rectangle +(1,1);
\fill[backward] (24,6) rectangle +(1,1);
\fill[backward] (10,10) rectangle +(1,1);
\fill[backward] (13,9) rectangle +(1,1);
\fill[backward] (35,3) rectangle +(1,1);
\fill[backward] (16,8) rectangle +(1,1);
\fill[backward] (29,5) rectangle +(1,1);
\fill[backward] (18,8) rectangle +(1,1);
\fill[backward] (21,7) rectangle +(1,1);
\draw[thick,gray] (4,2) grid[step=1] (37,12);



%% file: X.tikz
\fill[forward] (12,4) rectangle +(1,1);
\fill[forward] (23,4) rectangle +(1,1);
\fill[forward] (17,3) rectangle +(1,1);
\fill[forward] (11,5) rectangle +(1,1);
\fill[forward] (15,5) rectangle +(1,1);
\fill[forward] (6,2) rectangle +(1,1);
\fill[both] (26,5) rectangle +(1,1);
\fill[forward] (4,2) rectangle +(1,1);
\fill[both] (27,6) rectangle +(1,1);
\fill[forward] (5,3) rectangle +(1,1);
\fill[forward] (8,2) rectangle +(1,1);
\fill[forward] (19,2) rectangle +(1,1);
\fill[forward] (17,5) rectangle +(1,1);
\fill[forward] (26,7) rectangle +(1,1);
\fill[forward] (6,4) rectangle +(1,1);
\fill[both] (29,6) rectangle +(1,1);
\fill[forward] (7,3) rectangle +(1,1);
\fill[forward] (18,3) rectangle +(1,1);
\fill[forward] (21,2) rectangle +(1,1);
\fill[forward] (22,3) rectangle +(1,1);
\fill[forward] (27,8) rectangle +(1,1);
\fill[forward] (5,5) rectangle +(1,1);
\fill[forward] (8,4) rectangle +(1,1);
\fill[forward] (19,4) rectangle +(1,1);
\fill[forward] (9,3) rectangle +(1,1);
\fill[forward] (28,7) rectangle +(1,1);
\fill[forward] (26,9) rectangle +(1,1);
\fill[forward] (21,4) rectangle +(1,1);
\fill[forward] (31,8) rectangle +(1,1);
\fill[both] (22,5) rectangle +(1,1);
\fill[forward] (9,5) rectangle +(1,1);
\fill[forward] (11,2) rectangle +(1,1);
\fill[forward] (13,5) rectangle +(1,1);
\fill[forward] (15,2) rectangle +(1,1);
\fill[both] (24,5) rectangle +(1,1);
\fill[forward] (25,4) rectangle +(1,1);
\fill[forward] (31,10) rectangle +(1,1);
\fill[forward] (12,6) rectangle +(1,1);
\fill[forward] (14,3) rectangle +(1,1);
\fill[forward] (17,2) rectangle +(1,1);
\fill[both] (23,6) rectangle +(1,1);
\fill[forward] (11,4) rectangle +(1,1);
\fill[forward] (15,4) rectangle +(1,1);
\fill[both] (26,4) rectangle +(1,1);
\fill[forward] (16,3) rectangle +(1,1);
\fill[both] (25,6) rectangle +(1,1);
\fill[forward] (20,3) rectangle +(1,1);
\fill[forward] (14,5) rectangle +(1,1);
\fill[forward] (5,2) rectangle +(1,1);
\fill[forward] (17,4) rectangle +(1,1);
\fill[forward] (15,6) rectangle +(1,1);
\fill[both] (26,6) rectangle +(1,1);
\fill[forward] (16,5) rectangle +(1,1);
\fill[forward] (7,2) rectangle +(1,1);
\fill[forward] (18,2) rectangle +(1,1);
\fill[forward] (25,8) rectangle +(1,1);
\fill[forward] (29,8) rectangle +(1,1);
\fill[forward] (22,2) rectangle +(1,1);
\fill[forward] (5,4) rectangle +(1,1);
\fill[forward] (9,2) rectangle +(1,1);
\fill[forward] (10,3) rectangle +(1,1);
\fill[forward] (13,2) rectangle +(1,1);
\fill[forward] (7,4) rectangle +(1,1);
\fill[forward] (18,4) rectangle +(1,1);
\fill[forward] (31,7) rectangle +(1,1);
\fill[forward] (22,4) rectangle +(1,1);
\fill[forward] (12,3) rectangle +(1,1);
\fill[forward] (23,3) rectangle +(1,1);
\fill[forward] (30,8) rectangle +(1,1);
\fill[forward] (10,5) rectangle +(1,1);
\fill[forward] (13,4) rectangle +(1,1);
\fill[both] (24,4) rectangle +(1,1);
\fill[forward] (31,9) rectangle +(1,1);
\fill[forward] (12,5) rectangle +(1,1);
\fill[forward] (14,2) rectangle +(1,1);
\fill[both] (23,5) rectangle +(1,1);
\fill[both] (25,5) rectangle +(1,1);
\fill[forward] (16,2) rectangle +(1,1);
\fill[forward] (6,3) rectangle +(1,1);
\fill[forward] (20,2) rectangle +(1,1);
\fill[forward] (31,11) rectangle +(1,1);
\fill[both] (23,7) rectangle +(1,1);
\fill[forward] (14,4) rectangle +(1,1);
\fill[forward] (4,3) rectangle +(1,1) ;
\fill[forward] (27,7) rectangle +(1,1);
\fill[both] (28,6) rectangle +(1,1);
\fill[forward] (8,3) rectangle +(1,1);
\fill[forward] (19,3) rectangle +(1,1);
\fill[forward] (25,7) rectangle +(1,1);
\fill[forward] (16,4) rectangle +(1,1);
\fill[forward] (26,8) rectangle +(1,1);
\fill[forward] (6,5) rectangle +(1,1);
\fill[forward] (20,4) rectangle +(1,1);
\fill[forward] (29,7) rectangle +(1,1);
\fill[forward] (21,3) rectangle +(1,1);
\fill[forward] (14,6) rectangle +(1,1);
\fill[forward] (27,9) rectangle +(1,1);
\fill[forward] (28,8) rectangle +(1,1);
\fill[forward] (8,5) rectangle +(1,1);
\fill[forward] (10,2) rectangle +(1,1);
\fill[forward] (9,4) rectangle +(1,1);
\fill[forward] (21,5) rectangle +(1,1);
\fill[forward] (12,2) rectangle +(1,1);
\fill[forward] (23,2) rectangle +(1,1);
\fill[forward] (30,7) rectangle +(1,1);
\fill[forward] (10,4) rectangle +(1,1);
\fill[forward] (11,3) rectangle +(1,1);
\fill[forward] (15,3) rectangle +(1,1);
\fill[backward] (31,5) rectangle +(1,1);
\fill[backward] (34,4) rectangle +(1,1);
\fill[backward] (12,10) rectangle +(1,1);
\fill[backward] (27,4) rectangle +(1,1);
\fill[backward] (30,6) rectangle +(1,1);
\fill[backward] (32,6) rectangle +(1,1);
\fill[backward] (33,5) rectangle +(1,1);
\fill[backward] (13,8) rectangle +(1,1);
\fill[backward] (29,4) rectangle +(1,1);
\fill[backward] (16,7) rectangle +(1,1);
\fill[backward] (15,8) rectangle +(1,1);
\fill[backward] (20,7) rectangle +(1,1);
\fill[backward] (21,6) rectangle +(1,1);
\fill[backward] (12,9) rectangle +(1,1);
\fill[backward] (22,7) rectangle +(1,1);
\fill[backward] (14,9) rectangle +(1,1);
\fill[backward] (28,5) rectangle +(1,1);
\fill[backward] (17,8) rectangle +(1,1);
\fill[backward] (30,5) rectangle +(1,1);
\fill[backward] (19,8) rectangle +(1,1);
\fill[backward] (32,5) rectangle +(1,1);
\fill[backward] (11,10) rectangle +(1,1);
\fill[backward] (15,7) rectangle +(1,1);
\fill[backward] (20,6) rectangle +(1,1);
\fill[backward] (31,6) rectangle +(1,1);
\fill[backward] (22,6) rectangle +(1,1);
\fill[backward] (36,2) rectangle +(1,1);
\fill[backward] (12,8) rectangle +(1,1);
\fill[backward] (28,4) rectangle +(1,1);
\fill[backward] (14,8) rectangle +(1,1);
\fill[backward] (27,5) rectangle +(1,1);
\fill[backward] (30,4) rectangle +(1,1);
\fill[backward] (19,7) rectangle +(1,1);
\fill[backward] (32,4) rectangle +(1,1);
\fill[backward] (11,9) rectangle +(1,1);
\fill[backward] (24,6) rectangle +(1,1);
\fill[backward] (10,10) rectangle +(1,1);
\fill[backward] (13,9) rectangle +(1,1);
\fill[backward] (35,3) rectangle +(1,1);
\fill[backward] (16,8) rectangle +(1,1);
\fill[backward] (29,5) rectangle +(1,1);
\fill[backward] (18,8) rectangle +(1,1);
\fill[backward] (21,7) rectangle +(1,1);
\draw[thick,gray] (0,0) grid[step=1] (37,12);
\draw[line width=.06cm,red] (36,2) ++(.14,.14) rectangle +(.72,.72);
\draw[line width=.06cm,red] (35,3) ++(.14,.14) rectangle +(.72,.72);
\draw[line width=.06cm,red] (34,4) ++(.14,.14) rectangle +(.72,.72);
\draw[line width=.06cm,red] (33,5) ++(.14,.14) rectangle +(.72,.72);
\draw[line width=.06cm,red] (32,5) ++(.14,.14) rectangle +(.72,.72);
\draw[line width=.06cm,red] (31,5) ++(.14,.14) rectangle +(.72,.72);
\draw[line width=.06cm,red] (30,5) ++(.14,.14) rectangle +(.72,.72);
\draw[line width=.06cm,red] (29,5) ++(.14,.14) rectangle +(.72,.72);
\draw[line width=.06cm,red] (28,5) ++(.14,.14) rectangle +(.72,.72);
\draw[line width=.06cm,red] (27,5) ++(.14,.14) rectangle +(.72,.72);
\draw[line width=.06cm,red] (25,6) ++(.14,.14) rectangle +(.72,.72);
\draw[line width=.06cm,red] (24,6) ++(.14,.14) rectangle +(.72,.72);
\draw[line width=.06cm,red] (23,6) ++(.14,.14) rectangle +(.72,.72);
\draw[line width=.06cm,red] (23,6) ++(.14,.14) rectangle +(.72,.72);
\draw[line width=.06cm,red] (22,6) ++(.14,.14) rectangle +(.72,.72);
\draw[line width=.06cm,red] (21,7) ++(.14,.14) rectangle +(.72,.72);
\draw[line width=.06cm,red] (20,7) ++(.14,.14) rectangle +(.72,.72);
\draw[line width=.06cm,red] (19,7) ++(.14,.14) rectangle +(.72,.72);
\draw[line width=.06cm,red] (18,8) ++(.14,.14) rectangle +(.72,.72);
\draw[line width=.06cm,red] (17,8) ++(.14,.14) rectangle +(.72,.72);
\draw[line width=.06cm,red] (16,8) ++(.14,.14) rectangle +(.72,.72);
\draw[line width=.06cm,red] (15,8) ++(.14,.14) rectangle +(.72,.72);
\draw[line width=.06cm,red] (14,8) ++(.14,.14) rectangle +(.72,.72);
\draw[line width=.06cm,red] (13,9) ++(.14,.14) rectangle +(.72,.72);
\draw[line width=.06cm,red] (12,9) ++(.14,.14) rectangle +(.72,.72);
\draw[line width=.06cm,red] (11,9) ++(.14,.14) rectangle +(.72,.72);
\draw[line width=.06cm,red] (10,10) ++(.14,.14) rectangle +(.72,.72);

\draw[line width=.06cm,blue] (31,10) ++(.14,.14) rectangle +(.72,.72);
\draw[line width=.06cm,blue] (31,9) ++(.14,.14) rectangle +(.72,.72);
\draw[line width=.06cm,blue] (31,8) ++(.14,.14) rectangle +(.72,.72);
\draw[line width=.06cm,blue] (30,8) ++(.14,.14) rectangle +(.72,.72);
\draw[line width=.06cm,blue] (29,8) ++(.14,.14) rectangle +(.72,.72);
\draw[line width=.06cm,blue] (28,8) ++(.14,.14) rectangle +(.72,.72);
\draw[line width=.06cm,blue] (28,7) ++(.14,.14) rectangle +(.72,.72);
\draw[line width=.06cm,blue] (27,7) ++(.14,.14) rectangle +(.72,.72);
\draw[line width=.06cm,blue] (27,6) ++(.14,.14) rectangle +(.72,.72);
\draw[line width=.06cm,blue!60!red] (26,6) ++(.14,.14) rectangle +(.72,.72);
\draw[line width=.06cm,blue] (26,5) ++(.14,.14) rectangle +(.72,.72);
\draw[line width=.06cm,blue] (25,5) ++(.14,.14) rectangle +(.72,.72);
\draw[line width=.06cm,blue] (24,5) ++(.14,.14) rectangle +(.72,.72);
\draw[line width=.06cm,blue] (24,4) ++(.14,.14) rectangle +(.72,.72);
\draw[line width=.06cm,blue] (23,4) ++(.14,.14) rectangle +(.72,.72);
\draw[line width=.06cm,blue] (22,4) ++(.14,.14) rectangle +(.72,.72);
\draw[line width=.06cm,blue] (21,4) ++(.14,.14) rectangle +(.72,.72);
\draw[line width=.06cm,blue] (20,4) ++(.14,.14) rectangle +(.72,.72);
\draw[line width=.06cm,blue] (19,4) ++(.14,.14) rectangle +(.72,.72);
\draw[line width=.06cm,blue] (18,4) ++(.14,.14) rectangle +(.72,.72);
\draw[line width=.06cm,blue] (18,3) ++(.14,.14) rectangle +(.72,.72);
\draw[line width=.06cm,blue] (17,3) ++(.14,.14) rectangle +(.72,.72);
\draw[line width=.06cm,blue] (16,3) ++(.14,.14) rectangle +(.72,.72);
\draw[line width=.06cm,blue] (15,3) ++(.14,.14) rectangle +(.72,.72);
\draw[line width=.06cm,blue] (14,3) ++(.14,.14) rectangle +(.72,.72);
\draw[line width=.06cm,blue] (14,2) ++(.14,.14) rectangle +(.72,.72);
\draw[line width=.06cm,blue] (13,2) ++(.14,.14) rectangle +(.72,.72);
\draw[line width=.06cm,blue] (12,2) ++(.14,.14) rectangle +(.72,.72);
\draw[line width=.06cm,blue] (11,2) ++(.14,.14) rectangle +(.72,.72);
\draw[line width=.06cm,blue] (10,2) ++(.14,.14) rectangle +(.72,.72);
\draw[line width=.06cm,blue] (9,2) ++(.14,.14) rectangle +(.72,.72);
\draw[line width=.06cm,blue] (8,2) ++(.14,.14) rectangle +(.72,.72);
\draw[line width=.06cm,blue] (7,2) ++(.14,.14) rectangle +(.72,.72);
\draw[line width=.06cm,blue] (6,2) ++(.14,.14) rectangle +(.72,.72);
\draw[line width=.06cm,blue] (5,2) ++(.14,.14) rectangle +(.72,.72);
\draw[line width=.06cm,blue] (4,2) ++(.14,.14) rectangle +(.72,.72);

\path (2.75,-2) node (x1) {$(x_1,t)_n$};
\draw (x1)--(6.5,2.5);
\path (33,13.25) node (y1) {$(y_1,t')_n$};
\draw (y1)--(31.5,10.5);
\path (10.5,13.25) node (x2) {$(x_2,t')_n$};
\draw (x2)--(11.5,10.5);
\path (39.5,0) node (y2) {$(y_2,t)_n$};
\draw (y2)--(36.5,2.5);

\draw (0.500000000000000,0)-- ++(0,-.2) node[below] {0};
\draw (5.50000000000000,0)-- ++(0,-.2) node[below] {5};
\draw (10.5000000000000,0)-- ++(0,-.2) node[below] {10};
\draw (15.5000000000000,0)-- ++(0,-.2) node[below] {15};
\draw (20.5000000000000,0)-- ++(0,-.2) node[below] {20};
\draw (25.5000000000000,0)-- ++(0,-.2) node[below] {25};
\draw (30.5000000000000,0)-- ++(0,-.2) node[below] {30};
\draw (35.5000000000000,0)-- ++(0,-.2) node[below] {35};
\draw (0,0.500000000000000)-- ++(-.2,0) node[left] {0};
\draw (0,5.50000000000000)-- ++(-.2,0) node[left] {5};
\draw (0,10.5000000000000)-- ++(-.2,0) node[left] {10};

%% file: s.tikz
\begin{scope}
        \draw[->] (0,0) -- (12,0);
        \draw[->] (0,0) -- (0,5);
        \path
          (5,1.6) coordinate (n0)
          (7,2.4) coordinate (n-n0)
          (6,1.6) coordinate (n2s)
          (6,2.4) coordinate (n2t-s);
        \draw[help lines,col1]
          (n2s) ++ (-.1,0) -- +(210:1.2) node[below] {$(X_1,s)_{n/2}$};
        \draw[help lines,col4]
          (n2t-s) ++ (-.1,0) -- +(135:1.2) node[above left] {$(Y_1,t-s)_{n/2}$};
        \draw[help lines,col3]
          (n2s) ++ (.1,0) -- +(330:1.2) node[below right] {$(Y_2,s)_{n/2}$};          
        \draw[help lines,col2]
          (n2t-s) ++ (.1,0) -- +(45:1.2) node[above right] {$(X_2,t-s)_{n/2}$};          
        \draw[very thick] (0,0) -- (4.5,1.6) -- (n0);
        \draw[col1, very thick, ->] (n0) -- (n2s);
        \draw[col4, very thick,->] (n0) -- (5.8,2.4) -- (n2t-s);
        \draw[very thick] (n-n0) -- (7.5,2.4)-- (12,4);
        \draw[->, very thick, col2] (7,2.4) -- (n2t-s);
        \draw[very thick, col3,->] (n-n0) -- (6.2,1.6) -- (n2s);

        \draw (5,0.1) -- (5,-0.1);
        \draw (6,0.1) -- (6,-0.1);
        \draw (7,0.1) -- (7,-0.1);

        \draw (.1,1.6) -- (-0.1,1.6);
        \draw (0.1,2.4) -- (-0.1,2.4);
        \draw (0.1,4) -- (-0.1,4);

        \node[below] at (5,0) {$n_0$};
        \node[below] at (6,0) {$\frac n2$};
        \node[below] at (7,0) {$n-n_0$};
        \node[below] at (12,0) {$n$};

        \node[left] at (0,0) {$0$};
        \node[left] at (0,1.6) {$s$};
        \node[left] at (0,2.4) {$t-s$};
        \node[left] at (0,4) {$t$};

        \node[below] at (6, -0.75) {\bf Steps};
        \node[rotate=90, above] at (-1, 2) {\bf Row};
        \node at (13.5,0) {};
\end{scope}

%% file: XX.tikz
\begin{scope}[scale=.311,shift={(0,-18)},forward/.style={fill=blue!40!white}, backward/.style={red!40!white}, both/.style={red!60!blue!70}, fpath/.style={blue},bpath/.style={red},bothpath/.style={purple}]
  \begin{scope}[shift={(5,3)},fill=blue,fill opacity=.3]
    \fill (0,0) rectangle +(1,1);
    \fill (0,1) rectangle +(1,1);
    \fill (1,1) rectangle +(1,1);
    \fill (0,2) rectangle +(1,1);
    \fill (1,2) rectangle +(1,1);
    \fill (0,3) rectangle +(1,1);
    \fill (2,1) rectangle +(1,1);
    \fill (3,1) rectangle +(1,1);
    \fill (2,2) rectangle +(1,1);
    \fill (4,1) rectangle +(1,1);
    \fill (5,1) rectangle +(1,1);
    \fill (6,1) rectangle +(1,1);
    \fill (5,2) rectangle +(1,1);
    \fill (5,3) rectangle +(1,1);
    \fill (6,2) rectangle +(1,1);
    \fill (7,1) rectangle +(1,1);
    \fill (6,3) rectangle +(1,1);
    \fill (7,2) rectangle +(1,1);
    \fill (8,1) rectangle +(1,1);
    \fill (8,2) rectangle +(1,1);
    \fill (9,1) rectangle +(1,1);
    \fill (7,3) rectangle +(1,1);
    \fill (7,4) rectangle +(1,1);
    \fill (8,3) rectangle +(1,1);
    \fill (10,1) rectangle +(1,1);
    \fill (9,2) rectangle +(1,1);
    \fill (11,1) rectangle +(1,1);
    \fill (8,4) rectangle +(1,1);
    \fill (9,3) rectangle +(1,1);
    \fill (10,2) rectangle +(1,1);
    \fill (11,2) rectangle +(1,1);
    \fill (12,1) rectangle +(1,1);
    \fill (10,3) rectangle +(1,1);
    \fill (9,4) rectangle +(1,1);
    \fill (13,1) rectangle +(1,1);
    \fill (10,4) rectangle +(1,1);
    \fill (11,3) rectangle +(1,1);
    \fill (9,5) rectangle +(1,1);
    \fill (12,2) rectangle +(1,1);
    \fill (10,5) rectangle +(1,1);
    \fill (14,1) rectangle +(1,1);
    \fill (12,3) rectangle +(1,1);
    \fill (13,2) rectangle +(1,1);
    \fill (11,4) rectangle +(1,1);
    \fill (12,4) rectangle +(1,1);
    \fill (15,1) rectangle +(1,1);
    \fill (13,3) rectangle +(1,1);
    \fill (11,5) rectangle +(1,1);
    \fill (14,2) rectangle +(1,1);
    \fill (13,4) rectangle +(1,1);
    \fill (14,3) rectangle +(1,1);
    \fill (12,5) rectangle +(1,1);
    \fill (15,2) rectangle +(1,1);
    \fill (16,1) rectangle +(1,1);
    \fill (17,1) rectangle +(1,1);
    \fill (14,4) rectangle +(1,1);
    \fill (16,2) rectangle +(1,1);
    \fill (15,3) rectangle +(1,1);
    \fill (13,5) rectangle +(1,1);
    \fill (18,1) rectangle +(1,1);
    \fill (15,4) rectangle +(1,1);
    \fill (16,3) rectangle +(1,1);
    \fill (17,2) rectangle +(1,1);
    \fill (14,5) rectangle +(1,1);
    \fill (15,5) rectangle +(1,1);
    \fill (19,1) rectangle +(1,1);
    \fill (17,3) rectangle +(1,1);
    \fill (18,2) rectangle +(1,1);
    \fill (16,4) rectangle +(1,1);
    \fill (20,1) rectangle +(1,1);
    \fill (18,3) rectangle +(1,1);
    \fill (16,5) rectangle +(1,1);
    \fill (19,2) rectangle +(1,1);
    \fill (17,4) rectangle +(1,1);
    \fill (18,4) rectangle +(1,1);
    \fill (19,3) rectangle +(1,1);
    \fill (20,2) rectangle +(1,1);
    \fill (21,2) rectangle +(1,1);
    \fill (20,3) rectangle +(1,1);
    \fill (21,3) rectangle +(1,1);
    \fill (21,4) rectangle +(1,1);
    \fill (21,5) rectangle +(1,1);
    \fill (22,4) rectangle +(1,1);
    \fill (23,4) rectangle +(1,1);
    \fill (21,6) rectangle +(1,1);
    \fill (22,5) rectangle +(1,1);
    \fill (23,5) rectangle +(1,1);
    \fill (22,6) rectangle +(1,1);
    \fill (23,6) rectangle +(1,1);
    \fill (24,5) rectangle +(1,1);
    \fill (23,7) rectangle +(1,1);
    \fill (24,6) rectangle +(1,1);
    \fill (25,6) rectangle +(1,1);
    \fill (24,7) rectangle +(1,1);
    \fill (25,7) rectangle +(1,1);
    \fill (26,6) rectangle +(1,1);
    \fill (26,7) rectangle +(1,1);
    \fill (27,6) rectangle +(1,1);
    \fill (26,8) rectangle +(1,1);
    \fill (27,7) rectangle +(1,1);
  \end{scope}
  \begin{scope}[fill=red,fill opacity=.3]
    \fill (20,5) rectangle +(1,1);
    \fill (23,4) rectangle +(1,1);
    \fill (12,7) rectangle +(1,1);
    \fill (22,5) rectangle +(1,1);
    \fill (14,7) rectangle +(1,1);
    \fill (27,4) rectangle +(1,1);
    \fill (28,3) rectangle +(1,1);
    \fill (5,7) rectangle +(1,1);
    \fill (19,6) rectangle +(1,1);
    \fill (17,6) rectangle +(1,1);
    \fill (9,5) rectangle +(1,1);
    \fill (10,6) rectangle +(1,1);
    \fill (8,6) rectangle +(1,1);
    \fill (11,5) rectangle +(1,1);
    \fill (16,4) rectangle +(1,1);
    \fill (24,5) rectangle +(1,1);
    \fill (25,4) rectangle +(1,1);
    \fill (13,5) rectangle +(1,1);
    \fill (15,5) rectangle +(1,1);
    \fill (26,5) rectangle +(1,1);
    \fill (18,4) rectangle +(1,1);
    \fill (7,7) rectangle +(1,1);
    \fill (20,4) rectangle +(1,1);
    \fill (21,3) rectangle +(1,1);
    \fill (6,8) rectangle +(1,1);
    \fill (22,4) rectangle +(1,1);
    \fill (21,6) rectangle +(1,1);
    \fill (23,3) rectangle +(1,1);
    \fill (14,6) rectangle +(1,1);
    \fill (12,6) rectangle +(1,1);
    \fill (23,6) rectangle +(1,1);
    \fill (5,6) rectangle +(1,1);
    \fill (4,8) rectangle +(1,1);
    \fill (17,5) rectangle +(1,1);
    \fill (27,3) rectangle +(1,1);
    \fill (9,7) rectangle +(1,1);
    \fill (8,5) rectangle +(1,1);
    \fill (19,5) rectangle +(1,1);
    \fill (11,7) rectangle +(1,1);
    \fill (10,5) rectangle +(1,1);
    \fill (24,4) rectangle +(1,1);
    \fill (25,3) rectangle +(1,1);
    \fill (13,7) rectangle +(1,1);
    \fill (26,4) rectangle +(1,1);
    \fill (15,7) rectangle +(1,1);
    \fill (16,6) rectangle +(1,1);
    \fill (6,7) rectangle +(1,1);
    \fill (20,6) rectangle +(1,1);
    \fill (7,6) rectangle +(1,1);
    \fill (21,5) rectangle +(1,1);
    \fill (18,6) rectangle +(1,1);
    \fill (20,3) rectangle +(1,1);
    \fill (22,6) rectangle +(1,1);
    \fill (12,5) rectangle +(1,1);
    \fill (23,5) rectangle +(1,1);
    \fill (4,7) rectangle +(1,1);
    \fill (14,5) rectangle +(1,1);
    \fill (3,8) rectangle +(1,1);
    \fill (19,4) rectangle +(1,1);
    \fill (5,8) rectangle +(1,1);
    \fill (8,7) rectangle +(1,1);
    \fill (9,6) rectangle +(1,1);
    \fill (10,7) rectangle +(1,1);
    \fill (24,6) rectangle +(1,1);
    \fill (11,6) rectangle +(1,1);
    \fill (25,5) rectangle +(1,1);
    \fill (15,6) rectangle +(1,1);
    \fill (13,6) rectangle +(1,1);
    \fill (26,3) rectangle +(1,1);
    \fill (29,2) rectangle +(1,1);
    \fill (18,5) rectangle +(1,1);
    \fill (21,4) rectangle +(1,1);  
  \end{scope}
\draw[thick,gray] (0,0) grid[step=1] (37,12);
\draw[line width=.06cm,blue] (5,3) ++(.14,.14) rectangle +(.72,.72);
\draw[line width=.06cm,blue] (29,8) ++(.14,.14) rectangle +(.72,.72);
\draw[line width=.06cm,red] (3,8) ++(.14,.14) rectangle +(.72,.72);
\draw[line width=.06cm,red] (27,3) ++(.14,.14) rectangle +(.72,.72);

\path (2.75,-2) node (x1) {$(X_1,s)_n$};
\draw (x1)--(5.5,3.5);
\path (27,13.25) node (y1) {$(Y_1,t-s)_n$};
\draw (y1)--(29.5,8.5);
\path (8.5,13.25) node (x2) {$(X_2,t-s)_n$};
\draw (x2)--(3.5,8.5);
\path (29.5,-2) node (y2) {$(Y_2,s)_n$};
\draw (y2)--(27.5,3.5);

\end{scope}